\documentclass[letterpaper, 11pt]{amsart}

\usepackage{amssymb}
\usepackage{amsmath}
\usepackage{amsfonts}
\usepackage{amscd}
\usepackage{amsthm}
\usepackage{color}
\usepackage{soul}
\usepackage{enumerate}
\usepackage{tikz}
\usetikzlibrary{matrix,arrows,calc,decorations.pathreplacing,decorations.pathmorphing,fadings}

\copyrightinfo{2014}{Dingyu Yang}

\newtheorem{theorem}{Theorem}[section]
\newtheorem{lemma}[theorem]{Lemma}
\newtheorem{proposition}[theorem]{Proposition}
\newtheorem{corollary}[theorem]{Corollary}

\theoremstyle{definition}
\newtheorem{definition}[theorem]{Definition}
\newtheorem{example}[theorem]{Example}
\newtheorem{cexample}[theorem]{Counterexample}

\theoremstyle{remark}
\newtheorem{remark}[theorem]{Remark}

\numberwithin{theorem}{section}

\newcommand{\R}{\mathbb{R}}

\usepackage{hyperref}

\begin{document}

\title[A choice-independent Kuranishi theory]{The polyfold--Kuranishi correspondence I:\\A choice-independent theory of Kuranishi structures}

\author[Dingyu Yang]{Dingyu Yang}
\address{Courant Institute of Mathematical Sciences, New York University, 251 Mercer Street, New York, N.Y. 10012-1185, United States}
\curraddr{Institut de Math\'{e}matiques de Jussieu - Paris Rive Gauche, Universit\'{e} Pierre et Marie Curie, UMR 7586, Bureau 15-25/501, 4 place Jussieu, Case 247, 75252 Paris cedex 05, France}
\email{dingyu\_yang@cantab.net}



\date{February 27, 2014}



\begin{abstract}\addcontentsline{toc}{section}{Abstract}
This is the first paper in a series which proposes and develops the \emph{polyfold Fredholm structure--Kuranishi structure correspondence}, identifying these two abstract perturbative structures which are indispensable for constructing and understanding symplectic invariants in the most general settings. In this paper, I present my version of the theory of Kuranishi structures in full generality. This theory is independent of all the choices made in the construction (including the choices of good coordinate systems); and it uses the \emph{equivalence of Kuranishi structures} as the germ to capture the intrinsic underlying structure to which the perturbation theory descends. This is the first theory in the literature that has these two properties. This choice-independent theory is essential for canonically and functorially identifying the polyfold Fredholm theory of Hofer-Wysocki-Zehnder with the theory of Kuranishi structures. The next two papers in this series will be on the forgetful functor and the globalization functor in the respective directions of the polyfold--Kuranishi correspondence, as well as illustrating the use of the correspondence with a few sample applications.
\end{abstract}

\maketitle
\newpage
\tableofcontents
\newpage

\listoffigures
\newpage

\section*{Introduction}
In symplectic geometry, a wide range of interesting problems are often reduced to counting the signed solutions of first order elliptic partial differential equations defined on varying domains (satisfying certain constraints to make them index 0), or more generally constructing invariants on the solution spaces of those PDEs. For example, the moduli spaces of stable pseudoholomorphic curves in symplectic manifolds are used to define Floer cohomologies, symplectic field theory and finite energy foliations and to prove many results regarding periodic orbits and the properties of the background symplectic manifolds. 

However, in such approaches, there is an intrinsic difficulty: the analysis is not compatible with the compactness and transversality. Bubbling off can occur in the limit of a sequence of holomorphic curves with an energy bound, and this phenomenon is not smooth in the usual sense and causes the issue of compactness. The symmetry of problems creates the issue of transversality, namely, moduli spaces cannot be made transverse using geometric perturbations. The interplay between the analysis and compactification is responsible for a wealth of information encoded in the moduli spaces. On the other hand, to extract invariants from the moduli spaces, we need to abstactly perturb the section that characterizes the moduli space as the zero set, so that the zeros are transversely cut out and then we can assign invariants to the resulting regular solution space of the perturbed section.

To abstractly perturb a space $X$, we first need to invent an abstract perturbative space with sufficient structure to allow a perturbation mechanism to perturb itself into a regular space (an invariant weighted branched orbifold), and then we must demonstrate that we can equip $X$ with such a structure, so that $X$ is an example of the abstract perturbative space.

In the literature of symplectic geometry, there are only two mainstream abstract perturbative spaces, polyfold Fredholm structures and Kuranishi structures. Those two structures originated from different contexts and are of completely different flavors and philosophies.

The theory of Kuranishi structures was invented by Fukaya-Ono in \cite{FO} and further improved by Fukaya-Oh-Ohta-Ono in \cite{FOOO} during their development of the (Lagrangian) Floer theory. They package the local models inspired by variations of complex structures and the gauge theory into a global geometric object. Basically one uses finite dimensional differential geometric techniques to look at local finite dimensional reductions of a Fredholm problem. Rather than solving $f=0$, one looks at the space $V$ solving $f\in L|_U$ for a local finite rank bundle $L|_U$ filling up the cokernel of $f$, and now the solution of $f$ locally is the solution of the section $s:=f|_{V}$ in a finite dimensional bundle $L|_{V}$. Due to the Fredholm property, different ways of local finite dimensional reductions of a Fredholm problem are related if $L|_U$ and $L'|_{U'}$ are related (for example, $L|_{U\cap U'}\overset{\text{subbundle}}{\subset} L'|_{U\cap U'}$), and one can try to patch those local data in different dimensions coherently into a global object, and then gradually modify it to have better properties, and try to perturb it into a regular space. Here the local model of this global geometric object is a familiar object in finite dimensional geometry, but to be useful in the context of symplectic geometric applications, the dimensions of local models are allowed to vary. Therefore, one has to consider coordinate changes between local models of different dimensions, and in this case, a complete understanding of how coordinate changes between local models behave in the global perturbation theory is rather involved. Moreover, there is no canonical choice of coherent system of local finite dimensional reductions. A natural question is whether we get the same answer using different coherent systems of local finite dimensional reductions and different choices made during the later modifications and constructions during perturbations?

A series of preprints and notes have addressed various counterexamples, issues and subtleties present in the Kuranishi structure theory; see \cite{FOOO}, \cite{KH}, \cite{DY}, \cite{MW}, \cite{FOOONEW}. In the latest two detailed preprints (circa Aug/Sep 2012), McDuff-Wehrheim \cite{MW} provides a construction of Kuranishi theory for the special case of Kuranishi structures with trivial stabilizers, while Fukaya-Oh-Ohta-Ono \cite{FOOONEW} provides more details and clarifications for the construction in their book \cite{FOOO}.

The current status of the Kuranishi structure theories in the literature (except the results announced in \cite{DY}) can be summarized as follows: (I) The various proposed theories all rely on a fixed choice of a ``good coordinate system'' (ordered finite cover) for a given Kuranishi structure once and for all, and their methods are inadequate to compare different choices. (II) Various theories all work on the level of atlas, not descending to a germ that is intrinsic to the underlying geometric object represented by one particular choice of atlas which so happens to be chosen. The right notion of \emph{germ} of Kuranishi structures is not defined (except announced in \cite{DY}). What I meant by a right notion of germ is that one could imagine that different choices of Kuranishi structures representing the same germ capturing the underlying object should behave the same in the perturbation theory; ideally on this germ level, the choice of a coherent global system of local finite dimensional reductions should become canonical; and all the constructions for the perturbation theory and comparisons of choices involved should stay in this germ. (III) An issue important to the Hausdorffness--the topological matching condition--has been overlooked.

In this paper, I establish a choice-independent theory of Kuranishi structures in full generality; this is the first theory in the literature that deals with (I) and (II) where I call the germ as the Kuranishi structure equivalence class, as well as addressing (III).

\begin{theorem}
\begin{enumerate}
\item If we define a Kuranishi structure \ref{KURANS} as in \cite{FOOO} together with two additional conditions, the maximality condition and the topological matching condition, we can extract an ordered finite cover called a good coordinate system \ref{EXISTGCS}; and moreover, we can extract a Hausdorff good coordinate system with a stronger form of coordinate changes \ref{LEVEL1CCHANGE}, which we call a level-1 good coordinate system \ref{LEVEL0GCS}, \ref{LEVELONEGCS}. 
\item A level-1 good coordinate system admits a compact invariant multisectional perturbation theory \ref{PERTN}.
\item There are notions of (level-1) chart-refinements and (concerted\footnote{A Kuranishi embedding is \emph{concerted} if the coordinate change directions after the Kuranishi embedding do not strictly change.} level-1) Kuranishi embeddings between (level-1) good coordinate systems \ref{CHARTREFINEMENT}, \ref{KEMB}, \ref{LEVEL1CRMENT}, \ref{LEVEL1KEMBED}. There always exists a concerted level-1 Kuranishi embedding for a given concerted Kuranishi embedding, after fixing a level-1 good coordinate system for the domain good coordinate system of the Kuranishi embedding \ref{EMBEDDING0TO1EXTENDINGLEVEL1GCS}. 
\item There is a combinatorial process called tripling \ref{TRIPLING} which provides a chart-refinement of a level-1 good coordinate system. Equipping the domain good coodinate system of a general Kuranish embedding with a level-1 structure and then applying the tripling to the resulting level-1 good coordinate system, we can achieve a concerted Kuranishi embedding chart-refining the original Kuranishi embedding \ref{CONCERTEDPROOF}. 
\item The fiber product exists for a pair of Kuranishi embeddings with a common domain good coordinate system \ref{FIBPROD}, as follows: Applying item (4) and (3) and one more tripling, we have a pair of concerted level-1 Kuranishi embeddings with a common level-1 good coordinate system such that the pair chart-refines the original one and is now admissible (the coordinate change directions from one target in the new pair to the other target via the domain do not strictly change). Applying the tripling to this picture again to get the latest level-1 pair level-1 chart-refining the intermediate pair, and then there is a mechanism to form fiber product Kuranishi charts such that the induced coordinate changes between the fiber product Kuranishi charts naturally exist and are compatible. The fiber product good coordinate system is naturally level-1, and completes a level-1 fiber product commutative square with the latest level-1 pair. 
\item There is a notion of equivalence class of Kuranishi structures which is a global germ and captures the intrinsic underlying abstract perturbative structure \ref{KSEQUIV}, \ref{KSEQUIVPF}, as well as a similar notion for (level-1) good coordinate systems (see \ref{GCSEQUIV}, the proof of \ref{GCSEQUIVPF}). 
\item All the different choices made in extracting a level-1 good coordinate system from a fixed Kuranishi structure lead to equivalent level-1 good coordinate systems \ref{INDEPGCSYSTEM}, \ref{INDEPLEVEL1}. The level-1 good coordinate systems for equivalent Kuranishi structures are equivalent \ref{GCSEMBEDDINGFORKS}, \ref{EMBEDDING0TO1EXTENDINGLEVEL1GCS}. 
\item Equivalent level-1 good coordinate systems admit the same perturbation theory \ref{COMPAREPERTURBATION}. We have a well-defined perturbation theory for each Kuranishi structure equivalence class---the intrinsic underlying object---independent of all the choices made (with two different sets of full choices comparable in a common level-1 refinement) \ref{FULLINDEP}.
\end{enumerate}
\end{theorem}
On the other hand, the polyfold approach is to work with compatified space as a whole in infinite dimension, and look at the analysis involved and find a new smooothness concept that describes the compactified space smoothly and allows some structure to make the solution space generic, which on the formal level looks like the usual Sards-Smale theory in the usual smoothness situation. Polyfold Fredholm theory was invented by Hofer, Wysocki and Zehnder in a series of papers to develop the foundation for symplectic field theory, \cite{PolyfoldI}, \cite{PolyfoldII}, \cite{PolyfoldIII}, \cite{Polyfoldinteg}, \cite{Polyfoldanalysis}, \cite{PolyfoldGW}, and survey \cite{Clay}. This theory bases on the formalism in the usual differential geometry and we have invertible coordinate changes, but the local models are of the form $r(U)$ for an sc-smooth retraction $r: U\to U$. This involves a new notion of smoothness of maps, where the dimension of a connected component of $r(U)$ can vary, but it coincides with the usual smoothness in finite dimensions (when $U$ is finite dimensional). As a result, the domain reparametrization and bubbling-off/gluing are sc-smooth. The weaker notion of smoothness means that the implicit function theorem only holds for sections satisfying specific local forms, called \emph{sc-Fredholm sections}. A polyfold Fredholm structure is essentially an sc-Fredholm section $f:B\to E$ for which the orbit space of the zero set is compact and where restrtciting the base $B$ to an invariant open subset around the zero set is allowed.

The main picture of this series of papers is the discovery of a way of identifying the two kinds of aforementioned abstract perturbative spaces---Kuranishi structures and polyfold Fredholm structures---and their respective perturbation theories, which I call the \emph{polyfold Fredholm structure--Kuranishi structure correspondence}, or polyfold--Kuranishi correspondence for short. Thus, the only two mainstream perturbation theories in the literature are unified, and transferring results between these two theories can be very fruitful. In particular, signed solution counts (or more generally invariants of solution spaces) using these two markedly different frameworks give the same results.

\section*{Acknowledgements}

This paper is based on a part of the author's Ph.D. thesis at Courant Institute, New York University. The author is extremely grateful and indebted to Professor Helmut Hofer for his suggestion of this exciting project (he envisioned that a forgetful functor from polyfold theory to Kuranishi structures should exist) and his guidance, discussion, support and encouragement throughout author's graduate study. The author is also deeply thankful to his wife, Xiaoxue Yu, for her presence and support throughout author's Ph.D.
\newpage

\specialsection*{A PHILOSOPHICAL INTRODUCTION TO THIS CHOICE-INDEPENDENT THEORY OF KURANISHI STRUCTURES}

The main goal of part I of this paper series titled \emph{The polyfold--Kuranishi correspondence} is to provide a choice-independent theory of Kuranishi structures, as it is crucial for understanding the \emph{polyfold Fredholm structure--Kuranishi structure correspondence}, or the \emph{polyfold--Kuranishi correspondence} for short. Half of the correspondence is based on the forgetful construction starting from polyfold Fredholm structures and producing Kuranishi structures. The ambiguity in the outputs arises from different choices made during the construction and can be formulated as examples of the \emph{Kuranishi structure equivalence}. A Kuranishi structure equivalence class can also be viewed as a global germ for Kuranishi structures, as is evident from self-contained considerations on Kuranishi structures alone. The polyfold--Kuranishi correspondence is only useful if it preserves the perturbation theories. Therefore, we want a perturbation theory of Kuranishi structures that works for the outputs of the forgetful construction and also factors through Kuranishi structure equivalence. The outputs of the forgetful construction have additional structures that might be convenient, but we insist on constructing a perturbation theory that is applicable to the Kuranishi structures defined in a minimal way and arised from general contexts. 

In the first half of this paper, we will introduce the definition of Kuranishi structures in our theory, where we essentially added two more conditions to the definition given in \cite{FOOO}. The first condition is the \emph{maximality} condition, which roughly corresponds to the condition in \cite{KH}, \cite{FOOONEW} and the condition in \cite{MW}.  The second condition is the \emph{topological matching} condition which is new but is absolutely crucial for achieving Hausdorffness necessary for later constructions. To capture the pointwise-formulated information in the definition of a Kuranishi structure in a manageable way, we extract an ordered finite cover, called \emph{a good coordinate system}. We then further shrink each chart in this finite cover to (i) establish that whenever two charts overlap in the \emph{identification space} obtained by gluing charts together, there exists a coordinate change in the direction of increasing order, and (ii) obtain Hausdorffness of the identification space in the \emph{relative topology}. We call such a good coordinate system a \emph{Hausdorff good coordinate system}. We then introduce the new notion of \emph{refinement} and \emph{chart-refinement} in the contexts of Kuranishi structures and Hausdorff good coordinate systems. Roughly speaking, if we shrink charts of an object $A$ to form an object $B$ and coherently embed charts of $B$ into charts of possibly higher dimensions of an object $C$, then $C$ is said to be a refinement of $A$, $B$ is said to be a chart-refinement, and embedding from $B$ to $C$ is called a \emph{Kuranishi embedding}. Here objects can be simultaneously Kuranishi structures or Hausdorff good coordinate systems. Two objects are said to be equivalent if they admit a common refinement. In particular, two Kuranishi structures are Kuranishi structure equivalent, or \emph{equivalent} for short, if they admit a common Kuranishi structure refinement. This defines the aforementioned Kuranishi structure equivalence. There is also a notion of a \emph{chart-refinement of a Kuranishi embedding map between good coordinate systems}.

A \emph{choice-independent theory} of Kuranishi structures is a theory that is independent of all choices made in the construction of the theory, whereas in the literature all the existing theories depend on a fixed choice of a good coordinate system throughout and the dependence on different such choices, which inevitably involves the notion of \emph{refining a good coordinate system}, is not examined. We also want the perturbation theory to factor through Kuranishi structure equivalence, for which the existing perturbation theories of Kuranishi structures are inadequate. This underlines the necessity to construct such a choice-independent theory in the discussion of the polyfold-Kuranishi correspondence.

In the second half of this paper, we will show that the Kuranishi structure equivalence is an equivalence relation, and construct a perturbation theory that is independent of all the choices made during the construction and that factors through the Kuranishi structure equivalence. The key structure to achieve both goals is a new concept called a \emph{level-1 good coordinate system}, where coordinate changes have extra compatible structures for extending perturbations across coordinate changes and for forming fiber products. An important and nontrivial result is that we can always chart-refine and equip extra structures to a Kuranishi embedding between Hausdorff good coordinate systems to obtain a level-1 Kuranishi embedding of level-1 good coordinate systems. We explain a combinatorial process called \emph{tripling}, which is a functorial process needed (i) to show the independence of choices of level-1 structures and (ii) for forming a fiber product of a suitable pair of level-1 Kuranishi embeddings mapping from the same level-1 good coordinate system. The latter in turn is used to show the independence of choices of Hausdorff good coordinate systems and that Kuranishi structure equivalence and equvialence of Hausdorff good coordinate systems are equivalence relations.

We emphasize that this choice-independent theory of Kuranishi structures, especially the step of turning everything into level-1, cannot be drastically simplied even if we restrict to special case in which the Kuranishi structures arise from the output of the forgetful constructions. This is because, for such a Kuranishi structure, even though a level-1 structure is almost compatibly given from the forgetful construction, there is still an important piece of key data not naturally available, and to compatibly construct it for all the coordinate changes is a non-trivial task and there is no shortcut.

The definition and theory of Kuranishi structures were proposed by Fukaya and Ono in \cite{FO} to describe and ultimately perturb compactified moduli spaces in symplectic geometry to obtain invariants for applications. The definition was modified in \cite{FOOO}, most notably by not using germs for charts and coordinate changes and by including a strengthened tangent bundle condition. There are several preprints that discuss and clarify various other subtleties concerning Kuranishi structures (\cite{KH}, \cite{DY}, \cite{MW}, \cite{FOOONEW} and \cite{Joyce}), and my initial understanding of ideas around Kuranishi structures is from studying \cite{FO}, \cite{FOOO} and \cite{KH}; and an outline of my global germ approach to the theory of Kuranishi structures and the definition of Kuranishi structure equivalence are given in \cite{DY}, from which this paper builds and develops into current form with all the details and proofs included (the most important changes are that the metric quotient topology, which would forbid any dimension jump across coordinate changes, has been replaced by a new topological matching condition leading to a Hausdorff relative topology, and the form of the level-1 structure is more refined in this paper).

My version of the theory of Kuranishi structures described in this paper defines Kuranishi structures (before discussing new ideas of equivalence classes via common refinement) as a minimal modification of the form of the definition of Kuranishi structures in \cite{FOOO} for our purpose. This is because (i) the definition in \cite{FOOO} is the one to which I originally produced my forgetful construction before discussing the independence of choice, (ii) theory in \cite{MW} with nontrivial stablizers is not available at the time of writing, (iii) theories of Kuranishi structures in the literature all rely on a fixed choice of a good coordinate system with the dependence of such choice unexamined, so even if we implant the notion of refinement in our theory into theories of \cite{MW} and \cite{FOOONEW}, it is not obvious that perturbation mechanisms there will be compatible with refinement in order to factorize through Kuranishi structure equivalence, and the methods of trying to establish compatibility of perturbation with refinement are highly nontrivial and not available in the literature until this paper and they are similar in efforts to the current work, plus the extra significant efforts of learning their new theories in details if we were to base our definition prior to taking equivalence on the new definitions/conditions of \cite{MW} and \cite{FOOONEW} and to use their perturbation theories.

All the main ideas, constructions and proofs at each stage of our theory of Kuranish structures are new. I will also discuss how our theory addresses various known subtleties one might encounter in formulating such a theory.

In this paper for brevity, I will use \emph{ep-groupoids}\footnote{Here an ep-groupoid is a Lie groupoid where the source map $s$ and target map $t$ are local diffeomorphisms, and every point $z$ in the object space $M$ has an open neighborhood $W_z$ in $M$ such that $t:s^{-1}(\overline{W_z})\to M$ is proper. See \cite{PolyfoldIII}.} to mean \'{e}tale-proper\footnote{Here \'{e}tale means that the source map $s$ and target map $t$ are local diffeomorphisms, and proper means that $(s,t)$ is a proper map.} Lie groupoids in the usual sense, since both definitions agree in the finite dimensional setting in this paper and ep-groupoids are more adapted to infinite dimensions. I will use the term \emph{vector bundle ep-groupoids}, where the object spaces and the morphism spaces are vector bundles, the structure maps preserve the linear structures on the fibers and the orbit space of course does not have the vector bundle structure, to distinguish them from general fiber bundle ep-groupoids. I use $\subset$ to mean \emph{subset of or equal to}.

\specialsection*{KURANISHI STRUCTURES, GOOD COORDINATE SYSTEMS AND REFINEMENTS\label{COMPLETETHEORY}}

\section{Kuranishi structures}

Roughly speaking, a Kuranishi structure is a global object formed by coherently patching together local models cut out by sections in finite dimensional bundles with finite group actions, where the dimensions of the local bundles can vary. Local models varying in dimension from point to point will include many interesting examples incapable of being described by traditional differential geometry, with the most important examples being moduli spaces in symplectic geometry; and the structure provided by a Kuranishi structure is sufficient to have a well-defined perturbation theory perturbing examples into compact regular spaces to which the invariants can then be assigned.

We now describe the precise form of a local model, also called a \emph{Kuranishi chart}, or simply a \emph{chart}. We always use the underline notation in this paper to denote the quotient of group action on objects and maps.

\begin{definition} An \emph{$n$-manifold} $V$ \emph{with cornes of the corner index at most $m$} is a manifold based on local models which are relative open sets in $[0,\infty)^m\times \R^{n-m}\subset \R^n$. A point $x\in V$ is said to be of the corner index $k$ if under some local chart, it maps to a coordinate with $k$ zeros.
\end{definition}

\begin{definition}[\cite{FOOO} Kuranishi chart] Let $X$ be a topological space and let $p\in X$ be given. Let $G_p$ be a finite group acting smoothly and effectively on a finite dimensional vector bundle $E_p\to V_p$ over a manifold $V_p$ with corners, equipped with a smooth equivariant section $s_p: V_p\to E_p$, and let $\psi_p: \underline{s_p^{-1}(0)}\to X$ be a homeomorphism onto an open neighborhood $X_p$ of $p$ in $X$. 
Then $$(G_p, s_p:V_p\to E_p, \psi_p: \underline{s_p^{-1}(0)}\to X_p),\;\;\text{or}\;\; C_p\;\; \text{for short},$$ is said to be a \emph{Kuranishi chart} of $X$ based at $p$ (or centered at $p$). $X_p$ is called the \emph{coverage} on $X$ by the chart $C_p$, or simply coverage.
\end{definition}

\begin{remark}
\begin{enumerate}
\item In the above definition, we chose the setting to be a possibly nontrivial vector bundle $(E_p\to V_p, G_p)$ with a global quotient orbifold action, by considering locality, a grouping convention and notational simplicity. Since a Kuranishi structure will be pointwisely formulated using the local model described by a Kuranishi chart, only the behaviour of $s_p$ around a point $x\in V_p$ with $\psi_p(\underline{x})=p$ matters and this can be captured in the simplest local form of the local model used. Later in getting an ordered finite cover with nice properties from a Kuranishi structure, we allow grouping of charts into a chart of a slightly general form; and we choose to use this convention because it is intuitive and notationally simple for the starting basic definition, and think it benefits more than insisting local models in the much later modification stay in this same more complicated form from the very beginning.

\item We allow the bundle $E_p$ to be possibly non-trivial rather than always being a trivial bundle as in FOOO \cite{FOOO}, mostly as it is simpler to write than $V_p\times F_p$. As explained in item (1), this makes little difference as the behavior of $s_p$ around a point $x$ corresponding to $p$ is captured in the bundle-trivializing neighborhood $V_p'$ of $x$ in $V_p$ anyway.

\item $G_p$ is required to be effective, which is required in a possible triangulation of the zero set of compact transverse multisectional perturbation of $s_p, p\in X$, which is not covered in this paper. For the content in this paper, effectiveness is not necessary. For the polyfold-Kuranishi correspondence, we of course need the effectiveness to have the orbifold--ep-group correspondence.
\item The more important point concerns the description of orbifolds. We focus on the base here, as the bundle case is the same. We can choose either 
\begin{enumerate}[(a)]
\item a local orbifold representative (namely, a global quotient $(G_p, V_p)$) as we have done,
\item a fixed choice of finitely many local orbifold representatives of the same dimension $$(G_y, V_y, \varphi_y:\underline{V_y}\to X_p), y\in S_p,\text{ $S_p$ finite index set, }\bigcup_{y\in S_p}\varphi_y(\underline{V_y})=X_p$$ such that for any $y, y'\in S_p$ and any point $z\in \varphi_y(\underline{V_y})\cap\varphi_{y'}(\underline{V_{y'}})$, there exists $y''\in S_p$ such that $z\in \varphi_{y''}(\underline{V_{y''}})$ and there exist equivariant embeddings\footnote{By definition, an equivariant embedding induces an isomorphism between the stabilizer groups of every point in the domain and of its image.} from $(G_{y''}, V_{y''}, \varphi_{y''})$ to $(G_y, V_y, \varphi_y)$ and to $(G_{y'}, V_{y'}, \varphi_{p'})$ respectively (in particular, $\varphi_{y''}(\underline{V_{y''}})\subset\varphi_y(\underline{V_y})\cap\varphi_{y'}(\underline{V_{y'}})$), 
\item an orbifold representative $(V_p, \mathbb{V}_p, \psi:\underline{V_p}\to X_p)$ expressed as an \'etale-proper Lie groupoid, or
\item an orbifold which is an equivalence class of representatives.
\end{enumerate}
The option (d) will add an extra layer of structure, which is not a real complication, as a map between orbifolds is a generalized map. We cannot, however, use an orbifold germ at a point in the local model to describe a germ of $x$ in $X$, because this will create a problem when describing the composition of coordinate changes later on.
\item Because of observation (4), we can only use (a), (b), (c), or (d) to define Kuranishi charts. The resulting definition of Kuranishi structures, to be introduced below, is just a way of describing a coordinate representative of the intrinsic structure underlying the Kuranishi structure, and we need a notion of a germ to describe the intrinsic structure. We will later achieve the germ notion in Kuranishi structure equivalence classes in a global manner, with all the Kuranishi charts regarded as a whole. This is the first time such a germ notion for Kuranishi structures is given in the literature.
\item As alluded to in item (1), a local model in the ordereed finite cover with nice properties later can be described using either a special case of (4)(b) (where always $\varphi_{y''}(\underline{V_{y''}})=\varphi_y(\underline{V_y})\cap\varphi_{y'}(\underline{V_{y'}}))$ or (4)(c). We will adopt (4)(c); but using (4)(b) will be equally fine, so effective group action is not needed here, agreeing with the point in (3). 
\end{enumerate}
\end{remark}

We will often \emph{restrict a Kuranishi chart} in the following way, which we will also call \emph{shrinking a chart}. During this process, the coverage on $X$ by the restricted chart is important. Hence we also introduce useful notations for the restricted or shrunken charts and its coverage on $X$:

\begin{definition}[restricting/shrinking charts] Let $U$ be a $G_p$-invariant open subset of $V_p$ and define the restricted coverage $X_p|_U:=\psi_p(\underline{(s_p|_U)^{-1}(0)})$ and $$\psi_p|_U:=\psi_p|_{\underline{(s_p|_U)^{-1}(0)}}:\underline{(s_p|_U)^{-1}(0)}\to X_p|_U.$$ If $p\in X_p|_U$ (that is, the base point is still covered), we define $$C_p|_{U}:=(G_p, s_p|_{U}, \psi_p|_{U}),$$ which we often refer to as \emph{restricting} or \emph{shrinking} a chart to $U$.
\end{definition}

One key difference between a Kuranishi structure and \emph{the zero set of a global section in an orbifold bundle} is that the we do not require the dimensions of $E_p$ in local models in a Kuranishi structure to be constant. This is necessary because examples in applications originate from local finite dimensional reductions which use tranversality and are local in nature, and such local constructions which are further apart are not related to each other; moreover, this construction does not involve consideration of the dimension, so it is not clear how to impose a uniform dimension a priori.

The dimension is an invariant preserved by diffeomorphisms, so in order to define a global object with local models being Kuranishi charts with $E_p$ in different dimensions, we cannot hope to have the usual diffeomorphic coordinate changes from a Kuranishi chart to another. In this case, how are the local models coherently patched together into a global object? Suppose we have $p_1, p_2\in X$ and that $X_{p_1}\cap X_{p_2}\not=\emptyset$, that is, the coverages on $X$ by Kuranishi charts at $p_1$ and $p_2$ overlap. Letting $q$ be a point in the overlap $X_{p_1}\cap X_{p_2}$, we require that a restricted version, depending on $p_i$ and still covering $q$, of the Kuranishi chart at $q$ equivariantly embeds into the Kuranishi charts at $p_i$, and such an equivariant embedding is an equivariant bundle embedding intertwining the sections, and is called a \emph{pre-coordinate change} from the Kuranishi chart at $q$ to the Kuranish ichart at $p_i$, for $i=1, 2$. Since the equivariant embedding exists but is not required to be an open map, we have $\text{dim} E_q\leq \text{dim} E_{p_i}$ but $\text{dim} E_q< \text{dim} E_{p_i}$ is also allowed. In a sense, local models at points in the overlap $X_{p_1}\cap X_{p_2}$ are local submodels that mediate between $C_{p_1}$ and $C_{p_2}$.

We now fix some notations and then introduce the precise form of a pre-coordinate change.

\begin{definition} 
\begin{enumerate}
\item (quotient of a map) Let $f: (A, G_A)\to (B, G_B)$ be an equivariant map from $A$ with the group action $G_A$ to $B$ with the group action $G_B$. We use $\underline{f}$ to denote the quotient map $\underline{f}: \underline{A}\to\underline{B}$.  
\item ($G_A-G_B$-embedding) Let $f: (A, G_A)\to (B, G_B)$ be a smooth equivariant map between manifolds with smooth global group actions with an injective induced group morphism $f_\ast: G_A\to G_B$. $f$ is said to be a \emph{$G_A-G_B$-embedding} if $f$ is an embedding, $\underline{f}:\underline{A}\to \underline{B}$ is injective, and for every $x\in A$, $f$ induces an isomorphic stablizer-mapping from $(G_A)_x$ to $(G_B)_{f(x)}$.
\end{enumerate}
\end{definition}

\begin{remark} Observe that if $f(A)\cap g(f(A))\not=\emptyset$ for some $g\in G_B$, then $g\in f_\ast(G_A)$, and hence $f(A)=g(f(A))$. Indeed, suppose not, we have $f(x_1)=g(f(x_2))$ for some $x_1, x_2\in A$, but $g\not\in f_\ast(G_A)$, then $\underline{x_1}\not=\underline{x_2}$\footnote{If not, then $x_1=h(x_2)$ for some $h\in G_A$, then $f(h(x_2))=g(f(x_2))$, so $(f_\ast h)( f(x_2))=g (f(x_2))$. So $(f_\ast h)\cdot g^{-1}\in Stab(f(x_2))=f_\ast (Stab(x_2))\subset f_\ast (G_A)$, so $f_\ast (h)\cdot g^{-1}=f_\ast(h')$ for some $h'\in G_A$, so $g=f_\ast(h'^{-1}\cdot h)\in f_\ast (G_A)$, a contradiction.} and $\underline{f}(\underline{x_1})=\underline{f}(\underline{x_2})$, $\underline{f}$ is not injective, a contradiction.
\end{remark}

The precise form of a pre-coordinate change from the local model centered at $q\in X_p$ to the local model at $p$ is the following:

\begin{definition}[\cite{FOOO} pre-coordinate change] Let $C_q$ and $C_p$ be two Kuranishi charts where $q\in X_p$, and let $V_{pq}$ be a $G_q$-invariant open subset of $V_q$ such that $q\in X_q|_{V_{pq}}$. Let $\hat\phi_{pq}: E_q|_{V_{pq}}\to E_p$ be a $G_q-G_p$ bundle embedding that covers $\phi_{pq}:V_{pq}\to V_p$ such that $$\hat\phi_{pq}\circ (s_q|_{V_{pq}})=s_p\circ\phi_{pq}\;\;\text{and}\;\; \psi_p\circ\underline{(\phi_{pq}|_{(s_q|_{V_{pq}})^{-1}(0)})}=\psi_q|_{\underline{(s_q|_{V_{pq}})^{-1}(0)}}.$$ A \emph{pre-coordinate change} from $C_q$ to $C_p$ consists of the tuple $(\phi_{pq}, \hat\phi_{pq}, V_{pq})$, and $V_{pq}$ is called the \emph{domain of the pre-coordinate change}.
\end{definition}

\begin{remark}
Let $q\in X_p$, a pre-coordinate change from $C_q$ to $C_p$ is just a $G_q-G_p$-embedding from $C_q|_{V_{pq}}$ to $C_p$.
\end{remark}

The dimension of local models varying from point to point gives us the flexibility to account for some interesting examples, but we also need to control the extra dimensions when the dimension of local models increases across a pre-coordinate change. Since we will only be interested in compatible perturbations of those sections $s_p$, a dimension jump going from $C_q$ to $C_p$ for $q\in X_p$ does not change the behavior of the zero sets of the sections under perturbations, if the extra dimensions of the base and the fibers in $C_p$ transverse to directions from $C_q$ can be identified naturally, which is exactly what the tangent bundle condition specifies.

In order to define the tangent bundle condition, we recall a notion called \emph{normal linearization}:

Given a submanifold $U$ in $W$, the normal bundle $N_UW$ is the quotient bundle $TW|_U/TU$ defined by the fiberwise quotient. By the definition of the pre-coordinate change, we have $s_p|_{\phi_{pq}(V_{pq})}: \phi_{pq}(V_{pq})\to \hat\phi_{pq}(E_q|_{V_{pq}})$ and $\hat\phi_{pq}(E_q|_{V_{pq}})$ is a subbundle of $E_p|_{\phi_{pq}(V_{pq})}$. For any $x\in \phi_{pq}(V_{pq})$, dropping the dependence of $x$ in the notation, we can choose an open neighborhood $U$ of $x$ in $\phi_{pq}(V_{pq})$ such that we have an embedded open tubular neighborhood $W$ of $U$ in $V_p$, which is the image of an open neighborhhood of the zero section of $N_UW$ under the normal exponential map using a metric on $V_p$. We smoothly extend $(\hat\phi_{pq}(E_q|_{V_{pq}}))|_U$ to a bundle $\tilde E$ over $W$ and define the quotient bundle $Q:=E_p|_{W}/\tilde E$ over $W$. (Here for the definition of the normal linearization, the group action is irrelevant. The construction is natural and can be made invariant anyway.) Fixing a bundle connection choice $\nabla$ on $Q$, we denote $$\nabla^N (s_p|_W):=\nabla (s_p|_W/\tilde E): TW\to Q,$$ and the restriction $(\nabla^N(s_p|_W))|_{TW|_{U}}$ descends to $$d^N s_p: N_UW\to Q|_{U},$$ as $\nabla^N(s_p|_{U}: U\to Q|_U)=0$. Note that $N_UW=(N_{\phi_{pq}(V_{pq})} V_p)|_U$ and $Q|_U=(E_p|_{\phi_{pq}(V_{pq})}/\hat\phi_{pq}(E_q|_{V_{pq}}))|_U$. Then we recall a couple of general simple facts in differential geometry: (1) the connection $\nabla s$ at a zero $x$ of $s$ does not depend on $\nabla$ (indeed, $\tilde\nabla_x s=\nabla_x s+ A (s(x))=\nabla_x s$ for some bundle endomorphism-valued 1-form $A$), and (2) the linearization of $d_x s$ in a local coordinate around $x$ can be thought of as $\nabla_x s$ for some $\nabla$. Lastly, we remark that $(d^Ns_p)_x$ does not depend on the chosen extension $\tilde E$, because once we know that the normal linearization does not depend on the connection (or local coordinate we choose) over zeros, we observe that the condition of the normal linearization at zero $x$ is a pointwise condition for $d_x(s_p)$, where $d_x(s_p)$ is defined using a local coordinate. So $(d^N s_p)_x: (N_{\phi_{pq}(V_{pq})} V_p)_x\to (E_p|_{\phi_{pq}(V_{pq})}/\hat\phi_{pq}(E_q|_{V_{pq}}))_x$ is an intrinsic notion for zero $x\in s_p^{-1}(0)\cap \phi_{pq}(V_{pq})$, independent of all the choices made. Later, we will see by shrinking charts to $C_q|_{U'_q}$ and $C_p|_{U'_p}$ and denoting the restricted domain of the pre-coordinate change by $U'_{pq}$, we can choose $U$ to be the entire $\phi_{pq}(U'_{pq})$ and construct an open embedded invariant tubular neighborhood $W$ of $U=\phi_{pq}(U'_{pq})$ in $U'_p$.

\begin{definition}[\cite{FOOO} coordinate change, tangent bundle condition]\label{level0} Let $q\in X_p$. A pre-coordinate change from $C_q$ to $C_p$ is said to satisfy the \emph{tangent bundle condition}, if the normal linearization $$d^{N,q}s_p: N_{\phi_{pq}(V_{pq})}V_p\to E_p|_{\phi_{pq}(V_{pq})}/\hat\phi_{pq}(E_q|_{V_{pq}}),$$ which is well-defined over $s_p^{-1}(0)\cap \phi_{pq}(V_{pq})$, is a fiberwise isomorphism over $s_p^{-1}(0)\cap \phi_{pq}(V_{pq})$. A pre-coordinate change $(\phi_{pq}, \hat\phi_{pq}, V_{pq})$ from $C_q$ to $C_p$ that satisfies the tangent bundle condition is called a \emph{coordinate change}. A coordinate change from $C_q$ to $C_p$ is abbreviated as $C_q\to C_p$. We denote $V_{pq}:=\text{dom}(C_q\to C_p)$ and call it the \emph{domain of the coordinate change}.
\end{definition}

There might be more than one way to go from one chart to another, so we require the following compatibility condition.

Before discussing the compatibility condition, we notice that $G_p$ acts on $V_p$ and $E_p$, so
$$\begin{CD}
E_p @> g >> E_p\\
@A s_p AA @ A s_p AA\\
V_p @> g >>V_p
\end{CD}.$$
Therefore, $G_p$ induces automorphisms of $C_p$, namely, $$g\in G_p\mapsto (g: (G_p, s_p, \psi_p)\to (G_p, s_p, \psi_p)).$$
Here we do not need the action structure of $G_p$.

\begin{definition}[compatibility]\label{COMPATIBLE} Let $q\in X_p$, $r\in X_q\cap X_p$, and $(\phi_{pq}, \hat\phi_{pq}, V_{pq})$, $(\phi_{qr}, \hat\phi_{qr}, V_{qr})$, and  $(\phi_{pr}, \hat\phi_{pr}, V_{pr})$ be three coordinate changes. 

$(\phi_{pq}, \hat\phi_{pq}, V_{pq})$ and $(\phi_{qr}, \hat\phi_{qr}, V_{qr})$ can be composed over $V_{qr}\cap (\phi_{qr})^{-1}(V_{pq})$, which can be the empty set. Those three coordinate changes are said to be \emph{compatible} if over the common domain of the definitions, $$V_{pqr}:=V_{pr}\cap V_{qr}\cap (\phi_{qr})^{-1}(V_{pq}),$$ the composed coordinate change $\hat\phi_{pq}\circ\hat\phi_{qr}$ equals the direct coordinate change $\hat\phi_{pr}$ up to the $G_p$-action on $C_p$. A collection of coordinate changes are said to be \emph{compatible} if any three such coordinate changes among three charts in the collection are compatible.
\end{definition}

Having introduced all the ingredients, we can now define:

\begin{definition}[Kuranishi structure, maximality condition, topological\\matching condition]\label{KURANS} Let $X$ be a compact metrizable topological space. A \emph{Kuranishi structure} on $X$ is a tuple $$(X, \{C_p\}_{p\in X}, \{C_q\to C_p\}_{q\in X_p})$$ consisting of
\begin{enumerate}[(a)]
\item for all $p\in X$, a Kuranishi chart $C_p$, and 
\item for all $q\in X_p$, a coordinate change $C_q\to C_p$,
\end{enumerate}
such that 
\begin{enumerate}[(1)]
\item all  the coordinate changes are compatible,
\item the \emph{maximality condition} holds: for all $q\in X_p$, if $z\in V_q$, $z'\in V_p$, and $\underline{z}$ and $\underline{z'}$ are connected via a sequence of quotiented coordinate changes each of which moves in either direction, denoted by $\underline{z}\sim \underline{z'}$, then $z\in V_{pq}$ and $\underline{\phi_{pq}(z)}=\underline{z'}$, and
\item the \emph{topological matching} condition holds: for any two charts $V_p$ and $V_{p'}$ where $p, p'\in X$ and any pair of $v\in \underline{s_p^{-1}(0)}$ and $v'\in \underline{s_{p'}^{-1}(0)}$ such that there exist a sequence $v_n\in \underline{V_{p}}$ and a sequence $v'_n\in \underline{V_{p'}}$ satisfying $v_n\sim v_n'$, $v_n\to v$ in $\underline{V_p}$ and $v'_n\to v'$ in $\underline{V_{p'}}$, we have $v\sim v'$.
\end{enumerate}
\end{definition}

\begin{remark}
\begin{enumerate}
\item We observe that the above definition is the same as the version of Fukaya-Oh-Ohta-Ono in \cite{FOOO} together with my addition of the maximality condition and the topological matching condition. 
\item We need to pick $V_{pq}$ in specifying the data for the Kuranishi structure, and the domains $V_{pq}, q\in X_p, p\in X$ of the coordinate changes are then required to satisfy the maximality condition specified in my definition.
\item The maximality condition seems related to the condition in \cite{KH} and \cite{FOOONEW} and the tameness condition in \cite{MW}.
\item The topological matching condition is completely new in the literature, and it says that if two zeros can be approximated by the same sequence up to the identification $\sim$, then those two zeros are identified in $X$. This condition is essential in ruling out some pathological examples and for achieving Hausdorffness. Intuitively, a Kuranishi structure provides a coherent neighborhood of $X$, so that when two points $x$ and $y$ in $X$ can be approximated by the same sequence in this coherent neighborhood, if $x\not=y$, as indicated by the topology from the charts, $y$ is in the closure of $x$ and vice versa, but the topology of $X$ is Hausdorff and so $x$ cannot be in the closure of $y$ using the topology of $X$, and vice versa. Topology matching says that $x$ has to be $y$.
\item The definition of Kuranishi structure says that if $q\in X_p$, then there must exist a coordinate change $C_q\to C_p$; on the other hand, the coordinate changes are indexed by $(q,p)\in X\times X$ satisfying $q\in X_p$ in the definition. Therefore, the coordinate change $C_q\to C_p$ exists in a Kuranishi structure if and only if $q\in X_p$.
\item If the maximality condition holds, then whenever $q\in X_p$ and $r\in X_q\cap X_p$, $V_{pqr}$ is nonempty and is actually an open neighborhood in $V_r$ around a point mapping to $r$ in $X$. Indeed, let $x_r$, $x_q$, and $x_p$ be the points in $V_r$, $V_q$, and $V_p$ whose quotients are mapped to $r$ through $\psi_r$, $\psi_q$, and $\psi_p$ respectively. Since $\underline{x_q}\overset{\underline{\phi_{qr}}}{\leftarrow}\underline{x_r}\overset{\underline{\phi_{pr}}}{\to}\underline{x_p}$, $\underline{x_q}\sim\underline{x_p}$, by maximality we have $x_q\in V_{pq}$. Since $V_{pq}$ is $G_q$-invariant, $G_q(x_q)\subset V_{pq}$; and since and $\phi_{qr}$ is $G_r-G_q$-invariant, we have $x_r\in \phi_{qr}^{-1}(G_q (x_q))\subset(\phi_{qr})^{-1}(V_{pq})$. So $(\phi_{qr})^{-1}(V_{pq})$ is an open neighborhood of $x_r$. As $V_{qr}$ and $V_{pr}$ are also open neighborhood around $x_r$, $V_{pqr}=V_{pr}\cap V_{qr}\cap (\phi_{qr})^{-1}(V_{pq})$ is an open neighborhood around $x_r$. Alternatively, by the maximality of $C_q\to C_p$, $\phi_{qr}(V_{pr}\cap V_{qr})\subset V_{pq}$, thus $V_{pqr}=V_{pr}\cap V_{qr}$. This is one of the reasons why a condition like the maximality condition is important: the compatibility condition places more constraints when the maximality condition holds. Without the maximality condition, some coordinate changes have smaller domains of coordinate changes than they should, and this issue might persist in an ordered finite cover we choose, which will cause problems in the compatibility of the extensions during the inductive constructions on the latter (when two regions which are unrelated in the earlier part of induction become related later on). 
\item We will see that the maximality and topological matching condition will also play an important role in establishing Hausdorffness.
\end{enumerate}
\end{remark}

\begin{cexample}[counterexample to the maximality condition]\label{DISKEX} Let $X=\{x, y, z\}$, and let $C_z$, and $C_y$ and $C_x$ be Kuranishi charts with coordinate changes $C_z\to C_y$, $C_z\to C_x$ and $C_y\to C_x$ satisfying the maximality condition. Suppose we have a point $w\in V_{yz}\backslash s_z^{-1}(0)$ such that $\underline{\phi_{xz}(w)}=\underline{\phi_{xy}(\phi_{yz}(w))}$ (namely $w$ also belongs to $V_{xz}\cap \phi_{yz}^{-1}(V_{xy})$). We remove $w$ from $V_{yz}$ and keep everything else unchanged, then the maximality condition fails. This will cause problems in the compatibility of the extensions during the inductive constructions later.
\end{cexample}

\begin{cexample}[failure of the maximality condition cannot be corrected without shrinking the Kuranishi charts]\label{NONCOEXISTENCE} Consider $X=\{x, y, z\}$. Let $\lambda:\R\to \R$ be a periodic function such that $\lambda=\lambda(\cdot+6)$ and it has only zeros at $-1+6k$, $1+6k$, $3+6k$, $k\in\mathbb{N}$, here at least one set of the zeros are not transverse for the existence of such a function. Let $V_x:=(-4, 2)\times (-1,1)$, define $s_x: V_x\to E_x:=V_x\times \R^2$ by $s_x(a, b)=(\lambda|_{[-4,2]}(a), b)$, and let $\psi_x: (-1, 0)\mapsto x, (1, 0)\mapsto y, (-3, 0)\mapsto z$. Let $V_y:=(-2, 4)\times (-1,1)$, define $s_y: V_y\to E_y:=V_y\times \R^2$ by $s_y(a, b)=(\lambda|_{[-2,4]}(a), b)$, and let $\psi_y: (-1, 0)\mapsto x, (1, 0)\mapsto y, (3, 0)\mapsto z$.  Let $V_z:=(-1, 1)$, define $s_z: V_z\to E_z:=V_z\times \R^2$ by $s_z(a)=\lambda(a-3)$, and let $\psi_x: 0\mapsto z$. Define $\phi_{xz}:(-1,1)\mapsto (-4, 2)$ by $a\mapsto (a-3, 0)$, $\phi_{yz}:(-1, 1)\mapsto (2, 4)$ by $a\mapsto a+3$, $\phi_{xy}: V_{xy}:=(-2, 2)\times (-1, 1)\to (-4, 2)\times (-1, 1)$ by $(a, b)\mapsto (a, b)$, and $\phi_{yx}: V_{yx}:=(-2, 2)\times(-1,1)\to (-2, 4)\times (-1,1)$ by $(a, b)\mapsto (a, b)$. The space $M$ formed by gluing $V_x, V_y$ and $V_z$ together using coordinate changes looks like a strip folded around, twisted by 90 degrees and with its two ends intersected along the center line of the strip. See figure \ref{figure01}. The maximality condition does not hold. If we add points to $V_{xy}$ to foce validity of the maximality condition valid, we will have to make new $\tilde V_{xy}:=(-4,-2)\times\{0\}\cup (-2, 2)\times (-1, 1)$ and this is not open in $V_y$, failing a requirement for coordinate changes. Therefore failure of the maximality condition in this case cannot be corrected without shrinking the Kuranishi charts.
\end{cexample}

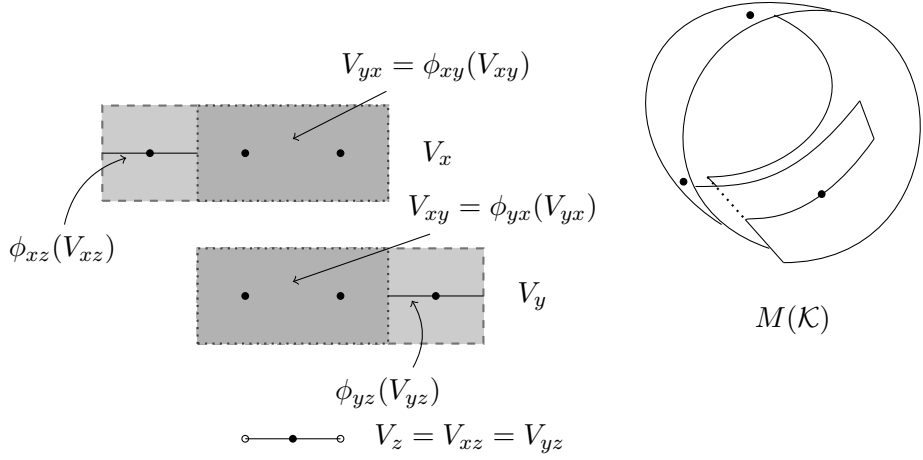
\begin{figure}[htb]
\begin{center}
\begin{tikzpicture}[scale=19/30]

\filldraw [color=black!20] (6,4) rectangle (12,6);
\draw [thick, dashed, color=black!55] (6,4) rectangle (12,6);
\filldraw [color=black!30] (8, 4) rectangle (12, 6);
\draw [thick, dotted, color=black!75] (8, 4) rectangle (12, 6);

\filldraw [color=black!20] (8,1) rectangle (14,3);
\draw [thick, dashed, color=black!55] (8,1) rectangle (14,3);
\filldraw [color=black!30] (8, 1) rectangle (12, 3);
\draw [thick, dotted, color=black!75] (8, 1) rectangle (12, 3);

\draw (6, 5) -- (8, 5);
\draw (12, 2)--(14, 2);

\filldraw (7,5) circle (2pt);
\filldraw (9,5) circle (2pt);
\filldraw (11,5) circle (2pt);
\filldraw (9,2) circle (2pt);
\filldraw (11,2) circle (2pt);
\filldraw (13,2) circle (2pt);
\filldraw (10,-1) circle (2pt);
\draw (9,-1) circle (2pt);
\draw (11,-1) circle (2pt);

\draw (9, -1) -- (11, -1);

\node at (12.5,5) [right] {$V_x$};
\node at (14.5, 2) [right] {$V_y$};
\node at (11.5, -1) [right] {$V_z=V_{xz}=V_{yz}$};

\node (p1) at (6.5, 3.5) [below left] {$\phi_{xz}(V_{xz})$}; 
\path[->] (p1) edge [bend left] (6.5, 4.9);

\node (p2) at (10.8, 6.3) [above right] {$V_{yx}=\phi_{xy}(V_{xy})$}; 
\path[->] (p2) edge (10, 5.3);

\node (p3) at (12.2, 3.3) [above right] {$V_{xy}=\phi_{yx}(V_{yx})$}; 
\path[->] (p3) edge (10, 2.3);

\node (p4) at (12, 0.5) [below] {$\phi_{yz}(V_{yz})$}; 
\path[->] (p4) edge [bend right] (12.5, 1.9);

\draw (20, 3) .. controls (17,4) and (18, 8) .. (21, 8);
\draw (21, 8) .. controls (24, 8) and (24, 2.7) .. (20.3, 2.7);
\draw (20.3, 2.7)-- (19.5, 3.6);
\draw (19.5, 3.6) .. controls (19.6, 3.7) and (20.8, 3.3) .. (22.2, 5.3);
\draw (18.45, 4.3) .. controls (19.5, 4.3) and (20.5, 4.3) .. (21.9, 6.1);
\draw (21.9, 6.1) -- (22.2, 5.3);
\draw (20.1, 7.9) .. controls (22.5, 6.8) and (21, 4.5) .. (18.7, 4.5);
\draw (18.7, 4.5) -- (18.9, 4.3);
\draw (19, 3.5) .. controls (16.3,5.2) and (17, 9) .. (21.3, 8);
\filldraw (21.1,4.14) circle (2pt);
\filldraw (18.2, 4.4) circle (2pt);
\filldraw (19.6, 7.9) circle (2pt);
\draw [thick,dotted] (18.7, 4.5)--(19.5, 3.6);

\node at (20.5, 1.5) {$M(\mathcal{K})$};
\end{tikzpicture}
\end{center}

\caption[Counterexample 1.14.]{Counterexample 1.14.}
\label{figure01}
\end{figure}

\begin{cexample}[counterexamples to the topological matching\\condition]\label{NONMATCHING} The following examples fail the topological matching condition.
\begin{enumerate}
\item Let $V_x=D_2(0)\subset \R^2$ be a standard open disk centered at 0 of radius 2, and define $s_x: V_x\to E_x:=D_2(0)\times \R^2$ so that $s_x$ has the only zeros $(0, 0)$ and $(1, 0)$ which are mapped under $\psi_x$ to $x$ and $z$ respectively, $\pi_{\R^2}(s_x(\cdot, 0))=(t(\cdot), 0)$, and $(1,0)$ is a nondegenerate zero. Let $V_y=D_2(0)$ and $s_y: V_y\to E_y:=D_2(0)\times\R^2$ so that $s_y$ has the only zeros being $(0,0)$ and $(1,0)$ which are mapped under $\psi_y$ to $y$ and $z$ respectively, $\pi_{\R^2}(s_y(\cdot, 0))=(t(\cdot),0)$, and $(1,0)$ is a nondegenerate zero. Let $V_z=(0,2)$ and $s_z:=t: V_z\to V_z\times \R$ with the only zero $(1, 0)$ of $s_z$ mapping to $z$. We can extend these data into a Kuranishi structure without the topological matching condition, with $\phi_{xz}: V_{xz}:=(0,2)\to V_x$ and $\phi_{yz}: V_{yz}:=(0,2)\to V_y$ being inclusions $(0,2)\times\{0\}$ in $D_2(0)\subset \R^2$. The topology from $X$ says that $x$ and $y$ are discrete, but information from the charts essentially indicates that $x$ is in the closure of $y$ and vice versa. See figure \ref{figure02}.
\item Let $s: D_2(0)\to D_2(0)\times \R^2$ with only two zeros at $(0,0)$ and $(1,0)$. Let $s_x$ be a copy of $s$ with $(0,0)$ and $(1,0)$ mapping to $x$ and $z$; let $s_y$ be another copy of $s$ with $(0,0)$ and $(1,0)$ maps to $y$ and $z$; and let $s_z$ be a copy of $s|_{D_2(0)\backslash\{0\}}$ with $(1,0)$ mapping to $z$. Let $\hat\phi_{xz}$ and $\hat\phi_{yz}$ be inclusions. This is similar to (1) but without dimension jump.
\item  Let $X$ consist of three points $x, y$, and $z$. Letting $V_x=D_2(0)\subset \R^2$, and define $\psi_x: (0,0)\mapsto x$, and $(1,0)\mapsto z$. Similarly, let $V_y$ be another copy of same disk and define $\psi_y: (0, 0)\mapsto y, (1,0)\mapsto z$. Let $V_z:=D_2(0)\backslash\{0\}\subset \R^2$ be a standard open disk centered at zero of radius 2 and punctured at 0 and define $\psi_z: (1,0)\mapsto z$, we can extend these into three Kuranishi charts with coordinate changes $C_z\to C_x$ and $C_z\to C_y$, so that $s_x^{-1}(0)=\{(0,0),(1,0)\}$, $s_y^{-1}(0)=\{(0,0),(1,0)\}$, $s_z^{-1}(0)=\{(1,0)\}$, $V_{xz}=V_{yz}=V_z$. The resulting object is almost a Kuranishi structure except the failure of the topological matching condition. We cannot achieve Hausdorffness in any ordered finite covers extracting from these three examples and they will have perturbation issues.
\end{enumerate}
\end{cexample}

\begin{figure}[htb]
  \begin{center}
\begin{tikzpicture}[scale=38/50]
\hspace{0cm}

\filldraw[color=black!20] (6,3) circle (2);
\draw[thick, dashed, color=black!40] (6,3) circle (2);
\filldraw[red] (6,3) circle (2pt);
\filldraw (7,3) circle (2pt);
\draw (6,3)--(8,3);
\node at (6,0.5) {$V_x$};
\path[->] node at (8,5)[above]{$\phi_{xz}(V_{xz})$} edge [bend right] (6.5,3);

\filldraw[color=black!20] (12,3) circle (2);
\draw[thick, dashed, color=black!40] (12,3) circle (2);
\filldraw[green] (12,3) circle (2pt); 
\filldraw (13,3) circle (2pt);
\draw (12,3)--(14,3);
\node at (12,0.5){$V_y$};
\path[->] node at (14,5)[above]{$\phi_{yz}(V_{yz})$} edge [bend right] (12.5,3);

\draw (16, 3)--(18, 3);
\filldraw (17,3) circle (2pt);
\draw (16,3) circle (2pt) (18,3) circle (2pt);
\node at (17,1.5){$V_z=V_{xz}=V_{yz}$};

\draw [rotate around={65:(14,-2)}](12,-2) circle (2 and 0.15);
\draw [rotate around={15:(14,-2)}](14,-2) arc (0:180: 2 and 0.15);
\draw [dashed, rotate around={15:(14,-2)}](10,-2) arc (180:360: 2 and 0.15);
\draw [rotate around={37:(14,-2)}] (14,-2)--(11.8,-2);
\draw [rotate around={15:(14,-2)}](10,-2) coordinate (p1);
\draw [rotate around={37:(14,-2)}](11.8,-2)..controls (10.8,-1.15)..(p1);
\draw [rotate around={65:(14,-2)}](10,-2) coordinate (p2);
\draw [rotate around={37:(14,-2)}](11.8,-2)..controls (10.85, -3.15)..(p2);
\filldraw [red, rotate around={36:(14,-2)}](11.8,-2) circle (2pt);
\filldraw [green, rotate around={38:(14,-2)}](11.8,-2) circle (2pt);
\filldraw [rotate around={37:(14,-2)}]({(11.7+14)/2},-2) circle (2pt);
\node at (9.5, -3) [left] {$M(\mathcal{K})$};
\end{tikzpicture}
  \end{center}
\caption[Counterexample 1.15.(1)]{Counterexample 1.15.(1)}
\label{figure02}
\end{figure}
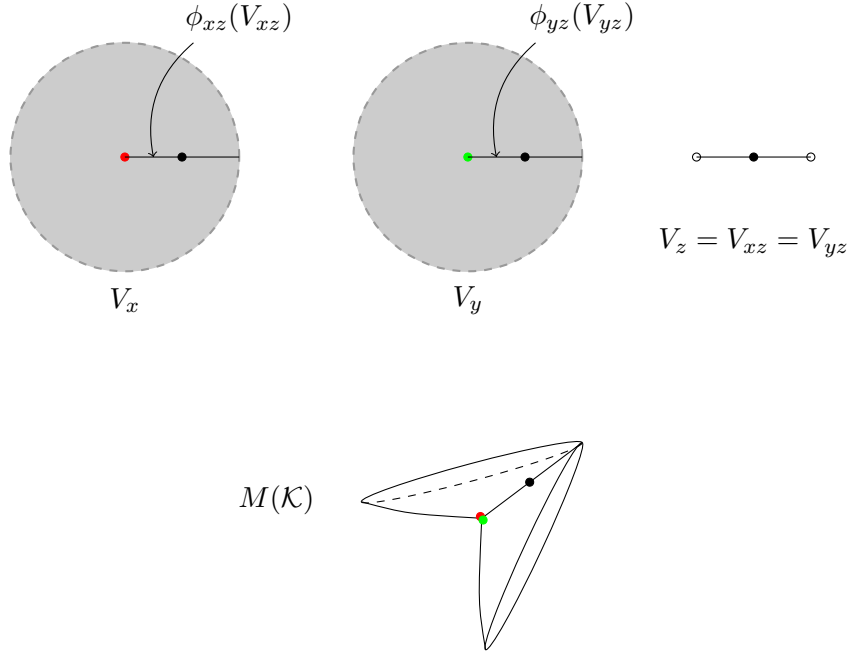

Now, we will give an alternative formulation of the maximality condition we often use later on. From this alternative picture, we will observe that another consequence of the maximality condition is that we can recover the domain $V_{pq}$ of the coordinate change unambiguously from $V_q$ and $V_p$. This is necessary because we often need to shrink charts during our constructions and need to know the ``maximal'' domain of the coordinate changes between the shrunken charts. 

\begin{definition}[the identification space as a set] Let $\mathcal{K}$ be a Kuranishi structure, and let $\sim$ be the relation on $\bigsqcup_{p\in X}\underline{V_p}$ saying two points are connected via a sequence of quotiented coordinate changes in the Kuranishi structure each of which moves in either direction. The set $M=M(\mathcal{K}):=(\bigsqcup_{p\in X}\underline{V_p})/\sim$ is called the identification space of the Kuranishi structure $\mathcal{K}$. For all $p\in X$, let $\iota_p: \underline{V_p}\to M$ be the map naturally induced from the inclusion $\underline{V_p}\to\bigsqcup_{p\in X}\underline{V_p}$, let $\text{quot}_p: V_p\to\underline{V_p}$ denote the quotient map, and denote $t_p:=\iota_p\circ \text{quot}_p: V_p\to M$. 
\end{definition}

If the maximality condition holds, the following lemma implies that $V_{pq}$ is determined by how $V_q$ and $V_p$ intersect in the identification space $M$ via $t_q$ and $t_p$. We do not need a topology on $M$ so far.

\begin{lemma}\label{MAXIMALITYEQUIV} The maximality condition is equivalent to the following: $$(i)\;\; V_{pq}=t_q^{-1}(t_q(V_q)\cap t_p(V_p))$$ for all $q\in X_p, p\in X$ and (ii) $\iota_p: \underline{V_p}\to M$ is injective for all $p\in X$.
\end{lemma}
\begin{proof}
(1) The maximality condition implies these two conditions:

We always have that $V_{pq}\subset t_q^{-1}(t_q(V_q)\cap t_p(V_p))$. If $z\in t_q^{-1}(t_q(V_q)\cap t_p(V_p))$, then $t_q(z)\in t_q(V_q)\cap t_p(V_p)$, namely $\underline{z}\in \underline{V_q}$ can be identified with some $\underline{z'}\in \underline{V_p}$ using $\sim$, and by the maximality, $z\in V_{pq}$, so we have $t_q^{-1}(t_q(V_q)\cap t_p(V_p))\subset V_{pq}$. Hence the equality $V_{pq}= t_q^{-1}(t_q(V_q)\cap t_p(V_p))$ holds.

If the maximality holds, then from the consequence just proved, we know that $V_{pp}=t_p^{-1}(t_p(V_p))=V_p$. From this, $\phi_{pp}\circ\phi_{pp}$ and $\phi_{pp}$ agreeing up to $G_p$-action, and $\phi_{pp}$ being invertible (being an embedding onto itself, without dimension increase or stablizer group size increase), we know that $\underline{\phi_{pp}}=\underline{Id}:\underline{V_p}\to\underline{V_p}$. If $\underline{z}, \underline{z'}\in \underline{V_p}$ satisfy $\iota_p(\underline{z})=\iota_p(\underline{z'})$, then $\underline{z}\sim \underline{z'}$. By the maximality condition, $\underline{\phi_{pp}}(\underline{z})=\underline{z'}$, and hence $\underline{z'}=\underline{\phi_{pp}}(\underline{z})=\underline{Id}(\underline{z})=\underline{Id(z)}=\underline{z}$. Hence $\iota_p$ is injective.

(2) These two conditions imply the maximality condition:

Let $q\in X_p$, and assume $\underline{z}\sim \underline{z'}$ for some $z\in V_q$ and $z'\in V_p$, then $t_q(z)=t_p(z')$, hence $z\in t_q^{-1}(t_q(V_q)\cap t_p(V_p))=V_{pq}$. Now $\iota_p(\underline{\phi_{pq}(z)})=\iota_p(\underline{z'})$ but $\iota_p$ is injective, so $\underline{\phi_{pq}}(\underline{z})=\underline{\phi_{pq}(z)}=\underline{z'}$. \end{proof}

\begin{remark}The identification space $M$ will be only used as a set in determining the domain of the coordinate changes of shrunken Kuranishi charts. In general, $M$ is not Hausdorff for any reasonable definition of topology.
\end{remark}

\section{A good coordinate system for a Kuranishi structure}

If we have coordinate charts centered at each point of a compact manidold, we can extract a finite cover consisting of coordinate charts with invertible coordinate changes among them. This section discusses the Kuranishi structure analogue of this fact.

\begin{definition}[a non-antisymmetric total order]
A finite set $S$ with a relation $\leq$ is said to be a \emph{non-antisymmetric total order}, or \emph{order} for short in our context, if
\begin{enumerate}
\item for all $x\in S$, $x\leq x$,
\item for all $x, y, z \in S$, if $x\leq y$ and $y\leq z$ then $x\leq z$, and
\item for all $x, y\in S$, $x\leq y$ or $y\leq x$ (or both).
\end{enumerate}
\end{definition}

We use a slightly different notation for Kuranishi charts, because we want to give a stand-alone definition of a good coordinate system and also define a good coordinate system for a Kuranishi structure. Denote $$\tilde C_x:=(G_x, s_x:U_p\to E_p, \psi_x: \underline{s_x^{-1}(0)}\to \tilde X_x),$$ where $\tilde X_x:=\psi_x(\underline{s_x^{-1}(0)})$ is an open neighborhood of $x$ in $X$ and is often referred to as the \emph{coverage on $X$ by the chart $\tilde C_x$}.

A good coordinate system in a first approximation is a finite collection of Kuranishi charts $\tilde C_x, x\in S$ indexed by a non-antisymmetric total order $(S, \leq)$, such that if $$\tilde X_y\cap \tilde X_x\not=\emptyset, \;\;y,x\in S, \;\; y\leq x$$ there exists a coordinate change $\tilde C_y\to \tilde C_x$ in the direction dictated by the order $y\leq x$, and $X$ is covered by the coverages of the charts: $X=\bigcup_{x\in S}\tilde X_x$. There are also some other technical conditions.

Since the condition for coordinate changes are relaxed, we need slightly general form of coordinate changes in the context of good coordinate systems:

\begin{definition}[coordinate change in the context of good coordinate systems]\label{GCSCC} Let $\tilde C_y$ and $\tilde C_x$ be two Kuranishi charts for $X$ such that $\tilde X_y\cap \tilde X_x\not=\emptyset$, and let $U_{xy}$ be a $G_y$-invariant open subset of $U_y$. Let $\hat\phi_{xy}: E_y|_{U_{xy}}\to E_x$ be a $G_y-G_x$ bundle embedding that covers $\phi_{xy}:U_{xy}\to U_x$ such that $$\hat\phi_{xy}\circ (s_y|_{U_{xy}})=s_x\circ\phi_{xy}\;\;\text{and}\;\; \psi_x\circ\underline{(\phi_{xy}|_{(s_y|_{U_{xy}})^{-1}(0)})}=\psi_y|_{\underline{(s_y|_{U_{xy}})^{-1}(0)}}.$$ A \emph{coordinate change} from $\tilde C_y$ to $\tilde C_x$ consists of such tuple $(\phi_{xy}, \hat\phi_{xy}, U_{xy})$ that satisfies the tangent bundle condition and the condition $\psi_y^{-1}(\tilde X_y\cap\tilde X_x)\subset\underline{U_{xy}}$, and $U_{xy}$ is the \emph{domain of the coordinate change}.
\end{definition}

Note that in this context, it is not necessary that $y\in \tilde X_y|_{U_{xy}}$ for such a coordinate change. But the domain of such a coordinate change is always non-empty by $\psi_y^{-1}(\tilde X_y\cap\tilde X_x)\subset\underline{U_{xy}}$.

\begin{definition}[compatibility]\label{GCSCOMPATIBLE} Let $$\text{$(\phi_{yz}, \hat\phi_{yz}, U_{yz})$, $(\phi_{xz}, \hat\phi_{xz}, U_{xz})$, and $(\phi_{xy}, \hat\phi_{xy}, U_{xy})$}$$ be three coordinate changes. Those three coordinate changes are said to be \emph{compatible} if over the common domain of the definitions, $$U_{xyz}:=U_{xz}\cap U_{yz}\cap (\phi_{yz})^{-1}(U_{xy}),$$ the composed coordinate change $\hat\phi_{xy}\circ\hat\phi_{yz}$ equals the direct coordinate change $\hat\phi_{xz}$ up to the $G_x$-action on $\tilde C_x$. A collection of coordinate changes is said to be \emph{compatible} if any three such coordinate changes in the collection are compatible.
\end{definition}

\begin{definition}[good coordinate system]\label{LEVEL0GCS} A \emph{good coordinate system} for a compact metrizable space $X$ consists of:
\begin{enumerate} 
\item a finite set $S\subset X$ with a non-antisymmetric total order $\leq$, and
\item a collection of Kuranishi charts $$\tilde C_x=(G_x, s_x: U_x\to E_x, \psi_x: \underline{s_x^{-1}(0)}\to \tilde X_x), \;\;x\in S$$ for $X$ indexed by $S$, such that:
\begin{enumerate}
\item $X=\bigcup_{x\in S}\tilde X_x$,
\item if $\text{dim} E_y<\text{dim} E_x$, then $y\leq x$ and $x\not\leq y$,
\end{enumerate}
and
\item a collection of coordinate changes $\{\tilde C_y\to \tilde C_x\}_{(y,x)\in I(\mathcal{G})}$ as defined in the \ref{GCSCC}, where the index set is $I(\mathcal{G}):=\{(y, x)\in S\times S\;|\; y\leq x, \tilde X_y\cap \tilde X_x\not=\emptyset\}$ such that
\begin{enumerate}
\item the collection of the coordinate changes is compatible,
\item the maximality condition holds: for any coordinate change $\tilde C_y\to \tilde C_x$ in the collection and for any pair $z_1\in U_y$ and $z_2\in U_x$ such that $\underline{z_1}$ and $\underline{z_2}$ can be identified via a sequence of quotiented coordinate changes pointing in either directions each of which is quotiented from a coordinate change in the given collection, which is denoted by $\underline{z_1}\sim \underline{z_2}$, we have $z_1\in U_{xy}$ and $\underline{z_2}=\underline{\phi_{xy}}(\underline{z_1})$, and
\item the topological matching condition is satisfied: for any two charts $U_y$ and $U_x$ where $y, x\in S$ and any pair of $v\in \underline{s_y^{-1}(0)}$ and $v'\in \underline{s_x^{-1}(0)}$ such that there exist a sequence $v_n\in \underline{U_y}$ and a sequence $v'_n\in \underline{U_x}$ satisfying $v_n\sim v_n'$, $v_n\to v$ in $\underline{U_y}$ and $v'_n\to v'$ in $\underline{U_x}$, we have $v\sim v'$.
\end{enumerate}
\end{enumerate}
\end{definition}

\begin{remark}\label{TRIVIALFACT}
\begin{enumerate}
\item We now explain why we require the condition $$\psi_y^{-1}(\tilde X_y\cap\tilde X_x)\subset\underline{U_{xy}}$$ in definition \ref{LEVEL0GCS} of a coordinate change. If we remove this condition, it is possible to have a good coordinate system where we have a coordinate change $\tilde C_y\to \tilde C_x$ such that two points $w\in U_y$ and $z\in U_x$ map to the same point in $X$, but $w\not\in U_{xy}$, and this does not violate the maximality condition, because we can have $\underline{w}\not\sim \underline{z}$. This is simply against the philosophy of how coordinate changes work and will have compatibility issues.
\item For a coordinate change $C_q\to C_p$ in a Kuranishi structure, always $\psi_q^{-1}(X_q\cap X_p)\subset\underline{V_{pq}}$. Indeed, if $z\in V_q$ with $\underline{z}\in \psi_q^{-1}(X_q\cap  X_p)$, then by denoting $r:=\psi_q(\underline{z})\in X$, we have $r\in X_q\cap X_p$, so by definition we have coordinate changes $C_r\to C_q$ and $C_r\to C_p$ with $r\in X_q|_{V_{qr}}$ and $r\in X_p|_{V_{pr}}$, so $\underline{z}\sim \psi_p^{-1}(r)$. Therefore, by the maximality condition, $z\in V_{pq}$, and hence $\underline{z}\in \underline{V_{pq}}$.
\item By item 2, if a coordinate change $C_y\to C_x$ is restricted from a coordinate change in a Kuranishi structure, we automatically have $\psi_y^{-1}(\tilde X_y\cap\tilde X_x)\subset\underline{U_{xy}}$. Therefore, if we only work with a good coordinate system induced from a Kuranishi structure (to be precisely defined in \ref{GCSFORKS}) which we always do in this paper, we do not need to impose the condition $\psi_y^{-1}(\tilde X_y\cap\tilde X_x)\subset\underline{U_{xy}}$, but we also want to provide a stand-alone definition of a good coordinate system including this essential property.
\item If we have a coordinate change $\tilde C_y\to\tilde C_x$, then the domain $U_{xy}$ is nonempty; and since $U_{xy}$ is open and $G_y$-invariant in $U_y$ and $\phi_{xy}$ is an equivariant embedding, we have $\text{dim} E_y\leq \text{dim} E_x$ and $|G_y|\leq |G_x|$.
\end{enumerate}
\end{remark}

A good coordinate system we are interested in is extracted from a given Kuranishi structure.

\begin{definition}\label{CCRESTRICTED} A coordinate change $C_y|_{U_y}\to C_x|_{U_x}$ is said to be \emph{restricted} from $(C_y\to C_x)=(\phi_{xy}, \hat\phi_{xy}, V_{xy})$ if the domain $U_{xy}$ of the coordinate change $C_y|_{U_y}\to C_x|_{U_x}$ satisfies $U_{xy}=\phi_{xy}^{-1}(U_x)\cap U_y$ with the intersection taken in $V_x$, and $(C_y|_{U_y}\to C_x|_{U_x})=(\phi_{xy}|_{U_{xy}}, \hat\phi_{xy}|_{E_y|_{U_{xy}}}, U_{xy})$.
\end{definition}

\begin{definition}[good coordinate system for a Kuranishi structure]\label{GCSFORKS} A good coordinate system $$\mathcal{G}:=(X, (S,\leq), \{\tilde C_x\}_{x\in S}, \{\tilde C_y\to \tilde C_x\}_{y\leq x, \tilde X_y\cap \tilde X_x\not=\emptyset}),$$ or $(X, \{\tilde C_x\}_{x\in S})$ for short, is said to be a \emph{good coordinate system for the Kuranishi structure} $\mathcal{K}:=(X, \{C_p\}_{p\in X}, \{C_q\to C_p\}_{q\in X_p})$ if 
\begin{enumerate}
\item $\tilde C_x=C_x|_{U_x}$ is a restriction of $C_x$ for all $x\in S$,
\item for any coordinate change $\tilde C_y\to\tilde C_x$ in the collection,
\begin{enumerate}
\item if $y\in X_x$, $\tilde C_y\to\tilde C_x$ is restricted from $C_y\to C_x$, and
\item if $y\not\in X_x$, $\tilde C_y\to \tilde C_x$ is an inverse of the restriction of a coordinate change in $\mathcal{K}$ across which the dimension and stabilizer group size remain the same.
\end{enumerate}
\end{enumerate}
\end{definition}

We give a simple counterexample to show that in order to obtain a good coordinate system we cannot just take a finite collection of Kuranishi charts from Kuranishi structure without shrinking charts. Consider the following Kuranishi structure for $\{x, y, z\}$. Let $V_x=(-1,\infty)\times\R\times\{0\}\subset \R^3$ and define $s_x$ with its only zeros as $(0,0,0)$ and $(2,0,0)$ which are respectively identified as $z$ and $x$ under $\psi_x$ such that $(0,0,0)$ is a nondegenerate zero and $s_x|_{(-1,1)\times\{0\}\times \{0\}}=t: (-1,1)\times\{0\}\times\{0\}\to \R\times\{0\}\times\{0\}$; let $V_y=(-\infty,1)\times \{0\}\times \R\subset\R^3$ and define $s_y$ with its only zero as $(0,0,0)$ and $(-2,0,0)$ which are identified as $z$ and $y$, such that $(0,0,0)$ is a nondegenerate zero and $s_y|_{(-1,1)\times\{0\}\times \{0\}}=t$; and $V_z=(-1,1)\times\{0\}\times\{0\}\subset \R^3$ and define $s_z=t$ with the only zero $(0,0,0)$ of $s_z$, identified as $z$. The coordinate changes of this Kuranishi structure are induced from how $V_x, V_y$ and $V_z$ embed in $\R^3$. Because we have $z\in X_x$ and $z\in X_y$, not other inclusions, we only have two coordinate changes, and this is a Kuranishi structure. The index set is already finite, but we do not have coordinate change between $C_x$ and $C_y$ despite the fact that $X_x\cap X_y\not=\emptyset$, both charts have the bundles in the same dimension and are acted upon by trivial stabilizer groups.

\begin{remark}\label{SAMEWAVE}
We include here two very important remarks:
\begin{enumerate}[(I)]
\item We observe that we have specified exactly what the collection of coordinate changes is in the definition \ref{LEVEL0GCS}. This has two consequences (i) the composition of two coordinate changes up to group actions is not necessarily a coordinate change in the collection (just as in the case of Kuranishi structure), and (ii) later on, we will shrink $C_x|_{U_x}$ to $C_{U'_x}, x\in S$, which might affect the index set $I(\mathcal{G})$. A good coordinate system in this form is an intermediate notion, as we will ultimately making it Hausdorff and level-1 by shrinking and further constructions; and we can shrink charts to avoid issue (i) and achieve that the collection of coordinate changes is closed up to group action under composition, as we will show in proposition \ref{COORDEXIST}. Also since $X_x|_{U_x}, x\in S$ are open in $X$, whether two of them intersect or not will not change if we shrink $U_x$ slightly; so after a slight shrinking from $U_x$ to $U'_x$, we might have $X_x|_{U'_x}\not=X_x|_{U_x}$, but we still have $I(\mathcal{G}')=I(\mathcal{G})$, where $\mathcal{G}':=\{C_x|_{U'_x}\}_{x\in S}$. The reason we choose the collection of coordinate changes indexed by $I(\mathcal{G})$ rather than the closure of collection of coordinate changes up to group action under composition is because: we added some new inverted coordinate changes as in definition \ref{GCSFORKS}.(2)(b), and if we also include compositions of coordinate changes in our collection, the composed coordinate changes might not satisfy the maximality condition, similar to the example \ref{NONCOEXISTENCE}.
\item If we use $\sim_{\mathcal{K}}$ induced from coordinate changes in the Kuranishi structure $\mathcal{K}$ for the definition \ref{GCSFORKS} which has more coordinate changes to use under identification, will it make any difference? The answer is no. First observe that $\underline{z}\sim \underline{z'}$ implies $\underline{z} \sim_{\mathcal{K}}\underline{z'}$. Indeed, although we added some inverted coordinate changes in forming the good coordinate system, we are allowed to have coordinate changes in either direction in defining $\sim_{\mathcal{K}}$. For the other direction, just note that the coordinate changes in the good coordinate system are all induced in two ways from the coordinate changes in Kuranishi structures which already satisfy the maximality condition defined via $\sim_{\mathcal{K}}$.
\end{enumerate}
\end{remark}

The finite index set $S$ and order $\leq$ are a finite way to organize local models of varying stabilizer groups and dimensions in a Kuranishi structure. The order $\leq$ on $S$ is required to be compatible with the order by the bundle dimension. We will show in a nontrivial theorem \ref{HAUSIDSPACE} that we can have a Hausdorff good coordinate system by shrinking. Then we will show that by grouping certain Kuranishi charts of the same bundle dimension, the data can be organized to allow inductive constructions. Thus a Hausdorff good coordinate system provides a convenient ordered finite cover for working with Kuranishi structures.

\begin{remark}\label{TEMPRESOLUTION} We now explain the intuition behind the condition $\tilde X_y\cap\tilde X_x\not=\emptyset$ for the existence of coordinate changes in definition \ref{LEVEL0GCS}, as it is not spelt out in the literature. In a Kuranishi structure, we can have two charts $C_x$ and $C_y$ such that their coverages $X_x$ and $X_y$ overlap in $X$ but only the section of a subbundle of $E_x$ restricted over a submaifold of $V_x$ is identified with the section of a subbundle of $E_y$ restricted over a submanifold of $V_y$ via a mediating local model of the lower dimension. Moreover, there can be a third chart $C_o$ larger in dimension than $C_x$ and $C_y$ in which part of $C_x$ and $C_y$ including their identified parts embed, and this can cause compatibility issues in extending perturbations. For example, to visualize, consider $V_x, V_y, V_o$ to be embedded in $\R^3$ with coordinate changes induced from the embeddings in $\R^3$. Let $V_x$ and $V_y$ be two surfaces intersect along a line $V_z$ of a smaller order, and those two surfaces are tangent along this line, and let $V_o$ be $\R^3$. In this case without knowing the last chart $C_o$, the extensions of perturbations in the earlier stage of the induction can be incompatibly chosen. So to resolve this situation, we require that a coordinate change exists between the charts if their coverages for $X$ intersect. We remark that this intersection condition does not avoid this issue, but provides a condition sufficient for a further construction which ultimately resolves this issue. A condition that resolves it directly will make the formulation of a good coordinate system less elegant and not minimal. If we want to obtain a good coordinate system from these charts, we have to shrink $C_x$ and $C_y$ so that their coverages for $X$ is disjoint in order to satisfy the existence condition for coordinate changes.
\end{remark}

The following is a key theorem and it will be proved in the next section. 

\begin{theorem}[existence of a good coordinate system] For any Kuranishi structure, a good coordinate system always exists.
\end{theorem}

The original proof by Fukaya and Ono in \cite{FO} used a background chart $C_z$ containing both $C_y|_{U_y}$ and $C_x|_{U_x}$ to construct a coordinate change $C_y|_{U_y}\to  C_x|_{U_x}$, since it is unclear whether there are coordinate changes between different choices of $C_z$, so the resulting coordinate change $C_y|_{U_y}\to C_x|_{U_x}$ might depend on the choice of $C_z$, and all the resulting coordinate changes might not be compatible as a result. Moreover, even if we can choose compatible coordinate changes by extra arguments and hence a good coordinate system, two different choices of such good coordinate systems cannot be compared. We will use a different argument in section \ref{EXISTGCS}.

Before proceeding to that section, we show that we can also go backwards.

\begin{proposition}[inducing a Kuranishi structure]\label{GCSKS} 
\begin{enumerate}
\item A good coordinate system naturally induces a Kuranishi structure.
\item If a good coordinate system $\mathcal{G}$ is obtained from a Kuranishi structure, then we can choose an induced Kuranishi structure $\mathcal{K}(\mathcal{G})$ such that $\mathcal{G}$ is a good coordinate system for $\mathcal{K}(\mathcal{G})$.
\end{enumerate}
\end{proposition}

\begin{proof}
\begin{enumerate}
\item For $p\in X$, denote $S_p:=\{x\in S\;|\; p\in \tilde X_x\}$, pick a smallest $y\in S_p$ according to $\leq$, and choose an invariant open neighborhood $V_p$ of $w$ in $U_y$ where $w$ corresponds to $p$, and define $C_p:=\tilde C_y|_{V_p}$. 

The coordinate changes among $\{C_p\}_{p\in X}$ will be induced from coordinate changes of this good coordinate system with the domain specified by the maximality condition. We only need to show for $q\in X_p$, we have $C_q\to C_p$. Suppose, $q\in X_p$, $C_q:=\tilde C_z|_{V_q}$, $C_p:=\tilde C_y|_{V_p}$, then $q\in \tilde X_y|_{V_p}\subset \tilde X_y$, so $y\in S_q$, so by definition $z\leq y$ (and $\tilde X_z\cap \tilde X_y$ contains $q$), by the definition of a good coordinate system, we have $\tilde C_z\to \tilde C_y$, and this induces $C_q:=\tilde C_z|_{V_q}\to \tilde C_y|_{V_p}=:C_p$.

\item For $p\in S$, define $C_p:=\tilde C_p$, the chart in the good coordinate system. For $p\in X\backslash S$, denote $S_p:=\{x\in S\;|\; p\in \tilde X_x\}$, as before, pick a smallest $y\in S_p$ according to $\leq$, and choose an invariant open subset $V_p$ of $w$ in $U_y$ where $w$ corresponds to $p$, and define $C_p:=\tilde C_y|_{V_p}$. 

The coordinate changes among $\{C_p\}_{p\in X}$ will again be induced from coordinate changes of this good coordinate with the domain specified by the maximality condition. We will only need to verify that for $q\in X_p$, we have $C_q\to C_p$. The only new cases are, (a) when $q\in S$, $p\in S$, (b) $q\in S, p\not\in S$, and (c) $q\not\in S, p\in S$.

Case (a): Since the charts of good coordinate system are shrunken from a original Kuranishi structure, $\tilde C_x:=\hat C_x|_{U_x}$, so we must have $q\in \hat X_p$, so we have $\hat C_q\to\hat C_p$, which induces $C_q:=\tilde C_q\to \tilde C_p=: C_p$.

Case (b): Since $C_p:=\tilde C_y|_{V_p}$ for some $y\in S$, and hypothesis says $q\in X_p$, we must have $q\in \tilde X_y$, and from case (a), we have $\tilde C_q\to \tilde C_y$, and this induces $C_q:=\tilde C_q\to \tilde C_y|_{V_p}=: C_p$.

Case (c) follows from item (1) and an inclusion.

So we have a new Kuranishi structure $\mathcal{K}(\mathcal{G})$ from $\mathcal{G}$; and $\mathcal{G}$ is also a good coordinate system for $\mathcal{K}(\mathcal{G})$ by the construction.
\end{enumerate}
\end{proof}

The above also allows us to consider examples/nonexamples of Kuranishi structures derived from simple examples/nonexamples of good coordinate systems.

\section{The existence of a good coordinate system}\label{EXISTGCS}

This section gives a short proof of the existence of a good coordinate system for a Kuranishi structure.

The idea of my proof is as follows. If we can (i) define an order compatible with the order by dimension and (ii) find a way to shrink charts such that whenever two charts intersect there exists a coordinate change or we can naturally create a coordinate change between these two charts in the direction determined by the order, then we can invoke the compactness of $X$ to extract a finite number of charts covering $X$. Since the existence condition for coordinate changes in the definition of a good coordinate system is relaxed from the existence condition in a Kuranishi structure, it is necessary to create some new coordinate changes, but we do not want to change the information provided by the original Kuranishi structure, hence we impose the condition \ref{GCSFORKS}.2.(b).

\begin{theorem} A good coordinate system $\mathcal{G}$ for a Kuranishi structure $\mathcal{K}$ exists. Moreover, $\mathcal{G}$ can be chosen such that for all $x\in S$, $\tilde X_x$ is the interior of its closure $\overline{\tilde X_x}$, where the closure of $\tilde X_x$ is taken in $X$.
\end{theorem}
\begin{proof} Since $X$ is metrizable, choose a metric $d$ on $X$. Define $$B_r(p):=\{q\in X\;|\; d(q,p)<r\}.$$ For all $p\in X$, since the coverage $X_p$ is an open neighborhood of $p$ in $X$, we can choose $B_{r_p}(p)\subset X_p$ for some $r_p>0$ depending on $p$.

Denote $quot_p: V_p\to \underline{V_p}$, and define $$U_p:=V_p\backslash (s_p^{-1}(0)\backslash quot_p^{-1}(\psi_p^{-1}(B_{r_p/2}(p)))),$$ which is an invariant open subset of $V_p$ such that $p\in X_p|_{U_p}$.
Define $$\tilde C_p:=C_p|_{U_p}\text{ and }\tilde X_p:=X_p|_{U_p},$$ so $\tilde X_p=B_{r_p/2}(p)$. Therefore, $\tilde X_x$ is the interior of its closure $\overline{\tilde X_x}$.

Since $\tilde X_p, p\in X$ is an open cover for the compact $X$, we can choose a finite index set $S$, such that $X=\bigcup_{x\in S}\tilde X_x$.

For $y, x\in S$, define $y\leq x$ if
\begin{enumerate}[(i)]
\item $\text{dim} E_y<\text{dim} E_x$, or
\item $\text{dim} E_y=\text{dim} E_x$ and $|G_y|\leq |G_x|$, where $|G|$ is the size of the group $G$.
\end{enumerate}

So $\leq$ is defined for all pairs of points from $S$.

We have constructed the order $(S, \leq)$ and Kuranishi charts $\tilde C_x, x\in S$ so far. We only need to show that if $\tilde X_y\cap \tilde X_x\not=\emptyset$, $y\leq x$, then there exists a coordinate change $\tilde C_y\to \tilde C_x$ such that condition 2 in definition \ref{GCSFORKS} holds. Namely, (i) if $y\in X_x$, $\tilde C_y\to\tilde C_x$ is restricted from $C_y\to C_x$, and (ii) if $y\not\in X_x$, $\tilde C_y\to \tilde C_x$ is an inverse of the restriction of a coordinate change in $\mathcal{K}$ across which the dimension and stabilizer group size remain the same.

Starting from $\tilde X_y\cap \tilde X_x\not=\emptyset$ and $y\leq x$, we have three cases:

\begin{enumerate}[(a)]
\item $r_y\leq r_x$. We denote $z\in \tilde X_y\cap \tilde X_x\not=\emptyset$, then $$d(y,x)\leq d(y, z)+d(z,x)<r_y/2+r_x/2\leq r_x/2+r_x/2=r_x,$$ so $y\in B_{r_x}(x)\subset X_x$. By the existence condition in the definition of the Kuranishi structure, there exists a coordinate change $C_y\to C_x$. Then define $\tilde C_y:=C_y|_{U_y}\to C_x|_{U_x}=:\tilde C_x$ as being restricted from $C_y\to C_x$ as in definition \ref{CCRESTRICTED}. Observe that the domain $U_{xy}$ of the coordinate change $\tilde C_y\to \tilde C_x$ determined by the maximality of the Kuranishi structure is same as $U_{xy}:=\phi_{xy}^{-1}(U_x)\cap U_y$ with the intersection taken in $V_y$, since the maximality condition already holds for $C_y\to C_x$.
\item $r_x< r_y$, and $y\in X_x$. Then $C_y\to C_x$ exists in the Kuranishi structure and define $(\tilde C_y\to \tilde C_x):=(C_y|_{U_y}\to C_x|_{U_x})$.
\item $r_x< r_y$, and $y\not\in X_x$. Then $\tilde C_y\to \tilde C_x$ cannot be a restriction of a coordinate change $C_y\to C_x$, as the latter does not exist; and we have to create it, but this case is easy too. The same argument as in case (a) shows that $x\in X_y$, and the coordinate change $C_x\to C_y$ exists in the Kuranishi structure. So, $\text{dim} E_x\leq \text{dim} E_y$ and $|G_x|\leq |G_y|$. Since we have hypothesis $y\leq x$ and we know $\text{dim} E_x\leq \text{dim} E_y$ from this case, according to the definition of $y\leq x$, we are left with the alternative that $\text{dim} E_y=\text{dim} E_x$ and $|G_y|\leq |G_x|$, and the latter and $|G_x|\leq |G_y|$ give $|G_y|=|G_x|$. Thus, we have $\text{dim} E_x=\text{dim} E_y$ and $|G_x|=|G_y|$. Therefore, the given coordinate change $C_x\to C_y$ provides a $G_x-G_y$ bundle diffeomorphism $\hat\phi_{yx}: E_x|_{V_{yx}}\to E_y$ onto its image intertwining sections and the stabilizer groups are isomorphic. Define $U_{yx}:=\phi_{yx}^{-1}(U_y)\cap U_x$ with the intersection taken in $V_x$. So in this case, $\phi_{yx}(U_{yx})$ is $G_y$-invariant in $U_y$ and we can invert this $G_x-G_y$ bundle diffeomorphism $(\phi_{yx}, \hat\phi_{yx}, U_{yx})$ into a $G_y-G_x$ bundle diffeomorphism $(\phi_{xy}, \hat\phi_{xy}, U_{xy})$, where $$\phi_{xy}:=(\phi_{yx})^{-1}: \phi_{yx}(U_{yx})\to U_{yx}\subset U_x, \hat\phi_{xy}:=(\hat\phi_{yx})^{-1}, U_{xy}:=\phi_{yx}(U_{yx});$$ and we define $$(\tilde C_y\to \tilde C_x):=(\phi_{xy}, \hat\phi_{xy}, U_{xy}).$$ 
This satisfies the definition of the coordinate change in \ref{GCSCC}. By construction, $\tilde C_y\to \tilde C_x$ is the inverse of $C_x|_{U_x}\to C_y|_{U_y}$.
\end{enumerate}

By construction of (a) and (b), if $y\in X_x$, $\tilde C_y\to \tilde C_x$ is induced from the coordinate change $C_y\to C_x$ (not just up to group action, which is why we separate the case (b) out from (c)). Since the above three cases exhaust all the possibilities, if $y\not\in X_x$, we are in case (c), and $\tilde C_y\to \tilde C_x$ is an inverse of the restriction $C_x|_{U_x}\to C_y|_{U_y}$ of the coordinate change $C_x\to C_y$ in $\mathcal{K}$ across which the dimension and the stabilizer group size remain the same.

The compatibility of coordinate changes follows from the same property of the Kuranishi structure, since the invertible coordinate changes from which newly added coordinate changes are inverted are already compatible among themselves and with the rest coordinate changes (in cases (a) and (b)).

We already discussed in remark \ref{SAMEWAVE}.(II) that it makes no difference to the maximality condition and the topological matching condition of a good coordinate system obtained from a Kuranishi structure whether we use $\sim$ or $\sim_{\mathcal{K}}$. The maximality condition for coordinate changes from restriction holds using $\sim_{\mathcal{K}}$ and hence using $\sim$. The maximality condition for coordinate changes from inversion also holds using $\sim_\mathcal{K}$ (since it holds for the restricted coordinate changes from which they are inverted), and hence using $\sim$.

Since the addition of extra inverted coordinate changes (which are diffeomorphisms intertwining isomorphic group actions) does not change how $\underline{U_x}, x\in S$ are glued in the identification space $M$ (the set appeared in the equivalent formulation of the maximality condition), topological matching condition in the current setting is just a subcase of the topological matching condition of the Kuranishi structure.
\end{proof}

\begin{remark}
\begin{enumerate}
\item From the above proof, it seems that an alternative definition of the order by defining $y\leq' x$ if and only if $r_y\leq r_x$ (while removing item 2.(b) in definition \ref{LEVEL0GCS}) will make the proof much shorter (only case (a) now) and all coordinate changes in the good coordinate system are restricted from coordinate changes in the Kuranishi structure. However, this order is very sensitive to the choice of the metric and the radii of the chosen balls in $X_x$ and it is impossible to tell the order between $x$ and $y$ from the knowledge of charts $C_x$ and $C_y$. Moreover, later on we will group the charts of each dimension together, so we prefer to define an order which is compatible with the order by dimension. The disagreement of the orders $\leq$ and $\leq'$ occurs when (i) $\tilde X_y\cap\tilde X_x=\emptyset$ which does not matter, since the order is used to organize coordinate changes and in this case there is no coordinate change; and (ii) a coordinate change between charts based at those two points is inverted, and in this case $\leq'$ between these two points is valid in only one way while $\leq$ between these two points holds in both ways. In our case, the inverted diffeomorphic coordinate changes intertwining isomorphic groups are all harmless. We have included condition 2.(b) in definition \ref{LEVEL0GCS} specifically to rule out the order $\leq'$ and choose the current definition of $\leq$.
\item Later we will show that two different choices of good coordinate systems obtained using the method in the above proof are equivalent.
\end{enumerate}
\end{remark}

The bases $U_x$ appearing in the charts in a good coordinate system for a Kuranishi structure are restricted from  $V_x$ in a way that we do nothing on $V_x$ away from $X_x$, so we may have a very non-Hausdorff space if we glue together all the $U_x$'s in a good coordinate system with respect to any reasonable topology. We will deal with this issue in the next section.

\section{Obtaining a Hausdorff good coordinate system}

In this section, we consider the Hausdorffness aspect of a good coordinate system obtained from a Kuranishi structure $\mathcal{K}$. We will write a Kuranishi chart as $C_x|_{U_x}$ rather than $\tilde C_x$, as we often restrict charts to smaller bases.

Hausdorffness is important for the existence of partitions of unity and the compatibility of the later constructions. Hausdorffness can fail even in the simplest case where we have two charts of the same dimension with a coordinate change between them. Indeed, if we glue these two charts by identifying each point in the domain of the coordinate change with its image in the target, Hausdorffness might fail, as we have just glued two open sets along open subsets. It becomes more involved to discuss Hausdorffness of a good coordinate system on a global scale where dimensions are not constant and the identification $\sim_{\mathcal{K}}$ becomes less visual. Here $\sim_{\mathcal{K}}$ is defined using coordinate changes of the Kuranishi structure from which we extracted the good coordinate system. In this general case, we need to define an appriopriate topology to even discuss Hausdorffness, and we need to improve the properties that $\sim_{\mathcal{K}}$ might have.

We will first give a conceptual overview of what will happen in this section culminating in theorem \ref{BIGHAUSDORFF}. Next we will motivate the discussions by some examples. Then we achieve a nice property where we can visualize $\sim_{\mathcal{K}}$. Then we define the relative topology for the identification space with dimension jumping. In last subsection, we will proceed to obtain a Hausdorffness good coordinate system with respect to this relative topology by considering a special sequence of shrinkings of good coordinate systems.

\subsection{Conceptual overview and definitions}

From the last section, we have obtained a good coordinate system $\mathcal{G}$ such that for all $x\in S$, $X_x|_{U_x}$ is the interior of its the closure $\overline{X_x|_{U_x}}$ taken in $X$. We fix such a $\mathcal{G}$ in this and the next section. We can form an identification space $M(\mathcal{G})$ of $\mathcal{G}$ by gluing charts together via coordinate changes in $\mathcal{G}$. In general it is impossible to define any reasonable Hausdorff topology on $M(\mathcal{G})$.

What we will do is to shrink $\mathcal{G}$ to $\mathcal{G}'$ to achieve better properties.

\begin{definition}[shrinking]\label{GCSSHRINKING} A good coordinate system $$\mathcal{G}':=(X, (S,\leq), \{C_x|_{U'_x}\}_{x\in S}, \{C_y|_{U'_y}\to C_x|_{U'_x}\}_{(y,x)\in I(\mathcal{G}')})$$ is called a \emph{shrinking} of another good coordinate system $$\mathcal{G}:=(X, (S,\leq), \{C_x|_{U_x}\}_{x\in S}, \{C_y|_{U_y}\to C_x|_{U_x}\}_{(y,x)\in I(\mathcal{G})}),$$ if $U'_x$ is a $G_x$-invariant open subset of $U_x$ with $x\in X_x|_{U'_x}$ for all $x\in S$, and for all $y, x\in S$, $X_y|_{U_y}\cap X_x|_{U_x}\not=\emptyset$ implies $X_y|_{U'_y}\cap X_x|_{U'_x}\not=\emptyset$ (hence $X_y|_{U_y}\cap X_x|_{U_x}\not=\emptyset$ is equivalent to $X_y|_{U'_y}\cap X_x|_{U'_x}\not=\emptyset$, so $I(\mathcal{G'})=I(\mathcal{G})$), and coordinate changes $C_y|_{U'_y}\to C_x|_{U'_x}$ among the restricted charts are restricted from coordinate changes in $\mathcal{G}$. We also say $U'_x, x\in S$ is a shrinking of a good coordinate system $C_x|_{U_x}, x\in S$, or $\mathcal{G}'$ is \emph{shrunken} from $\mathcal{G}$.
\end{definition}

We will introduce two useful properties of a good coordinate system.

\begin{definition}\label{STRONGLYINT}
Let $\mathcal{G}$ be a good coordinate system for a Kuranishi structure $\mathcal{K}$.
\begin{enumerate}
\item (identification-matching) $\mathcal{G}$ is said to be \emph{identification-matching} if $M(\mathcal{G})=M(\mathcal{K})|_{\mathcal{G}}$, where $M(\mathcal{K})|_{\mathcal{G}}$ is the identification space of the Kuranishi structure $\mathcal{K}$ restricted to the part corresponding to the good coordinate system.
\item (strongly intersecting) $\mathcal{G}$ is said to be \emph{strongly intersecting} if for all $y, x\in S$, the condition that $\underline{U_y}$ and $\underline{U_x}$ intersect in $M(\mathcal{K})|_{\mathcal{G}}$ implies that $X_y|_{U_y}\cap X_x|_{U_x}\not=\emptyset$.
\end{enumerate}
\end{definition}

\begin{remark} For a strongly intersecting good coordinate system $\mathcal{G}$, if $\underline{U_y}$ and $\underline{U_x}$ intersect in $M(\mathcal{G})$ and $y\leq x$, then there exists a coordinate change $C_y|_{U_y}\to C_x|_{U_x}$ from the data in the definition. A shrinking of a strongly intersecting good coordinate system is strongly intersecting by definition \ref{GCSSHRINKING}.
\end{remark}

In general $\mathcal{G}$ is neither identification-matching nor strongly intersecting, but as a nontrivial result \ref{COORDEXIST} shows, we can always find a shrinking $\mathcal{G}'$ of $\mathcal{G}$ that satisfies these two properties, as well as retaining the property of $\mathcal{G}$: $X_x|_{U'_x}$ is the interior of its closure $\overline{X_x|_{U'_x}}$ taken in $X$. Now we fix such a $\mathcal{G}'$.

We will show that we can choose a \emph{precompact} shrinking $\mathcal{G}'':=\{C_x|_{U''_x}\}_{x\in S}$ of $\mathcal{G}'$, which additionally requires that $\overline{U''_x}$ to be compact in $U'_x$. This requirement is for the following purpose: We want to choose a $U''_x\subset U'_x$ with a natural compactification in $U'_x$ such that every sequence escaping in infinity in $U''_x$ has a subsequence converges to a point in $U'_x$. So we can discuss ``boundary points of $U''_x$ at infinity'' as interior points in $U'_x$. The precompact shrinking allows us to define the relative topology. Moreover, we can always choose $\mathcal{G}''$ such that $M(\mathcal{G}'')$ can be equipped with a Hausdorff relative topology relative to $\mathcal{G}'$ (or equivalently relative to $\mathcal{G}$ or to $\mathcal{K}$). This topology is the topology we will explicitly use when we do constructions later.

\begin{theorem}\label{BIGHAUSDORFF} Let $\mathcal{K}$ be a Kuranishi structure. We can always choose a good coordinate system $\mathcal{G}$ for $\mathcal{K}$ in the sense of \ref{GCSFORKS} such that $\mathcal{G}$ is strongly intersecting and identification-matching, and $M(\mathcal{G})$ can be equipped with a relative topology $\mathcal{T}(\mathcal{G}, \mathcal{K})$ which is Hausdorff.
\end{theorem}

The theorem is the ultimate goal of this section. 

\subsection{Motivating examples}

We motivate the need for the strongly intersecting and identification-matching properties by discussions and examples.

In a good coordinate system $$(X, (S, \leq), \{C_x|_{U_x}\}_{x\in S}, \{C_y|_{U_y}\to C_x|_{U_x}\}_{y\leq x, X_y|_{U_y}\cap X_x|_{U_x}\not=\emptyset}),$$ by the definition, a coordinate change between two charts $C_y|_{U_y}$ and $C_x|_{U_x}$ exists if and only if $X_y|_{U_y}\cap X_x|_{U_x}\not=\emptyset$ for $y, x\in S$. This condition gives rise to two issues:
\begin{enumerate}[(a)]
\item We can have $C_z|_{U_z}\to C_y|_{U_y}$ and $C_y|_{U_y}\to C_x|_{U_x}$ in a good coordinate sysem such that one of them is created by inverting the restriction of a coordinate change in a Kuranishi structure, $U_{yz}\cap \phi_{yz}^{-1}(U_{xy})\not=\emptyset$ and $X_z|_{U_z}\cap X_x|_{U_x}=\emptyset$. So the composed coordinate change $\phi_{xy}\circ\phi_{yz}$ naturally arises but is not included in the data of coordinate changes. $\phi_{xy}\circ\phi_{yz}$ will play a role in identifying parts together, but it could be impossible to achieve the maximality condition for $\phi_{xy}\circ\phi_{yz}$ and satisfy the definition \ref{GCSCC} at the same time without further shrinking $U_x, x\in S$. We cannot induce $\phi_{xy}\circ\phi_{yz}$ from the Kuranishi structure giving rise to the good coordinate system either, because one of the $\phi_{yz}$ and $\phi_{xy}$ is obtained by inversion and does not exist in the Kuranishi structure.
\item When two charts from a good coordinate system intersect in $M\backslash X$, where $M$ is the identification space, we do not know whether there is effectively a direct coordinate change between the charts. Because nonzero points of sections from these two charts might be identified via coordinate changes from a common chart of lower order or to a common chart of higher order according to the order $\leq$, even though there is no direct coordinate changes between these two charts.
\end{enumerate}

An example to illustrate (a) and (b) is the following figure \ref{figure03}, and it will cause compatibility issue for further construction.

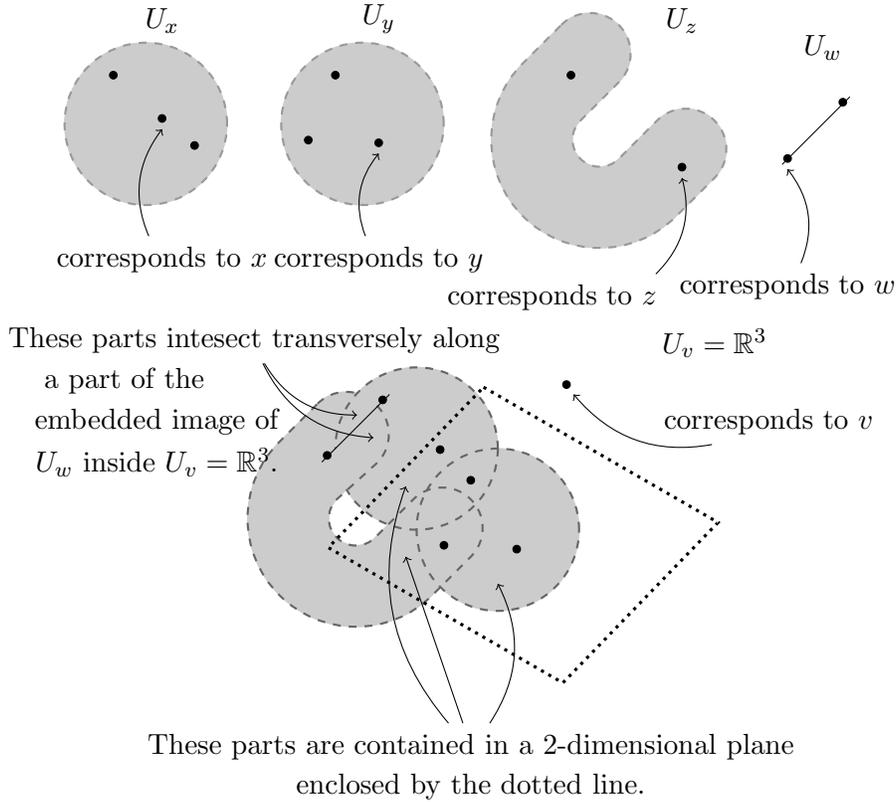
\begin{figure}[htb]
  \begin{center}

\begin{tikzpicture}[scale=36/50]
\hspace{0 cm}

\filldraw[color=black!20] (6,2) circle (1.5);
\draw[thick, dashed, color=black!40] (6,2) circle (1.5);
\filldraw[color=black!20] (10, 2) circle (1.5);
\draw[thick, dashed, color=black!40](10, 2) circle (1.5);

\filldraw (5.4,2.9) circle (2pt);
\filldraw (6.3,2.1) circle (2pt);
\path[->] node at (6.3, -0.5){corresponds to $x$} edge[bend left] (6.3, 1.9);
\filldraw (6.9,1.6) circle (2pt);

\filldraw (9.5, 2.9) circle (2pt);
\filldraw (9, 1.7) circle (2pt);
\filldraw (10.3, 1.65) circle (2pt);
\path[->] node at (10.3, -0.5){corresponds to $y$} edge[bend left] (10.3, 1.45);

\begin{scope}[shift={(0.5cm,-1cm)}]
\filldraw[color=black!20, rotate around={-45:(16,2)},](12,2) arc (180: 0: 0.75)  -- (13.5, 1) arc (180: 360: 0.5) -- (14.5, 2) arc (180:0:0.75) --(16,1) arc (0:-180:2)--(12,2);
\draw[thick, dashed, color=black!40][rotate around={-45:(16,2)},](12,2) arc (180: 0: 0.75)  -- (13.5, 1) arc (180: 360: 0.5) -- (14.5, 2) arc (180:0:0.75) --(16,1) arc (0:-180:2)--(12,2);
\end{scope}

\filldraw (13.85,2.9) circle (2pt);
\filldraw (15.9,1.2) circle (2pt);

\path[->] node at (13.5, -1.2){corresponds to $z$} edge[bend right] (15.9, 1);

\begin{scope}[shift={(8.5,5.5)}]
\draw (10.5, -3)--(9.25,-4.25);
\end{scope}

\filldraw (18.87,2.4) circle (2pt);
\filldraw (17.85,1.36) circle (2pt);
\path[->] node at (17.85, -1){corresponds to $w$} edge[bend right] (17.85, 1.16);

\node [above] at (6.3, 3.5){$U_x$};
\node [above] at (10.3, 3.5){$U_y$};
\node [above] at (15.9, 3.5){$U_z$};
\node [above] at (18.5, 3){$U_w$};

\node [above] at (16.5, -2.5){$U_v=\mathbb{R}^3$};

\path[->] node at (17.5, -3.5){corresponds to $v$} edge[bend left] (13.9, -3);

\filldraw[color=black!20] (11, -4) circle (1.5);

\filldraw[color=black!20] (12.5,-5.5) circle (1.5);

\begin{scope}[shift={(-4cm,-8cm)}]
\filldraw[color=black!20, rotate around={-45:(16,2)},](12,2) arc (180: 0: 0.75)  -- (13.5, 1) arc (180: 360: 0.5) -- (14.5, 2) arc (180:0:0.75) --(16,1) arc (0:-180:2)--(12,2);
\draw[thick, dashed, color=black!60, rotate around={-45:(16,2)}](12,2) arc (180: 118: 0.75)
(12.8,2.75) arc (90: 0: 0.75)  -- (13.5, 1) arc (180: 360: 0.5) -- (14.5, 2) arc (180:0:0.75) --(16,1) arc (0:-180:2)--(12,2);
\filldraw[rotate around={-45:(16,2)}] (12.8, 1.45) circle (2pt) (12.8, 2.9) circle (2pt) (15.5,1.8) circle (2pt) (15, 3) circle (2pt) (16.5, 2.7) circle (2pt) (14.2, 3) circle (2pt) (15, 5.5) circle (2pt);
\end{scope}
\draw[thick, dashed, color=black!60] (12.5, -4) arc (0: 180: 1.5) (12.5, -4) arc (0: -150: 1.5);

\draw[thick, dashed, color=black!60](12.5,-5.5) circle (1.5);
\draw (10.5, -3)--(9.25,-4.25);

\begin{scope}[shift={(4.2,2)}]
\draw[very thick, dotted, rotate around={-30:(12, -6)}] (8,-7) -- (13, -7)--(12,-11)--(7,-11)--(8,-7);
\end{scope}

\path[->] node (n) at (12, -9.5){These parts are contained in a 2-dimensional plane}
  edge (10.8, -6);
\node at (12,-10.25){enclosed by the dotted line.};
\path[->] (n) edge[bend left](10.8, -4.7);
\path[->] (n) edge[bend right] (12.5, -6.5);

\path[->] node (m) at (8, -2){These parts intesect transversely along} edge[bend right] (9.9, -3.38);
\path[->] (m) edge[bend right] (10.2, -3.8);
\node at (5.6,-2.75){a part of the};
\node at (6.2,-3.5){embedded image of};
\node at (6.2,-4.25){$U_{w}$ inside $U_v=\mathbb{R}^3$.};
\end{tikzpicture}

  \end{center}
\caption[An example to illustrate (a) and (b).]{An example to illustrate (a) and (b).}
\label{figure03}
\end{figure}

In the example in figure \ref{figure03}, we order points from $S:=\{x, y, z, w, v\}$ using the order of the dimensions of the bases. The points drawn in a given base are the points corresponding to the zeros of the section over that base. $X_z|_{U_z}\cap X_y|_{U_y}\not=\emptyset$, so there are coordinate changes between them in both ways. Similarly, $X_y|_{U_y}\cap X_x|_{U_x}\not=\emptyset$, so there are coordinate changes between them. On the other hand, $X_z|_{U_z}\cap X_x|_{U_x}=\emptyset$, so there is no coordinate change between $C_z|_{U_z}$ and $C_x|_{U_x}$ given in the data of the good coordinate system. However, there are effectively coordinate changes between $C_z|_{U_z}$ and $C_x|_{U_x}$ on a part of the right region of their intersection in $M(\mathcal{G})$ induced by $\phi_{xy}\circ\phi_{yz}$ and its inverse, and those naturally arise in the picture. We would need the maximality condition on this induced $\phi_{xz}$, but the domain $U_{xz}$ has to include a part of the 1-dimensional curve $\phi_{zw}(U_{zw})$ on the left region of the intersection, which is of course impossible: $U_{xz}$ then has to have a 1-dimensional component which is not open in $U_z$, and there is no way to have a direct coordinate change between $C_z|_{U_z}$ and $C_x|_{U_x}$ over this component of $U_{xz}$. This does not violate the maximality condition of the good coordinate system, since this coordinate change is not included in the data and therefore not restrained by the maximality condition. This explains issue (a), and the nature of this issue is similar to figure \ref{figure01}. We have seen that via $\phi_{xy}\circ\phi_{yz}$, parts of $C_z|_{U_z}$ and $C_x|_{U_x}$ are identified. Two more identifications occur via considering coordinate changes from both charts into $C_v|_{U_v}$ and coordinate changes from $C_w|_{U_w}$ into both charts. This explains issue (b).

Both issues (a) and (b) are resolved by the strongly intersecting property.

Another important observation is that the identification space $M(\mathcal{G})$ for the good coordinate system $\mathcal{G}$ for a Kuranishi structure $\mathcal{K}$ might not be the same as identification space $M(\mathcal{K})|_{\mathcal{G}}$ obtained from the identification space of the Kuranishi structure $\mathcal{K}$ by restricting to the part corresponding to the good coordinate system. See figure \ref{figure04}. This phenomenon must happen away from $X$, but can be arbitrarily near a point in $X$ (here we can describe the nearness in an ad hoc way using the topology of any single chart covering that point in $X$). This does not violate the maximality condition for the good coordinate system, as there are fewer coordinate changes in the good coordinate system by the definition.

\begin{figure}[htb]
  \begin{center}

\begin{tikzpicture}[scale=18/50]
\coordinate (P1) at (2,10);
\filldraw [color=black!20] (P1) circle (2);
\draw [thick, dashed, color=black!50] (P1) circle (2);
\filldraw (P1) circle (2pt);
\filldraw (3.5, 10) circle (2pt);
\coordinate[label=below: $V_x$] (V1) at (2,8); 
\node at (0,10) [left] {$\mathcal{K}$};

\coordinate (P2) at (7,10);
\filldraw [color=black!20] (P2) circle (2);
\draw [thick, dashed, color=black!50] (P2) circle (2);
\filldraw (P2) circle (2pt);
\filldraw (5.5, 10) circle (2pt);
\coordinate[label=below: $V_y$] (V2) at (7,8); 

\coordinate (P3) at (11,10);
\draw (10,10)--(12,10);
\filldraw (P3) circle (2pt);
\coordinate[label=below: $V_z$] (V3) at (11,9);

\coordinate (P4) at (2,3);
\filldraw [color=black!20] (P4) circle (2);
\filldraw (P4) circle (2pt);
\node at (0,3) [left] {$M(\mathcal{K})$};

\coordinate (P5) at (5,3);
\filldraw [color=black!20] (P5) circle (2 and 1);
\draw[thick, dashed, color=black!50] (3,3) arc (-180:128:2 and 1);
\draw[thick, dashed, color=black!50] (4,3) arc (0:335:2);
\filldraw (P5) circle (2pt);
\filldraw (3.5, 3) circle (2pt);

\coordinate (P6) at (3.5,3);
\draw (2.5,3)--(4.5,3);
\filldraw (P6) circle (2pt);

\coordinate (P7) at (18,6.5);
\filldraw [color=black!20] (P7) circle (2);
\draw[thick, dashed, color=black!50] (P7) circle (2);
\filldraw (19.5, 6.5) circle (2pt);
\filldraw [color=black!40] (P7) circle (0.8);
\draw [thick, dotted, color=black!70] (P7) circle (0.8);
\filldraw (P7) circle (2pt);
\path[->] node at (18, 3) {$U_x$} edge [bend left] (18, 6.2);

\coordinate (P8) at (23,6.5);
\filldraw [color=black!20] (P8) circle (2);
\draw [thick, dashed, color=black!50] (P8) circle (2);
\filldraw [color=black!40] (P8) circle (1.5);
\draw [thick, dotted, color=black!70] (P8) circle (1.5);
\filldraw (P8) circle (2pt);
\filldraw (21.5, 6.5) circle (2pt);
\path[->] node at (23, 3) {$U_y$} edge [bend right] (23, 6.0);

\node at (23,10) {$\mathcal{G}$ for $\mathcal{K}$};

\coordinate (P9) at (27,6.5);
\draw (26,6.5)--(28,6.5);
\filldraw (P9) circle (2pt);
\coordinate[label=below: ${U_z=V_z}$] (V9) at (27,5.5);

\node at (23, 1.5) {no coordinate change};

\coordinate (P10) at (18,-0.5);
\filldraw [color=black!40] (P10) circle (0.8);
\draw [thick, dotted, color=black!70] (P10) circle (0.8);
\filldraw (P10) circle (2pt);
\node at (16,-0.5) [left ]{$M(\mathcal{G})$};

\coordinate (P11) at (24,-0.5);
\filldraw [color=black!40] (P11) circle (1.5);
\draw [thick, dotted, color=black!70] (P11) circle (1.5);
\filldraw (P11) circle (2pt);

\coordinate (P12) at (20.5,-0.5);
\draw (19.5,-0.5)--(21.5,-0.5);
\filldraw (P12) circle (2pt);

\coordinate (P13) at (2,-5);
\filldraw [color=black!40] (P13) circle (0.8);
\draw [thick, dotted, color=black!70] (P13) circle (0.8);
\filldraw (P13) circle (2pt);
\node at (1,-5) [left] {$M(\mathcal{K})|_{\mathcal{G}}$};

\coordinate (P14) at (5,-5);
\filldraw [color=black!40] (P14) circle (1.5 and 0.75);
\draw [thick, dotted, color=black!70] (P14) circle (1.5 and 0.75);
\filldraw (P14) circle (2pt);

\coordinate (P15) at (3.5,-5);
\draw (2.5,-5)--(4.5,-5);
\filldraw (P15) circle (2pt);

\end{tikzpicture}

  \end{center}
\caption[Non-identification-matching.]{Non-identification-matching.}
\label{figure04}
\end{figure}
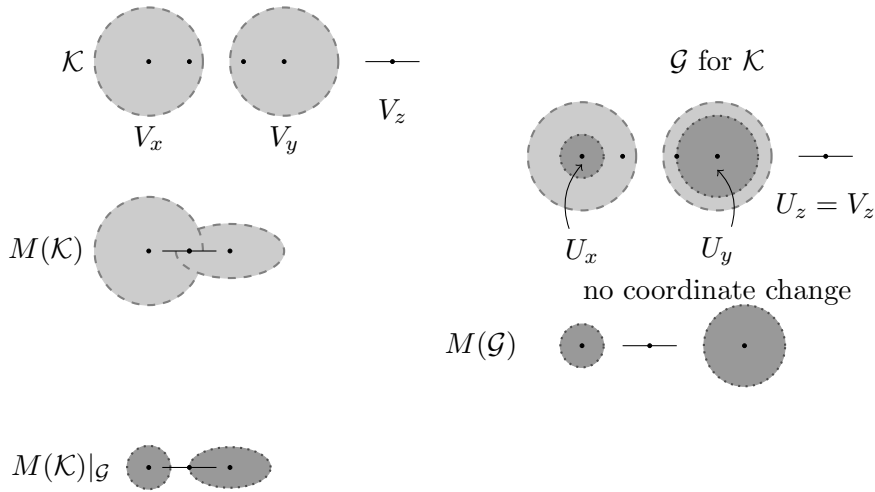

In the example in figure \ref{figure04}, the open regions depicted are the bases of Kuranishi charts, and inside each base the dots represent points that correspond to the zeros of the section over that base. We have coordinate changes $C_z\to C_x$ and $C_z\to C_y$ in the Kuranishi structure $\mathcal{K}$, since $z\in X_x\cap X_y$. But $X_x|_{U_x}=\{x\}$, $X_y|_{U_y}=\{y\}$ and $X_z|_{U_z}=\{z\}$ are disjoint, so $I(\mathcal{G})=\emptyset$ and there is no coordinate change in the good coordinate system $\mathcal{G}$. Therefore, the identification space $M(\mathcal{G})$ obtained by gluing bases via coordinate changes in $\mathcal{G}$ is just the disjoint union of those three bases. On the other hand, $M(\mathcal{K})|_{\mathcal{G}}$ is connected. The discrepancy of identification spaces occurs away from $X$ but is arbitrarily close to $z$. So in some sense, the good coordinate system $\mathcal{G}$ does not truly capture how $X$ is situated in the identification space $M(\mathcal{K})$ of the Kuranishi structure $\mathcal{K}$.

After resolving issues (a) and (b), we will be able to specify when we can reduce such the identification $\sim_{\mathcal{K}}$ to the identification via a single coordinate change in a good coordinate system, which is key to the understanding of how to obtain identification-matching and Hausdorffness.

\subsection{Strongly intersecting and Hausdorffness in each dimension}
Now we start to prove the Hausdorffness for each dimension through a sequence of results. First, we establish the strongly intersecting property that will resolve issues (a) and (b) simultaneously:

\begin{proposition}[strongly intersecting property]\label{COORDEXIST} Let $M(\mathcal{K})$ be the identification space (as a set) of the Kuranishi structure $\mathcal{K}$. Let $\mathcal{G}:=\{C_x|_{U_x}\}_{x\in S}$ be a good coordinate system obtained from $\mathcal{K}$ such that $X_x|_{U_x}$ is the interior of its closure $\overline{X_x|_{U_x}}$ taken in $X$, for all $x\in S$ (recall, by construction $X_x|_{U_x}=B_{r_x/2}(x)$). 

Then there exists a shrinking $\mathcal{G}'=\{C_x|_{U'_x}\}_{x\in S}$ of $\mathcal{G}$ in the sense of \ref{GCSSHRINKING}, such that for all $x\in S$, $x\in X_x|_{U'_x}=\text{interior}(\overline{X_x|_{U'_x}})\subset \overline{X_x|_{U'_x}}\subset X_x|_{U_x}$ (with the closure taken in $X$), and the following strongly intersecting property holds: If $\underline{U_x'}$ and $\underline{U_y'}$ for $x, y\in S$ intersect in $M(\mathcal{K})|_{\mathcal{G}'}$, then $X_x|_{U'_x}\cap X_y|_{U'_y}\not=\emptyset$ (so a direct coordinate change exists in $\mathcal{G}'$).
\end{proposition}

\begin{remark} In stating proposition \ref{COORDEXIST}, we need $X_x|_{U_x}$ to be the interior of its closure $\overline{X_x|_{U_x}}$, because it is a property readily available, and it also allows us to easily choose an invariant subset $U_x(1)\subset U_x$ which inherits this property and also satisfies $\overline{X_x|_{U_x(1)}}\subset X_x|_{U_x}$.
\end{remark}

We will prove this proposition and some results below constructively by considering special sequences of shrinkings:

\begin{definition}\label{STRONGSHRINKINGSEQ}
Suppose that a good coordinate system $\mathcal{G}:=\{C_x|_{U_x}\}_{x\in S}$ is constructed from a Kuranishi structure such that $X_x|_{U_x}$ is the interior of its closure $\overline{X_x|_{U_x}}$ taken in $X$ for all $x\in S$.
\begin{enumerate}
\item A \emph{sequence of shrinkings} $\mathcal{G}(k):=\{C_x|_{U_x(k)}\}_{x\in S}, k\in \mathbb{N}$ of $\mathcal{G}$ is a collection of shrinkings of $\mathcal{G}$ in the sense of \ref{GCSSHRINKING} such that $\mathcal{G}(k+1)$ is also a shrinking of $\mathcal{G}(k)$ for all $k\in\mathbb{N}$. We also say $U_x(k), x\in S, k\in\mathbb{N}$ is a \emph{sequence of shrinkings} of $\{C_x|_{U_x}\}_{x\in S}$.
\item A shrinking $\{C_x|_{U'_x}\}_{x\in S}$ of $\mathcal{G}:=\{C_x|_{U_x}\}_{x\in S}$ is said to be \emph{precompact} in $\mathcal{G}$ if for all $x\in S$, $\overline{U'_x}$ with the closure taken in $U_x$ is compact in $U_x$ (so in particular $\overline{X_x|_{U'_x}}\subset X_x|_{U_x}$ where the closure of $X_x|_{U'_x}$ is taken in $X$) and $X_x|_{U'_x}=\text{interior}(\overline{X_x|_{U'_x}})$.
\item A \emph{good shrinking sequence} of $\mathcal{G}$ is a sequence of shrinkings $$U_x(k), x\in S, k\in\mathbb{N}$$ of $\mathcal{G}$ such that $\{C_x|_{U_x(1)}\}_{x\in S}$ is precompact in $\mathcal{G}$, $X_x|_{U_x(k+1)}=X_x|_{U_x(k)}$ for all $x\in S$ and all $k\in\mathbb{N}$, and $\bigcap_k \overline{U_x(k)}=\overline{(s_x|_{U_x(1)})^{-1}(0)}$.
\item A \emph{strong shrinking sequence} of $\mathcal{G}$ is a sequence of shrinkings $$U_x(k), x\in S, k\in \mathbb{N}$$ of $\mathcal{G}$ such that $\{C_{U_x(1)}\}_{x\in S}$ is precompact in $\mathcal{G}$, $\{C_x|_{U_x(k+1)}\}_{x\in S}$ is a precompact shrinking for $C_x|_{U_x(k)}, x\in S$ for all $k\in\mathbb{N}$ and $\bigcap_k \overline{U_x(k)}\subset \overline{(s_x|_{U_x(1)})^{-1}(0)}$.
\end{enumerate}
\end{definition}

\begin{remark}
\begin{enumerate}
\item In definition \ref{STRONGSHRINKINGSEQ}.(3), we need $$X_x|_{U_x(1)}=\text{interior}(\overline{X_x|_{U_x(1)}})$$ to obtain $U'_x$ with the property $X_x|_{U'_x}=\text{interior}(\overline{X_x|_{U'_x}})$ in the conclusion of proposition \ref{COORDEXIST}.
\item A subsequence of a good shrinking sequence is still a good shrinking sequence, and a subsequence of a strong shrinking sequence is still a strong shrinking sequence.
\end{enumerate}
\end{remark}

\begin{lemma}\label{GSSE} Let $\{C_x|_{U_x}\}_{x\in S}$ be a good coordinate system with $$X_x|_{U_x}=\text{interior}(\overline{X_x|_{U_x}}),$$ then a precompact shrinking $\{C_x|_{U_x(1)}\}_{x\in S}$ for $\{C_x|_{U_x}\}_{x\in S}$ always exists. Moreover, we can also achieve that if $X_x|_{U_x(1)}\cap X_y|_{U_y(1)}=\emptyset$, then $$\overline{X_x|_{U_x(1)}}\cap \overline{X_y|_{U_y(1)}}=\emptyset,$$ for all pairs $x, y\in S$. In particular, a good sequence of $\{C_x|_{U_x}\}_{x\in S}$ always exists, and a strong sequence of $\{C_x|_{U_x}\}_{x\in S}$ always exists.
\end{lemma}

\begin{proof}
\begin{enumerate}
\item Choose a $G_x$-invariant Riemannian metric on $U_x$, which induces a metric\footnote{The metric is defined using the infimum of lengths of paths connecting two points that are connected, and $\infty$ if two points are in different connected components.} with the metric topology compatible with the manifold topology of $U_x$.

For each $x\in S$, we can choose an open subset $Y_x$ of $X$ such that $x\in Y_x\subset\overline{Y_x}\subset X_x|_{U_x}$, $X=\bigcup_{x\in S} Y_x$ and whenever $X_x|_{U_x}\cap X_y|_{U_y}\not=\emptyset$, we have $Y_x^o\cap Y_y^o\not=\emptyset$, where $Y_x^o:=\text{interior}(\overline{Y_x})$.

Then we have:
\begin{enumerate}[(i)]
\item $X_x|_{U_x}\cap X_y|_{U_y}\not=\emptyset$ if and only if $Y_x^o\cap Y_y^o\not=\emptyset$; and 
\item $Y^o_x\cap Y^o_y=\emptyset$ implies $\overline{Y_x}\cap \overline{Y_y}=\emptyset$.
\end{enumerate}

Indeed, the other direction in (i) is trivial. (ii) follows from (i), because if $\overline{Y_x}\cap\overline{Y_y}\not=\emptyset$, then intersection of bigger sets is not empty: $X_x|_{U_x}\cap X_y|_{U_y}\not=\emptyset$, which implies $Y^o_x\cap Y^o_y\not=\emptyset$.

Denote $quot_x: U_x\to \underline{U_x}$, and $Z_x:=quot_x^{-1}(\psi_x^{-1}(Y_x^o))\subset U_x$. Since $\overline{Y_x}$ is compact in $X_x|_{U_x}$, $\overline{Z_x}$ is a compact subset of $U_x$. 

For all $u\in \overline{Z_x}$, choose an open metric ball $B_{r_u}(u)\subset U_x$ with $r_u$ small enough such that the closure $\overline{B_{r_u}(u)}$ is compact in $U_x$ and $r_{gu}=r_u$ for all $u\in \overline{Z_x}, g\in G_x$. Since the metric is induced from a $G_x$-invariant Riemannian metric and $u\mapsto r_u$ is chosen to be $G_x$-invariant, we have $g(B_{r_u}(u))=B_{r_{gu}}(gu)$, for all $g\in G_x$. Define $O_{\underline{u}}:=\bigcup_{u\in\underline{u}} B_{r_u}(u)$, so $O_{\underline{u}}$ is invariant and its closure is compact in $U_x$. Consider $O_{\underline{u}}, \underline{u}\in \underline{\overline{Z_x}}=\psi_x^{-1}(\overline{Y_x})$ a cover for the compact $\overline{Z_x}$ where each open set in the cover is $G_x$-invariant, we can choose a finite cover indexed by $T_x$. Define $\hat U_x(1):=\bigcup_{\underline{u}\in T_x} O_{\underline{u}}$ and its closure is compact in $U_x$. Define $U_x(1):=\hat U_x(1)\backslash ((s_x|_{\hat U_x(1)})^{-1}(0)\backslash Z_x)$. Then $\overline{U_x(1)}\subset\overline{\hat U_x(1)}$ is compact in $U_x$, and $X_x|_{U_x(1)}=Y_x^o$. By (i) above, $\{C_x|_{U(1)_x}\}_{x\in S}$ is a shrinking of $\{C_x|_{U_x}\}_{x\in S}$ by definition \ref{GCSSHRINKING}, and (ii) implies the extra property for $U_x(1), x\in S$ in the conclusion.

\item For a good shrinking sequence: Use the metric on $U_x(1)$ induced from the one chosen for $U_x$ in part (i). For $w\in (s_x|_{U_x(1)})^{-1}(0)$, we can choose an open ball $B_{r_w}(w)\subset U_x(1)$, where the choice $w\mapsto r_w$ is $G_x$-invariant. For $k\geq 2$, define $$U_x(k):=\bigcup_{w\in (s_x|_{U_x(1)})^{-1}(0)}B_{r_w/k}(w).$$

\item For a strong shrinking sequence: The part (1) gives us a precompact shrinking $U_x(1), x\in S$ of $\{C_x|_{U_x}\}_{x\in S}$. We then repeat it with $\{C_x|_{U_x(1)}\}_{x\in S}$ in the place of $\{C_x|_{U_x}\}_{x\in S}$, using the metric on $U_x(1)$ induced from the one chosen for $U_x$ in part (i), and using some $r^{(2)}_u$ in the place of $r_u$ with the additional property that $r^{(2)}_u\leq r_u/2$ whenever $r^{(2)}_u$ is defined. This yields $U_x(2), x\in S$. In general, at the $(n+1)$th step, we yield $U_x(n+1), x\in S$ using $\{C_x|_{U_x(n)}\}_{x\in S}$ and $r^{(n+1)}_u\leq r^{(n)}_u/2$ whenever $r^{(n+1)}_u$ is defined, for all $n\in\mathbb{N}$.
\end{enumerate}
\end{proof}

\begin{lemma}\label{TOPOLOGYLEMMA} Suppose we have an open set $U$, a compact subset $Y\subset U$, and a sequence of open subsets $U_k\subset U$, with $U_{k+1}\subset U_k$, $\overline{U_1}$ compact in $U$, and $\bigcap_k \overline{U_k}\subset Y$, where the closures are taken in $U$. Suppose we choose a point $y_k$ in each $U_k$, then the sequence $y_k$ has a convergent subsequence converging to a point in $Y$. (For example, we will apply this to $U=U_x$, $U_k=U_x(k)$, and $Y=\overline{(s_x|_{U_x(1)})^{-1}(0)}$.)
\end{lemma}

\begin{proof}
The sequence $y_k$ in the compact $\overline{U_1}$ has a subsequnece $y_{k_j}, j\in\mathbb{N}$ converging to $y\in \overline{U_1}$. Since we definitely have $y_{k_j}, j\geq i$ lying in the closed $\overline{U_i}$, so $y\in \overline{U_i}$. So $y\in \bigcap_i\overline{U_i}\subset Y$.
\end{proof}

\begin{proof} (of proposition \ref{COORDEXIST} on the strongly intersecting property)

Applying lemma \ref{GSSE}, we can choose a good shrinking sequence $$U_x(k), x\in S, k\in\mathbb{N},$$ with the property that if $X_x|_{U_x(1)}\cap X_y|_{U_y(1)}=\emptyset$, $\overline{X_x|_{U_x(1)}}\cap \overline{X_y|_{U_y(1)}}=\emptyset$. We claim that for large enough $k$, $C_x|_{U_x(k)}, x\in S$ is the desired good coordinate system shrinking in the conclusion of proposition \ref{COORDEXIST}.

Suppose the claim is not true. Then, we have a subsequence $k_1<k_2<\cdots\to\infty$ such that for all $n\in \mathbb{N}$, there exists $x(k_n), y(k_n)\in S$, $w(k_n)\in U_{x(k_n)}(k_n)$ and $z(k_n)\in U_{y(k_n)}(k_n)$ with $\underline{w(k_n)}\sim_{\mathcal{K}} \underline{z(k_n)}$ but $X_{x(k_n)}|_{U_{x(k_n)}(k_n)}\cap X_{y(k_n)}|_{U_{y(k_n)}(k_n)}=\emptyset$. Then by a part of the definition of the good shrinking sequence being precompact, $$X_{x(k_n)}|_{U_{x(k_n)}(1)}\cap X_{y(k_n)}|_{U_{y(k_n)}(1)}=X_{x(k_n)}|_{U_{x(k_n)}(k_n)}\cap X_{y(k_n)}|_{U_{y(k_n)}(k_n)}=\emptyset,$$ therefore $\overline{X_{x(k_n)}|_{U_{x(k_n)}(1)}}\cap \overline{X_{y(k_n)}|_{U_{y(k_n)}(1)}}=\emptyset$.

Since $S$ is finite, there exists a subsequence $k'_n, n\in\mathbb{N}$ of $k_n, n\in \mathbb{N}$ such that $x(k'_n)=x$ and $y(k'_n)=y$ for some fixed $x, y\in S$ indepedent of $n$. Using the properties of the good shrinking sequence $$\{U_x(1)\}_{x\in S}, \{U_x(k_1')\}_{s\in S}, \{U_x(k_2')\}_{s\in S}, \cdots$$ and applying lemma \ref{TOPOLOGYLEMMA}, we have, for a subsequence $k_n''$ of $k_n'$, that $w(k_n'')$ converges in $U_x$ to $w\in \overline{(s_x|_{U_x(1)})^{-1}(0)}$ and $z(k_n'')$ converges in $U_y$ to $z\in \overline{(s_y|_{U_y(1)})^{-1}(0)}$.

In summary there exist a subsequence $\{k_n''\}_{n\in\mathbb{N}}$ and $x, y\in S$ such that
\begin{enumerate}
\item $\overline{X_x|_{U_x(1)}}\cap \overline{X_y|_{U_y(1)}}=\emptyset$, 
\item $\underline{U_{x}(k_n'')}$ and $\underline{U_y(k_n'')}$ intersect in $M(\mathcal{K})$, namely, they contain $\underline{w(k_n'')}$ and $\underline{z(k_n'')}$ respectively with $\underline{w(k_n'')}\sim_{\mathcal{K}} \underline{z(k_n'')}$,
\item $w(k_n'')\to w$ in $U_x$ and $z(k_n'')\to z$ in $U_y$, such that $$\underline{w}\in \underline{\overline{(s_x|_{U_x(1)})^{-1}(0)}}\subset \underline{(s_x|_{U_x})^{-1}(0)},\;\; \underline{z}\in \underline{\overline{(s_y|_{U_y(1)})^{-1}(0)}}\subset \underline{(s_y|_{U_y})^{-1}(0)}.$$
\end{enumerate}

By the topological matching condition of $\mathcal{K}$, $\underline{w}\sim \underline{z}$. Therefore, $$\psi_{x}(\underline{w})=\psi_{y}(\underline{z})\in \overline{X_x|_{U_x(1)}}\cap \overline{X_y|_{U_y(1)}}\not=\emptyset,$$ a contradiction to item (1) above.
\end{proof}

We have established proposition \ref{COORDEXIST} and obtained a strongly intersecting good coordinate system shrinking $\mathcal{G}'$. Now we can rule out issue (a) for $\mathcal{G}'$:

\begin{corollary} For a strongly intersecting good coordinate system $$\mathcal{G}':=(X, \{C_x|_{U_x'}\}_{x\in S}, \{C_y|_{U'_y}\to C_x|_{U'_x}\}_{(y,x)\in I(\mathcal{G}')}),$$ the set of coordinate changes indexed by $I(\mathcal{G}')$ is closed under the composition. Namely, if $(z, y), (y, x)\in I(\mathcal{G}')$ such that $$(\phi_{yz}, \hat\phi_{yz}, U'_{yz})\text{ and }(\phi_{xy}, \hat\phi_{xy}, U'_{xy})$$ are composable ($U'_{yz}\cap \phi_{yz}^{-1}(U'_{xy})\not=\emptyset$), then $(z, x)\in I(\mathcal{G}')$, so the direct coordinate change is already included in the data (with the maximality condition satisfied).
\end{corollary}
\begin{proof} Since $U'_{yz}\cap \phi_{yz}^{-1}(U'_{xy})\not=\emptyset$, $U'_z$ and $U'_x$ intersect in $M(\mathcal{K})$. Then by the strongly intersecting property, we have $X_z|_{U_z'}\cap X_x|_{U'_x}\not=\emptyset$. By transitivity, $z\leq x$. Therefore, $(z,x)\in I(\mathcal{G}')$.
\end{proof}

We have the following stronger condition equivalent to the existence condition for coordinate changes in $\{C_x|_{U'_x}\}_{x\in S}$, which answers issue (b):

\begin{corollary}\label{GCSEQUIVCCEXIST} Suppose $\mathcal{G}':=\{C_x|_{U_x'}\}_{x\in S}$ is a strongly intersecting good coordinate system for a Kuranishi structure $\mathcal{K}$. Then $\underline{U'_x}$ and $\underline{U'_y}$ for $x, y\in S$ intersect in $M(\mathcal{K})$ if and only if there is already a coordinate change given in the data of $\mathcal{G}'$ between charts $C_x|_{U'_x}$ and $C_y|_{U'_y}$ in the direction dictated by the order $\leq$.
\end{corollary}

\begin{proof} If $\underline{U'_x}$ and $\underline{U'_y}$ for $x, y\in S$ intersect in $M$, then $X_x|_{U'_x}\cap X_y|_{U'_y}\not=\emptyset$. Then there is a coordinate change in the direction of $\leq$, by the definition of a good coordinate system.

The other direction is trivial, since by the definition, the domain of a coordinate change is non-empty.
\end{proof}

\begin{corollary}[Being strongly intersecting is identification-matching]\label{SETMATCHING} Let $\mathcal{K}$ be a Kuranishi structure with the identification space $M(\mathcal{K})$. We can obtain a strongly intersecting good coordinate system $\mathcal{G}$ by proposition \ref{COORDEXIST}, and we can define the identification space $M(\mathcal{G})$ by gluing using coordinate changes in $\mathcal{G}$. We can also use the inclusion of $\mathcal{G}$ into $\mathcal{K}$ to define the restriction $M(\mathcal{K})|_{\mathcal{G}}$. Then $M(\mathcal{G})=M(\mathcal{K})|_{\mathcal{G}}$ as sets. Namely, points from charts in $\mathcal{G}$ getting identified in $\mathcal{K}$ are already identified in $\mathcal{G}$ using coordinate changes in $\mathcal{G}$.
\end{corollary}

\begin{proof} Using the definitions and corollary \ref{GCSEQUIVCCEXIST}.
\end{proof}

So far, we showed that:

\begin{theorem}\label{STRONGGCS} Let $\mathcal{K}:=\{C_x\}_{x\in X}$ be a Kuranishi structure. There exists a good coordinate system $\mathcal{G}:=\{C_x|_{U_x}\}_{x\in S}$ for $\mathcal{K}$ such that 
\begin{enumerate}
\item (identification-matching) $M(\mathcal{G})=M(\mathcal{K})|_{\mathcal{G}}$,
\item (strongly intersecting) if $y, x\in S$ with $y\leq x$ and $\underline{U_x}$ and $\underline{U_x}$ intersect in $M(\mathcal{G})$ (same as $M(\mathcal{K})|_{\mathcal{G}}$), then there exists $C_y|_{U_y}\to C_x|_{U_x}$ already given in the data of $\mathcal{G}$ (namely, $(y,x)\in I(\mathcal{G})$ as in \ref{LEVEL0GCS}), and
\item $X_x|_{U_x}$ is the interior of $\overline{X_x|_{U_x}}$ with the closure taken in $X$.
\end{enumerate}
\end{theorem}

\subsection{The relative topology}

Suppose now $\mathcal{G}:=\{C_x|_{U_x}\}_{x\in S}$ is a good coordinate system that satisfies the conclusion of theorem \ref{STRONGGCS}. We choose a precompact shrinking $\mathcal{G}':=\{C_x|_{U'_x}\}_{x\in S}$ for $\mathcal{G}$. We can define a topology $\mathcal{T}(\mathcal{G}',\mathcal{G})$ on the identification space $M(\mathcal{G}')$ as follows.

For any point $z\in M(\mathcal{G}')$, by the definition of $M(\mathcal{G}')$, there exists a possibly non-unique $x^0\in S$ with the largest order such that $z\in \iota_{x^0}(\underline{U'_{x^0}})$ where $\iota_{x^0}:\underline{U_{x^0}}\to M(\mathcal{G}')$. We define $i_z:=\text{dim} E_{x^0}$, and observe that it is well-defined independent of $x^0$ chosen for $z$ (since there exists an invertible coordinate change between charts centered at any two such $x^0$ in $\mathcal{G}'$, by the strongly intersecting property). 

We also observe that there exists a possibly non-unique $x^+\in S$ with the largest order such that $z\in\overline{\underline{U'_{x^+}}}$ where the closure is taken in $\underline{U_{x^+}}$ (and is compact in $\underline{U_{x^+}}$). The existence of a choice of $x^+$ follows by considering the set $J_z:=\{x\in S \;|\; z\in \iota_{x}(\underline{\overline{U'_x}})\}$, where the closure is always taken in $\underline{U_x}$. Since $J_z$ is non-empty (a choice $x^0\in J_z$), it has a largest element, a choice of which is $x^+$. We define $j_z:=\text{dim} E_{x^+}$, and observe that it is well-defined independent of $x^+$ chosen for $z$ (since there exists an invertible coordinate change between charts centered at any two such $x^+$ in $\mathcal{G}$, by the strongly intersecting property). 

Since by the definition $i_z\leq j_z$, we have either $i_z=j_z$, or $i_z<j_z$.

\begin{definition}\label{JUMPINGPOINT}
Let $z\in M(\mathcal{G}')$, $z$ is said to be an \emph{interior point}, if $i_z=j_z$; and a \emph{jumping point} if $i_z<j_z$. The notions of an interior point and a jumping point depend on $\mathcal{G}'$.
\end{definition}

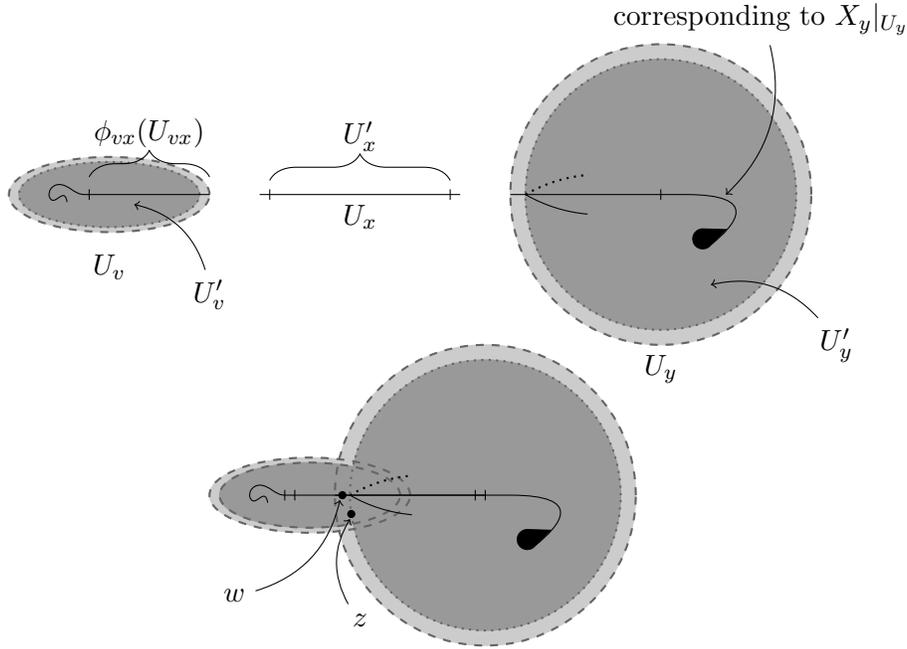
\begin{figure}[htb]
  \begin{center}

\begin{tikzpicture}[scale=2/3]
\filldraw[black!20] (0,6) circle (2 and 0.75);
\draw[thick, dashed, black!60] (0,6) circle (2 and 0.75);
\filldraw[black!40] (0,6) circle (1.8 and 0.65);
\draw[thick, dotted, black!60] (0,6) circle (1.8 and 0.65);
\draw (2,6)--(-0.5,6)..controls (-0.7, 6) and (-.8, 6.2).. (-1, 6.2) arc (90:190:0.2)..controls (-1.1,5.7) and (-0.9, 6.2)..(-0.84,5.85);
\draw (-0.4,6) arc (180:178:3) (-0.4,6)arc (180:182:3);

\draw [decorate,decoration={brace,amplitude=10pt}, xshift=0, yshift=10pt]
(-.4,6)--(2,6) node [midway,yshift=7.5pt, above]{$\phi_{vx}(U_{vx})$};

\node at (0,5)[below]{$U_v$};
\path[->]node at (2,4){$U'_v$} edge [bend right] (0.5,5.8);

\draw (3,6)--(7,6);
\draw (3.2,6) arc (180:178:3) (3.2,6)arc (180:182:3);
\draw (6.8,6) arc (0:2:3) (6.8,6)arc (0:-2:3);

\draw [decorate,decoration={brace,amplitude=10pt}, xshift=0, yshift=8pt]
(3.2,6)--(6.8,6) node [midway,yshift=7.5pt, above]{$U'_x$};
\node at ({(3.2+6.8)/2},6)[below]{$U_x$};

\filldraw[black!20] (11,6) circle (3);
\draw[thick, dashed, black!60] (11,6) circle (3);
\filldraw[black!40] (11,6) circle (2.7);
\draw[thick, dotted, black!60] (11,6) circle (2.7);
\draw (8,6)--(11.5,6).. controls (13,6) and (12.5,5.4).. (12,5);
\filldraw(12,5) arc (-35: -270:0.2)--(12.3,5.3)--(12,5);
\draw (8.3,6) arc (-120:-95: 3);
\draw[thick, dotted] (8.3,6) arc (120:95:3);

\draw (11,6) arc (0:3:2) (11,6) arc (0:-3:2);

\node at (11,3)[below]{$U_y$};
\path[->]node at (14.5,3){$U'_y$} edge [bend right] (12,4.2);
\path[->]node at (13,9.5){corresponding to $X_y|_{U_y}$} edge [bend left] (12.3,6.0);

\filldraw[black!20] (4,0) circle (2 and 0.75);
\filldraw [black!20](7.5,0) circle (3);
\filldraw [black!40](7.5,0) circle (2.7);
\draw[thick, dashed, black!60] (4,0) circle (2 and 0.75);
\filldraw[black!40] (4,0) circle (1.8 and 0.65);
\draw[thick, dashed, black!60] (4,0) circle (1.8 and 0.65);
\draw (3.5,0)--(7.5,0);
\draw (3.7,0) arc (180:178:3) (3.7,0)arc (180:182:3);
\draw (7.3,0) arc (0:2:3) (7.3,0)arc (0:-2:3);

\draw (3.5,0) arc (180:178:3) (3.5,0) arc (180:182:3);
\draw (7.5,0) arc (0:3:2) (7.5,0) arc (0:-3:2);

\draw (4.8,0) arc (-120:-95: 3);
\draw[thick, dotted] (4.8,0) arc (120:95:3);

\begin{scope}[shift={(-3.5,-6)}]
\draw (8,6)--(11.5,6).. controls (13,6) and (12.5,5.4).. (12,5);
\filldraw(12,5) arc (-35: -270:0.2)--(12.3,5.3)--(12,5);
\end{scope}
\begin{scope}[shift={(4,-6)}]
\draw(-0.5,6)..controls (-0.7, 6) and (-.8, 6.2).. (-1, 6.2) arc (90:190:0.2)..controls (-1.1,5.7) and (-0.9, 6.2)..(-0.84,5.85);
\end{scope}

\draw[thick,dashed, black!60] (7.5,0) circle (3);
\draw[thick,dotted,black!60] (7.5,0) circle (2.7);

\filldraw(4.65,0) circle (2pt);
\path (4.8,0) arc (180: 188: 2.7) coordinate (x);
\filldraw (x) circle (2pt);

\path[->]node at (2.5,-2){$w$} edge [bend right] (4.6,-.1);
\path[->]node at (5,-2.5){$z$} edge [bend left] (4.8,-.48);

\node at (4,-4) {$w$ is an interior point, $i_w=j_w=2$.};
\node at (4.4,-4.75) {$z$ is a jumping point, $i_z=2<j_z=3$.};
\end{tikzpicture}

  \end{center}
\caption[An interior point and a jumping point.]{An interior point and a jumping point.}
\label{figure05}
\end{figure}

For an interior point $z$, an open basis of $z$ in $M(\mathcal{G}')$ is defined to be induced (quotiented by $\sim$ and followed by $\iota_{x^0}$) from an open basis of $\iota_{x^0}^{-1}(z)$ in $$D'_{x^0}:=\underline{U'_{x^0}\backslash \bigcup_{x^0\leq y,\;y\not\leq x^0,\;X_{x^0}|_{U'_{x^0}}\cap X_y|_{U'_y}\not=\emptyset}\overline{U'_{y x^0}}},$$ where the closure is taken in $U_{x^0}$. Notice that this set $D'_{x^0}$ induces a set in $M(\mathcal{G}')$ consisting of interior points. Such a definition of a basis for $z$ is compatible with another basis defined using different basis of $z$ in $D'_{x^0}$ or using different $x^0$.

For a jumping point $z$, by the definition of a jumping point, $i_z<j_z$, so $$z\in \iota_{x^+}(\overline {\underline{U'_{x^+}}}\backslash \underline{U'_{x^+}})\text{ but }z\not\in \iota_{y}(\underline{\overline{U'_y}})\text{ for all }y\text{ such that }x^+\leq y, y\not\leq x^+.$$ Since $\underline{\overline{U'_{x^+}}}$ is a compact subset in $\underline{U_{x^+}}$, $z\in\iota_{x^+}(\underline{U_{x^+}})$. We can choose an open basis $O_m$ of $\iota_{x^+}^{-1}(z)$ in $$D_{x^+}:=\underline{U_{x^+}\backslash \bigcup_{x^+\leq y,\;y\not\leq x^+,\;X_{x^+}|_{U_{x^+}}\cap X_y|_{U'_y}\not=\emptyset}\overline{\tilde U_{y x^+}}},$$
where $\tilde U_{y x^+}$ is the domain of coordinate change of $C_{x^+}|_{U_{x^+}}\to C_y|_{U'_y}$ determined by the maximality condition (or simply by restriction). We can intersect $\iota_{x^+}(\underline{O_m})$ with $M(\mathcal{G}')$ in $M(\mathcal{G})$ to define an open basis of $z$ in $M(\mathcal{G}')$. Any two such choices (using different basis $O_m$ or different $x^+$) are compatible as well. See figure \ref{figure06}.

\begin{figure}[htb]
  \begin{center}

\begin{tikzpicture}[scale=.7]
\hspace{0cm},

\draw [very thick, dashed] (6.5,8) arc (0:20:10) coordinate (x0) (6.5,8) arc (0:-20:10) coordinate (x);
\draw [very thick, dashed] (x) arc (200:180:10) arc (180:160:10) coordinate (y); 

\begin{scope}[shift={(5,0)}]
\draw [very thick, dashed] (6.5,8) arc (0:20:10) coordinate (y0) (6.5,8) arc (0:-20:10) coordinate (y);
\end{scope}

\draw [very thick, dashed] (x)--(y) (x0)--(y0);

\filldraw[black!20] (6.5,8) arc (0:20:10)--(y0) arc (20:-20:10)--(x) arc (-20:0:10);

\draw [blue, very thick, dotted] (6.8,8) arc (0:18:10) coordinate (z0) (6.8,8) arc (0:-18:10) coordinate (z);
\draw [blue, very thick, dotted] (z) arc (198:180:10) arc (180:162:10) coordinate (u); 

\begin{scope}[shift={(4.4,0)}]
\draw [blue, very thick, dotted](6.8,8) arc (0:18:10) coordinate (u0) (6.8,8) arc (0:-18:10) coordinate (u);
\end{scope}

\filldraw[black!40] (6.8,8) arc (0:18:10)--(u0) arc (18:-18:10)--(z) arc (-18:0:10);

\draw[thick, dashed] (6.5,8) arc (0:20:10)--(y0) arc (20:-20:10)--(x) arc (-20:0:10);

\filldraw[black!50] (z0) arc (18:-18:10) arc (198:162:10);

\filldraw[black!30] (x0) arc (20:-20:10) arc (200:160:10);

\filldraw[black!20] (2,6) rectangle (5.9, 10);

\filldraw[black!40] (2.5,6.5) rectangle (5.9, 9.5);

\draw[very thick, dashed] (2,6) rectangle (8, 10);

\draw[blue, very thick, dotted] (2.5,6.5) rectangle (7.5, 9.5);

\draw [blue, very thick, dotted](z)--(u) (z0)--(u0);

\draw[very thick, dashed] (6.5,8) arc (0:20:10)--(y0) arc (20:-20:10)--(x) arc (-20:0:10);

\draw[blue,very thick, dotted] (6.8,8) arc (0:18:10)--(u0) arc (18:-18:10)--(z) arc (-18:0:10);

\draw (0,8)--(3.5,8);
\draw[blue, very thick, dotted] (0.5,8)--(3,8);

\draw [decorate,decoration={brace,amplitude=10pt}, xshift=0, yshift=2pt]
(0,8)--(3.5,8) node [midway,yshift=7.5pt, above]{$U_v$};

\draw [decorate,decoration={brace,amplitude=10pt}, xshift=0, yshift=-6pt]
(3,8)--(0.5,8) node [midway,yshift=-7.5pt, below]{$U'_v$};

\filldraw (2.5, 8) circle (2pt);

\node at(2.3, 7.7) {$z$};
\node at (4, 7.4){$U'_x$};
\node at (4,6.2){$U_x$};
\node at (10, 7.4){$U'_y$};
\node at (10,4.5){$U_y$};

\begin{scope}[shift={(0,-8)}]

\draw[very thick, dashed] (2,6) rectangle (8, 10);

\filldraw[yellow] (2,6) rectangle (6.3, 10);

\draw[blue, very thick, dotted] (2.5,6.5) rectangle (7.5, 9.5);

\filldraw[black!30] (2.5, 8) circle (.3);
\node at(2.5, 7.7) {$z$};
\filldraw (2.5, 8) circle (2pt);

\draw[blue, very thick, dotted] (6.8,8) arc (0:18:10) coordinate (a1) (6.8,8) arc (0:-18:10) coordinate(b1) (11.2,8) arc (0:18:10) coordinate (c1) (11.2,8) arc (0:-18:10) coordinate (d1);
\draw[blue, very thick, dotted] (a1)--(c1) (b1)--(d1);
\draw[blue, very thick, dotted] (a1) arc (162:198:10);
\end{scope}

\node at(6.5, -3.5) {Here $x^+=x$. The yellow area is $D_{x^+}$.};
\node at (6.5,-4.25){The grey disk is an element of the basis of $\iota_{x^+}^{-1}(z)$ in $D_{x^+}$.};
\begin{scope}[shift={(0,-0.75)}]
\filldraw[black!30] (2.5,-5) arc (-90:90:0.3) --(2.5,-5);
\filldraw (2.5,-4.7) circle (2pt);
\draw (2.5,-4.7)--(2.2, -4.7);
\draw (2.2, -4.7) circle (2pt);

\node at(2.5, -5) {$z$};

\node at(8, -5.5){This is an element of the basis of $z$ in $(M(\mathcal{G}'),\mathcal{T}(\mathcal{G}',\mathcal{G}))$.};
\end{scope}
\end{tikzpicture}

  \end{center}
\caption[An element of an basis at a jumping point in the relative topology.]{An element of an basis at a jumping point in the relative topology.}
\label{figure06}
\end{figure}

From the form of $D'_{x^0}$ and $D_{x^+}$, it is obvious that they generate a topology.

\begin{definition}[relative topology]\label{RELATIVETOP}
An open set in $M(\mathcal{G}')$ is a union of elements from bases of various points in $M(\mathcal{G}')$. This defines the \emph{relative topology} $\mathcal{T}(\mathcal{G}',\mathcal{G})$ for $M(\mathcal{G}')$.
\end{definition}

\begin{remark}\label{TWOCASESISONE} We have three important remarks to make:
\begin{enumerate}
\item After defining everything, we now observe that no matter what kind of point $z$ is, we can always use $D_{x^+}$ to induce an open basis.
\item It is called the relative topology because a basis of a jumping point $z$ for the topology $\mathcal{T}(\mathcal{G}', \mathcal{G})$ of $M(\mathcal{G}')$ is defined in a dimension-constant background $\iota_{x^+}(\underline{D_{x^+}})\subset M(\mathcal{G})$ as a relatively open basis. 
\item Notice that $\mathcal{T}(\mathcal{G}',\mathcal{G})$ is completely different from the quotient topology on $M(\mathcal{G}')$. See figure \ref{figure07}. $\mathcal{T}(\mathcal{G}',\mathcal{G})$ is a relative topology of the quotient topology on $M(\mathcal{G})$. This topology $\mathcal{T}(\mathcal{G}',\mathcal{G})$ is what we will use explicitly in the constructions in the next chapter.
\end{enumerate}
\end{remark}

\begin{figure}[htb]
  \begin{center}

\begin{tikzpicture}[scale=3/5]
\hspace{0 cm}

\draw (2,6)--(7,6);
\draw (2.5,6) arc (180:178:3) (2.5,6)arc (180:182:3);
\draw (3.2,6) arc (180:178:3) (3.2,6)arc (180:182:3);
\draw (6.5,6) arc (0:2:3) (6.5,6)arc (0:-2:3);

\draw [decorate,decoration={brace,amplitude=10pt}, xshift=0, yshift=8pt]
(3.2,6)--(7,6) node [midway,yshift=7.5pt, above]{$s_y^{-1}(0)$};

\draw [decorate,decoration={brace,amplitude=10pt}, xshift=0, yshift=-12pt]
(6.5,6)--(2.5,6) node [midway,yshift=-7.5pt, below]{$U'_y$};
\node at ({(2.5+6.5)/2},6)[below]{$U_y$};

\filldraw[black!20] (11,6) circle (3);
\draw[thick, dashed, black!60] (11,6) circle (3);
\filldraw[black!40] (11,6) circle (2.5);
\draw[thick, dotted, black!60] (11,6) circle (2.5);
\draw (8,6)--(11.5,6).. controls (13,6) and (12.5,5.4).. (12,5);
\filldraw(12,5) arc (-35: -270:0.2)--(12.3,5.3)--(12,5);

\draw (11,6) arc (0:3:2) (11,6) arc (0:-3:2);

\node at (11,3)[below]{$U_x$};
\path[->]node at (14.5,3){$U'_x$} edge [bend right] (12,4.2);
\path[->]node at (13,9.5){$s_x^{-1}(0)$} edge [bend left] (12.3,6.0);

\filldraw [black!20](7.5,-1) circle (3);
\filldraw [black!40](7.5,-1) circle (2.5);

\draw[very thick, dotted](5,-1) circle (0.4);
\begin{scope}
\clip (7.5,-1) circle (2.5);
\clip (5,-1) circle (.4);
\filldraw[black!70] (7.5,-1) circle (2.5) (5,-1) circle (.4);
\end{scope}
\draw[very thick](5,-1)--(4.6,-1);

\draw (2.5,-1)--(7.5,-1);
\draw (3,-1) arc (180:178:3) (3,-1)arc (180:182:3);
\draw (3.7,-1) arc (180:178:3) (3.7,-1)arc (180:182:3);
\draw (7,-1) arc (0:2:3) (7,-1)arc (0:-2:3);
\draw (7.5,-1) arc (0:2:3) (7.5,-1)arc (0:-2:3);

\begin{scope}[shift={(-3.5,-7)}]
\draw (8,6)--(11.5,6).. controls (13,6) and (12.5,5.4).. (12,5);
\filldraw(12,5) arc (-35: -270:0.2)--(12.3,5.3)--(12,5);
\end{scope}

\draw[thick,dashed, black!60] (7.5,-1) circle (3);
\draw[thick,dotted,black!60] (7.5,-1) circle (2.5);

\filldraw(5,-1) circle (2pt);

\path[->]node at (4,-2){$z$} edge [bend right] (4.9,-1.2);

\node at (7,-4.5) {relative topology};
\node at (16,-4.5) {quotient topology};

\node at (7,-7) {an open neighborhood of $z$};
\node at (7,-7.75) {in $M(\mathcal{G}')$ with};
\node at (7,-8.5) {the relative topology};

\node at (16,-7) {an open neighborhood of $z$};
\node at (16,-7.75) {in $M(\mathcal{G}')$ with};
\node at (16,-8.5){the quotient topology};

\begin{scope}[shift={(0,-5)}]
\filldraw(5,-1) circle (2pt);

\draw[very thick, dotted](5,-1) circle (0.4);
\begin{scope}
\clip (7.5,-1) circle (2.5);
\clip (5,-1) circle (.4);
\filldraw[black!70] (7.5,-1) circle (2.5) (5,-1) circle (.4);
\end{scope}

\draw[very thick](5,-1)--(4.6,-1);
\end{scope}

\begin{scope}[shift={(9,0)}]
\filldraw [black!20](7.5,-1) circle (3);
\filldraw [black!40](7.5,-1) circle (2.5);
\begin{scope}[shift={(-9,0)}]
\draw[very thick,dotted] (14,-1).. controls (14.5,-1) and (14.7,-0.5).. (15,-0.5)arc (90:-90:0.5).. controls (14.7,-1.5) and (14.5, -1)..(14,-1);
\filldraw[black!70] (14,-1).. controls (14.5,-1) and (14.7,-0.5).. (15,-0.5)arc (90:-90:0.5).. controls (14.7,-1.5) and (14.5, -1)..(14,-1);
\draw[very thick](13.6,-1)--(14,-1);
\end{scope}

\draw (2.5,-1)--(7.5,-1);
\draw (3,-1) arc (180:178:3) (3,-1)arc (180:182:3);
\draw (3.7,-1) arc (180:178:3) (3.7,-1)arc (180:182:3);
\draw (7,-1) arc (0:2:3) (7,-1)arc (0:-2:3);
\draw (7.5,-1) arc (0:2:3) (7.5,-1)arc (0:-2:3);

\begin{scope}[shift={(-3.5,-7)}]
\draw (8,6)--(11.5,6).. controls (13,6) and (12.5,5.4).. (12,5);
\filldraw(12,5) arc (-35: -270:0.2)--(12.3,5.3)--(12,5);
\end{scope}

\draw[thick,dashed, black!60] (7.5,-1) circle (3);
\draw[thick,dotted,black!60] (7.5,-1) circle (2.5);

\filldraw(5,-1) circle (2pt);

\path[->]node at (4,-2){$z$} edge [bend right] (4.9,-1.2);
\end{scope}

\begin{scope}[shift={(0,-5)}]
\filldraw (14,-1) circle (2pt);
\draw[very thick,dotted] (14,-1).. controls (14.5,-1) and (14.7,-0.5).. (15,-0.5)arc (90:-90:0.5).. controls (14.7,-1.5) and (14.5, -1)..(14,-1);
\filldraw[black!70] (14,-1).. controls (14.5,-1) and (14.7,-0.5).. (15,-0.5)arc (90:-90:0.5).. controls (14.7,-1.5) and (14.5, -1)..(14,-1);
\draw[very thick](13.6,-1)--(14,-1);
\end{scope}
\end{tikzpicture}

  \end{center}
\caption[The relative topology versus the quotient topology.]{The relative topology versus the quotient topology.}
\label{figure07}
\end{figure}
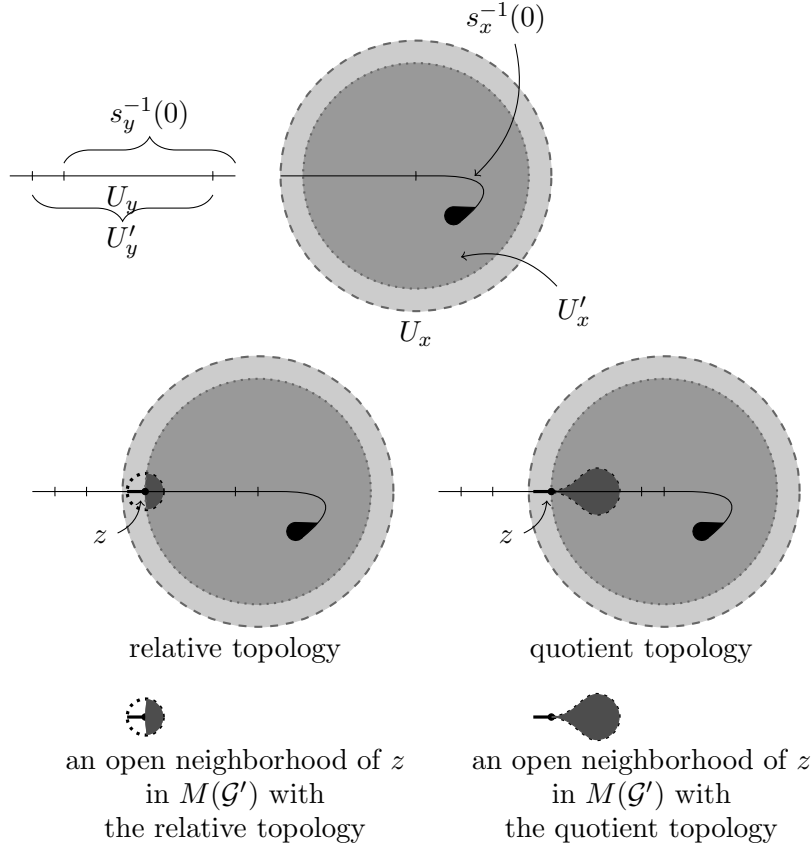

\subsection{Obtaining Hausdorffness from a strong shrinking sequence}

Now we have obtained the strongly intersecting property for a shrinking $\mathcal{G}$ of a starting $\tilde{\mathcal{G}}$ (using a good shrining sequence), established the existence of a strong shrinking sequence $\mathcal{G}(k), k\in\mathbb{N}$ of $\mathcal{G}$ and introduced the relative topology $\mathcal{T}(\mathcal{G}(k), \mathcal{G}(k'))$ for $M(\mathcal{G}(k))$, for $k'<k$; and we are ready to prove the Hausdorffness of $(M(\mathcal{G}(k)),\mathcal{T}(\mathcal{G}(k),\mathcal{G}(k')))$ for $k$ large enough.

First we establish a convenient lemma that will be used in the main proof.
\begin{lemma}\label{IDENTIFY} Let $\{C_x|_{U_x'}\}_{x\in S}$ be a good coordinate system. Then if a sequence $u_n$ in the quotiented domain $\underline{U'_{x_2 x_1}}\subset \underline{U'_{x_1}}$ of the coordinate change converges to $y$ in $\underline{(s_{x_1}|_{U'_{x_1}})^{-1}(0)}$, such that $\underline{\phi_{x_2 x_1}}(u_n)$ converges to $z$ in $\underline{(s_{x_2}|_{U'_{x_2}})^{-1}(0)}$, then $y\in \underline{U'_{x_2 x_1}}$, and $\underline{\phi_{x_2 x_1}}(y)=z$.
\end{lemma}

\begin{proof} By the topological matching condition, we have $y\sim z$. By the maximality condition, $y\in \underline{U_{x_2 x_1}}$. Then because $\underline{\phi_{xy}}(y)$ and $z$ are the limits in $\underline{U_{x_2}}$ of the same converging sequence $\underline{\phi_{x_2 x_1}}(u_n)$, and $\underline{U_{x_2}}$ is Hausdorff, we have $\underline{\phi_{x_2 x_1}}(y)=z$. 
\end{proof}

We now prove the existence of a Hausdorff good coordinate system.

\begin{theorem}\label{HAUSIDSPACE} Let $\mathcal{G}:=\{C_x|_{U_x}\}_{x\in S}$ be a good coordinate system satisfying the properties in the conclusion of theorem \ref{STRONGGCS}. For a strong shrinking sequence $U_x(k), x\in S, k\in\mathbb{N}$ of $\mathcal{G}$, define $\mathcal{G}(0):=\mathcal{G}$ and $\mathcal{G}(k):=\{C_x|_{U_x(k)}\}_{x\in S}$, then $$(M(\mathcal{G}(k)), \mathcal{T}(\mathcal{G}(k),\mathcal{G}))$$ is Hausdorff for sufficiently large $k$. 
\end{theorem}

\begin{proof} First observe that the relative topology on $M(\mathcal{G}(k))$ is the same if we define it relative to any $\mathcal{G}(j)$ for $0\leq j\leq k-1$. This is why we require the sequence of shrinkings to be a strong shrinking sequence.

We claim that for sufficiently large $k$, $M(\mathcal{G}(k))$ is Hausdorff for $\mathcal{T}(\mathcal{G}(k), \mathcal{G})$.

Suppose this is not the case. Then there exists a subsequence $k_n, n\in\mathbb{N}$ with $k_1<k_2<\cdots\to \infty$ of $\mathbb{N}$ such that $M(\mathcal{G}(k_n))$ is not Hausdorff in the relative topology $\mathcal{T}(\mathcal{G}(k_n), \mathcal{G}(k_{n-1}))=\mathcal{T}(\mathcal{G}(k_n), \mathcal{G})$. Here the convention is $\mathcal{G}(k_0):=\mathcal{G}$. 

Then for each $k_n$, there will be a pair of points $\hat y(k_n)$ and $\hat z(k_n)$ causing the non-Hausdorffness in the identified union $(\bigsqcup_{x\in S}\underline{U_x(k_n)})/\sim$. That is, $\hat y(k_n)\not=\hat z(k_n)$, but any pair of open neighborhoods around $\hat y(k_n)$ and $\hat z(k_n)$ respectively in the topology $\mathcal{T}(\mathcal{G}(k_n), \mathcal{G}(k_{n-1}))$ will intersect in $M(\mathcal{G}(k_n))$.

By remark \ref{TWOCASESISONE}.1., the local bases around $\hat y(k_n)$ and $\hat z(k_n)$ are defined respectively using $\iota_{x^+_{\hat y(k_n)}}(D_{x^+_{\hat y(k_n)}}(k_{n-1}))$ and $\iota_{x^+_{\hat z(k_n)}}(D_{x^+_{\hat z(k_n)}}(k_{n-1}))$ respectively, where a choice of $x^+_z$ for a point $z$ is defined just before definition \ref{JUMPINGPOINT} and $D_{x^+_z}(k_{n-1})\subset \underline{U_{x^+_z}(k_{n-1})}$ is defined immediately before definition \ref{RELATIVETOP} together with remark \ref{TWOCASESISONE}. We denote $$y(k_n):=\iota_{x^+_{\hat y(k_n)}}^{-1}(\hat y(k_n))\text{ and }z(k_n):=\iota_{x^+_{\hat z(k_n)}}^{-1}(\hat z(k_n)).$$ 

The non-Hausdorffness then means that there exist a sequence of points $y(k_n)_m$, $m\in\mathbb{N}$ converges to $y(k_n)$ in $D_{x^+_{y(k_n)}}(k_{n-1})$ and another sequence of points $z(k_n)_m$, $m\in\mathbb{N}$ converges to $z(k_n)$ in $D_{x^+_{z(k_n)}}(k_{n-1})$ such that $y(k_n)_m\sim z(k_n)_m$.

Since $S$ is finite, by passing to a subsequence $k_n', n\in\mathbb{N}$ of $k_n, n\in\mathbb{N}$, we can assume $x^+_{\hat y(k_n')}=x_1$ and $x^+_{\hat z(k_n')}=x_2$ are fixed indepedent of $k_n'$. We will use the convention $\mathcal{G}(k_0'):=\mathcal{G}$ to define the relative topology for $M(\mathcal{G}(k_1'))$.

We have two possibilities, either $x_1\leq x_2$ or $x_2\leq x_1$. We will only consider the first case, as the the argument for the second case is the same after swapping $x_1$ and $x_2$.

We apply lemma \ref{TOPOLOGYLEMMA} with $U=\underline{U_{x_i}}, U_n=\underline{U_{x_i}(k_n')}, Y=\underline{\overline{(s_{x_i}|_{U_{x_i}(k_1')})^{-1}(0)}}$ for $i=1$, $2$. By passing to a subsequence $k_n''$ of $k_n'$, we have that $y(k_n'')$ converges to $y\in \overline{\underline{(s_{x_1}|_{U_{x_1}(k_1'')})^{-1}(0)}}\subset \underline{(s_{x_1}|_{U_{x_1}})^{-1}(0)}$ in $\underline{U_{x_1}}$, and that $z(k_n'')$ converges to $z\in \overline{\underline{(s_{x_2}|_{U_{x_2}(k_1'')})^{-1}(0)}}\subset  \underline{(s_{x_2}|_{U_{x_2}})^{-1}(0)}$ in $\underline{U_{x_2}}$.

We will first establish a coordinate change $C_{x_1}|_{U_{x_1}}\to C_{x_2}|_{U_{x_2}}$ by the strongly intersecting property. By the definition, the pair $y(k_1'')\in \underline{U_{x_1}(k_1'')}$ and $z(k_1'')\in \underline{U_{x_2}(k_1'')}$ causes non-Hausdorffness. Since $\underline{U_{x_1}(k''_1)}$ and $\underline{U_{x_2}(k''_1)}$ are open, the approximating sequences $y(k_1'')_n$ and $z(k_2'')_n$ lie in the respective charts for large enough $n$, and for each of such $n$ they are identified. So by the strongly intersecting property, there is a direct coordinate change $C_{x_1}|_{U_{x_1}}\to C_{x_2}|_{U_{x_2}}$, since we are discussing the case $x_1\leq x_2$.

We have $y\in \underline{U_{x_1}}$ and $z\in \underline{U_{x_2}}$, and then by passing to a diagonal subsequence $y(k_m'')_m, z(k_m'')_m, m\in\mathbb{N}$, we see that the hypothesis of lemma \ref{IDENTIFY} holds, so that $y$ is identified with $z$, and $y\in U_{x_2 x_1}$, so $y$ has an open neighborhood in the domain of coordinate change. This in turn implies that $y(k_n'')$ is in the domain of the coordinate change for all $n\geq N$ for some large $N$, so the limiting sequence $y(k_N'')_m$ approximating $y(k_N'')$, which is identified with $z(k_N'')_m$ limiting to $z(k_N'')$, lies inside the domain of the coordinate change as well for all $m\geq L(N)$ for some large $L(N)$. Therefore for all $m\geq L(N)$, $\underline{\phi_{x_2 x_1}}(y(k_N'')_m)$ is identified with $z(k_N'')_m$, and hence equals $z(k_N'')_m$ by the maximality condition and the Hausdorffness of $\underline{U_{j_2}}$. Because $\underline{\phi_{x_2 x_1}}(y(k_N'')_m)$ converges to $\underline{\phi_{x_2 x_1}}(y(k_N''))$ and $z(k_N'')_m$ converges to $z(k_N'')$, we have $\underline{\phi_{x_2 x_1}}(y(k_N''))=z(k_N'')$, namely, $\hat y(k_N'')=\hat z(k_N'')$, and this contradicts the choice $\hat y(k_N'')\not=\hat z(k_N'')$.
\end{proof}

\begin{remark}
\begin{enumerate}
\item We need $X_x|_{U'_x}=\text{interior}(\overline{X_x|_{U'_x}})$ in the conclusion of proposition \ref{COORDEXIST}, because we will need to extract a strong shrinking sequence from it again in the proof of theorem \ref{HAUSIDSPACE}.
\item We now comment on the relation between Hausdorffness in the relative topology and the Hausdorffness in each dimension. Denote $N:=\text{max}\{\text{dim} E_x\;|\; x\in S\}$ and denote $$M_i:=M_i(\mathcal{G}'):=\underline{U_i}:=(\bigsqcup_{x\in S,\;\text{dim} E_x=i} \underline{U_x})/\sim.$$
The topology on $M_i$ can be defined in the usual way using only charts of dimension i without dimension jump (namely, generated by bases of $z$ in $D'_{x^0(z)}$ for some choice of $x^0(z)$, for $z\in M_i$). By remark \ref{TWOCASESISONE}.(1), if $M(\mathcal{G}')$ is $\mathcal{T}(\mathcal{G}',\mathcal{G})$-Hausdorff, then $M_i$ for all $i\leq N$ is Hausdorff in the usual topology. However, an easy example such as figure \ref{figure08} illustrates that $\mathcal{G}'$ can be non-Hausdorff if we only have Hausdorffness of $M_i$ for all $i\leq N$.
\item In \cite{Dingyu}, I proved theorem \ref{HAUSIDSPACE} in two steps. First I considered a good shrinking sequence to achieve Hausdorffness of $M_i$ for all $i$ simultaneously. Then I grouped the charts in dimension $i$ together for each $i$ and considered a strong shrinking sequence of this slightly generalized good coordinate system to achieve the result. The structures in both steps are the same, except in the first step the topology is the usual topology for each $i$, which by remark \ref{TWOCASESISONE}.(1) is a special case of relative topology and the result achieved using a good shrinking sequence can of course be achieved using a strong shrinking sequence. Therefore, I can combine both steps into the one step. As a result, the discussion of grouping charts is not necessary here and will be systematically introduced in grouping using a total order partition.
\end{enumerate}
\end{remark}

\begin{figure}[htb]
  \begin{center}

\begin{tikzpicture}[scale=1]
\filldraw[dashed, color=black!20] (2,0) circle (2);
\filldraw (2,0) circle (2pt);
\draw[dotted, color=black!40] (2.5, -1) -- (2.5, 1);
\begin{scope}
\clip (2.5, -1) rectangle (3.5, 1);
\filldraw[dotted, color=black!40] (2.5,0) circle (1);
\filldraw (3,0) circle (2pt);
\end{scope}
\path (2,-2.5) node {$U_y$};
\path[->] (4,-1.5) node [below right]{$U_{xy}$} edge [bend right](3,-0.5);

\begin{scope}
\clip (6.5, -1) rectangle (7.5, 1);
\filldraw[dotted, color=black!40] (6.5,0) circle (1);
\filldraw (7,0) circle (2pt);
\end{scope}
\path[->] (6,0) edge (9,0);
\path[->] (6.5, -2.5) edge (6.5, 2.5);
\path[->] (7.5, -0.8) edge (5.5, 0.8);
\filldraw (8.5, 0.5) circle (2pt);
\path[->] (8, 1) node [above right]{$\phi_{xy}(U_{xy})$} edge [bend right] (7, 0.5);

\path (8.5, -2.5) node[below] {$U_x=(0,\infty)\times\mathbb{R}\times\mathbb{R}$};
\end{tikzpicture}

  \end{center}
\caption[A non-Hausdorff identification space with the part from each dimension being Hausdorff.]{A non-Hausdorff identification space (of a good coordinate system for a Kuranishi structure) with the part from each dimension being Hausdorff.}
\label{figure08}
\end{figure}
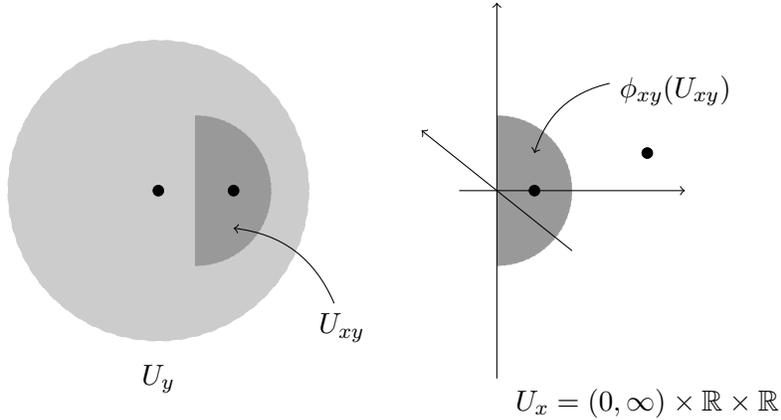

\begin{proof}[Proof of theorem \ref{BIGHAUSDORFF} by defining $\mathcal{T}(\mathcal{G},\mathcal{K})$] Changing the notations slightly, from the conclusion in theorem \ref{HAUSIDSPACE}, we can pick a precompact shrinking $\mathcal{G}$ of a good coordinate system $\tilde{\mathcal{G}}$ obtained from a Kuranishi structure $\mathcal{K}$ such that $M(\mathcal{G})$ is a Hausdorff space for the relative topology $\mathcal{T}(\mathcal{G},\tilde{\mathcal{G}})$. 

We now define $\mathcal{T}(\mathcal{G},\mathcal{K})$, and show it only depends on $\mathcal{G}$ and $\mathcal{K}$, it is well-defined, and agrees with $\mathcal{T}(\mathcal{G},\tilde{\mathcal{G}})$. We denote the collection of Kuranishi charts indexed by $S$ in the Kuranishi structure by $\mathcal{K}_S:=\{C_x|_{V_x}\}_{x\in S}$. Define $\mathcal{\mathcal{T}(\mathcal{G}, \mathcal{K}}):=\mathcal{T}(\mathcal{G}, \mathcal{K}_S)$ defined in the same way as $\mathcal{T}(\mathcal{G}, \tilde{\mathcal{G}})$. The well-definedness and recovering $\mathcal{T}(\mathcal{G},\tilde{\mathcal{G}})$ follow because by the definitions, the topologies on $M(\mathcal{G})$ satisfy the identity: $\mathcal{T}(\mathcal{G}, \mathcal{K})=\mathcal{T}(\mathcal{G}, \tilde{\mathcal{G}})$.
\end{proof}

\section{Refinements and equivalences}\label{REFANDEQUIV}

We have seen so far that starting from a Kuranishi structure, we can extract a good coordinate system with the strongly intersecting property among other nice attributes. We can define a relative topology such that the identification space obtained by gluing $\underline{U_x}, x\in S$ together using the quotiented coordinate changes of the good coordiate system is Hausdorff in that topology.

In obtaining a Hausdorff good coordinate system, we made various choices and later on when we upgrade it into a stronger form called a \emph{level-1 good coordinate system} and perform a perturbation, we will have made even more choices. We want to be able to compare results with different choices made, in particular at this stage to compare two different Hausdorff good coordinate systems obtained using the recipe of the last two sections, which the current literature simply avoids. In order to compare two such good coordinate systems, we need to be able to perform shrinking, refine the index set, and coherently replace each chart by a chart of possibly higher dimension. Those modifications to a good coordinate system are captured by the notion of a \emph{refinement}. Two good coordinate systems are essentially the same and said to be \emph{equivalent} if they can be refined to be the same good coordinate system. There is an analogous definition of refinement and hence equivalence for Kuranishi structures. An important and highly non-trivial result is that good coordinate systems obtained from equivalent Kuranishi structures are equivalent. In this paper, we will show that our perturbation method factors through the equivalence classes of Kuranishi structures; and later in the series, we will see that the polyfold-Kuranishi correspondence becomes well-defined on the level of equivalence classes of Kuranishi structures.

The notion of equivalence classes are not just a gadget to absorb ambiguity of choices. A Kuranishi structure (or a good coordinate system) merely provides a coordinate representative of the intrinsic underlying structure, and two different data sets can essentially describe the same underlying structure. Their equivalence classes provide germs to capture those intrinsic structures in a more general way than equivalence classes of atlases captures instrinsic structures of orbifolds. The equivalence classes are global germs and are first announced in \cite{DY}.

\begin{definition}[chart embedding between Kuranishi charts based at the same point]\label{SAMECOORD} Let $C_p$ and $C'_p$ be two Kuranishi charts based at the same point $p$ of $X$. An \emph{chart embedding} $(\phi_p, \hat\phi_p, V_p)=(C_p\to C'_p)$ is a coordinate change from $C_p$ to $C'_p$ such that the domain $\text{dom}(C_p\to C'_p)$ of the coordinate change is the entire $V_p$.
\end{definition}

\begin{remark} This definition allows that $\text{dim} E_p'>\text{dim} E_p$. The existence of such an embedding implies that $X_p\subset X_p'$. The definition of a chart embedding above is same as \ref{level0}, except that chart in the target is based at the same point as the chart of the domain.
\end{remark}

\begin{definition}[Kuranishi embedding of Kuranishi structures]\label{REFINE} A \emph{Kuranishi embedding} from a Kuranishi structure $\mathcal{K}''$ on $X$ to another Kuranishi structure $\mathcal{K}'$ on $X$ is a collection of chart embeddings $\{C''_p\to C'_p\}_{p\in X}$ as in \ref{SAMECOORD}, such that for all $q\in X_p''$ ($\subset X_p'$), the following square is commute up to the $G_p$-action on $C'_p$:

$$\begin{CD}
C_q'@>>> C_p'\\
@AAA @AAA\\
C_q'' @>>> C_p''
\end{CD}$$
\end{definition}
\begin{definition}[chart-refinement of Kuranishi structures] A Kuranishi embedding is \emph{open} if the chart embeddings $(\phi_p,\hat\phi_p, V''_p)$ in $(C''_p\to C'_p)$ are open maps for all $p\in X$. In particular, $\text{dim} E''_p=\text{dim} E'_p$ for all $p\in X$. An open Kuranishi embedding $\mathcal{K}''\to \mathcal{K}'$ is also called a \emph{chart-refinement}.
\end{definition}

\begin{definition}[refinement of Kuranishi structures]\label{KUREQUIV} A Kuranishi structure $\mathcal{K}$ is said to \emph{have a refinement} $\mathcal{K}'$ if there exists a diagram $$\mathcal{K}\overset{\alpha}{\leftarrow}\mathcal{K}''\overset{\beta}{\to}\mathcal{K}',$$ where $\alpha$ is a chart-refinement (namely, an open Kuranishi embedding) and $\beta$ is a Kuranishi embedding. We also say that the diagram $\mathcal{K}\overset{\alpha}{\leftarrow}\mathcal{K}''\overset{\beta}{\to}\mathcal{K}'$ is a \emph{refinement}, and $\mathcal{K}'$ refines $\mathcal{K}$.
\end{definition}
\begin{definition}[equivalence of Kuranishi structures]\label{KSEQUIV} Two Kuranishi structures $\mathcal{K}_1$ and $\mathcal{K}_2$ are \emph{Kuranishi equivalent}, or simply \emph{equivalent}, if they have a common refinement. 
\end{definition}

We now define the notion of the equivalence of good coordinate systems via common refinements of good coordinate systems. As usual, we will omit tilde in the notation of good coordinate systems\footnote{Because tilde was only introduced to discuss the relation between a Kuranishi structure and a good coordinate system obtained from it.}. We will write a typical chart of a good coordinate system by $C_x|_{U_x}$ to signify the base of the bundle in the chart being $U_x$, and it can be understood as a chart in a stand-alone good coordinate system previously denoted by $\tilde C_x:=(G_x, s_p: U_x\to E_x, \psi_x)$ or a restricted chart $C_x|_{U_x}$ from a chart $C_x$ with a possibly bigger base, but the data of this $C_x$ is not available unless otherwise included.

\begin{definition}[chart-refinement of good coordinate systems]\label{CHARTREFINEMENT} Let $$\mathcal{G}=(X, (S, \leq), \{C_x|_{U_x}\}_{x\in S}, \{C_y\to C_x\}_{(y,x)\in I(\mathcal{G}))})$$ and $\mathcal{G}'=(X, (S', \leq'), \{C'_{x'}|_{U'_{x'}}\}_{x'\in S'}, \{C'_{y'}\to C'_{x'}\}_{({y'},{x'})\in I(\mathcal{G}')})$ be two good coordinate systems for $X$. $\mathcal{G}'$ is said to chart-refine $\mathcal{G}$ if there is a map $\mu: S'\to S$ such that if $y'\leq' x'$, then $\mu(y')\leq \mu(x')$, and for each $x'\in S'$, we have an open embedding $C'_{x'}|_{U'_{x'}}\to C_{\mu(x')}|_{U_{\mu(x')}}$ with the domain of the coordinate change being the entire $U'_{x'}$, such that for all $(y', x')\in I(\mathcal{G}')$,
$$\begin{CD}
C_{\mu(y')}|_{U_{\mu(y')}}@>>> C_{\mu(x')}|_{U_{\mu(x')}}\\
@AAA @AAA\\
C'_{y'}|_{U'_{y'}} @>>> C'_{x'}|_{U'_{x'}}
\end{CD}$$ is commutative up to group actions. We will also write $\mathcal{G}'\overset{\sim}{\to}\mathcal{G}$, for short, and call it a \emph{chart-refinement} of good coordinate systems. $\mathcal{G}'$ is also said to be a \emph{chart-refinement} of $\mathcal{G}$.
\end{definition}

The chart embeddings $C'_{x'}|_{U'_{x'}}\to C_{\mu(x')}|_{U_{\mu(x')}}$ in a chart-refinement are open and $\mu$ between index sets preserves the orders. See figure\ref{figure09}.
\begin{figure}[htb]
  \begin{center}

\begin{tikzpicture}[scale=7/15]
\hspace{0 cm}

\node at (12,11.5) {chart-refinement};
\node at (1,6) {$\mathcal{G}$};
\node at (1,5.25) {$x\leq y$};

\node at (1,-1) {$\mathcal{G}'$};
\node at (1,-1.75) {$v\leq w$};
\node at (1,-2.5) {$w\leq z$};
\node at (1,-3.25) {$z\leq w$};
\node at (1,-4) {$v\leq z$};

\node at (4,-4) {$U'_v$};
\node at (7.5,-4) {$U'_w$};
\node at (12,-5) {$U'_z$};

\node at (2,-9.5) {$\mu(v)=x$};
\node at (2.5,-10.25) {$\mu(w)=\mu(z)=y$};

\path [->] node at (4,4){corresponds to $x$} edge [bend  right] (5,5.8);

\draw (3,6)--(7,6);

\draw (4,6) arc (180:178:3) (4,6)arc (180:182:3);
\filldraw (5,6) circle (2pt);

\draw [decorate,decoration={brace,amplitude=10pt}, xshift=0, yshift=8pt]
(4,6)--(7,6) node [midway,yshift=7.5pt, above]{$s_x^{-1}(0)$};

\node at (4.5,6)[below]{$U_x$};

\filldraw[black!20] (11,6) circle (3);
\draw[thick, dashed, black!60] (11,6) circle (3);
\draw (8,6)--(11.5,6).. controls (13,6) and (12.5,5.4).. (12,5);
\filldraw(12,5) arc (-35: -270:0.2)--(12.3,5.3)--(12,5);

\draw (9,6) arc (0:3:2) (9,6) arc (0:-3:2);
\filldraw (11,6) circle (2pt);

\node at (11,3)[below]{$U_y$};
\path[->]node at (13,9.5){$s_y^{-1}(0)$} edge [bend left] (12.3,6.0);

\begin{scope}[shift={(12,0)}]
\draw (3,6)--(7,6);

\draw (4,6) arc (180:178:3) (4,6)arc (180:182:3);
\filldraw (5,6) circle (2pt);
\end{scope}

\begin{scope}[shift={(10,0)}]
\filldraw[black!20] (11,6) circle (3);
\draw[thick, dashed, black!60] (11,6) circle (3);
\draw (8,6)--(11.5,6).. controls (13,6) and (12.5,5.4).. (12,5);
\filldraw(12,5) arc (-35: -270:0.2)--(12.3,5.3)--(12,5);

\draw (9,6) arc (0:3:2) (9,6) arc (0:-3:2);
\filldraw (11,6) circle (2pt);

\end{scope}

\begin{scope}[shift={(0,-16)}]
\filldraw[black!20] (11,6) circle (3);

\filldraw[black!40] (8.8,6) circle (0.6);

\filldraw[black!40] (11.5,6) circle (2.3);
\draw[thick, dashed, black!70] (11.5,6) circle (2.3);
\filldraw (11.5,6) circle (2pt);

\draw[thick, dashed, black!70] (8.8,6) circle (0.6);
\filldraw (8.8,6) circle (2pt);

\draw[thick, dashed, black!60] (11,6) circle (3);
\draw (8,6)--(11.5,6).. controls (13,6) and (12.5,5.4).. (12,5);
\filldraw(12,5) arc (-35: -270:0.2)--(12.3,5.3)--(12,5);

\draw (9,6) arc (0:3:2) (9,6) arc (0:-3:2);
\filldraw (11,6) circle (2pt);

\end{scope}

\begin{scope}[shift={(2,-16)}]
\draw (3,6)--(7,6);

\draw (4,6) arc (180:178:3) (4,6)arc (180:182:3);
\filldraw (5,6) circle (2pt);
\end{scope}

\begin{scope}[shift={(10,-8)}]

\filldraw[black!40] (8.8,6) circle (0.6);
\draw (8.2,6)--(9.4,6);

\filldraw[black!40] (11.5,6) circle (2.3);
\draw[thick, dashed, black!70] (11.5,6) circle (2.3);
\filldraw (11.5,6) circle (2pt);

\draw[thick, dashed, black!70] (8.8,6) circle (0.6);
\filldraw (8.8,6) circle (2pt);

\draw (9.2,6)--(11.5,6).. controls (13,6) and (12.5,5.4).. (12,5);
\filldraw(12,5) arc (-35: -270:0.2)--(12.3,5.3)--(12,5);

\draw (9,6) arc (0:3:2) (9,6) arc (0:-3:2);

\begin{scope}[shift={(2,0)}]
\draw (3,6)--(7,6);

\draw (4,6) arc (180:178:3) (4,6)arc (180:182:3);
\filldraw (5,6) circle (2pt);
\end{scope}

\end{scope}

\begin{scope}[shift={(-1,-8)}]
\filldraw[black!40] (8.8,6) circle (0.6);
\draw[thick, dashed, black!70] (8.8,6) circle (0.6);
\filldraw (8.8,6) circle (2pt);
\draw (9,6) arc (0:3:2) (9,6) arc (0:-3:2);
\draw (8.2,6)--(9.4,6);
\end{scope}

\begin{scope}[shift={(0,-8)}]

\filldraw[black!40] (11.5,6) circle (2.3);
\draw[thick, dashed, black!70] (11.5,6) circle (2.3);
\filldraw (11.5,6) circle (2pt);

\draw (9.2,6)--(11.5,6).. controls (13,6) and (12.5,5.4).. (12,5);
\filldraw(12,5) arc (-35: -270:0.2)--(12.3,5.3)--(12,5);

\begin{scope}[shift={(-1,0)}]
\draw (3,6)--(7,6);

\draw (4,6) arc (180:178:3) (4,6)arc (180:182:3);
\filldraw (5,6) circle (2pt);
\end{scope}

\end{scope}

\path [->] node at (8,2){corresponds to $y$} edge [bend  right] (11,5.8);
\path [->] node at (2,-5.5){corresponds to $v$} edge [bend  left] (4, -2.2);
\path [->] node at (6.5,-6.5){corresponds to $w$} edge [bend  right] (7.8, -2.2);
\path [->] node at (15,-6){corresponds to $z$} edge [bend  left] (11.5,-2.2);

\path [->] node at (10,-14) (a) {images of $U'_{x'}, x'\in S'$} edge [bend  left] (6.5, -10.2);
\path[->] (a) edge (8.7, -10.2);
\path[->] (a) edge[bend right] (10.5, -10.6);
\node at (10,-15){under open embeddings of charts};
\end{tikzpicture}

  \end{center}
\caption[A chart-refinement of good coordinate systems.]{A chart-refinement of good coordinate systems.}
\label{figure09}
\end{figure}
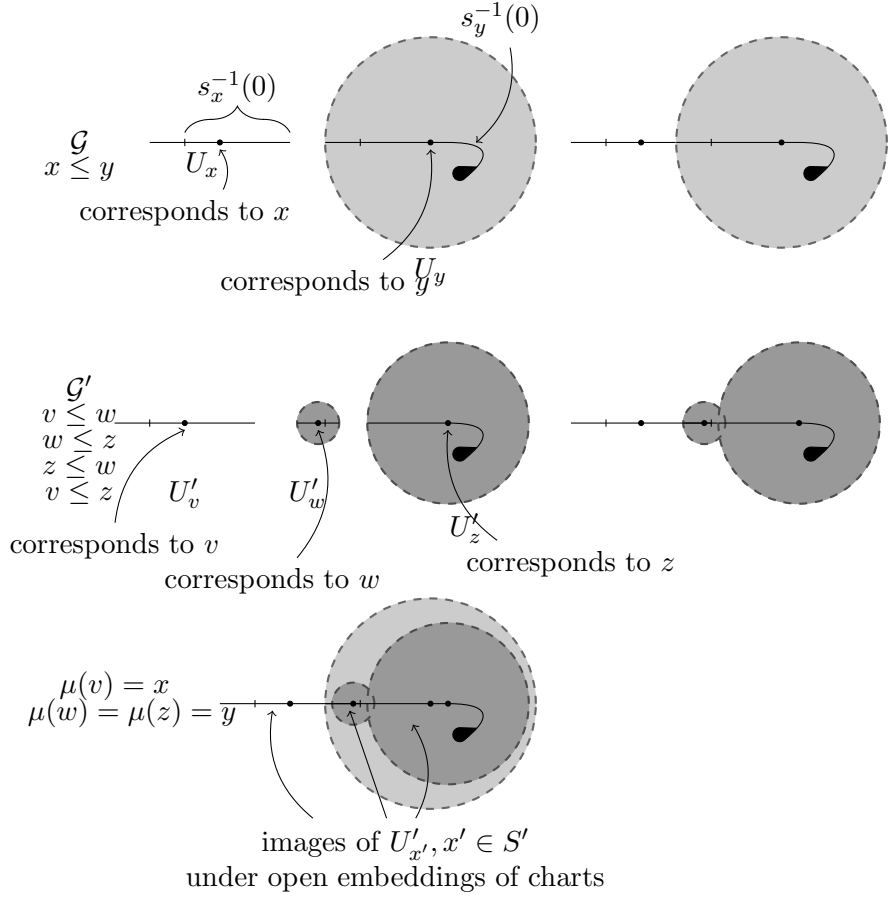

\begin{remark} Since $X'_{y'}|_{U'_{y'}}\subset X_{\mu(y')}|_{U_{\mu(y')}}$ and $y'\leq x'$ implies $\mu(y')\leq \mu(x')$, we have that if $(y', x')\in I(\mathcal{G}')$, then $(\mu(y'), \mu(x'))\in I(\mathcal{G})$. Therefore, the above square always exists if $(y', x')\in I(\mathcal{G}')$.
\end{remark}

\begin{example} A shrinking defined in definition \ref{GCSSHRINKING} is a chart-refinement with $(S, \leq)=(S', \leq')$ and $\mu=Id$.
\end{example}

So when $C'_{x'}|_{U'_{x'}}\to C_x|_{U_x}$ are inclusions, in taking a chart-refinement, we are allowed to shrink the charts and break charts into smaller charts of the same dimensions with a finer index set $S'$, while keeping $X$ globally covered.

\begin{definition}[Kuranishi embedding of good coordinate systems]\label{KEMB} Let $\mathcal{G}$ and $\mathcal{G}'$ be two good coordinate systems indexed by the same set $S=S'$ but with possibly different orders. Suppose we have Kuranishi chart embeddings $C_x|_{U_x}\to C'_x|_{U'_x}$ for all $x\in S=S'$ as in definition \ref{SAMECOORD}. If $(y,x)\in I(\mathcal{G})$, namely, $X_y|_{U_y}\cap X_x|_{U_x}\not=\emptyset$, and $y\leq x$, then we have $(y, x)\in I(\mathcal{G'})$ or $(x, y)\in I(\mathcal{G'})$ (or both)\footnote{Since $S=S'$, $X'_y|_{U'_y}\cap  X'_x|_{U'_x}\supset X_y|_{U_y}\cap X_x|_{U_x}\not=\emptyset$ and $\leq'$ is a non-antisymmetric total order, the implication follows from the definition of $I(\mathcal{G}')$.}. This collection of Kuranishi chart embeddings $C_x|_{U_x}\to C'_x|_{U'_x}, x\in S=S'$ is said to be a \emph{Kuranishi embedding} $\mathcal{G}\Rightarrow \mathcal{G}'$ if for every $(y, x)\in I(\mathcal{G})$, if $(y, x)\in I(\mathcal{G}')$, then the square
$$\begin{CD}
C'_y|_{U'_y}@>>> C'_x|_{U'_x}\\
@AAA @AAA\\
C_y|_{U_y} @>>> C_x|_{U_x}
\end{CD}$$
is commutative up to group actions;
or if we have $(x, y)\in I(\mathcal{G}')$, then the square
$$\begin{CD}
C'_y|_{U'_y}@<<< C'_x|_{U'_x}\\
@AAA @AAA\\
C_y|_{U_y} @>>> C_x|_{U_x}
\end{CD}$$
is commutative up to group actions.
\end{definition}

The two good coordinate systems from a Kuranishi embedding have the same index set, but across a Kuranishi embedding the dimensions of charts can increase and the coordinate change directions can change. See figure \ref{figure10}.

\begin{figure}[htb]
  \begin{center}

\begin{tikzpicture}[scale=7/15]
\hspace{0 cm}

\node at (12,11.5) {Kuranishi embedding};
\node at (0,6) {$\mathcal{G}$};
\node at (0,5.25) {$x\leq y$};

\node at (0,-1) {$\mathcal{G}'$};
\node at (0,-1.75) {$y\leq x$};

\node at (5,-6) {$U'_x$};
\node at (12,-5) {$U'_y$};

\path [->] node at (4,4){corresponds to $x$} edge [bend  right] (5,5.8);

\draw (3.5,6)--(6,6);

\draw (4,6) arc (180:178:3) (4,6)arc (180:182:3);
\filldraw (5,6) circle (2pt);

\draw [decorate,decoration={brace,amplitude=5pt}, xshift=0, yshift=8pt]
(4,6)--(6,6) node [midway,yshift=7.5pt, above]{$s_x^{-1}(0)$};

\node at (4.5,6)[below]{$U_x$};

\begin{scope}[shift={(-1,0)}]
\filldraw[black!20] (11,6) circle (1.5);
\draw[thick, dashed, black!60] (11,6) circle (1.5);
\draw (9.5,6)--(11,6).. controls (12.5,6) and (12,5.4).. (11.5,5);
\filldraw(11.5,5) arc (-35: -270:0.2)--(11.8,5.3)--(11.5,5);

\filldraw (11,6) circle (2pt);

\node at (11,4.5)[below]{$U_y$};
\path[->]node at (13,8){$s_y^{-1}(0)$} edge [bend left] (12,5.9);
\end{scope}

\begin{scope}[shift={(14,0)}]
\draw (3.5,6)--(6,6);

\draw (4,6) arc (180:178:3) (4,6)arc (180:182:3);
\filldraw (5,6) circle (2pt);
\end{scope}

\begin{scope}[shift={(10,0)}]

\filldraw[black!20] (11,6) circle (1.5);
\draw[thick, dashed, black!60] (11,6) circle (1.5);
\draw (9.5,6)--(11,6).. controls (12.5,6) and (12,5.4).. (11.5,5);
\filldraw(11.5,5) arc (-35: -270:0.2)--(11.8,5.3)--(11.5,5);

\filldraw (11,6) circle (2pt);

\draw (10,6) arc (0:3:2) (10,6) arc (0:-3:2);
\filldraw (11,6) circle (2pt);

\end{scope}

\begin{scope}[shift={(0,-16)}]

\filldraw[black!40] (8,6) circle (3.5);
\filldraw[black!40] (11,6) circle (2.5);

\filldraw[black!20] (11,6) circle (1.5);
\draw[thick, dashed, black!60] (11,6) circle (1.5);

\draw[thick, dashed, black!60] (8,6) circle (3.5);
\draw[very thick, black!80] (4.5, 6) arc (-120:-100:6);
\draw[very thick,dotted, black!80] (4.5, 6) arc (120:100:6);

\draw[thick, dashed, black!60] (11,6) circle (2.5);
\draw (9.5,6)--(11,6).. controls (12.5,6) and (12,5.4).. (11.5,5);
\filldraw(11.5,5) arc (-35: -270:0.2)--(11.8,5.3)--(11.5,5);

\draw (10,6) arc (0:3:2) (10,6) arc (0:-3:2);
\filldraw (11,6) circle (2pt);

\end{scope}

\begin{scope}[shift={(4,-16)}]
\draw (3.5,6)--(7,6);

\draw (4,6) arc (180:178:3) (4,6)arc (180:182:3);
\filldraw (5,6) circle (2pt);
\end{scope}

\begin{scope}[shift={(10,-8)}]

\filldraw[black!40] (8.5,6) circle (3.5);

\filldraw[black!40] (11.5,6) circle (2.5);
\draw[thick, dashed, black!70] (11.5,6) circle (2.5);
\filldraw (11.5,6) circle (2pt);

\draw[thick, dashed, black!70] (8.5,6) circle (3.5);
\draw[very thick, black!80] (5, 6) arc (-120:-100:6);
\draw[very thick,dotted, black!80] (5, 6) arc (120:100:6);

\filldraw (9.5,6) circle (2pt);
\draw (8.5,6)--(9,6);

\draw (9,6)--(11.5,6).. controls (13,6) and (12.5,5.4).. (12,5);
\filldraw(12,5) arc (-35: -270:0.2)--(12.3,5.3)--(12,5);

\end{scope}

\begin{scope}[shift={(-4,-8)}]
\filldraw[black!40](8.5,6) circle (3.5);
\draw[thick, dashed, black!70] (8.5,6) circle (3.5);
\draw[very thick, black!80] (5, 6) arc (-120:-100:6);
\draw[very thick,dotted, black!80] (5, 6) arc (120:100:6);

\begin{scope}
\clip (8.5,6) circle (3.5);

\draw (8.5,6)--(11.5,6).. controls (13,6) and (12.5,5.4).. (12,5);
\filldraw(12,5) arc (-35: -270:0.2)--(12.3,5.3)--(12,5);
\filldraw (9.5,6)circle (2pt);
\filldraw (11.5,6)circle (2pt);
\end{scope}

\end{scope}

\begin{scope}[shift={(0,-8)}]

\filldraw[black!40] (11.5,6) circle (2.5);
\draw[thick, dashed, black!70] (11.5,6) circle (2.5);
\filldraw (11.5,6) circle (2pt);

\draw (9,6)--(11.5,6).. controls (13,6) and (12.5,5.4).. (12,5);
\filldraw(12,5) arc (-35: -270:0.2)--(12.3,5.3)--(12,5);

\end{scope}

\path [->] (4,3.5) [below] edge [bend right] (5.4, -1.8);

\path [->] node at (8,2){corresponds to $y$} edge [bend  left] (10,5.8);
\path [->] node at (8,2){corresponds to $y$} edge [bend  left] (11.5,-1.8);

\draw [decorate,decoration={brace,amplitude=5pt}, xshift=0, yshift=-8pt]
(10,-10)--(7.5,-10) node [midway,yshift=-7.5pt, above]{};

\path [->] node at (10,-14) (a) {images of $U_x, x\in S$ under} edge [bend left] (8.7, -10.6);
\path[->] (a) edge[bend right] (10.5, -11);
\node at (10,-15){the Kuranishi embedding};
\end{tikzpicture}

  \end{center}
\caption[A Kuranishi embedding.]{A Kuranishi embedding.}
\label{figure10}
\end{figure}
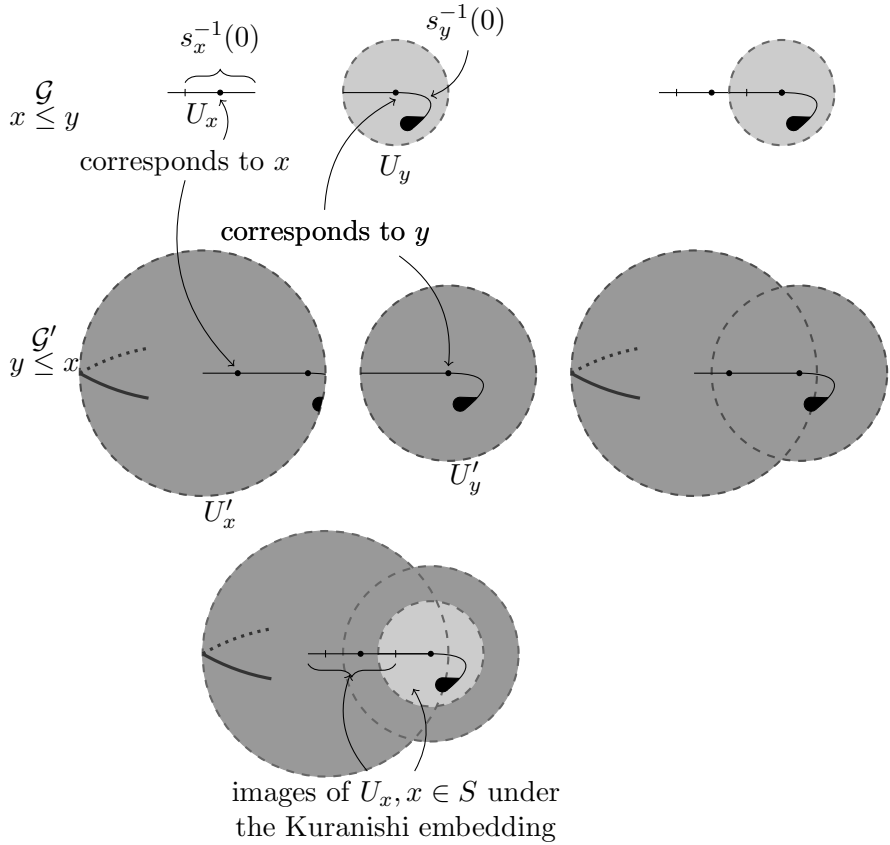

\begin{definition}[concertedness]\label{CONCERTEDNESS} A Kuranishi embedding $\mathcal{G}\Rightarrow \mathcal{G}'$ is called \emph{concerted} if for each pair $(y, x)\in S\times S$ such that $U_y$ and $U_x$ intersect in the identification space, we have $y\leq x$ and $y\leq' x$ at the same time, or $x\leq y$ and $x\leq' y$ at the same time. Equivalently, for each pair $(y,x)\in S\times S$ such that $(y,x)\in I (\mathcal{G})$ and $(x, y)\not\in I(\mathcal{G})$, we have $(y, x)\in I(\mathcal{G}')$.
\end{definition}

So in a concerted Kuranishi embedding, we can always swap one of the horizontal arrows in the second ``commutative square'' (up to group actions) in definition \ref{KEMB} to make it into the form of the first ``commutative square''.

\begin{definition}[chart-refinement of a Kuranishi embedding]\label{CROFKEMB} Let $\mathcal{G}\Rightarrow \mathcal{G'}$ and $\tilde{\mathcal{G}}\Rightarrow \tilde{\mathcal{G}'}$ be Kuranishi embeddings such that $\tilde{\mathcal{G}}$ chart-refines $\mathcal{G}$ and $\tilde{\mathcal{G}'}$ chart-refines $\mathcal{G}'$. Denote the index set of $\tilde{\mathcal{G}}$ by $\tilde S$ and in those two chart-refinements, the index sets are refined via $\mu$ and $\mu'$ respectively. $\tilde{\mathcal{G}}\Rightarrow\tilde{\mathcal{G}'}$ is said to be a \emph{chart-refinement} of $\mathcal{G}\Rightarrow\mathcal{G}'$ if $\mu=\mu'$ and for all $\tilde x\in \tilde S$, the square
$$\begin{CD}
C_{\mu(\tilde x)}|_{U_{\mu(\tilde x)}}@>>>C'_{\mu'(\tilde x)}|_{U'_{\mu'(\tilde x)}}\\
@AAA @AAA\\
\tilde C_{\tilde x}|_{\tilde U_{\tilde x}}@>>>\tilde C'_{\tilde x}|_{\tilde U'_{\tilde x}}
\end{CD}$$
commutes up to group actions.
\end{definition}

Figure \ref{figure11} illustrates how a chart-refinement of a Kuranishi embedding looks like.

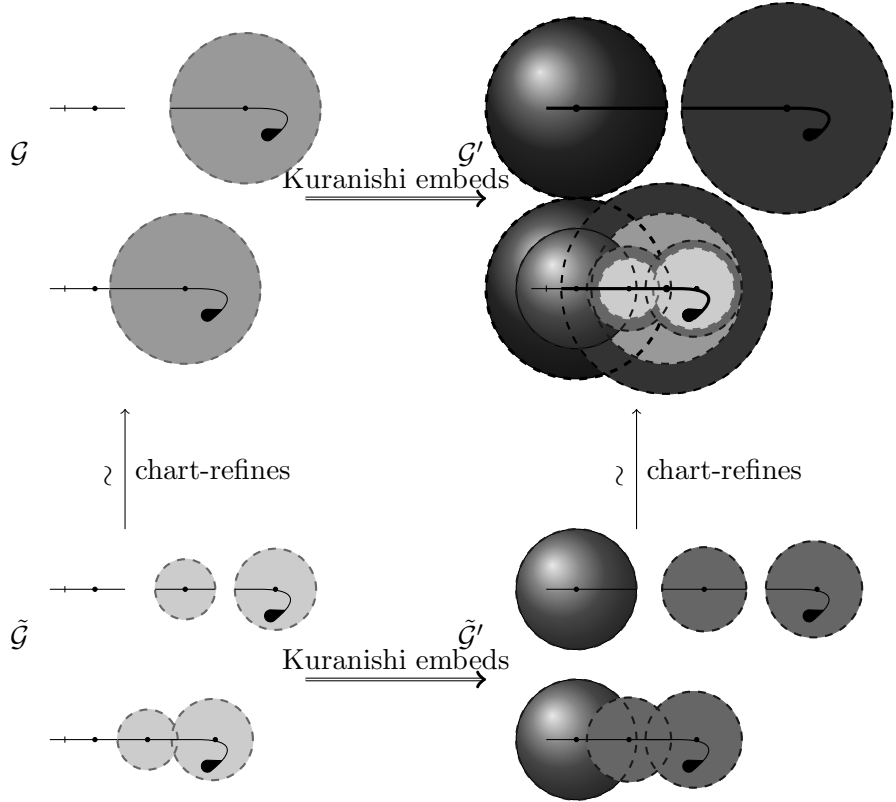
\begin{figure}[htb]
  \begin{center}

\begin{tikzpicture}[scale=20/50]
\hspace{0 cm}

\node at (17,11.5) {A chart-refinement of a Kuranishi embedding};

\node at (2.5,4.5){$\mathcal{G}$};
\node at (17.5,4.5){$\mathcal{G}'$};
\node at (2.5,-11.5){$\tilde{\mathcal{G}}$};
\node at (17.5,-11.5){$\tilde{\mathcal{G}}'$};

\path[->] (6,-8) edge node [left, shift={(.2,0.2)}, sloped]{$\sim$} (6, -4);

\node at (6,-6)[right]{chart-refines};

\path[->] (23,-8) edge node [left, shift={(.2,0.2)}, sloped]{$\sim$} (23, -4);

\node at (23,-6)[right]{chart-refines};

\draw[double, ->] (12,3) -- node [above]{Kuranishi embeds} (18, 3);
\draw[double, ->] (12,-13) -- node [above]{Kuranishi embeds} (18, -13);

\draw (3.5,6)--(6,6);

\draw (4,6) arc (180:178:3) (4,6)arc (180:182:3);
\filldraw (5,6) circle (2pt);

\begin{scope}[shift={(-1,0)}]
\filldraw[black!40] (11,6) circle (2.5);
\draw[thick, dashed, black!60] (11,6) circle (2.5);
\begin{scope}[shift={(-0.1,0)}]
\draw (8.6,6)--(11.5,6).. controls (13,6) and (12.5,5.4).. (12,5);
\filldraw(12,5) arc (-35: -270:0.2)--(12.3,5.3)--(12,5);
\end{scope}
\filldraw (11,6) circle (2pt);
\end{scope}

\draw (3.5,0)--(6,0);

\draw (4,0) arc (180:178:3) (4,0)arc (180:182:3);
\filldraw (5,0) circle (2pt);

\begin{scope}[shift={(-3,-6)}]
\filldraw[black!40] (11,6) circle (2.5);
\draw[thick, dashed, black!60] (11,6) circle (2.5);
\begin{scope}[shift={(-.1,0)}]
\draw (8.6,6)--(11.5,6).. controls (13,6) and (12.5,5.4).. (12,5);
\filldraw(12,5) arc (-35: -270:0.2)--(12.3,5.3)--(12,5);
\end{scope}
\filldraw (11,6) circle (2pt);
\end{scope}

\begin{scope}[shift={(0,-16)}]
\draw (3.5,6)--(6,6);

\draw (4,6) arc (180:178:3) (4,6)arc (180:182:3);
\filldraw (5,6) circle (2pt);

\begin{scope}[shift={(-1,0)}]
\filldraw[black!20] (12,6) circle (1.35);

\filldraw[black!20] (9,6) circle (1);

\draw[thick, dashed, black!60] (12,6) circle (1.35);

\draw[thick, dashed, black!60] (9,6) circle (1);
\draw (8,6)--(10,6);
\filldraw (9,6) circle (2pt);

\draw (10.65,6)--(11.5,6).. controls (13,6) and (12.5,5.4).. (12,5);
\filldraw(12,5) arc (-35: -270:0.2)--(12.3,5.3)--(12,5);

\filldraw (12,6) circle (2pt);
\end{scope}
\end{scope}

\begin{scope}[shift={(0,-21)}]
\draw (3.5,6)--(6,6);

\draw (4,6) arc (180:178:3) (4,6)arc (180:182:3);
\filldraw (5,6) circle (2pt);

\begin{scope}[shift={(-1,0)}]
\filldraw[black!20] (9.9,6) circle (1.35);

\filldraw[black!20] (7.75,6) circle (1);
\draw[thick, dashed, black!60] (9.9,6) circle (1.35);

\draw[thick, dashed, black!60] (7.75,6) circle (1);
\draw (6.75,6)--(8.75,6);
\filldraw (7.75,6) circle (2pt);

\begin{scope}[shift={(-2.1,0)}]
\draw (10.65,6)--(11.5,6).. controls (13,6) and (12.5,5.4).. (12,5);
\filldraw(12,5) arc (-35: -270:0.2)--(12.3,5.3)--(12,5);
\end{scope}
\filldraw (10,6) circle (2pt);
\end{scope}
\end{scope}

\begin{scope}[shift={(16,-21)}]
\filldraw[shading=ball, ball color=black!60] (5,6) circle (2);

\draw (4,6)--(7,6);

\filldraw (5,6) circle (2pt);

\begin{scope}[shift={(-1,0)}]
\filldraw[black!60] (9.9,6) circle (1.6);

\filldraw[black!60] (7.75,6) circle (1.4);

\draw[thick,dashed,black!90] (6,6) circle (2);

\draw[thick, dashed, black!90] (9.9,6) circle (1.6);

\draw[thick, dashed, black!90] (7.75,6) circle (1.4);
\draw (6.35,6)--(9.15,6);
\filldraw (7.75,6) circle (2pt);

\begin{scope}[shift={(-2.1,0)}]
\draw (10.4,6)--(11.5,6).. controls (13,6) and (12.5,5.4).. (12,5);
\filldraw(12,5) arc (-35: -270:0.2)--(12.3,5.3)--(12,5);
\end{scope}
\filldraw (10,6) circle (2pt);
\end{scope}
\end{scope}

\begin{scope}[shift={(16,-16)}]
\filldraw[shading=ball, ball color=black!60] (5,6) circle (2);

\draw (4,6)--(7,6);

\filldraw (5,6) circle (2pt);
\draw[thick,dashed,black!90] (5,6) circle (2);

\begin{scope}[shift={(1.5,0)}]
\filldraw[black!60] (7.75,6) circle (1.4);
\draw[thick, dashed, black!90] (7.75,6) circle (1.4);
\draw (6.35,6)--(9.15,6);
\filldraw (7.75,6) circle (2pt);
\end{scope}

\begin{scope}[shift={(3,0)}]
\filldraw[black!60] (9.9,6) circle (1.6);

\draw[thick, dashed, black!90] (9.9,6) circle (1.6);

\begin{scope}[shift={(-2.1,0)}]
\draw (10.4,6)--(11.5,6).. controls (13,6) and (12.5,5.4).. (12,5);
\filldraw(12,5) arc (-35: -270:0.2)--(12.3,5.3)--(12,5);
\end{scope}
\filldraw (10,6) circle (2pt);
\end{scope}
\end{scope}

\begin{scope}[shift={(16,0)}]

\filldraw[shading=ball, ball color=black!80] (5,6) circle (3);

\draw[very thick] (4,6)--(8,6);

\filldraw (5,6) circle (3pt);
\draw[very thick,dashed] (5,6) circle (3);

\begin{scope}[shift={(1,0)}]
\filldraw[black!80] (11,6) circle (3.5);
\draw[thick, dashed] (11,6) circle (3.5);

\begin{scope}[shift={(-0.1,0)}]
\draw[very thick] (7.6,6)--(11.5,6).. controls (13,6) and (12.5,5.4).. (12,5);
\filldraw(12,5) arc (-35: -270:0.2)--(12.3,5.3)--(12,5);
\end{scope}

\filldraw (11,6) circle (3pt);
\end{scope}

\begin{scope}[shift={(0,-6)}]
\filldraw[shading=ball, ball color=black!80] (5,6) circle (3);
\filldraw[black!80] (8,6) circle (3.5);
\filldraw[black!40] (8,6) circle (2.5);

\draw[very thick] (4,6)--(8,6);

\filldraw (5,6) circle (3pt);
\draw[very thick,dashed] (5,6) circle (3);

\end{scope}

\draw (3,0)--(6,0);

\draw (4,0) arc (180:178:3) (4,0)arc (180:182:3);
\filldraw (5,0) circle (2pt);

\begin{scope}[shift={(-3,-6)}]

\draw (8.5,6)--(11.5,6).. controls (13,6) and (12.5,5.4).. (12,5);
\filldraw(12,5) arc (-35: -270:0.2)--(12.3,5.3)--(12,5);

\filldraw (11,6) circle (2pt);
\end{scope}

\begin{scope}[shift={(0,-6)}]
\filldraw[shading=ball, ball color=black!60] (5,6) circle (2);

\draw (4,6)--(7,6);

\filldraw (5,6) circle (2pt);

\filldraw[black!60] (8.9,6) circle (1.6);

\filldraw[black!60] (6.75,6) circle (1.4);

\filldraw[black!20] (8.9,6) circle (1.35);

\filldraw[black!20] (6.75,6) circle (1);

\begin{scope}[shift={(-1,0)}]

\draw[thick,dashed,black!90] (6,6) circle (2);

\draw[thick, dashed, black!90] (9.9,6) circle (1.6);

\draw[thick, dashed, black!90] (7.75,6) circle (1.4);
\draw[thick, dashed, black!90] (9,6) circle (2.5);

\draw (6.35,6)--(9.15,6);
\filldraw (7.75,6) circle (2pt);

\begin{scope}[shift={(-2.1,0)}]
\draw (10.4,6)--(11.5,6).. controls (13,6) and (12.5,5.4).. (12,5);
\filldraw(12,5) arc (-35: -270:0.2)--(12.3,5.3)--(12,5);
\end{scope}
\filldraw (10,6) circle (2pt);
\end{scope}

\begin{scope}[shift={(0,0)}]
\draw (3.5,6)--(6,6);

\draw (4,6) arc (180:178:3) (4,6)arc (180:182:3);
\filldraw (5,6) circle (2pt);

\begin{scope}[shift={(-1,0)}]

\draw[thick, dashed, black!60] (9.9,6) circle (1.35);

\draw[thick, dashed, black!60] (7.75,6) circle (1);
\draw (6.75,6)--(8.75,6);
\filldraw (7.75,6) circle (2pt);

\begin{scope}[shift={(-2.1,0)}]
\draw (10.65,6)--(11.5,6).. controls (13,6) and (12.5,5.4).. (12,5);
\filldraw(12,5) arc (-35: -270:0.2)--(12.3,5.3)--(12,5);
\end{scope}
\filldraw (10,6) circle (2pt);
\end{scope}
\end{scope}

\end{scope}

\begin{scope}[shift={(-3,-6)}]

\draw[thick, dashed] (11,6) circle (3.5);

\begin{scope}[shift={(-0.1,0)}]
\draw[very thick] (7.6,6)--(11.5,6).. controls (13,6) and (12.5,5.4).. (12,5);
\filldraw(12,5) arc (-35: -270:0.2)--(12.3,5.3)--(12,5);
\end{scope}

\filldraw (11,6) circle (3pt);

\end{scope}
\end{scope}

\end{tikzpicture}

  \end{center}
\caption[A chart-refinement of a Kuranishi embedding.]{A chart-refinement of a Kuranishi embedding.}
\label{figure11}
\end{figure}

The following result is proved in the second part of this paper using a combinatorial process called \emph{tripling}.

\begin{proposition} A Kuranishi embedding of good coordinate systems always has a concerted chart-refinement.
\end{proposition}

\begin{definition}[general embedding]\label{GEMB} A \emph{general embedding} between two good coordinate systems $\mathcal{G}\hookrightarrow \mathcal{G}'$ is a composition of a Kuranishi embedding followed by a chart-refinement, namely $\mathcal{G}\Rightarrow\mathcal{G}''\overset{\sim}{\to}\mathcal{G}'$ for some good coordinate system $\mathcal{G}''$. 
\end{definition}

\begin{definition}[refinement of good coordinate systems]\label{REFINEMENTMAP} A good coordinate system $\mathcal{G}'$ refines another good coordinate system $\mathcal{G}$ if there exists a third good coordinate system $\mathcal{G}''$ such that we have a chart-refinement $\mathcal{G''}\overset{\sim}{\to}\mathcal{G}$ and a general embedding $\mathcal{G}''\hookrightarrow \mathcal{G}'$. We also say that $\mathcal{G}'$ is a \emph{refinement} of $\mathcal{G}$, and $\mathcal{G}$ has a \emph{refinement} $\mathcal{G}'$. The diagram $$\mathcal{G}\overset{\sim}{\leftarrow}\mathcal{G}''\hookrightarrow\mathcal{G}'$$ is also called a \emph{refinement} of good coordinate systems. More explicitly, a refinement map is of the form 
$$\mathcal{G}\overset{\sim}{\leftarrow}\mathcal{G}''\Rightarrow\mathcal{G}'''\overset{\sim}{\to}\mathcal{G}'.$$
\end{definition}

\begin{definition}[equivalence of good coordinate systems]\label{GCSEQUIV} If two good coordinate systems have a common refinement, then they are said to be \emph{equivalent} as good coordinate systems.
\end{definition}

The following theorem is true, and it requires some highly non-trivial work and is proved in the second part of this paper.

\begin{theorem} Kuranishi equivalence of Kuranishi structures is an equivalence relation, and equivalence of good coordinate systems is also an equivalence relation.
\end{theorem}

We finish this section with two results required later, which also demonstrate how refinement works.

\begin{proposition}\label{COMMONREFINED} Let $\mathcal{G}$ and $\mathcal{G}'$ be two good coordinate systems for a Kuranishi structure $\mathcal{K}$ as defined in \ref{GCSFORKS}. Then $\mathcal{G}$ and $\mathcal{G}'$ both refine some common good coordinate system $\mathcal{G}''$. In fact, we can find a chart-refinement $\tilde{\mathcal{G}}$ of $\mathcal{G}$ and a chart-refinemement $\tilde{\mathcal{G}'}$ of $\mathcal{G}'$ such that there exist natural Kuranishi embeddings $\mathcal{G}''\Rightarrow\tilde{\mathcal{G}}$ and $\mathcal{G}''\Rightarrow\tilde{\mathcal{G}}'$ for some $\mathcal{G}''$.
\end{proposition}

\begin{proof} Suppose we have chosen two good coordinate systems $C_x|_{U_x}, x\in S$ and $C_x|_{U_x'}, x'\in S'$ for the Kuranishi structure $\mathcal{K}$. Choose a metric $d$ on $X$ so that we can define $B_r(p)\subset X$ for $p\in X$ using $d$. For all $p\in X$, choose $B_{r_p}(p)\subset X_p$ such that $B_{r_p}(p)\subset X_{x(p)}|_{U_{x(p)}}\cap X_{x'(p)}|_{U_{x'(p)}}$ for some $x(p)\in S$ and $x'(p)\in S'$ depending on $p$. The latter is possible, since the coverages on $X$ by charts in a good coordinate system provides an open cover for $X$. Choose a finite $S''$ such that $X=\bigcup_{x''\in S''} B_{r_{x''}/2}(x'')$.

The rest of the construction leading to a good coodinate system $\mathcal{G}''$ using the method in \ref{EXISTGCS} is entirely canonical, except the definition of any of inverted coordinate changes is only unique up to group actions. But this ambiguity will be already absorbed in the Kuranishi embeddings to be defined, since the squares are only required to commute up to group actions.

We now have $C_{x''}|_{U''_{x''}}, x''\in S''$. Since $$X_{x''}|_{U''_{x''}}=B_{r_{x''}/2}(x'')\subset X_{x(x'')}|_{U_{x(x'')}}\cap X_{x'(x'')}|_{U_{x'(x'')}},$$
we can choose a $G_{x''}$-invariant open subset $U'''_{x''}\subset U''_{x''}$ such that $X_{x''}|_{U'''_{x''}}=X_{x''}|_{U''_{x''}}$ such that $\phi_{x(x'') x''}(U'''_{x''})\subset U_{x(x'')}$\text{ and }$\phi_{x'(x'')x''}(U'''_{x''})\subset U'_{x'(x'')}$, where $\phi_{x(x'')x''}$ and $\phi_{x'(x'')x''}$ are the coordinate changes given in the data of the Kuranishi structure (they exist since $x''\in X_{x(x'')}$ and $x''\in X_{x'(x'')}$).

The desired good coordinate system is $\mathcal{G}'':=\{C_{x''}|_{U'''_{x''}}\}_{x''\in S''}$ with the order and the coordinate changes defined from $\mathcal{K}$ as in section \ref{EXISTGCS}.

Choose a $G_{x''}$-invariant open neighborhood $\tilde U_{x''}$ of $\phi_{x(x'') x''}(U'''_{x''})$ in $U_{x(x'')}$, and define $\tilde {\mathcal{G}}:=\{C_{x''}|_{\tilde U_{x''}}\}_{x''\in S''}$ with $\tilde \leq$ on $S''$ defined as $y''\tilde\leq x''$ if and only if\footnote{Or if and only if $x(y'')\leq x(x'')$ and $|G_{y''}|\leq |G_{x''}|$.} (a) $\text{dim} E_{x(y'')}<\text{dim} E_{x(x'')}$ or (b) $\text{dim} E_{x(y'')}=\text{dim} E_{x(x'')}$ and $|G_{y''}|\leq |G_{x''}|$. Define $\mu(x''):=x(x'')$. We have a chart-refinement $\tilde{\mathcal{G}}\overset{\sim}{\to} \mathcal{G}$. The coordinate changes in $\mathcal{K}$ induces a Kuranishi embedding $\mathcal{G}''\Rightarrow\tilde{\mathcal{G}}$.

In the same vein, choose $\tilde U'_{x''}$ as a $G_{x''}$-invariant open neighborhood of $\phi_{x' (x'')x''}(U'''_{x''})$ in $U'_{x'(x'')}$. Define $\tilde {\mathcal{G}}':=\{C'_{x''}|_{\tilde U'_{x''}}\}_{x''\in S''}$, where $y''\tilde\leq' x''$ on $S''$ if and only if\footnote{Or if and only if $x'(y'')\leq' x'(x'')$ and $|G_{y''}|\leq |G_{x''}|$.} (i) $\text{dim} E_{x'(y'')}<\text{dim} E_{x'(x'')}$ or (ii) $\text{dim} E_{x'(y'')}=\text{dim} E_{x'(x'')}$ and $|G_{y''}|\leq |G_{x''}|$. Let $\mu'(x''):=x'(x'')$. We have $\tilde{\mathcal{G}}'\overset{\sim}{\to} \mathcal{G}'$. The coordinate changes in $\mathcal{K}$ induces $\mathcal{G}''\Rightarrow\tilde{\mathcal{G}}'$.\end{proof}

Figure \ref{figure12} gives a visual explanation of proposition \ref{COMMONREFINED}.
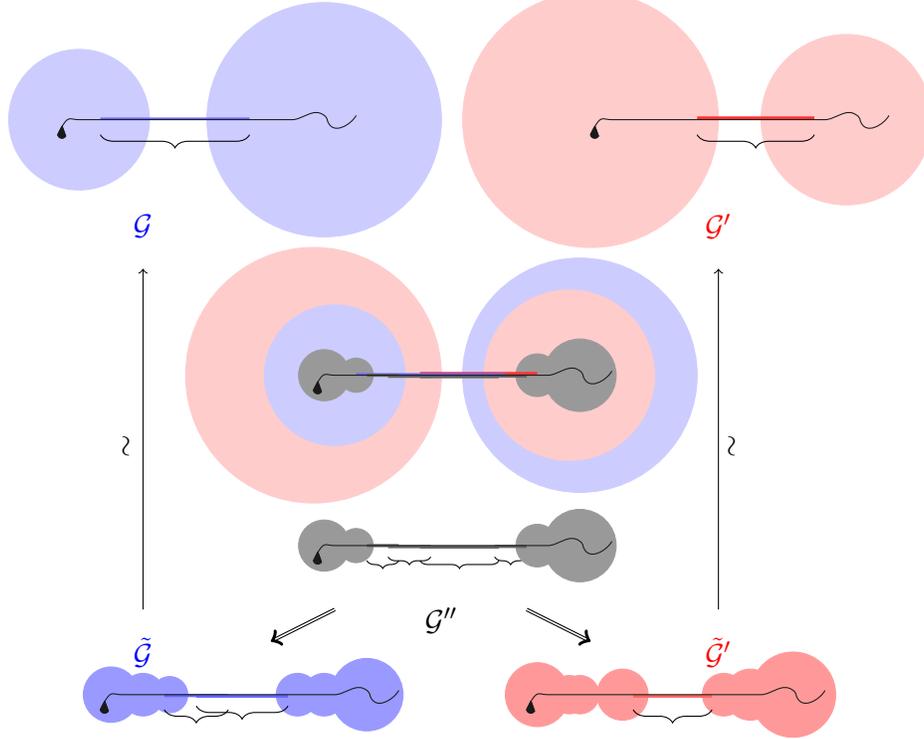
\begin{figure}[htb]
  \begin{center}

\begin{tikzpicture}[scale=17/60]
\hspace{0cm}

\node at (13,30) {the common good coordinate system $\mathcal{G}''$ constructed};
\node at (13,28.6){(after some choice made) to be refined by both $\mathcal{G}$ and $\mathcal{G}'$};
\node at (-2,17)[blue] {$\mathcal{G}$};

\node at (25,17)[red] {$\mathcal{G}'$};
\node at (12,-1.5) {$\mathcal{G}''$};
\node at (-2,-3)[blue] {$\tilde{\mathcal{G}}$};
\node at (25,-3)[red] {$\tilde{\mathcal{G}}'$};

\path[->] (-2,-1) edge node [left, shift={(.2,0.2)}, sloped]{$\sim$} (-2, 15);

\draw[double, ->] (7,-1) -- (4,-2.5);

\path[->] (25,-1) edge node [left, shift={(.2,-0.2)}, sloped]{$\sim$} (25, 15);

\draw[double, ->] (16,-1) -- (19,-2.5);

\filldraw [red!20] (6,10) circle (6); 
\filldraw [blue!20] (18.5,10) circle (5.5);
\filldraw [red!20] (18,10) circle (4);

\filldraw [blue!20] (7,10) circle (3.3);

\filldraw [black!40] (6.5,10) circle (1.2);
\filldraw [black!40] (8,10) circle (0.8);
\filldraw [black!40] (16.5,10) circle (1);
\filldraw [black!40] (18.5,10) circle (1.7);

\draw [very thick, red!80] (11,10.08)--(16.5,10.08);

\draw [thick, blue!60] (8,10.05) -- (15,10.05);

\draw [very thick, black!60] (8.5,10)-- (10,10);
\draw [very thick, black!60] (9.5,9.95)-- (11.5,9.95);
\draw [very thick, black!60] (11,9.93)-- (14.7, 9.93);
\draw [very thick, black!60] (14.5, 9.96)-- (16, 9.96);

\draw [black!90] (6, 9.3).. controls (6.5,10.3) and (6.5,10).. (7,10) -- (17,10) ..controls (17.5,10) and (18, 10.6) .. (18.5, 10.2) arc (60:5:0.3) .. controls (18.9,9.2) and (19.8,9.8) .. (20,10.2);
\draw [black!90] (6, 9.3) arc (190: 330 :0.2) -- (6.2,9.7);
\filldraw [black!90] (6.2,9.7)--(6, 9.3) arc (190: 330 :0.2) -- (6.2,9.7);

\begin{scope}[shift={(0,-8)}]

\draw [decorate,decoration={brace,amplitude=3pt}, xshift=0, yshift=-20pt]
(10,10)--(8.5,10) node [midway,yshift=7.5pt, below]{};
\draw [decorate,decoration={brace,amplitude=3pt}, xshift=0, yshift=-15pt]
(11.5,10)--(9.5,10) node [midway,yshift=7.5pt, below]{};
\draw [decorate,decoration={brace,amplitude=3pt}, xshift=0, yshift=-20pt]
(14.7,10)--(11,10) node [midway,yshift=7.5pt, below]{};
\draw [decorate,decoration={brace,amplitude=3pt}, xshift=0, yshift=-15pt]
(16,10)--(14.5,10) node [midway,yshift=7.5pt, below]{};

\filldraw [black!40] (6.5,10) circle (1.2);
\filldraw [black!40] (8,10) circle (0.8);
\filldraw [black!40] (16.5,10) circle (1);
\filldraw [black!40] (18.5,10) circle (1.7);
\draw [very thick, black!60] (8.5,10)-- (10,10);
\draw [very thick, black!60] (9.5,9.95)-- (11.5,9.95);
\draw [very thick, black!60] (11,9.93)-- (14.7, 9.93);
\draw [very thick, black!60] (14.5, 9.96)-- (16, 9.96);

\draw [black!90] (6, 9.3).. controls (6.5,10.3) and (6.5,10).. (7,10) -- (17,10) ..controls (17.5,10) and (18, 10.6) .. (18.5, 10.2) arc (60:5:0.3) .. controls (18.9,9.2) and (19.8,9.8) .. (20,10.2);
\draw [black!90] (6, 9.3) arc (190: 330 :0.2) -- (6.2,9.7);
\filldraw [black!90] (6.2,9.7)--(6, 9.3) arc (190: 330 :0.2) -- (6.2,9.7);
\end{scope}

\begin{scope}[shift={(-10,-15)}]

\filldraw [blue!40] (6.5,10) circle (1.3);
\filldraw [blue!40] (8,10) circle (1);
\filldraw [blue!40] (16.5,10) circle (1);
\filldraw [blue!40] (18.5,10) circle (1.7);
\filldraw [blue!40] (9.25,10) circle (.85);
\filldraw [blue!40] (15.25,10) circle (1);

\draw [decorate,decoration={brace,amplitude=5pt}, xshift=0, yshift=-20pt]
(12,10)--(9,10) node [midway,yshift=7.5pt, below]{};
\draw [decorate,decoration={brace,amplitude=5pt}, xshift=0, yshift=-15pt]
(14.8,10)--(10.5,10) node [midway,yshift=7.5pt, below]{};

\draw [very thick, blue!60] (9,9.96)-- (12,9.96);
\draw [very thick, blue!60] (10.5,9.93)-- (14.8, 9.93);

\draw [black!90] (6, 9.3).. controls (6.5,10.3) and (6.5,10).. (7,10) -- (17,10) ..controls (17.5,10) and (18, 10.6) .. (18.5, 10.2) arc (60:5:0.3) .. controls (18.9,9.2) and (19.8,9.8) .. (20,10.2);
\draw [black!90] (6, 9.3) arc (190: 330 :0.2) -- (6.2,9.7);
\filldraw [black!90] (6.2,9.7)--(6, 9.3) arc (190: 330 :0.2) -- (6.2,9.7);
\end{scope}

\begin{scope}[shift={(10,-15)}]

\filldraw [red!40] (6.5,10) circle (1.5);
\filldraw [red!40] (8,10) circle (0.9);
\filldraw [red!40] (16.5,10) circle (1.2);
\filldraw [red!40] (18.5,10) circle (2);
\filldraw [red!40] (8.5,10) circle (0.9);
\filldraw [red!40] (10.5,10) circle (1.2);
\filldraw [red!40] (15.25,10) circle (1);
\draw [very thick, red!60] (11,9.93)-- (14.7, 9.93);

\draw [decorate,decoration={brace,amplitude=5pt}, xshift=0, yshift=-20pt]
(14.7,10)--(11,10) node [midway,yshift=7.5pt, below]{};

\draw [black!90] (6, 9.3).. controls (6.5,10.3) and (6.5,10).. (7,10) -- (17,10) ..controls (17.5,10) and (18, 10.6) .. (18.5, 10.2) arc (60:5:0.3) .. controls (18.9,9.2) and (19.8,9.8) .. (20,10.2);
\draw [black!90] (6, 9.3) arc (190: 330 :0.2) -- (6.2,9.7);
\filldraw [black!90] (6.2,9.7)--(6, 9.3) arc (190: 330 :0.2) -- (6.2,9.7);
\end{scope}

\begin{scope}[shift={(-12, 12)}]

\filldraw [blue!20] (18.5,10) circle (5.5);

\filldraw [blue!20] (7,10) circle (3.3);

\draw [thick, blue!60] (8,10.05) -- (15,10.05);

\draw [decorate,decoration={brace,amplitude=5pt}, xshift=0, yshift=-20pt]
(15,10)--(8,10) node [midway,yshift=7.5pt, below]{};

\draw [black!90] (6, 9.3).. controls (6.5,10.3) and (6.5,10).. (7,10) -- (17,10) ..controls (17.5,10) and (18, 10.6) .. (18.5, 10.2) arc (60:5:0.3) .. controls (18.9,9.2) and (19.8,9.8) .. (20,10.2);
\draw [black!90] (6, 9.3) arc (190: 330 :0.2) -- (6.2,9.7);
\filldraw [black!90] (6.2,9.7)--(6, 9.3) arc (190: 330 :0.2) -- (6.2,9.7);
\end{scope}

\begin{scope}[shift={(13,12)}]

\filldraw [red!20] (6,10) circle (6); 

\filldraw [red!20] (18,10) circle (4);

\draw [very thick, red!80] (11,10.08)--(16.5,10.08);

\draw [decorate,decoration={brace,amplitude=5pt}, xshift=0, yshift=-20pt]
(16.5,10)--(11,10) node [midway,yshift=7.5pt, below]{};

\draw [black!90] (6, 9.3).. controls (6.5,10.3) and (6.5,10).. (7,10) -- (17,10) ..controls (17.5,10) and (18, 10.6) .. (18.5, 10.2) arc (60:5:0.3) .. controls (18.9,9.2) and (19.8,9.8) .. (20,10.2);
\draw [black!90] (6, 9.3) arc (190: 330 :0.2) -- (6.2,9.7);
\filldraw [black!90] (6.2,9.7)--(6, 9.3) arc (190: 330 :0.2) -- (6.2,9.7);
\end{scope}

\end{tikzpicture}

  \end{center}
\caption[Proposition \ref{COMMONREFINED}.]{Proposition \ref{COMMONREFINED}.}
\label{figure12}
\end{figure}

\begin{remark} Using that the equivalence of good coordinate systems is transitive, we can establish that  $\mathcal{G}$ and $\mathcal{G}'$ are equivalent. The above proposition is an ingredient of directly establishing a common refinement for $\mathcal{G}$ and $\mathcal{G}'$. The remaining step uses a fiber product construction and is the same step for proving that the equivalence is transitive.
\end{remark}

\begin{proposition}\label{KREFINEGCS} If a good coordinate system $\mathcal{G}$ is obtained from a Kuranishi structure $\mathcal{K}$ using the method in \ref{EXISTGCS}, then the Kuranishi structure $\mathcal{K}(\mathcal{G})$ naturally induced from this good coordinate system refines $\mathcal{K}$.
\end{proposition}

\begin{proof} Recall the construction in proposition \ref{GCSKS}.(1). For a good coordinate system $\mathcal{G}=\{C_x|_{U_x}\}_{x\in S}$ for $\mathcal{K}:=\{C_p\}_{p\in X}$, define $S_p:=\{x\in S\;|\; p\in X_x|_{U_x}\}$, and fix a choice $x(p)$ of the smallest elements of $S_p$. For all $p\in X$, choose a $G_p$-invariant open neighborhood of $W_p$ of $U_{x(p)}$ such that $p\in X_{x(p)}|_{W_p}$. Define $K_p:=C_{x(p)}|_{W_p}$. This induces a Kuranishi structure $\mathcal{K}(\mathcal{G}):=\{K_p\}_{p\in X}$ as shown in \ref{GCSKS}.

Now we want to show that $\mathcal{K}(\mathcal{G})$ refines $\mathcal{K}$. For any $p\in X$, by the definition $p\in X_{x(p)}$, so there is a coordinate changes $C_p\to C_{x(p)}$, and this induces a coordinate changes $C_p\to K_p:=C_{x(p)}|_{W_p}$. Since the coverage on $X$ by $K_p$ is an open neighborhood $Y_p:=X_{x(p)}|_{W_p}$ of $p$ in $X$, we can find an open neighborhood $Z_p\subset X_p\cap Y_p$ of $p$ in $X$, define $V'_p:=V_p\backslash (s_p^{-1}(0)\backslash quot_p^{-1}(\psi_p^{-1}(Z_p)))$ and so $X_p|_{V'_p}=Z_p$. Now we consider the restricted coordinate change $C_p|_{V'_p}\to C_{x(p)}|_{W_p}$ with the domain of the coordinate change between restricted charts being $O_{x(p) p}$. By the maximality condition, $Z_p\subset X_p|_{O_{x(p) p}}$. We just define $V''_p:=O_{x(p) p}$. By construction, $C_p|_{V''_p}\to K_p, p\in X$ are all Kuranishi chart embeddings and together provide a Kuranishi embedding from $\mathcal{K}':=\{C_p|_{V''_p}\}_{p\in X}$ to $\mathcal{K}(\mathcal{G})$. The chart-wise inclusion $\mathcal{K}'$ to $\mathcal{K}$ is clearly a chart-refinement. So we have established that $\mathcal{K}(\mathcal{G})$ refines $\mathcal{K}$ as Kuranishi structures.
\end{proof}

\newpage

\specialsection*{PERTURBATION THEORY, CHOICE-INDEPENDENCE AND EQUIVALENCE RELATIONS\label{PERTURBATIONTHEORY}}

In part one of this paper, we demonstrated that we can obtain a strongly intersecting Hausdorff good coordinate system from a Kuranishi structure. We will now show (i) how to compactly perturb the sections in the charts in such a good coordinate system into compatible and transverse ones, (ii) how to compare all the choices made in the theory via common refinements and (iii) that two equivalences defined in section \ref{REFANDEQUIV} are equivalence relations.

To achieve (ii) and (iii), we need to be able to turn Kuranishi chart embeddings in a Kuranishi embedding between two good coordinate systems into compatible submersions with extra properties in order to form a fiber product for a suitable pair of Kuranishi embeddings mapping from a common good coordinate system. For (i) we want to have a structure in place such that when we perturb a section in a chart into a transverse section, we should be able to naturally transfer the perturbation to a chart in higher dimension, so that perturbing the section in the chart of higher dimension using this transferred perturbation is automatically transverse over the perturbed zeros (where the validity of the tangent bundle condition is not guaranteed in general but guaranteed in this structure we want to have). We accomplish both goals by the \emph{level-1} structure, and we will introduce this concept next.

\section{A special metric recovering the submersion}

We need some preliminary results for constructing compatible submersions. We will always call \'{e}tale-proper Lie groupoids as ep-groupoids in this paper, as they coincide in the current finite dimensional setting, and denote ep-groupoids by their object spaces with the morphisms implicit.

\begin{definition}[$\pi^g_{AB}$]\label{SUBMERSIONMODEL} Let $B$ be an invariant submanifold of the object space of an ep-groupoid $A$. Let $g$ be an invariant Riemannian metric on $A$. Denote the normal exponential map by $\exp^{g}_{AB}: U(N_B A)\to W_{AB}$ preserving ep-groupoid structures (namely a functor), where $U(N_BA)$ is an invariant tubular neighborhood of the zero section in the normal bundle $N_B A$ and $W_{AB}$ is the image which is an immersed\footnote{Embedded if \emph{immersed} is omitted.} invariant tubular neighborhood of $B$ in $A$. Let $pr_{AB}: N_B A\to B$ denote the normal bundle projection functor. Define $\pi_{AB}^g:=pr_{AB}\circ (\exp^g_{AB})^{-1}: W_{AB}\to B$. Later, we will refer to this as \emph{constructing the submersion (functor) using the metric $g$}. For such a submersion $\pi_{AB}^g$, we also say $g$ \emph{recovers} it.
\end{definition}

\begin{lemma}\label{TOTALGEOD} Let $\pi_D: D\to C$ be a vector bundle ep-groupoid. Choose an invariant bundle metric $h$ on $D$, an invariant bundle connection $\nabla$ on $D$ compatible with $h$, and an invariant Riemannian metric $g^C$ on $C$. Then we can naturally equip an invariant Riemannian metric $g^D$ on $D$ such that the fibers of $D$ are totally geodesic with respect to $g^D$ and the induced metric on $D_x$ from $g^D$ is $h_x$. In particular, $\pi_D=\pi^{g^D}_{DC}$ with $W_{DC}=D$.
\end{lemma}

\begin{remark} Let $M$ be an ep-groupoid equipped with an invariant Riemannian metric $g$ and an invariant connection $\nabla$ compatible with $g$. Then the local stabilizer action $h\in G_x\subset Auto(U_x)$ is isometric and hence $(\nabla|_{U_x}) dh=0$. Therefore, a geodesic $\gamma\subset U_x$ will be mapped to geodesics under $G_x$. So there is no complication arising from the structures of ep-groupoids and ep-groupoid bundles for the arguments in this section, as the constructions will be natural and invariant.
\end{remark}

\begin{proof}
\begin{enumerate}
\item Definition of the connector $K^\nabla $ of a connection $\nabla$:

We can define a horizontal subspace $H_{(x,v)}\subset T_{(x,v)}D$ via the horizontal lift $h_{(x,v)}$. Let $y\in T_xC$, and pick a smooth path $\gamma:[0,\epsilon)\to C$ with $\gamma(0)=x$ and $\gamma'(0)=y$. Parallel transport $v$ along $\gamma$ using the connection $\nabla$ to a path $\hat\gamma: [0,\epsilon)\to D$ with $\hat\gamma(0)=(x,v)$. Define $h^\nabla_{(x,v)}(y):=\hat{\gamma}'(0)$. Denote $H^\nabla_{(x,v)}:=h^\nabla_{(x,v)}(T_xC)\subset T_{(x,v)}D$ and call it the horizontal subspace. Identify $t:T_{(x,v)}D_x\to D_x\cong T_{(x,0)}D_x$ by parallel transporting using the vector space structure. Since $H^\nabla_{(x,v)}$ and $T_{(x,v)}D_x$ are complementary subspaces in $T_{(x,v)}D$, we can define the horizontal projection $\pi^\nabla_{H}: T_{(x,v)}D\to H^\nabla_{(x,v)}$. We can define the connector $K^\nabla: T_{(x,v)}D\to D_x$ by $t\circ (1-\pi^\nabla_H)$\footnote{In local coordinate, $TD$ is of the form $(U\times\R^k)\times(\R^j\times\R^k)$ and $\nabla$ is $\tilde\nabla:=d+\Gamma(\cdot)(\cdot, \cdot)$ where $\Gamma: U\to L(\R^j\times\R^k,\R^k), x\mapsto \Gamma(x)(\cdot,\cdot)$. Denote $\hat\gamma(t):=(x+ty,v(t))$ the parallel transport of $(x,v)$ along $\gamma(t):=x+ty$. By the definition, $\tilde\nabla_{\gamma'}(v)|_{t=0}=v'(0)+\Gamma(0)(y,v)=0$, so $h_{(x,v)}(y):=\hat\gamma'(0)=(y, -\Gamma(0)(y,v))$. So $\pi^\nabla (x,v, y,w)=(x,v, y, -\Gamma(x)(y,v))$, and $K^\nabla$ locally is $(x,v,y,w)\mapsto t((x,v,y,w)-(x,v,y, -\Gamma(x)(y,v)))=t((x,v, 0, w+\Gamma(x)(y,v)))=(x, w+\Gamma(x)(y,v))$.}.

\item Definition of the Riemannian metric $g^D$ on $D$ as an ep-groupoid (not as a vector bundle ep-groupoid):
 
Now, $\pi_D^\ast (K^\nabla\oplus d\pi_D): TD\to \pi_D^\ast(D\oplus TC)$ gives a splitting of the tangent space $T_{(x,v)}D$ into $D_x$ and $T_xC$. Define the metric on $D$ by $$g^D:=(K^\nabla)^\ast h\oplus (d\pi_D)^\ast g^C.$$ Clearly the factors of the tangent space splitting are orthogonal to each other with respect to $g^D$. 

\item We claim that for all $x\in C$, the fiber $D_x$ is totally geodesic with respect to $g^D$. To show this, we prove that the second fundamental form $\text{II}$ on $D_x$ vanishes.

Since $\text{II}$ is symmetric, sufficiently to prove $g^D(\text{II}(w,w),u)=0$ for any non-zero $w\in T_{(x,v)}D_x$ and any non-zero $u\in T_{(x,v)}D$ normal to $D_x$. Such a $u$ comes from the horizontal lift of $y:=d\pi^D(u)\in T_{x}C$. Extend $y$ into a vector field $Y$ on $C$. $Y$ generates a horizontal flow of fibers starting from $D_x$ along the flow line and the flow preserves the induced metrics of the fibers. Denote the horizontal lift $U:=h_{(x,v)} Y$. Extend $w$ into any vector field $\tilde W$ in $D_x$ and let $W$ be the vector field on $D$ starting from $\tilde W$ and induced by the flow. So $U|_{(x,v)}=u$ and $W_{(x,v)}=w$. Denote the Levi-Civita connection associated with $g^\nabla$ by $\nabla^D$. Then
\begin{align*}
0&=U(g^D(W,W)))=2g^D(\nabla^D_U W,W)=2g^D(\nabla^D_W U,W)\\&=2W( g^D(U,W))-2g^D(U,\nabla^D_WW)=-2g^D(\nabla^D_WW)^N,U).
\end{align*}
The right-hand side at $(x,v)$ is precisely $-2g^D(\text{II}(w,w),u)$. In the above derivation, we have used that $\nabla^D_UW-\nabla^D_W U=[U,W]=0$, since $U$ and $W$ near $(x,v)$ come from local coordinate axes, hence commute, and the Levi-Civita connection is torsion-free.
\end{enumerate}
\end{proof}

\begin{proposition}\label{METRICCOMPATIBILITY} Let $B\subset A$ be an invariant submaifold in an ep-groupoid $A$, and $C\subset B$ be an invariant submanifold in $B$. Let $\pi: A\to B$ be an $F_A$-fiber bundle ep-groupoid with the fiber model $F_A$ diffeomorphic to $\R^m$ such that there exists a fiber bundle ep-groupoid diffeomorphism $\phi_A$ preserving ep-groupoid structures and mapping $A$ to an $\R^m$-fiber bundle $A'\to B$ with a vector bundle ep-groupoid structure; and let $\pi': B\to C$ be an $F_B$-fiber bundle ep-groupoid with the fiber model $F_B$ diffeomorphic to $\R^n$ such that there exists a fiber bundle ep-groupoid diffeomorphism $\phi_B$ preserving ep-groupoid structures and mapping $B$ to an $\R^n$-fiber bundle $B'\to C$ with a vector bundle ep-groupoid structure. Then there exists an invariant Riemannian metric $g$ on $A$ such that $\pi=\pi_{AB}^g$ with $A=W_{AB}$, and $\pi'=\pi_{BC}^g$ with $B=W_{BC}$ such that $\pi'\circ\pi(=\pi^g_{BC}\circ \pi^g_{AB})=\pi^g_{AC}$ with $A=W_{AC}$.
\end{proposition}

\begin{proof}
\begin{enumerate}
\item Convert the double fiber bundle ep-groupoid into a double vector bundle ep-groupoid:

$\phi_B$ maps $B\to C$ to $B'\to C$ (the vector bundle ep-groupoid structure on $B'$ to be used). Consider $(\phi_B^{-1})^\ast: A\to (\phi_B^{-1})^\ast A$ where $(\phi_B^{-1})^\ast A$ is a pullback bundle over $B'$, and this is mapped to $(\phi_B^{-1})^\ast A'$ using $(\phi_B^{-1})^\ast\phi_A$. $(\phi_B^{-1})^\ast A'\to B'$ is an $\R^m$-fiber bundle ep-groupoid over $B'$ with a vector bundle ep-groupoid structure, hence a vector bundle ep-groupoid. So the fiber bundle ep-groupoid diffeomorphism $(\phi_B^{-1})^\ast\phi_A: A\to A'':=(\phi_B^{-1})^\ast A'$ is a double fiber bundle ep-groupoid diffeomorphism from $A\to B\to C$ to $A''\to B'\to C$, where the latter is regarded as a double vector bundle ep-groupoid.

\item Convert the double vector bundle ep-groupoid into a direct sum of two vector bundle ep-groupoids.

Regard $A''\overset{\tilde\pi}{\to} B'\overset{\tilde\pi'}{\to}C$ as a double vector bundle ep-groupoid. For each $w\in C$, $A''|_{(\tilde\pi')^{-1}(w)\cong\R^n}$ is an $\R^m$-vector bundle ep-groupoid over a contractible space, hence is globally trivialized using an invariant vector bundle connection $\nabla^{w}$ on $\tilde\pi: A''|_{(\tilde\pi')^{-1}(w)}\to (\tilde\pi')^{-1}(w)$, smoothly depending on $w$. So there exists a vector bundle ep-groupoid diffeomorphism $\Psi^w: A''|_{(\tilde\pi')^{-1}(w)}\to (\tilde\pi')^{-1}(w)\times \R^m$ such that $w\mapsto \Psi^w$ is smooth. So $\Psi^w, w\in C$ fit together into a fiber bundle ep-groupoid diffeomorphism $\tilde\Psi$ from $A''\to C$ to an $\R^n\times\R^m$-fiber bundle ep-groupoid $\pi'': D\to C$ over $C$, where the coordinate changes between local bundle charts of $D$ is linear on $\R^n\times\R^m$ and preserve the splitting of this direct product by the construction. So we can regard $D$ as a direct sum of an $\R^n$-vector bundle ep-groupoid $D_1\to C$ and an $\R^m$-vector bundle ep-groupoid $D_2\to C$. Then $\tilde\Psi$ maps $A''\to B'\to C$ to $D=(D_1\oplus D_2)\to D_1\to C$ as a double vector bundle ep-groupoid diffeomorphism, where $D=(D_1\oplus D_2)\to C$ can also be regarded as an $\R^n\times \R^m$-vector bundle ep-groupoid.

\item Construct an invariant metric $g^D$ on $D$ recovering the double vector bundle ep-groupoid projection with totally geodesic fibers for each projection and directly recovering the projection composition.

Choose invariant bundle metrics $h_1$ on $D_1$ and $h_2$ on $D_2$, and use the direct sum metric $h:=h_1\oplus h_2$ on bundle $D$. Choose an invariant Riemannian metric $g^C$ on $C$. By invoking lemma \ref{TOTALGEOD}, we have an invariant Riemannian metric $g^D$ on $D$ such that $D_x$ is totally geodesic with respect to $g^D$ and the induced metric on $D_x$ is a constant metric $h_x=(h_1)_x\oplus (h_2)_x$. So a normal geodesic in $D$ normal to $C$ starting from $x\in C$ will stay in $D_x$ and is a straight line. Thus, $D_x$ is the image of the normal exponential map at $x$ normal to $C$. Because the bundle metric is a direct sum, $(D_1)_x$ is totally geodesic and the induced metric on $(D_1)_x$ is the constant metric $(h_1)_x$; and the fiber $D_{(x, u)}$ of $D\to D_1$ at $(x,u)$ is also totally geodesic and the induced metric on $D_{(x, u)}$ is the constant metric $I_{(x,u)}^\ast (h_2)$, where $I_{(x,u)}$ is the canonical identification of $D_{(x,u)}$ with $D_{(x,0)}=(D_2)_x$ using the vector space structure on $(D_1)_x$. So, a normal geodesic from $x\in C$ to $(x,u)\in (D_1)_x$ followed by a normal geodesic from $(x,u)\in D_1$ to $(x, w)\in D$ will stay inside $D_x$, and $(x,w)$ can be reached by a direct normal geodesic from $x\in C$. Namely, $\exp^{g^D}_{D_1C}\circ\exp^{g^D}_{DD_1}=\exp^{g^D}_{DC}$. 

\item Transfer the Riemannian metric back to the original double fiber bundle ep-groupoid recovering the double fiber bundle ep-groupoid projection as well as directly recovering the projection composition.

We now transfer the metric $g^D$ on the double bundle ep-groupoid $D\to D_1\to C$ back to a metric $g:=(\tilde\Psi\circ ((\phi_B^{-1})^\ast\phi_A))^\ast g^D$ on $A\to B\to C$. Since $(\tilde\Psi\circ ((\phi_B^{-1})^\ast\phi_A))^{-1}$ is now an isometry from $(D\to E\to C, g^D)$ to $(A\to B\to C, g)$, the respective geodesics are canonically identified and $g$ is the desired metric.
\end{enumerate}
\end{proof}

\begin{remark}\label{SINGLESUBMERSION}
Notice that the choice of invariant metric $g^C$ in constructing $g^D$ is irrelevant for our purpose. We also proved that for a fiber bundle ep-groupoid $\pi: A\to C$ that after a fiber bundle ep-groupoid diffeomorphism can be equipped with a vector bundle ep-groupoid structure, we can construct an invariant Riemannian metric $g$ on $A$ such that $\pi=\pi^g_{AC}$ with $W_{AC}=A$, and the fibers of $\pi$ are totally geodesic with respect to $g$.
\end{remark}

\section{The level-1 structure}

The definition of a level-1 good coordinate system $\hat{\mathcal{G}}'$, to be defined and constructed in this section, will only depend on a strongly intersecting Hausdorff good coordinate system $\hat{\mathcal{G}}$. The underlying good coordinate system of $\hat{\mathcal{G}}'$ forgetting the level-1 structure is a precompact shrinking of $\hat{\mathcal{G}}$. The definition does not depend on the Kuranishi structure from which $\hat{\mathcal{G}}$ originates.

\subsection{Changing the indexing set to a total order and grouping of Kuranishi charts}\label{REORGANIZEBYTOT}

We will start with a strongly intersecting Hausdorff good coordinate system $\hat{\mathcal{G}}$ indexed by $(\hat S, \hat \leq)$. We suggest the readers to recall these notions from definition \ref{STRONGLYINT} and theorem \ref{BIGHAUSDORFF}.

\begin{definition}[total order partition]\label{TOPART} A \emph{total order partition} $S$ of $(\hat S, \hat \leq)$ consists of disjoint subsets $\alpha$ of $\hat S$ indexed by $\alpha\in S$ with $\hat S=\bigsqcup_{\alpha\in S}\alpha$ such that
\begin{enumerate}
\item $\text{dim} E_x, x\in \alpha$ are the same. (Namely $\alpha\subset S_i:=\{x\in \hat S\;|\; \text{dim} E_x=i\}$ for some $i$ depending on $\alpha$.)
\item For $\beta,\alpha\in S$, if we have $y\hat \leq x$ for a pair $(y, x)\in \beta\times\alpha$, then we have $y' \hat\leq x'$ for all pairs $(y',x')\in \beta\times \alpha$. Moreover, the induced order $\leq$ on $S$, defined as $\beta\leq \alpha$ if and only if $y \hat\leq x$ for some $(y,x)\in\beta\times\alpha$, is a total order. 
\end{enumerate}
\end{definition}

\begin{example}[total order partition by the bundle dimension]\label{BUNDLEDIMEN}
Recall that for a good coordinate system, if $\text{dim} E_y< \text{dim} E_x$, then $y \hat\leq x$ and $x\hat{\not\leq} y$. Define $\alpha_i:=\{x\in \hat S\;|\; \text{dim}E_x=i\}$ and $N:=\max\{\text{dim} E_x\;|\;x\in\hat S\}$. So if $j<i$, $y\in S_j$, $x\in S_i$, we have $y\hat\leq x$ and $x\hat{\not\leq} y$. So the order $\leq$ on $S:=\{\alpha_0, \cdots, \alpha_N\}$ is compatible with $\hat\leq$ on $\hat S$ and $(S,\leq)$ is now a total order.
\end{example}

\begin{definition}[grouping Kuranishi charts, $(s_\alpha: U_\alpha\to E_\alpha, \psi_\alpha)$, or equivalently $(\mathfrak{s}_\alpha:\mathcal{U}_\alpha\to\mathcal{E}_\alpha, \Psi_\alpha)$]\label{ICHART} Let $(S,\leq)$ be a total order partition of $(\hat S, \hat \leq)$. We can view the subcollection $C_x|_{U_x}, x\in \alpha$ as a single Kuranishi chart slightly generalized, namely, as an orbifold bundle $E_\alpha\to U_\alpha$ with a preferred atlas consisting of finitely many charts, together with a section $s_\alpha: U_\alpha\to E_\alpha$ expressed in this atlas and a topological identification $\psi_\alpha$. More precisely, it can be viewed in either of the following ways:
\begin{enumerate}
\item We can view it as an orbifold bundle with a preferred representation by an ep-groupoid\footnote{An \'etale-proper Lie groupoid, as we are always in finite dimensions within this paper.} together with a section.

The object space will be $U_\alpha:=\bigsqcup_{x\in \alpha} U_x$, and the morphisms between points are induced by $G_x$ and direct coordinate changes (due to the strongly intersecting property), both of which are smooth and local diffeomorphisms. The orbit space $\underline{U_\alpha}$ is exactly $M_\alpha(\mathcal{G}):=(\bigsqcup_{x\in\alpha} U_x)/\text{morphisms}\equiv(\bigsqcup_{x\in \alpha} \underline{U_x})/\sim$. Similarly we can define $E_\alpha$, the morphisms between points in $E_\alpha$ and the orbit space $\underline{E_\alpha}$. Notice that $\underline{U_\alpha}$ is Hausdorff as it is a subspace without dimension jump of a Hausdorff space; therefore $\underline{E_\alpha}$ is also Hausdorff due to the bundle topology, and hence $E_\alpha$ is indeed an orbifold bundle representative by a bundle ep-groupoid. We will denote the ep-groupoids by the object spaces. $s_\alpha: U_\alpha\to E_\alpha$ is just a section functor from the ep-groupoid base to the ep-groupoid bundle. So it is $\bigsqcup_{x\in\alpha} s_x$ between object spaces and preserves the ep-groupoid structures. $\psi_\alpha$ is the homeomorphism $\underline{s_\alpha^{-1}(0)}\to X$ defined by gluing $\psi_x$ together, and we denote its image by $X_\alpha|_{U_\alpha}$.

\item We recall that an atlas for an orbifold is $$(\mathcal{M}, \{G_a, U_a, \varphi_a:\underline{U_a}\to \mathcal{M})\}_{a\in \alpha})$$ consisting of a Hausdorff topological space $\mathcal{M}$ and for $a\in \alpha$, a finite group $G_a$ acts smoothly and effectively on a smooth manifold $U_a$ and a homeomorphism $\varphi_a$ onto its image, such that $\bigcup_{a\in \alpha}\varphi_a(\underline{U_a})=\mathcal{M}$ and for any $a, b\in \alpha$ and any point $z\in \varphi_a(\underline{U_a})\cap\varphi_{b}(\underline{U_{b}})$, there exists $c\in \alpha$ such that $z\in \varphi_{c}(\underline{U_{c}})$ and there exist open equivariant embeddings\footnote{By the definition, an equivariant embedding induces a group isomorphism between the stabilizer groups of every point in the domain and of its image in the target.} from $(G_{c}, U_{c}, \varphi_{c})$ to $(G_a, U_a, \varphi_a)$ and from $(G_{c}, U_{c}, \varphi_{c})$ to $(G_{b}, U_{b}, \varphi_{b})$ respectively (in particular, $\varphi_{c}(\underline{U_{c}})\subset\varphi_a(\underline{U_a})\cap\varphi_{b}(\underline{U_{b}})$).

In our setting, we define $\mathcal{E}_\alpha:=(\bigsqcup_{x\in \alpha} \underline{E_x})/\sim$, where $v\sim v'$ if $v$ and $v'$ can be identified via the quotient of a direct bundle coordinate change pointing in the direction of increasing stabilizer group size, and $\mathcal{U}_\alpha:=M_\alpha(\mathcal{G}):=(\bigsqcup_{x\in \alpha}\underline{U_x})/\sim$. We have natural inclusions $incl_x: \underline{E_x}\to \mathcal{E}_\alpha$ induced by the inclusions, and they are indeed inclusions by the alternative characterization of the maximality condition. The finite collection of charts $(\mathcal{E}_\alpha, (G_x, E_x, incl_x)_{x\in \alpha})$ is not an orbifold atlas, because in general a point in the intersection of a pair from $incl_x(\underline{E_x}), x\in \alpha$ is not always covered by $incl_y(\underline{E_y})$ contained in the intersection of the pair, for some $y\in \alpha$. But we can canonically add finitely many charts to $$(\mathcal{E}_\alpha, (G_x, E_x, incl_x)_{x\in \alpha})$$ to achieve the atlas property. Indeed, define $$\tilde \alpha:=\{T\in 2^{\alpha}\;|\; \bigcap_{x\in T} incl_x(\underline{E_x})\not=\emptyset\}.$$
Define $x_T:=\text{min} T$ using $\leq$, then by the strongly intersecting property, we know that we have coordinate changes $C_{x_T}|_{U_{x_T}}\to C_y|_{U_y}$ for all $y\in T$ with the domains $U_{y x_T}$. Define $$G_T:=G_{x_T},\; E_T:=E_{x_T}|_{(\bigcap_{y\in T} U_{y x_T})}.$$ Then for any $T, T'\in \tilde\alpha$ such that $incl_{T}(\underline{E_T})\cap incl_{T'}(\underline{E_{T'}})\not=\emptyset$, there exists $T'':=T\cup T'\in \tilde \alpha$, and equivariant embeddings from the chart indexed by $T''$ into the charts indexed by $T$ and $T'$ are respectively induced by $C_{x_{T''}}|_{U_{x_{T''}}}\to C_{x_T}|_{U_{x_{T}}}$ and $C_{x_{T''}}|_{U_{x_{T''}}}\to C_{x_{T'}}|_{U_{x_{T'}}}$. Thus, we have $incl_{T''}(\underline{E_{T''}})\subset incl_{T}(\underline{E_T})\cap incl_{T'}(\underline{E_{T'}})$ (and the equality actually holds). Therefore, the collection $$(G_T, E_T, incl_T), T\in \tilde\alpha$$ is an orbifold bundle atlas for $\mathcal{E}_\alpha$. Similarly for $\mathcal{U}_\alpha$. $s_T, T\in \tilde\alpha$ gives a coordinate representation of $\mathfrak{s}_\alpha$ with $\mathfrak{s}_\alpha^{-1}(0)\subset \mathcal{U}_\alpha$. We denote $\Psi_\alpha: \mathfrak{s}_\alpha^{-1}(0)\to X$ the homeomorphism induced by gluing $\psi_x, x\in \alpha$ together and denote its image by $X_\alpha|_{\mathcal{U}_\alpha}$. Note that the canonical transition from $C_x:=(G_x, s_x, \psi_x), x\in \alpha$ to $$\mathfrak{s}_\alpha: (\mathcal{U}_\alpha, \{(G_T, U_T, \iota_T)\}_{T\in \tilde\alpha})\to (\mathcal{E}_\alpha, \{(G_T, E_T, incl_T)\}_{T\in\tilde\alpha})\text{ and }\Psi_\alpha,$$ where $\iota_T$ is the natural inclusion $\underline{U_T}\to \mathcal{U}_\alpha$, contains no more new information but merely conforms to the convention of orbifold atlases. So if we adopt this convention, we can just write $\mathfrak{s}_\alpha: \mathcal{U}_\alpha\to \mathcal{E}_\alpha$ and $C_x|_{U_x}, x\in \alpha$ interchangeably.
\end{enumerate}
\end{definition}

\begin{remark} Both grouping approaches are equivalent due to the orbifold--ep-groupoid correspondence. We will adopt the ep-groupoid approach, as it is the closest in the notation to $C_x|_{U_x}=(G_x, s_x: U_x\to E_x, \psi_x)$ when it is just a single chart. Since the charts $(s_\alpha, \psi_\alpha)$ are already expressed in local coordinates, using this approach to discuss coordinate changes between different $\alpha\in S$ expressed in these charts is not really a hybrid concept of orbifolds and ep-groupoids. Ep-groupoids induced from orbifold coordinates are visual and we can perform constructions invariant under morphisms and preserving coordinate changes in the same way as in the orbifold approach.
\end{remark}

After picking up the convention/interpretation of \ref{ICHART}.(1), we can treat $\{C_x|_{U_x}\}_{x\in \alpha}$ as a single chart for each $\alpha\in S$, the good coordinate system $\{C_x|_{U_x}\}_{x\in\hat S}$ at the start of this section will become $$C_\alpha|_{U_\alpha}:=(s_\alpha: U_\alpha\to E_\alpha, \psi_\alpha), \alpha\in S.$$ In this setup, we do not have a global group action as in the Kuranishi chart, instead we have the morphisms of an ep-groupoid acting. Here we do not want to pass to the Morita-equivalence class of this ep-groupoid, and want to remember and work with this preferred ep-groupoid choice.

Recall $I(\hat{\mathcal{G}}):=\{(y,x)\in \hat S\times \hat S\;|\; y \hat\leq x,\; X_y|_{U_y}\cap X_x|_{U_x}\not=\emptyset\}$. Now define $$I(\beta,\alpha):=\{(y, x)\in I(\hat{\mathcal{G}})\;|\; y\in \beta,\; x\in\alpha\}$$ and for $y\in \beta$ define $I_y(\beta,\alpha):=\{x\in \alpha\;|\; (y, x)\in I(\beta,\alpha)\}$ and define $$U_{\alpha\beta}:=\bigsqcup_{y\in \beta}(\bigcup_{x\in I_y(\beta,\alpha)} U_{xy}).$$

\begin{definition}[coordinate change between grouped charts]\label{EPGROUPOIDCC} A coordinate change $C_\beta|_{U_\beta}\to C_\alpha|_{U_\alpha}$ is $\hat\phi_{\alpha\beta}: E_\beta|_{U_{\alpha\beta}}\to E_\alpha$ covering $\phi_{\alpha\beta}$ with the domain of the coordinate change as $U_{\alpha\beta}$, and it is induced by $\hat\phi_{xy}, (y,x)\in I(\beta,\alpha)$. $\hat\phi_{\alpha\beta}$ can be factorized into $E_\beta|_{U_{\alpha\beta}}\to \hat\phi_{\alpha\beta}(E_\beta|_{U_{\alpha\beta}})\to E_\alpha$, where $\hat\phi_{\alpha\beta}(E_\beta|_{U_{\alpha\beta}}):=\bigsqcup_{x\in\alpha}(\bigcup_{y\in\beta}\hat\phi_{xy}(E_y|_{U_{xy}}))$, the first map in the above factorization is a Morita-refinement\footnote{A map between ep-groupoids is called a \emph{Morita-refinement}, if it is a local diffeomorphism between the object spaces and induces a homeomorphism between the orbit spaces and bijections between the corresponding stabilizers. (In HWZ's terminology in \cite{PolyfoldIII}, it is called an equivalence of ep-groupoids.)} and the second map is an equivariant embedding between ep-groupoids (preserving the induced morphism structure on $\hat\phi_{\alpha\beta}(E_\beta|_{U_{\alpha\beta}})$ and the morphism structure of $E_\alpha$).
\end{definition}

Denote 
\begin{align*}I(\mathcal{G})&:=\{(\beta,\alpha)\in S\times S\;|\; I(\beta,\alpha)\not=\emptyset\}\\
&\;\equiv \{(\beta,\alpha)\in S\times S\;|\;\beta\leq \alpha, X_\beta|_{U_\beta}\cap X_\alpha|_{U_\alpha}\not=\emptyset\}.
\end{align*} We can regard a \emph{grouped good coordinate system} $$\mathcal{G}:=(X, (S, \leq), \{C_\alpha|_{U_\alpha}\}_{\alpha\in S}, \{C_\beta|_{U_\beta}\to C_\alpha|_{U_\alpha}\}_{(\beta,\alpha)\in I(\mathcal{G})})$$ as a strongly intersecting Hausdorff good coordinate system slightly generalized and indexed by a total order.

\begin{remark} In our cases of interests later, $\alpha$ in the total order $(S,\leq)$ is either determined by the bundle dimension as in \ref{BUNDLEDIMEN}, where we write $C_{\alpha_i}|_{U_{\alpha_i}}$ as $C_i|_{U_i}$, or more generally the tuple of bundle dimensions. In general we use $\alpha,\beta, \gamma\in S$ rather than $i,j,k\in\mathbb{N}$.
\end{remark}

\begin{definition}[shrinking of a grouped good coordinate system]\label{SHRINKINGGGCS} A \emph{shrinking} $U'_\alpha,\alpha\in S$ of a good coordinate system $C_\alpha|_{U_{\alpha}},\alpha\in S$ is the same as the definition in \ref{GCSSHRINKING}, except we always require a \emph{chart shrinking of a grouped chart} to be of the following form: $U'_{\alpha}:=\bigsqcup_{x\in\alpha} U'_x$ where $U'_x$ is a $G_x$-invariant open subset of $U_x$ with $x\in X_x|_{U'_x}$ for all $x\in\alpha$ such that for all $y,x\in \alpha$, if $X_y|_{U_y}\cap X_x|_{U_x}\not=\emptyset$, then $X_y|_{U'_y}\cap X_x|_{U'_x}\not=\emptyset$.
\end{definition}

\begin{remark}\label{REMDEF}
In other words, a good coordinate system $C_\alpha|_{U'_\alpha}, \alpha\in S$ is said to be a \emph{shrinking} of $C_\alpha|_{U_\alpha}, \alpha\in S$ if the underlying good coordinate system indexed by $\hat S$ of the former is a shrinking of the underlying good coordinate system indexed by $\hat S$ of the latter. A \emph{precompact shrinking} and other notions of a grouped good coordinate system are defined in the same way as before with the above form of grouped chart shrinking.
\end{remark}

We discuss briefly why grouping charts in $\alpha$ of the same bundle dimension ensures that bases $U_x, x\in \alpha$ have the same dimension in a connected component of $U_\alpha$, hence the validity of ep-groupoid/orbifold description of $U_\alpha$ and $E_\alpha$.

\begin{definition}[connectedness] A Kuranishi structure $\mathcal{K}$ on $X$ is said to be \emph{connected}, if the identification space $M(\mathcal{K})$ is path-connected.
\end{definition}

\begin{remark}\label{CONNECTEDNESSREMARK} We note the following:
\begin{enumerate} 
\item We can use the convention that different components of $U_\alpha$ might have different dimensions (but bases in the same components of $U_\alpha$ will have the same dimension), or consider each connected component of a general Kuranishi structures in turn.
\item We can define a good coordinate system $\mathcal{G}$ to be \emph{connected} if $M(\mathcal{G})$ is path-connected. Figure \ref{figure04} shows that a connected Kuranishi structure can admit a nonconnected good coordinate system.
\item As coordinate changes satisfy the tangent bundle condition, if there is a coordinate change between charts at $q$ and $p$, then $\text{rank} E_p-\text{rank} E_q=\text{dim} U_p-\text{dim} U_q$. So, $\text{dim} E_p-\text{dim} E_q=2(\text{dim} U_p-\text{dim} U_q)$. Therefore, for a connected Kuranishi structure, $\text{dim} E_p-2\text{dim} U_p$ is constant independent of $p$. Therefore, for a good coordinate system obtained from it, $\text{dim} E_x-2\text{dim} U_x, x\in S$ is also constant.
\item Therefore for a good coordinate system obtained from a connected Kuranishi structure, bases $U_x$, $x\in \alpha\subset S_i:=\{x\in \hat S\;|\; \text{dim} E_x=i\}$, of charts of the same bundle dimension $i$ have the same dimension.
\end{enumerate}
\end{remark}

\subsection{The definition and construction of a level-1 coordinate change}

We introduce some notions before defining a level-1 coordinate change.

\begin{definition}[vector-bundle-like submersion]\label{VBLIKE} Let $\pi: W\to U$ be a fiber bundle ep-groupoid. The projection functor $\pi: W\to U$ is called \emph{vector-bundle-like submersion}, if there exists a fiber bundle ep-groupoid diffeomorphism $\Phi$ mapping $\pi: W\to U$ to a fiber bundle ep-groupoid $N\to U$ such that the bundle ep-groupoid projection $N\to U$ can be equipped with a vector bundle ep-groupoid structure.
\end{definition}

\begin{example}[$\pi^g_{AB}$ in definition \ref{SUBMERSIONMODEL}]\label{FSHRINKING} We will only consider the vector-bundle-like submersions in the following setting: Let $B\subset D$ be an invariant embedded submanifold in the object space of an ep-groupoid $D$ with the induced morphisms. Choose an invariant Riemannian metric $g$ on $D$, we can define $\pi':=\pi^g_{A'B}: A'\to B$ as in \ref{SUBMERSIONMODEL} mapping from an invariant immersed tubular neighborhood $A'$ of $B$ in $D$. (Here $\tilde A$ is not necessarily embedded in $D$, which is irrelevant here.) As $U(N_B A')$ appeared in \ref{SUBMERSIONMODEL} can be dilated into a vector bundle ep-groupoid, $\pi': A'\to B$ is a vector-bundle-like submersion. Via the unit disk bundle ep-groupoid of the vector bundle ep-groupoid model of $\pi'$, we can choose an invariant \emph{fiber-shrinking} of $A$ of $A'$ so that for each $x\in B$, the closure of the fiber $A_x:=(\pi'|_A)^{-1}(x)$ in $A'_x:=(\pi')^{-1}(x)$ is a compact disk with smooth boundary in $A'$, $x\mapsto \overline{A_x}$ is smooth, and $\pi:=\pi'|_{A}\equiv\pi_{AB}^g: A\to B$ is also a vector-bundle-like submersion.
\end{example}

\begin{remark} We will see that the notion of a vector-bundle-like submersion is a way to remember a vector bundle ep-groupoid structure of a fiber bundle ep-groupoid up to the identification by a bundle ep-groupoid diffeomorphism, so that when we require the compatibility condition involving fiber bundle ep-groupoids, we will not require the associated vector bundle ep-groupoids to satisfy the compatibility condition via the identifications.
\end{remark}

\begin{definition}[strong open neighborhood\footnote{Ep-groupoid structures are not relevant in this definition.}, strongly embeddedness, figure \ref{figure2_1}]\label{STRONGNBD} Let $U$ be a manifold and let $U'$ be an open subset of $U$ such that the closure $\overline{U'}$ is compact in $U$ and $U'$ is the interior of $\overline{U'}$. Let $U'_1$ be a submanifold in $U'$. $U'_1$ is said to be \emph{strongly embedded} in $(U', U, W)$, if 
\begin{enumerate}
\item there exists an embedded submanifold $U_1$ in $U$ such that
\begin{enumerate}
\item $\overline{U'_1}$ is compact in $U_1$, where the closure $\overline{U'_1}$ is taken in $U$, and 
\item $U'_1$ is an open interior of $\overline{U'_1}$,
\end{enumerate}
\item there exists a vector-bundle-like submersion $\pi: \hat W\to U_1$ for an immersed open subset $\hat W$ of $U$ such that $\pi^{-1}(V)$ is embedded over some open neighborhood $V$ of $\overline{U'_1}$ in $U_1$ (in particular, the fiber $\pi^{-1}(z)$ of $z\in U'_1$ does not degenerate as $z$ approaches $\overline{U'_1}\backslash U'_1$), and
\item $W=\pi^{-1}(U'_1)$ and $U'\cap \hat W=W$.
\end{enumerate}
We also say the $(U', U, W)$ is a \emph{strong open neighborhood} of $U'_1$.
\end{definition}

\begin{figure}[htb]
  \begin{center}

\begin{tikzpicture}[scale=1]
\hspace{0 cm}, 

\filldraw[thick, dashed, color=black!20] (8,0) circle (4 and 2.5);

\filldraw[color=black!45](5.37, 0.63) arc (158: 172.2:4) arc (172: 192: 1).. controls (5.2, -1) and (7, -1.6).. (8,-1.4) .. controls (9, -1.3) and (10.5, -.7).. (10.5, 0) .. controls (10.5, 0.7) and (9.4, 1.2).. (9,1.3)..controls (8, 1.5) and (5.9, 1.4)..(5.37, 0.63);

\begin{scope}
\clip (5.37, 0.63) arc (158: 168.2:4)--(10, -3)--(10, 2)--(5.37, 0.63);
\filldraw[dashed, color=black!60] 
(4,0).. controls (6, 1.5) and (7,-.2) .. (9, 0.8) arc (155:175:4).. controls (8, -0.8) and (7,-1.2) .. (6,-0.4) .. controls (5,.2) and (5,-0.2).. (4,0);

\clip (4,0).. controls (6, 1.5) and (7,-.2) .. (9, 0.8) arc (155:175:4).. controls (8, -0.8) and (7,-1.2) .. (6,-0.4) .. controls (5,.2) and (5,-0.2) ..(4,0);
\foreach \len in {4,4.2,...,11}
{
\draw[black!75] (\len +1, 2.4) arc (125:180:4);
}
\end{scope}

\begin{scope}
\clip (2.5, -0.4) rectangle (8.765, 0.4);
\draw[thick] (5.2,0.13) .. controls (5.8,0.2) and (6.2,0.3) ..(7,0) .. controls (8,-0.5) and (8.5,0.5) .. (9,0);
\end{scope}

\draw[very thick, densely dotted, color=black!75] (5.37, 0.63) arc (158: 172.2:4) arc (172: 192: 1).. controls (5.2, -1) and (7, -1.6).. (8,-1.4) .. controls (9, -1.3) and (10.5, -.7).. (10.5, 0) .. controls (10.5, 0.7) and (9.4, 1.2).. (9,1.3)..controls (8, 1.5) and (5.9, 1.4)..(5.37, 0.63);
\node at (10, 1.8) {$U$};

\node at (9.5, .2) {$U'$};

\path[->] node at (5.4,-2.7)[below] {$U'_1$} edge [bend left] (6,.2);

\path[->] node at (10,-3)[below] {$W$} edge [bend right] (8.02,-.47);

\draw[thick] (5.205, 0.138) circle (1.5pt);
\draw[thick] (8.76, 0.144) circle (1.5pt);

\end{tikzpicture}

  \end{center}
\caption[A strong open neighborhood.]{A strong open neighborhood.}
\label{figure2_1}
\end{figure}
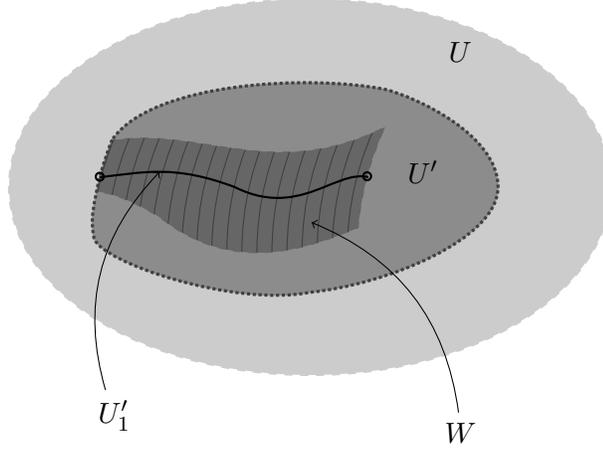

\begin{definition}[(precompact) shrinking of charts with index $J$]\label{PRECOMPACTCSHRINK} Let $\mathcal{G}$ be a strongly intersecting Hausdorff good coordinate system. Fix a subset $J$ of $S$. $\{U'_\alpha\}_{\alpha\in J}$ is said to be a \emph{shrinking} of $\{U_\alpha\}_{\alpha\in J}$ if for all $\alpha\in J$, $U'_\alpha$ is an invariant open subset of $U_\alpha$ and $\{C_\alpha|_{U'_\alpha}\}_{\alpha\in J}\bigsqcup \{C_\beta|_{U_\alpha}\}_{\beta\in S\backslash J}$ with induced coordinate changes among them is a shrinking of $\mathcal{G}$ in the sense of definition \ref{SHRINKINGGGCS}. A shrinking $\{U'_\alpha\}_{\alpha\in J}$ is called a \emph{precompact shrinking}, if additionally for all $\alpha\in J$, $\overline{U'_\alpha}$ is compact in $U_\alpha$ and $U'_\alpha$ is the interior of $\overline{U'_\alpha}$ 
\end{definition}

This is consistent with a (precompact) shrinking of good coordinate systems defined in \ref{SHRINKINGGGCS} and \ref{REMDEF} in the sense that the special case of a (precompact) shrinking of charts with $J=S$ is a precompact shrinking of good coordinate systems.


\begin{definition}[level-1 coordinate change]\label{LEVEL1CCHANGE} Let $$\mathcal{G}:=(X, (S,\leq), \{C_\alpha|_{U_\alpha}\}_{\alpha\in S})$$ be a strongly intersecting Hausdorff good coordinate system indexed by a total order $(S, \leq)$. Let $\beta, \alpha\in S$ with $(\beta,\alpha)\in I(\mathcal{G})$. Let $\{U'_\beta, U'_\gamma\}$ be a precompact shrinking of $\{U_\beta, U_\alpha\}$ as in definition \ref{PRECOMPACTCSHRINK}.

A \emph{level-1 coordinate change} $C_\beta|_{U_\beta'}\overset{\text{level-1}}{\to} C_\alpha|_{U'_\alpha}$ from $C_\beta|_{U'_\beta}$ to $C_\alpha|_{U'_\alpha}$ is a tuple $(C_\beta|_{U'_\beta}\to C_\alpha|_{U'_\alpha}, W'_{\alpha\beta}, \pi_{\alpha\beta}, \tilde\pi_{\alpha\beta}, \hat\pi_{\alpha\beta})$ consisting of the following:
\begin{enumerate}
\item a restricted coordinate change $$(C_\beta|_{U'_\beta}\to C_\alpha|_{U'_\alpha})=(\phi_{\alpha\beta}, \hat\phi_{\alpha\beta}, U'_\beta, U'_\alpha,  U'_{\alpha\beta})$$ such that the domain $U'_{\alpha\beta}$ of the coordinate change $C_\beta|_{U'_\beta}\to C_\alpha|_{U'_\alpha}$ is determined by restricting $C_\beta|_{U_\beta}\to C_\alpha|_{U_\alpha}$;
\item a fiber bundle ep-groupoid projection functor $\pi_{\alpha\beta}: W'_{\alpha\beta}\to \phi_{\alpha\beta}(U'_{\alpha\beta})$, where $W'_{\alpha\beta}$ is an embedded open tubular neighborhood of $\phi_{\alpha\beta}(U'_{\alpha\beta})$ in the object space of $U'_\alpha$ and regarded as an ep-groupoid with the morphisms being the $W'_{\alpha\beta}$-invariant morphisms in $U'_\alpha$, such that
\begin{enumerate}
\item $\pi_{\alpha\beta}$ is a vector-bundle-like submersion,
\item $(U'_\beta, U_\beta, W'_{\alpha\beta})$ is a strong open neighborhood of $\phi_{\alpha\beta}(U'_{\alpha\beta})$ as in \ref{STRONGNBD},
\item $W'_{\alpha\beta}\backslash \phi_{\alpha\beta}(U'_{\alpha\beta})$ contains no zeros of $s_\alpha|_{W'_{\alpha\beta}}$, and
\item for any $(y, x)\in\beta\times\alpha$, presenting the orbits of $\pi_{\alpha\beta}^{-1}(\phi_{xy}(U'_{xy}))$ under the morphisms of $U'_\alpha$ disjointly as $O^{xy}_1, \cdots, O^{xy}_{k_{xy}}$, $\iota^{xy}:\bigsqcup_i \overline{O^{xy}_i}\to U'_\alpha$, where $\iota^{xy}|_{\overline{O^{xy}_i}}$ is the inclusion, is injective;
\end{enumerate}
\item $\tilde\pi_{\alpha\beta}:\tilde E_{\alpha\beta}\to \hat\phi_{\alpha\beta}(E_\beta|_{U'_{\alpha\beta}})$, a fiberwise-isomorphic vector bundle map covering $\pi_{\alpha\beta}: W'_{\alpha\beta}\to \phi_{\alpha\beta}(U'_{\alpha\beta})$, where $\tilde E_{\alpha\beta}\to W'_{\alpha\beta}$ is a subbundle of the object space of $E_\alpha|_{W'_{\alpha\beta}}$, $\tilde E_{\alpha\beta}$ is an ep-groupoid after equipped with the $\tilde E_{\alpha\beta}$-invariant morphisms in $E_\alpha|_{U'_\alpha}$, and $\tilde E_{\alpha\beta}\to W'_{\alpha\beta}$ is a vector bundle ep-groupoid projection functor, such that 
\begin{enumerate}
\item $\tilde\pi_{\alpha\beta}$ is a functor between vector bundle ep-groupoids\footnote{Also a fiber bundle ep-groupoid projection functor.}, 
\item $\tilde E_{\alpha\beta}|_{\phi_{\alpha\beta}(U'_{\alpha\beta})}=\hat\phi_{\alpha\beta}(E_\beta|_{U'_{\alpha\beta}})$,
\item the zero set of $s_\alpha|_{W'_{\alpha\beta}}/\tilde E_{\alpha\beta}: W'_{\alpha\beta}\to E_\alpha|_{W'_{\alpha\beta}}/\tilde E_{\alpha\beta}$ is exactly $\phi_{\alpha\beta}(U'_{\alpha\beta})$, and 
\item the well-defined linearization of $s_\alpha|_{W'_{\alpha\beta}}/\tilde E_{\alpha\beta}$ over $\phi_{\alpha\beta}(U'_{\alpha\beta})$ is transverse; and
\end{enumerate}
\item a fiberwise-projection functor $\hat\pi_{\alpha\beta}: E_\alpha|_{W'_{\alpha\beta}}\to \tilde E_{\alpha\beta}$ between vector bundle ep-groupoids, where $E_\alpha|_{W'_{\alpha\beta}}$ is an ep-groupoid with the morphisms being the $E_\alpha|_{W'_{\alpha\beta}}$-invariant morphisms in $E_\alpha|_{U'_\alpha}$. 
\end{enumerate}
\end{definition}

\begin{remark}\label{LEVEL1CCREM} There are a few important remarks:
\begin{enumerate}[(I)]
\item The following commutative diagram
$$\begin{CD}
\tilde E_{\alpha\beta} @>\tilde\pi_{\alpha\beta}>> \hat\phi_{\alpha\beta}(E_\beta|_{U'_{\alpha\beta}})\\
@VVV @VVV\\
W'_{\alpha\beta} @>\pi_{\alpha\beta}>> \phi_{\alpha\beta}(U'_{\alpha\beta})
\end{CD}$$
of bundle ep-groupoid projection functors appeared in \ref{LEVEL1CCHANGE}.(2) is not invariant\footnote{But is invariant under the natural representation of stabilizer groups of the morphisms of $\phi_{\alpha\beta}(U'_{\alpha\beta})$, for example, hence they are bundle ep-groupoid projections.} under the morphisms of $E_\alpha|_{U'_\alpha}$, but we can always take the invariant hull $\text{orbit}(\tilde E_{\alpha\beta})$ of this picture, the union of the orbits of $\tilde E_{\alpha\beta}$ under the morphisms of $E_\alpha|_{U'_\alpha}$. We can easily convert the definition \ref{LEVEL1CCHANGE} into one which is invariant under $E_\alpha|_{U'_\alpha}$, or vice versa, with the contents staying the same. The current choice is purely a formalism: 
\begin{enumerate}[(i)] 
\item The form of coordinate changes is consistent with ones in the original definition: $E_\beta|_{U'_\beta}\to \hat\phi_{\alpha\beta}(E_\beta|_{U'_{\alpha\beta}})\to E_\alpha$ rather than $E_\beta|_{U'_\beta}\to \text{orbit}(\hat\phi_{\alpha\beta}(E_\beta|_{U'_{\alpha\beta}}))\to E_\alpha$.
\item If sticking to (i), being invariant under the morphisms of $E_\alpha|_{U'_\alpha}$ will require using $\text{orbit}(\cdot)$, making notations clustered in \ref{LEVEL1CCHANGE}. ($\text{orbit}(\cdot)$ is deferred to the inductive constructions in the proof.)
\end{enumerate}
Therefore, we think of the level-1 structure as data near the image ep-groupoid of the coordinate change and it can be made invariant under the morphisms in the target by taking the union of its orbits.
\item Define $\text{dim}\;\alpha:=\text{dim}\;E_x$ for any $x\in\alpha$. For a total order partition, $\text{dim}\;\alpha$ is well-defined. If $\text{dim}\;\beta=\text{dim}\;\alpha$, then the level-1 structure for $C_\beta|_{U'_\beta}\to C_\alpha|_{U'_\alpha}$ is trivial: $\pi_{\alpha\beta}=\text{Id}$, $\tilde\pi_{\alpha\beta}=\text{Id}$ and $\hat\pi_{\alpha\beta}=\text{Id}$.
\item \ref{LEVEL1CCHANGE}.(2)(d) will allow room to extend constructions on the orbit of the image of the coordinate change over $U'_\alpha$. The orbits of $\pi_{\alpha\beta}^{-1}(\phi_{xy}(U'_{xy}))$ under the morphisms of $U'_\alpha$ either are disjoint or coincide, so the disjoint presentation exists and is unique up to reordering. In particular, taking the closure of the disjoint presentation of the orbits of $\phi_{xy}(U'_{xy})$ will not create double points, which is not true for general $\mathcal{G}$. 
\item We can regard $C_\beta|_{U_\beta}\to C_\alpha|_{U_\alpha}$ as obtained from grouping charts, or a \emph{general coordinate change} between charts based on general ep-groupoids, defined to be of the form of a Morita-refinement followed by an embedding functor as in \ref{EPGROUPOIDCC}. The extensibility condition \ref{LEVEL1CCHANGE}(2)(d) is then: for every neighborhood $Q_z$ of every $z\in \phi_{\alpha\beta}(U'_{\alpha\beta})$ such that the $Q_z$-invariant morphisms in $U'_\alpha$ becomes the natural representation of the stabilizer group of $z$, with the orbits of $\pi_{\alpha\beta}^{-1}(Q_z)$ under the morphisms of $U'_\alpha$ disjointly listed as $O^{\alpha z}_1, \cdots, O^{\alpha z}_{k_{\alpha z}}$, $\iota^{\alpha z}:\bigsqcup_i \overline{O^{\alpha z}_i}\to U'_\alpha$ is injective, where $\iota|_{\overline{O^{\alpha z}_i}}$ is the inclusion.
\end{enumerate}
\end{remark}

\begin{proposition}[existence of a level-1 coordinate change]\label{EOFLEVEL1CC} Let $$(C_\beta|_{U_\beta}\to C_\alpha|_{U_\alpha})=(\phi_{\alpha\beta},\hat\phi_{\alpha\beta}, U_{\alpha\beta}), (\beta, \alpha)\in I(\mathcal{G})$$ be a coordinate change in a strongly intersecting Hausdorff good coordinate system $\mathcal{G}=\{C_\alpha|_{U_\alpha}\}_{\alpha\in S}$\footnote{See \ref{LEVEL1CCREM}.(4).}. After choosing an invariant Riemannian metric $g_{\alpha\beta}$ on $U_\alpha$, we can construct a level-1 coordinate change from it in the sense of definition \ref{LEVEL1CCHANGE}. Here the special case $\beta=\alpha$ is also trivially subsumed.
\end{proposition}

\begin{proof}
\begin{enumerate}
\item Construct a vector-bundle-like submersion ($W_{\alpha\beta}$ immersed):

We define the normal bundle ep-groupoid projection $pr_{\alpha\beta}:N_{\alpha\beta}:=N_{\phi_{\alpha\beta}(U_{\alpha\beta})} U_\alpha\to\phi_{\alpha\beta}(U_{\alpha\beta})$. Using $g_{\alpha\beta}$, we define the normal exponential map functor $\exp_{\alpha\beta}^{g_{\alpha\beta}}: U(N_{\alpha\beta})\to W_{\alpha\beta}$, which is a fiberwise diffeomorphism. Here $U(N_{\alpha\beta})$ is an invariant tubular neighborhood of the zero section in $N_{\alpha\beta}$ such that $\exp_{\alpha\beta}^{g_{\alpha\beta}}(U(N_{\alpha\beta})\cap (N_{\alpha\beta})_z)$ has a compact closure in $U_\alpha$ for all $z\in \phi_{\alpha\beta}(U_{\alpha\beta})$; and the image $W_{\alpha\beta}$ is an invariant \emph{immersed} tubular neighborhood of $\phi_{\alpha\beta}(U_{\alpha\beta})$. We define $\pi_{\alpha\beta}:=pr_{\alpha\beta}\circ(\exp_{\alpha\beta}^{g_{\alpha\beta}})^{-1}:W_{\alpha\beta}\to \phi_{\alpha\beta}(U_{\alpha\beta})$ and it is a fiber-bundle ep-groupoid projection functor and a vector-bundle-like submersion.

\item Choose a precompact shrinking $\{U'_\beta, U''_\alpha\}$ of $\{U_\beta, U_\alpha\}$:

We choose invariant open neighborhoods $U'_\beta$ of $U_\beta$ and $U''_\alpha$ of $U_\alpha$ such that $\overline{U'_\beta}$ is compact in $U_\beta$ with $U'_\beta$ being the interior of $\overline{U'_\beta}$, $\overline{U''_\alpha}$ is compact in $U_\alpha$ with $U''_\alpha$ being the interior of $\overline{U''_\alpha}$, and $\tilde{\mathcal{G}}:=\{C_{\beta}|_{U'_\beta}, C_{\alpha}|_{U''_{\alpha}}\}\bigsqcup\{C_\gamma|_{U_\gamma}\}_{\gamma\in S\backslash\{\beta,\alpha\}}$ with the induced coordinate changes among the charts is a shrinking of $\mathcal{G}$. In particular, $X$ is still covered and whether bases in any two given charts intersect or not in the identification space remains unchanged.

\item Shrink $C_\beta|_{U_\beta}$ to the chosen $U'_\beta$ in item (2) first and establish all the properties in \ref{LEVEL1CCHANGE}.(2) except the strong open neighborhood property:

Denote the induced domain of the coordinate change $C_{\beta}|_{U'_\beta}\to C_{\alpha}|_{U_\alpha}$ by $U^m_{\alpha\beta}$. Here $m$ stands for \emph{mixed}. For a sufficiently small invariant fiber-shrinking $W^m_{\alpha\beta}$ of $\pi_{\alpha\beta}^{-1}(\phi_{\alpha\beta}(U^m_{\alpha\beta}))$ as in \ref{FSHRINKING}, we have:
\begin{enumerate}[(i)]
\item $W^m_{\alpha\beta}$ is embedded in $U_\alpha$,
\item $W^m_{\alpha\beta}\backslash \phi_{\alpha\beta}(U_{\alpha\beta})$ contains no zeros of $s_p$, because $$X_\beta|_{U^m_{\alpha\beta}}=X_\alpha|_{\phi_{\alpha\beta}(U^m_{\alpha\beta})}$$ is an open set in $X$, and
\item item 2(d) in definition \ref{LEVEL1CCHANGE} holds, because $W^m_{\alpha\beta}$ is a fiber shrinking and $U'_\alpha$ is a precompact shrinking of $U_\alpha$.
\end{enumerate}

We now have a vector-bundle-like submersion $\pi: W^m_{\alpha\beta}\to \phi_{\alpha\beta}(U^m_{\alpha\beta})$ from an invariant embedded tubular neighborhood $W^m_{\alpha\beta}$ of $\phi_{\alpha\beta}(U^m_{\alpha\beta})$, with $W^m_{\alpha\beta}$ satisfying items 2(a), 2(c) and 2(d) in \ref{LEVEL1CCHANGE}.

\item Shrink $C_\alpha|_{U_\alpha}$ to $U''_\alpha$ for the strong open neighborhood property:

Denote the induced domain of the coordinate change $C_{\beta}|_{U'_\beta}\to C_{\alpha}|_{U''_\alpha}$ by $U^{(3)}_{\alpha\beta}$ and define $$U^{(3)}_\alpha:=(U''_\alpha\backslash \text{orbit}((\pi_{\alpha\beta}|_{W^m_{\alpha\beta}})^{-1}(\pi_{\alpha\beta}((U_\alpha\backslash U''_\alpha)\cap W^m_{\alpha\beta}))))\bigcup$$ $$\text{orbit}((\pi_{\alpha\beta}|_{W^m_{\alpha\beta}})^{-1}(\phi_{\alpha\beta}(U^{(3)}_{\alpha\beta}))).$$ 

See figure \ref{figure2_2}:
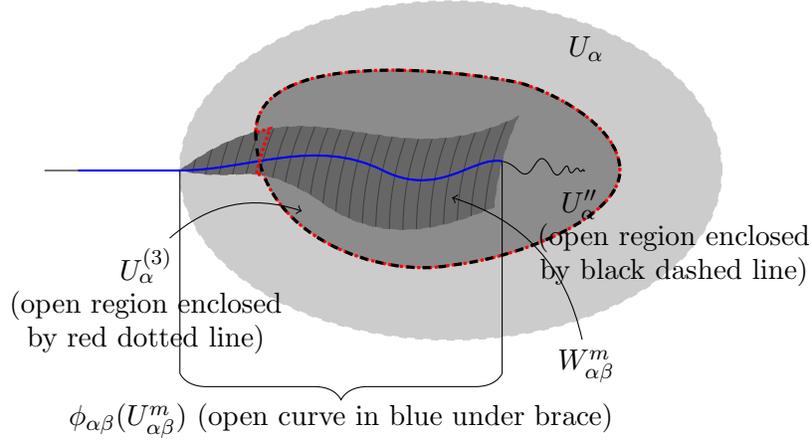
\begin{figure}[htb]
  \begin{center}

\begin{tikzpicture}[scale=0.9]
\hspace{0 cm}, 

\filldraw[thick, dashed, color=black!20] (8,0) circle (4 and 2.5);

\filldraw[color=black!45](5.12, 0.56).. controls (5.13, -1) and (7, -1.6).. (8,-1.4) .. controls (9, -1.3) and (10.5, -.7).. (10.5, 0) .. controls (10.5, 0.7) and (9.4, 1.2).. (9,1.3)..controls (8, 1.5) and (5.13, 1.8)..(5.12, 0.56);

\draw[very thick, densely dotted, color=red](5.12, 0.56).. controls (5.13, -1) and (7, -1.6).. (8,-1.4) .. controls (9, -1.3) and (10.5, -.7).. (10.5, 0) .. controls (10.5, 0.7) and (9.4, 1.2).. (9,1.3)..controls (8, 1.5) and (5.13, 1.8)..(5.12, 0.56);

\filldraw[dashed, color=black!60] (4,0).. controls (6, 1.5) and (7,-.2) .. (9, 0.8) arc (155:175:4).. controls (8, -0.8) and (7,-1.2) .. (6,-0.4) .. controls (5,.2) and (5,-0.2).. (4,0);
\begin{scope}
\clip (4,0).. controls (6, 1.5) and (7,-.2) .. (9, 0.8) arc (155:175:4).. controls (8, -0.8) and (7,-1.2) .. (6,-0.4) .. controls (5,.2) and (5,-0.2) ..(4,0);
\foreach \len in {4,4.2,...,11}
{
\draw[black!75] (\len +1, 2.4) arc (125:180:4);
}
\end{scope}

\draw[very thick, dashed](5.12, 0.56).. controls (5.13, -1) and (7, -1.6).. (8,-1.4) .. controls (9, -1.3) and (10.5, -.7).. (10.5, 0) .. controls (10.5, 0.7) and (9.4, 1.2).. (9,1.3)..controls (8, 1.5) and (5.13, 1.8)..(5.12, 0.56);

\draw (2,0) -- (4,0);
\draw (4,0) .. controls (5,0) and (6,0.5) .. (7,0);
\draw (7,0) .. controls (8,-0.5) and (8.5,0.5) .. (9,0);
\draw (9,0) .. controls (9.25,-0.25) and (9.25,0.5) .. (9.5,0);
\draw (9.5,0) .. controls (9.625,-0.23) and (9.7,0.23) .. (9.75,0);
\draw (9.75,0) .. controls (9.8,-0.12) and (9.85,0.12) .. (9.9,0);
\draw (9.9,0) .. controls (9.94,-0.05) and (9.95,0.05) .. (9.96,0)--(9.97,-0.008)--(9.98,-0.002)--(9.985,-0.001);

\begin{scope}
\clip (2.5, -0.4) rectangle (8.76, 0.4);
\draw[thick, blue] (2.5,0) -- (4,0) .. controls (5,0) and (6,0.5) ..(7,0) .. controls (8,-0.5) and (8.5,0.5) .. (9,0);
\end{scope}
\draw[very thick, densely dotted, color=red] (5.13, 0.57)--(5.335, 0.62) arc (158: 168.2:4) --(5.245,-0.06);
\node at (10, 1.8) {$U_\alpha$};
\node at (9.9, -.5) {$U''_{\alpha}$};
\node at (11.3,-1){(open region enclosed};
\node at (11.3,-1.5){by black dashed line)};

\draw (4,0)--(4,-3) (8.76,0.12)--(8.76,-3);
\draw [decorate,decoration={brace,amplitude=10pt}, xshift=0, yshift=0pt]
(8.76,-3)--(4,-3) node [midway,yshift=-7.5pt, below]{$\phi_{\alpha\beta}(U^m_{\alpha\beta})$ (open curve in blue under brace)};
\path[->] node at (10,-2.5)[below] {$W^m_{\alpha\beta}$} edge [bend right] (8.02,-.4);
\path[->] node at (3.5,-1)[below] {$U^{(3)}_{\alpha}$} edge [bend left] (5.8,-.6);
\node at (3.5,-2){(open region enclosed};
\node at (3.5,-2.5){by red dotted line)};
\end{tikzpicture}

  \end{center}
\caption[Constructing a strong open neighborhood.]{Constructing a strong open neighborhood.}
\label{figure2_2}
\end{figure}

Then the domain of the coordinate change $C_{\beta}|_{U'_\beta}\to C_{\alpha}|_{U^{(3)}_\alpha}$ is still $U^{(3)}_{\alpha\beta}$. Now, we have $W^{(3)}_{\alpha\beta}:=(\pi_{\alpha\beta}|_{W^m_{\alpha\beta}})^{-1}(\phi_{\alpha\beta}(U^{(3)}_{\alpha\beta}))$ as an embedded tubular neighborhood\footnote{An ep-groupoid with the morphisms being the $W^{(3)}_{\alpha\beta}$-invariant morphisms in $U^{(3)}_\alpha$.} of $\phi_{\alpha\beta}(U^{(3)}_{\alpha\beta})$ in $U'_\alpha$, with the vector-bundle-like submersion $\pi_{\alpha\beta}$ onto $\phi_{\alpha\beta}(U^{(3)}_{\alpha\beta})$, and $(U^{(3)}_\alpha, U_\alpha, W^{(3)}_{\alpha\beta})$ is now a strong open neighborhood of $\phi_{\alpha\beta}(U^{(3)}_{\alpha\beta})$. So far, we have chosen $C_\beta|_{U'_\beta}\to C_\alpha|_{U^{(3)}_\alpha}$ and $\pi_{\alpha\beta}|_{W^{(3)}_{\alpha\beta}}$ satisfying \ref{LEVEL1CCHANGE}.(2).

\item Establish $\tilde\pi_{\alpha\beta}$ over $W^{(3)}_{\alpha\beta}$ (the quotient section not yet transverse):
 
Now, by choosing an invariant connection on $E_\alpha|_{U^{(3)}_\alpha}$, we can parallel transport the fiber $(\hat\phi_{\alpha\beta}(E_\beta|_{U^{(3)}_{\alpha\beta}}))_z$ along the $g_{\alpha\beta}$-geodesic joining $z$ to $z'\in \pi_{\alpha\beta}^{-1}(z)$. This extends $\hat\phi_{\alpha\beta}(E_\beta|_{U^{(3)}_{\alpha\beta}})$ to an invariant subbundle ep-groupoid $\tilde\pi_{\alpha\beta}: \tilde E^{(3)}_{\alpha\beta}\to W^{(3)}_{\alpha\beta}$ of $E_\alpha|_{W^{(3)}_{\alpha\beta}}$, where the morphisms of the vector bundle ep-groupoid $E_\alpha|_{W^{(3)}_{\alpha\beta}}$ are the $E_\alpha|_{W^{(3)}_{\alpha\beta}}$-invariant morphisms in $E_\alpha|_{U^{(3)}_\alpha}$.

\item Establish the transversality of the quotient section over a fiber shrinking $W'_{\alpha\beta}$ of the bundle $W^{(3)}_{\alpha\beta}$ restricted to a smaller base:

With $\tilde E^{(3)}_{\alpha\beta}$ constructed, we now have the quotient section functor $s_\alpha|_{W^{(3)}_{\alpha\beta}}/\tilde E^{(3)}_{\alpha\beta}$, with its zero set exactly $\phi_{\alpha\beta}(U^{(3)}_{\alpha\beta})$. Its linearization, well-defined over $\phi_{\alpha\beta}(U^{(3)}_{\alpha\beta})$, is surjective over $s_\alpha^{-1}(0)\cap \phi_{\alpha\beta}(U^{(3)}_{\alpha\beta})$ by the tangent bundle condition. By the implicit function theorem, $s_{\alpha}|_{W^{(3)}_{\alpha\beta}}/\tilde E^{(3)}_{\alpha\beta}$ is transverse when restricting to an invariant fiber-shrinking $W'_{\alpha\beta}$ of $(\pi_{\alpha\beta}|_{W^{(3)}_{\alpha\beta}})^{-1}(\phi_{\alpha\beta}(U'_{\alpha\beta}))$, where $\phi_{\alpha\beta}(U'_{\alpha\beta})$ is an invariant open neighborhood of $s_\alpha^{-1}(0)\cap \phi_{\alpha\beta}(U^{(3)}_{\alpha\beta})$ in $\phi_{\alpha\beta}(U^{(3)}_{\alpha\beta})$. In the next item, we will see $U'_{\alpha\beta}$ is induced from a shrinking $U'_\alpha$ of $U^{(3)}_{\alpha}$ in the sense of \ref{PRECOMPACTCSHRINK}. 

\item Define a shrinking $U'_{\alpha}$ of $U^{(3)}_\alpha$ with the induced domain of coordinate change still being $U'_{\alpha\beta}$ while preserving $W'_{\alpha\beta}$ such that $\{U'_\beta, U'_\alpha\}$ is a precompact shrinking of $\{U_\beta, U_\alpha\}$ as in \ref{PRECOMPACTCSHRINK}:

Define $U'_\alpha:=U^{(3)}_\alpha\backslash (\overline{W^{(3)}_{\alpha\beta}}\backslash W'_{\alpha\beta})$, then clearly the domain of coordinate change $C_\beta|_{U'_\beta}\to C_\alpha|_{U'_{\alpha}}$ is still $U'_{\alpha\beta}$. 

Since $\overline{W^{(3)}_{\alpha\beta}}\backslash W'_{\alpha\beta}$ contains no zeros of $s_\alpha$, during $U''_\alpha$ changing\footnote{Not a shrinking, as $U^{(3)}_\alpha\not\subset U''_\alpha$ in general. See figure \ref{figure2_2}.} to $U^{(3)}_\alpha$ then shrinking to $U'_{\alpha}$, we have $X_\alpha|_{U''_\alpha}=X_\alpha|_{U^{(3)}_\alpha}=X_\alpha|_{U'_\alpha}$.

So the underlying strongly intersecting good coordinate system $$\mathcal{G}^m:=\{C_{\beta}|_{U'_\beta}, C_{\alpha}|_{U'_{\alpha}}\}\bigsqcup\{C_\gamma|_{U_\gamma}\}_{\gamma\in S\backslash\{\beta,\alpha\}}$$  with induced coordinate changes among them is still a shrinking of $\mathcal{G}$. We still have the same index set for coordinate changes, namely, $I(\mathcal{G}^m)=I(\tilde{\mathcal{G}})=I(\mathcal{G})$, and $X$ is still covered. Here $\tilde{\mathcal{G}}$ is defined within item (2) in the current proof.

\item Show that $(U'_\alpha, U_\alpha, W'_{\alpha\beta})$ is automatically a strong neighborhood of $\phi_{\alpha\beta}(U'_{\alpha\beta})$ by the construction:

Only need to show $W'_{\alpha\beta}$ is a vector-bundle-like submersion. Denote $N_{\alpha\beta}':=(\exp_{\alpha\beta}^{g_{\alpha\beta}})^{-1}(W'_{\alpha\beta})$, which is a disk bundle over $\phi_{\alpha\beta}(U'_{\alpha\beta})$ and thus can be equipped with a vector bundle structure, and $\Phi_{\alpha\beta}:=(\exp_{\alpha\beta}^{g_{\alpha\beta}}|_{N'_{\alpha\beta}})^{-1}: W'_{\alpha\beta}\to N'_{\alpha\beta}$ is a fiber bundle ep-groupoid diffeomorphism. $(N_{\alpha\beta}, \Phi_{\alpha\beta})$ makes $W'_{\alpha\beta}$ a vector-bundle-like submersion.

\item Definition of $\tilde E_{\alpha\beta}$ and the summary of the construction so far:

Define $\tilde E_{\alpha\beta}:=\tilde E^{(3)}_{\alpha\beta}|_{W'_{\alpha\beta}}$. We now have constructed the restricted coordinate change $C_\beta|_{U'_\beta}\to C_\alpha|_{U'_\alpha}$ with the domain of the coordinate change $U'_{\alpha\beta}$, a vector-bundle-like submersion $\pi_{\alpha\beta}: W'_{\alpha\beta}\to\phi_{\alpha\beta}(U'_{\alpha\beta})$ satisfying the entire item (2) in definition \ref{LEVEL1CCHANGE}, and $\tilde\pi_{\alpha\beta}: \tilde E_{\alpha\beta}\to W'_{\alpha\beta}$ such that the quotient section is transverse hence item (3) is also established.

\item Construct $\hat\pi_{\alpha\beta}$:

Choose a invariant bundle metric on $E_\alpha|_{W'_{\alpha\beta}}$. This induces a vector bundle ep-groupoid projection functor $\hat\pi_{\alpha\beta}: E_\alpha|_{W'_{\alpha\beta}}\to \tilde E_{\alpha\beta}$ which maps the invariant subbundle in $E_\alpha|_{W'_{\alpha\beta}}$ normal to $\tilde E_{\alpha\beta}$ with respect to this bundle metric to the zero section.

$$(C_\beta|_{U'_\beta}\to C_\alpha|_{U'_{\alpha}}, \pi_{\alpha\beta}:W'_{\alpha\beta}\to\phi_{\alpha\beta} (U'_{\alpha\beta}), \tilde\pi_{\alpha\beta}: \tilde E_{\alpha\beta}\to W'_{\alpha\beta}, \hat\pi_{\alpha\beta})$$ is the desired level-1 coordinate change.
\end{enumerate}
\end{proof}

\subsection{Compatibility of level-1 coordinate changes}

After defining a level-1 coordinate change and explaining how to upgrade a coordinate change to a level-1 coordinate change up to precompact shrinkings, we now define the compatibility condition among level-1 coordinate changes.

\begin{definition}[compatibility]\label{COMPATIBILITYONE} Let $\mathcal{G}:=(X, (S,\leq), \{C_\alpha|_{U_\alpha}\}_{\alpha\in S})$ be a strongly intersecting Hausdorff good coordinate system indexed by a total order $(S, \leq)$. Let $\gamma, \beta, \alpha\in S$ with $(\gamma,\beta), (\gamma,\alpha), (\beta, \alpha)\in I(\mathcal{G}')$. Let $\{U'_\gamma, U'_\beta, U'_\alpha\}$ be a precompact shrinking of $\{U_\gamma, U_\beta, U_\alpha\}$ as in definition \ref{PRECOMPACTCSHRINK}. Suppose we are given level-1 coordinate changes $C_\gamma|_{U_\gamma'}\overset{\text{level-1}}{\to} C_\alpha|_{U'_\alpha}$, $C_\gamma|_{U_\gamma'}\overset{\text{level-1}}{\to} C_\beta|_{U'_\beta}$ and $C_\beta|_{U_\beta'}\overset{\text{level-1}}{\to} C_\alpha|_{U'_\alpha}$. 

These three level-1 coordinate changes are said to be \emph{compatible}, if for any $(z,y,x)\in\gamma\times\beta\times\alpha$, we have that

\begin{enumerate}
\item over the common domain of the definitions, possibly empty, $$U'_{xyz}:=U'_{xz}\cap U'_{yz}\cap (\phi_{yz})^{-1}(U'_{xy})\overset{\text{maximality}}{\equiv} U'_{xz}\cap U'_{yz},$$ $\hat\phi_{xy}\circ\hat\phi_{yz}=h_{xyz}(\hat\phi_{xz}(\cdot))$ for some $h_{xyz}\in G_x$;
\item over the common domain of the definitions, $$W'_{xyz}:=h_{xyz}(W'_{xz})\cap W'_{xy}\cap \pi_{xy}^{-1}(\phi_{xy}(W'_{yz})),$$
\begin{align*}
\text{we have\quad\quad}((\phi_{xy})_\ast \pi_{yz})\circ\pi_{xy}&=\pi_{xz}(h_{xyz}^{-1}(\cdot)), \\
(\tilde \pi_{xy})^{-1}(\hat\phi_{xy}(\tilde E_{yz}))&=h_{xyz}(\tilde E_{xz}),\\
((\hat\phi_{xy})_\ast \tilde \pi_{yz})\circ \tilde\pi_{xy}|_{(\tilde \pi_{xy})^{-1}(\hat\phi_{xy}(\tilde E_{yz}))}&=\tilde \pi_{xz}(h_{xyz}^{-1}(\cdot)),\\
\pi_{xy}^{-1}(\phi_{xy}(\pi_{yz}^{-1}(\phi_{yz}(U'_{xz}\cap U'_{yz}))))&=h_{xyz}(\pi_{xz}^{-1}(\phi_{xz}(U'_{xz}\cap U'_{yz})))\text{ and}
\end{align*}
\item $(\hat\phi_{xy})_\ast\hat\pi_{yz}$ naturally extends to $$(\tilde\pi_{xy})^\ast((\hat\phi_{xy})_\ast\hat\pi_{yz}): E_x|_{W'_{xyz}}\to (\tilde \pi_{xy})^{-1}(\hat\phi_{xy}(\tilde E_{yz}))|_{W'_{xyz}}$$ via $\tilde{\pi}_{xy}$, which is defined by 
\begin{align*}(z,v)&\mapsto \tilde\pi_{xy}((z,v))\mapsto ((\hat\phi_{xy})_\ast\hat\pi_{yz})(\tilde\pi_{xy}((z,v)))\\&\mapsto (\tilde\pi_{xy}|_{(\tilde E_{xy})_z})^{-1}(((\hat\phi_{xy})_\ast\hat\pi_{yz})(\tilde\pi_{xy}((z,v)))),
\end{align*}
and we require over $W'_{xyz}$ that $$((\tilde\pi_{xy})^\ast((\hat\phi_{xy})_\ast\hat\pi_{yz}))\circ\hat\pi_{xy}=\hat\pi_{xz}(h^{-1}_{xyz}(\cdot)).$$
\end{enumerate}

Let $\{U'_\alpha\}_{\alpha\in S}$ be a precompact shrinking of $\{U_\alpha\}_{\alpha\in S}$. Denote $\mathcal{G}':=\{C_\alpha|_{U'_\alpha}\}_{\alpha\in S}$. A collection $\{C_\beta|_{U'_\beta}\overset{\text{level-1}}{\to} C_\alpha|_{U'_\alpha}\}_{(\beta,\alpha)\in I(\mathcal{G}')}$ of level-1 coordinate changes is said to be \emph{compatible} if any such three level-1 coordinate changes mapping among three charts as the above are compatible.
\end{definition}

\begin{remark} For the compatibility of level-1 coordinate changes among charts based on general ep-groupoids, we require the corresponding conditions to \ref{COMPATIBILITYONE} on the invariant neighborhoods in $U'_{\alpha\beta\gamma}:=U'_{\alpha\gamma}\cap U'_{\beta\gamma}\cap \phi_{\beta\gamma}^{-1}(U'_{\alpha\beta})$ with the natural representations of stabilizer groups similar to \ref{LEVEL1CCREM}.(IV). 
\end{remark}

The condition in item (2) above implies $W'_{xyz}=W'_{xy}\cap\pi_{xy}^{-1}(\phi_{xy}(W'_{yz}))$.

\begin{definition}[notation]\label{CLEVERNOTATION}
We have a compact way of writing the above, we can extend $\pi_{\alpha\beta}$ to a fiber bundle ep-groupoid projection functor $$\pi_{\alpha\beta}^o:\text{orbit}(W'_{\alpha\beta})\to \text{orbit}(\phi_{\alpha\beta}(U'_{\alpha\beta}))$$ by the morphisms in $U'_\alpha$, which is now invariant in $U'_\alpha$. See remark \ref{LEVEL1CCREM}.(I). Here $\text{orbit}(A)$ denotes the union of images of $A$ under the morphisms in $U'_\alpha$ or $E_\alpha|_{U'_\alpha}$, and $o$ stands for the orbit version of the map (mapping between the unions of the orbits under the morphisms, \emph{not} mapping between the orbit spaces which is always denoted by $\underline{\pi_{\alpha\beta}}$). Similarly, we can define $\tilde\pi_{\alpha\beta}^o, \hat\pi_{\alpha\beta}^o$, etc.
Then above compatibility identities in \ref{COMPATIBILITYONE}.(2) and (3) become
\begin{align*}((\phi_{\alpha\beta})_\ast \pi_{\beta\gamma})^o\circ\pi_{\alpha\beta}^o&=\pi_{\alpha\gamma}^o, \\
(\pi_{\alpha\beta}^o)^{-1}(\text{orbit}(\phi_{\alpha\beta}((\pi_{\beta\gamma}^o)^{-1}(\text{orbit}(\phi_{\beta\gamma}(U'_{\alpha\gamma}\cap U'_{\beta\gamma}))))))&=\\
(\pi_{\alpha\gamma}^o)^{-1}(\text{orbit}(\phi_{\alpha\gamma}(U'_{\alpha\gamma}\cap U'_{\beta\gamma})))&\\
((\hat\phi_{\alpha\beta})_\ast \tilde \pi_{\beta\gamma})^o\circ \tilde\pi_{\alpha\beta}^o|_{\text{orbit}((\tilde \pi_{\alpha\beta})^{-1}(\hat\phi_{\alpha\beta}(\tilde E_{\beta\gamma})))}&=\tilde \pi_{\alpha\gamma}^o,\\
((\tilde\pi_{\alpha\beta})^\ast((\hat\phi_{\alpha\beta})_\ast\hat\pi_{\beta\gamma}))^o\circ\hat\pi_{\alpha\beta}^o&=\hat\pi_{\alpha\gamma}^o.
\end{align*}
\end{definition}

\begin{remark} Observe that $\pi_{\alpha\beta}$ can be recovered from $\pi_{\alpha\beta}^o$ easily, etc.
\end{remark}

\subsection{A level-1 good coordinate system}

\begin{definition}[level-1 good coordinate system]\label{LEVELONEGCS} 

Let $$\mathcal{G}:=(X, (S,\leq), \{C_\alpha|_{U_\alpha}\}_{\alpha\in S}, \{C_\beta|_{U_\beta}\to C_\alpha|_{U_\alpha}\}_{(\beta.\alpha)\in I(\mathcal{G})})$$ be a strongly intersecting Hausdorff good coordinate system indexed by a total order $(S, \leq)$. Let $\{U'_\alpha\}_{\alpha\in S}$ be a precompact shrinking of $\{U_\alpha\}_{\alpha\in S}$ as in definition \ref{PRECOMPACTCSHRINK}. Denote $\mathcal{G}':=\{C_\alpha|_{U'_\alpha}\}_{\alpha\in S}$. Suppose for $(\beta,\alpha)\in I(\mathcal{G}')$, we are given a level-1 coordinate change $C_\beta|_{U_\beta'}\overset{\text{level-1}}{\to} C_\alpha|_{U'_\alpha}$ such that the underlying coordinate change $C_\beta|_{U_\beta'}\to C_\alpha|_{U'_\alpha}$ is restricted from $C_\beta|_{U_\beta}\to C_\alpha|_{U_\alpha}$, and the collection $\{C_\beta|_{U'_\beta}\overset{\text{level-1}}{\to} C_\alpha|_{U'_\alpha}\}_{(\beta,\alpha)\in I(\mathcal{G}')}$ is compatible as in \ref{COMPATIBILITYONE} (or equivalently \ref{CLEVERNOTATION}). Then $$\mathcal{G}':=(X, (S,\leq), \{C_\alpha|_{U'_\alpha}\}_{\alpha\in S}, \{C_\beta|_{U'_\beta}\overset{\text{level-1}}{\to}C_\alpha|_{U'_\alpha}\}_{(\beta.\alpha)\in I(\mathcal{G}')})$$ is called a \emph{level-1 good coordinate system} for $\mathcal{G}$.
\end{definition}

\begin{definition}[level-1 good coordinate system with a general index set]\label{LEVEL1GCSGENERALINDEX} Let $\hat{\mathcal{G}}$ be a strongly intersecting Hausdorff good coordinate system indexed by a general non-antisymmetric total order $(\hat S, \hat\leq)$. Let $(S, \leq)$ be a total order partition of $(\hat S, \hat \leq)$ as in \ref{TOPART}, obtained by using \ref{BUNDLEDIMEN} (or other methods). Let $\mathcal{G}$ indexed by $(S, \leq)$ be the associated grouped good coordinate system of $\hat{\mathcal{G}}$. If $\mathcal{G}'$ is a level-1 good coordinate system for $\mathcal{G}$ as in \ref{LEVELONEGCS}, denoting the associated level-1 good coordinate system indexed by $(\hat S,\hat\leq)$ of $\mathcal{G}'$ by $\hat{\mathcal{G}}'$, then $\hat{\mathcal{G}}'$ is said to be a \emph{level-1 good coordinate system} for $\hat{\mathcal{G}}$. See remark \ref{REMDEF} for the reason why the underlying good coordinate system of $\hat{\mathcal{G}}'$ without the level-1 structure is a shrinking of $\hat{\mathcal{G}}$.
\end{definition}

\begin{remark}
We can think of the level-1 structure of a level-1 good coordinate system as a \emph{globally compatible} way to make bundles of varying dimensions in Kuranishi charts ``uniform in dimension'' by having structures in place to precisely identify the section and its perturbations (respectively the topologies of base and total space) in $E_{\beta}|_{U'_{\alpha\beta}}$ with the section and its $(\pi_{\alpha\beta}, \tilde\pi_{\alpha\beta})$-controlled perturbations (respectively the $\pi_{\alpha\beta}$ bundle topology of the base and the $\tilde\pi_{\alpha\beta}$ bundle topology of the total space) in $E_\alpha|_{W'_{\alpha\beta}}$ for all $(\beta,\alpha)\in I(\mathcal{G}')$. We can also think of the level-1 structure as inverting coherent system of embeddings among the charts into coherent system of submersions with special properties, so that we can compatibily \emph{lift} constructions from charts of lower dimensions into charts of higher dimensions.
\end{remark}

\subsection{A level-1 Kuranishi embedding} We have defined the level-1 structure for a single object, a good coordinate system. Now we will define the level-1 structure for certain distinguished maps between two such objects, namely \emph{concerted Kuranishi embeddings} between good coordinate systems, resulting in concerted level-1 Kuranishi embeddings between level-1 good coordinate systems. There are clear analogies as well as subtle differences between level-1 chart embeddings and level-1 coordinate changes, and between a concerted level-1 Kuranishi embedding from one level-1 good coordinate system to another and a level-1 good coordinate system. The level-1 structures for concerted Kuranishi embeddings will be important in constructing fiber products, through which we can compare different choices made during working with Kuranishi structures and good coordinate systems and then show the choice-independence of the theory, and in showing two good coordinate systems being equivalent is an equivalence relation.

\begin{definition}[level-1 chart embedding]\label{LEVEL1CEMBEDDING} Let $$(\mathcal{G}\Rightarrow \tilde{\mathcal{G}}):=\{C_\alpha|_{U_\alpha}\to \tilde C_\alpha|_{\tilde U_\alpha}\}_{\alpha\in S}$$ be a Kuranishi embedding between two strongly intersecting Hausdorff good coordinate systems indexed by the same total order $(S,\leq)$.  Recall that a chart embedding $(C_\alpha|_{U_\alpha}\to \tilde C_\alpha|_{\tilde U_\alpha})=(\phi_\alpha,\hat\phi_\alpha, U_\alpha)$ is a coordinate change with the domain $U_\alpha$. A \emph{level-1 chart embedding} $C_\alpha|_{U'_\alpha}\overset{\text{level-1}}{\to}\tilde C_\alpha|_{\tilde U'_\alpha}$ for $\alpha\in S$ is a tuple $(C_\alpha|_{U'_\alpha}\to \tilde C_\alpha|_{\tilde U'_\alpha}, W'_{\alpha}, \Pi_{\alpha}, \tilde\Pi_{\alpha}, \hat\Pi_{\alpha})$ consisting of the following:
\begin{enumerate}
\item a restricted coordinate change $(C_\alpha|_{U'_\alpha}\to \tilde C_\alpha|_{\tilde U'_\alpha})=(\phi_{\alpha}, \hat\phi_{\alpha}, U'_{\alpha}, \tilde U'_\alpha)$, where 
\begin{enumerate}
\item $U'_\alpha$ is a precompact shrinking of $U_\alpha$ in $\mathcal{G}$ as in \ref{PRECOMPACTCSHRINK},
\item $\tilde U'_\alpha$ is a precompact shrinking of $\tilde U_\alpha$ in $\tilde{\mathcal{G}}$ as in \ref{PRECOMPACTCSHRINK}, and
\item $\phi_{\alpha}(U'_\alpha)\subset \tilde U'_\alpha$\footnote{So, the domain of the coordinate change $C_\alpha|_{U'_\alpha}\to \tilde C_\alpha|_{\tilde U'_\alpha}$ is $U'_\alpha$, same as determined by restricting $C_\alpha|_{U_\alpha}\to \tilde C_\alpha|_{\tilde U_\alpha}$.};
\end{enumerate}
\item a fiber bundle ep-groupoid projection functor $\Pi_{\alpha}: W'_{\alpha}\to \phi_{\alpha}(U'_{\alpha})$ such that
\begin{enumerate}
\item $W'_{\alpha}=\tilde U'_\alpha$\footnote{By the definitions of a shrinking with $J=\{\alpha\}$ in \ref{PRECOMPACTCSHRINK} and a chart embedding prior to grouping, $\phi_\alpha(U'_\alpha)$ is invariant under the morphisms of $\tilde U_\alpha$, so it makes sense that $\tilde U'_\alpha$ itself can be chosen to be a tubular neighborhood of $\phi_\alpha(U'_\alpha)$ in $\tilde U'_\alpha$, which is invariant under the morphisms of $\tilde U'_\alpha$ unlike in level-1 coordinate changes in a level-1 good coordinate system. We still use the notation $W'_{\alpha}$ to keep the analogy to a level-1 coordinate change.},
\item $\Pi_\alpha$ is a vector-bundle-like submersion,
\item $(\tilde U'_\alpha, \tilde U_\alpha, W'_{\alpha})$ is a strong open neighborhood of $\phi_{\alpha}(U'_{\alpha})$ as in definition \ref{STRONGNBD}, and
\item $W'_{\alpha}\backslash \phi_{\alpha}(U'_{\alpha})$ contains no zeros of $\tilde s_\alpha|_{W'_{\alpha}}$;
\end{enumerate}
\item $\tilde\Pi_{\alpha}:\tilde F_{\alpha}\to \hat\phi_{\alpha}(E_\alpha|_{U'_{\alpha}})$, a fiberwise-isomorphic bundle map covering $\pi_{\alpha}: W'_{\alpha}\to \phi_{\alpha}(U'_{\alpha})$, where $\tilde F_{\alpha}\to W'_{\alpha}=\tilde U'_\alpha$ is a subbundle of the object space of $\tilde E_\alpha|_{W'_{\alpha}}$, $\tilde F_\alpha$ is invariant under the morphisms of $\tilde E_\alpha|_{\tilde U'_\alpha}$ and becomes an ep-groupoid with those morphisms, and $\tilde F_\alpha \to W'_\alpha$ is a vector bundle ep-groupoid projection functor, such that
\begin{enumerate}
\item $\tilde\Pi_\alpha$ is a functor between vector bundle ep-groupoids\footnote{Also a fiber bundle ep-groupoid projection functor.},
\item $\tilde F_{\alpha}|_{\phi_{\alpha}(U'_{\alpha})}=\hat\phi_{\alpha}(E_\alpha|_{U'_{\alpha}})$, 
\item the zero set of $\tilde s_\alpha|_{W'_{\alpha}}/\tilde F_{\alpha}: W'_{\alpha}\to \tilde E_\alpha|_{W'_{\alpha}}/\tilde F_{\alpha}$ is exactly $\phi_{\alpha}(U'_{\alpha})$, and 
\item the well-defined linearization of $\tilde s_\alpha|_{W'_{\alpha}}/\tilde F_{\alpha}$ over $\phi_{\alpha}(U'_{\alpha})$ is transverse; and
\end{enumerate}
\item a fiberwise-projection functor\footnote{Where $\tilde E_\alpha|_{W'_{\alpha}}$ is obviously invariant under the morphisms of $\tilde E_\alpha|_{\tilde U_\alpha}$ by (2)(a).} $\hat\Pi_{\alpha}:\tilde E_\alpha|_{W'_{\alpha}}\to \tilde F_{\alpha}$ between vector bundle ep-groupoids.
\end{enumerate}
\end{definition}

\begin{definition}[level-1 Kuranishi embedding for a concerted Kuranishi embedding]\label{LEVEL1KEMBED} Let $(\mathcal{G}\Rightarrow\tilde{\mathcal{G}})=\{C_\alpha|_{U_\alpha}\to\tilde C_\alpha|_{\tilde U_\alpha}\}_{\alpha\in S}$ be a Kuranishi embedding between two strongly intersecting Hausdorff good coordinate systems indexed by the same total order $(S,\leq)$ (hence it is trivially concerted). A \emph{general concerted Kuranishi embedding} from $\hat{\mathcal{G}}$ with $(\hat S,\hat\leq)$ to $\hat{\tilde{\mathcal{G}}}$ with $(\hat S, \tilde\leq)$ can be reduced to the above form of a Kuranishi embedding with the same total orders, via the total order partition (or other methods): Define $$\alpha_{i,j}:=\{x\in \hat S\;|\;\text{dim}\hat E_x=i,\;\text{dim} \hat{\tilde{E}}_x=j\},\; i\leq N_1,\; j\leq N_2.$$ Define $\alpha_{i,j}\leq \alpha_{i',j'}$ if $i+j\leq i'+j'$. By the concertedness, for any $x, y\in \hat S$ with $\underline{U_x}$ and $\underline{U_y}$ intersecting, we have $x\hat\leq y$ and $x\tilde\leq y$, or $y\hat \leq x$ and $y\tilde\leq x$, or both. Here, when the order is only used, $\alpha_{ij}\leq\alpha_{i'j'}$ is the same as $x\hat\leq y, x\tilde\leq y$ for some (hence any) $(x,y)\in\alpha_{i,j}\times\alpha_{i',j'}$ (by the definition, the orders are compatible with the bundle dimension). The common total order index set is $S:=\{\alpha_{ij}\}_{i,j}$ with the order $\leq$.

A \emph{level-1 Kuranishi embedding} $$(\mathcal{G}'\overset{\text{level-1}}{\Rightarrow}\tilde{\mathcal{G}}')=\{C_\alpha|_{U'_\alpha}\overset{\text{level-1}}{\to}\tilde C_\alpha|_{\tilde U'_\alpha}\}_{\alpha\in S}$$ for $\mathcal{G}\Rightarrow\tilde{\mathcal{G}}$ consists of strongly intersecting Hausdorff level-1 good coordinate systems $\mathcal{G}'$ for $\mathcal{G}$ and $\tilde{\mathcal{G}}'$ for $\tilde{\mathcal{G}}$, and level-1 chart embeddings $$(C_\alpha|_{U'_\alpha}\overset{\text{level-1}}{\to}\tilde C_\alpha|_{\tilde U'_\alpha})=(C_\alpha|_{U'_\alpha}\to \tilde C_\alpha|_{\tilde U'_\alpha}, W'_\alpha, \Pi_\alpha, \tilde\Pi_{\alpha}, \hat\Pi_\alpha), \alpha\in S,$$ such that for all $\alpha\in S$, $C_\alpha|_{U'_\alpha}\to \tilde C_\alpha|_{\tilde U'_\alpha}$ is restricted from chart embeddings $C_\alpha|_{U_\alpha}\to \tilde C_\alpha|_{\tilde U_\alpha}$, and the following level-1 square is commutative up to group actions, for all $(\beta,\alpha)\in I(\mathcal{G}')=I(\mathcal{G})$:
$$\begin{CD}
\tilde C_\beta|_{\tilde U'_\beta}@>\text{level-1}>>\tilde C_\alpha|_{\tilde U'_\alpha}\\
@A\text{level-1}AA @AA\text{level-1}A\\
C_\beta|_{U'_\beta}@>\text{level-1}>>C_\alpha|_{U'_\alpha}
\end{CD}\;\;\;\;.$$
Here, the \emph{level-1 commutativity up to group actions}, or \emph{level-1 commutativity} for short, means that in addition to coordinate changes being commutative up to group actions, we have that
\begin{align*}((\phi_{\alpha})_\ast \pi_{\alpha\beta})^o\circ\Pi_{\alpha}&=((\tilde\phi_{\alpha\beta})_\ast\Pi_{\beta})^o\circ \widetilde{\pi_{\alpha\beta}}^o,\\
\text{orbit}((\tilde \pi_{\alpha\beta})^{-1}(\widetilde{\hat\phi_{\alpha\beta}}(\tilde F_{\beta})))&=\text{orbit}((\tilde \Pi_{\alpha})^{-1}(\hat\phi_{\alpha}(\tilde E_{\alpha\beta}))),\end{align*}
$$((\hat\phi_{\alpha})_\ast \tilde \pi_{\alpha\beta})^o\circ \tilde\Pi_{\alpha}|_{\text{orbit}((\tilde \pi_{\alpha\beta})^{-1}(\widetilde{\hat\phi_{\alpha\beta}}(\tilde F_{\beta})))}=$$ $$((\widetilde{\hat\phi_{\alpha\beta}})_\ast \tilde \Pi_{\beta})^o\circ \widetilde{\tilde\pi_{\alpha\beta}}^o|_{\text{orbit}((\tilde \Pi_{\alpha})^{-1}(\hat\phi_{\alpha}(\tilde E_{\alpha\beta})))},\text{ and}$$
$$((\tilde\Pi_{\alpha})^\ast((\hat\phi_{\alpha})_\ast\hat\pi_{\alpha\beta}))^o\circ\hat\Pi_{\alpha}=
((\widetilde{\tilde\pi_{\alpha\beta}})^\ast(\widetilde{(\hat\phi_{\alpha\beta}})_\ast\hat\Pi_{\beta}))^o\circ\widetilde{\hat\pi_{\alpha\beta}}^o,$$
using the notation convention explained in \ref{CLEVERNOTATION}.
\end{definition}

\begin{remark} There are three important remarks:
\begin{enumerate}[(I)]
\item (\textbf{convention}) We used $\tilde F_\alpha$ in definition \ref{LEVEL1CEMBEDDING}.(3), because $\tilde E_\alpha$ is already used. We used the capital letter for $\Pi_\alpha$ rather than $\pi_\alpha$, not to confuse it with the projection $E_\alpha\to U'_\alpha$. $\tilde \pi_{\alpha\beta}$ means the bundle ep-groupoid projection functor $\tilde E_{\alpha\beta}\to\hat\phi_{\alpha\beta}(E_\beta|_{U'_{\alpha\beta}})$ in $\mathcal{G}'$, whilst $\widetilde{\pi_{\alpha\beta}}$ means the vector-bundle-like submersion $\tilde W'_{\alpha\beta}\to \tilde\phi_{\alpha\beta}(\tilde U'_{\alpha\beta})$ in $\tilde{\mathcal{G}}'$. Observe that $\Pi_\alpha^o=\Pi_\alpha$, $\tilde\Pi_\alpha^o=\tilde\Pi_\alpha$ and $\hat\Pi_\alpha^o=\hat\Pi_\alpha$, which we already used in the commutativity identities in \ref{LEVEL1KEMBED}. Notice the similarities between level-1 good coordinate systems (resp. the compatibility of level-1 coordinate changes) and concerted level-1 Kuranishi embeddings (resp. the commutativity of level-1 squares).
\item (\textbf{general concerted level-1 Kuranishi embedding}, indexed by a common index set with possibly distinct orders) A \emph{concerted level-1 Kuranishi embedding indexed by a common general index set but with two possibly distinct orders} is defined via \ref{LEVEL1KEMBED}, in the same fashion as in \ref{LEVEL1GCSGENERALINDEX}. We will see later that we can construct a concerted level-1 Kuranishi embedding for $\mathcal{G}\Rightarrow\tilde{\mathcal{G}}$ after specifying a level-1 $\mathcal{G}'$ for $\mathcal{G}=\{C_x|_{U_x}\}_{x\in \hat S}$ with a general index set $(\hat S,\hat \leq)$. What happens to the index sets is the following: We first choose a total order partition $(S_1,\leq_1)$ for $(\hat S, \hat \leq)$, and group $\mathcal{G}$ into $\mathcal{G}_1$, and choose a level-1 good coordinate system $\mathcal{G}''_1$ for $\mathcal{G}_1$, and then recover the associated level-1 $\mathcal{G}''$ with the order $(\hat S, \hat\leq)$ for $\mathcal{G}''_1$. Secondly, we consider $\mathcal{G}''$ without the level-1 structure which maps to $\tilde{\mathcal{G}}$ in a concerted $\mathcal{G}''\Rightarrow\tilde{\mathcal{G}}$, then we choose the common total order partition by the double bundle dimension as in \ref{LEVEL1KEMBED} and group it into $\mathcal{G}''_2\Rightarrow\tilde{\mathcal{G}}_2$, then using the method in \ref{EMBEDDING0TO1EXTENDINGLEVEL1GCS} we can construct a level-1 Kuranishi embedding $\mathcal{G}'_2\Rightarrow\tilde{\mathcal{G}}'_2$ for it, and finally we recover the original index set $\hat S$ to get a concerted level-1 $\mathcal{G}'\Rightarrow\tilde{\mathcal{G}}'$ for the original $\mathcal{G}\Rightarrow\tilde{\mathcal{G}}$.
\item (\textbf{general level-1 Kuranishi embedding}, possibly not concerted) We can define a \emph{level-1 Kuranishi embedding for a general Kuranishi embedding} which is not necessarily concerted. In that case, we can only reduce to a common index set with two different orders (possibly not total orders). Moreover, when the respective orders between $\beta$ and $\alpha$ do not agree in $\mathcal{G}$ and $\tilde{\mathcal{G}}$, and $\underline{U_\beta}$ and $\underline{U_\alpha}$ intersect in the identification space, we need to require the level-1 compatibility squares with the top arrow direction reversed. The level-1 square can also be stated explicitly using identities similar to ones in \ref{LEVEL1KEMBED},$$((\tilde\phi_{\beta\alpha}\circ\phi_\alpha)_\ast \pi_{\alpha\beta})^o\circ((\tilde\phi_{\beta\alpha})_\ast(\Pi_\alpha))^o\circ\widetilde{\pi_{\beta\alpha}}^o=\Pi_\beta,\text{ etc.}$$ This notion appears just once for the completeness and is restricted from a concerted one (non-concerted $\mathcal{B}(\mathcal{G})\overset{\text{level-1}}{\Rightarrow}\mathcal{G}^i$ in the lower half of the fiber product square is never used); and we can always find a concerted chart-refinement of a Kuranishi embedding defined in \ref{CROFKEMB}.
\end{enumerate}
\end{remark}

\subsection{The simplest non-trivial inductive step for constructing a level-1 good coordinate system}

Having introduced the concept of level-1 good coordinate systems and level-1 Kuranishi embeddings, we now prove a key proposition in establishing compatible level-1 coordinate changes among three charts. We assume that $\alpha$, $\beta$ and $\gamma$ are distinct for this inductive step to be genuinely non-trivial\footnote{But the special cases of this result, where some of $\alpha,\beta, \gamma$ coincide, are trivial to show and the corresponding statements are subsumed in the statement of the proposition, except with the precompactness of $U''_\beta$ in $U'_\beta$ removed if $\beta=\gamma$.}.

\begin{proposition}[the $\{\alpha,\beta,\gamma\}$-step]\label{ALPHABETAGAMMA} Let $\mathcal{G}$ be a strongly intersecting Hausdorff good coordinate system indexed by a total order $(S,\leq)$, and let $\gamma, \beta, \alpha\in S$ be distinct with $(\gamma, \beta), (\gamma, \alpha), (\beta,\alpha)\in I(\mathcal{G})$. Suppose that we have chosen a level-1 coordinate change $C_\gamma|_{U'_\gamma}\overset{\text{level-1}}{\to} C_\beta|_{U'_\beta}$, where $\{U'_\gamma, U'_\beta\}$ is a precompact shrinking of $\{U_\gamma, U_\beta\}$ in the sense of \ref{PRECOMPACTCSHRINK}.

Then we can choose 
\begin{enumerate}[(i)]
\item $U''_\gamma=U'_\gamma$, and
\item a precompact shrinking $\{U''_\beta, U''_\alpha\}$ of $\{U'_\beta, U_\alpha\}$ such that $$(U''_\beta, U'_\beta, W''_{\beta\gamma})$$ is a strong open neighborhood of $\phi_{\beta\gamma}(U''_{\beta\gamma})$, where $U''_{\beta\gamma}$ is the domain of $C_\gamma|_{U''_\gamma}\to C_\beta|_{U''_\beta}$ and $W''_{\beta\gamma}$ is an invariant fiber shrinking of $$(\pi_{\beta\gamma})^{-1}(\phi_{\beta\gamma}(U''_{\beta\gamma})),$$
\end{enumerate}
such that we can construct level-1 coordinate changes $$C_\beta|_{U''_\beta}\overset{\text{level-1}}{\to} C_\alpha|_{U''_\alpha}\text{ and }C_\gamma|_{U''_\gamma}\overset{\text{level-1}}{\to} C_\alpha|_{U''_\alpha},$$ and those two coordinate changes and the restricted level-1 $C_\gamma|_{U''_\gamma}\overset{\text{level-1}}{\to} C_\beta|_{U''_\beta}$  with the above $W''_{\beta\gamma}$ from $C_\gamma|_{U'_\gamma}\overset{\text{level-1}}{\to} C_\beta|_{U'_\beta}$ are compatible as in \ref{COMPATIBILITYONE}, \ref{CLEVERNOTATION} (in particular, $\text{orbit}(W''_{\alpha\gamma}|_{\phi_{\alpha\gamma}(U''_{\alpha\gamma}\cap U''_{\beta\gamma})})=\text{orbit}((\pi_{\alpha\beta})^{-1}(\phi_{\alpha\beta}(W''_{\beta\gamma})))$.
\end{proposition}
\begin{remark} Item (i) says that we do not need to shrink the domain charts of level-1 coordinate changes already constructed during inductively constructng new level-1 structures for coordinate changes not equipped with level-1 structures yet, where the induction is in the direction of going up the order. Moreover, as we will see in item (1) of the proof, we do not need to shrink the target charts of the existing level-1 coordinate changes if not for creating room for establishing the compatibility, namely we can construct a level-1 coordinate change $C_\beta|_{U'_\beta}\to C_\alpha|_{U'_\alpha}$ with the same previous $U'_\beta$.
\end{remark}

\begin{proof}
\begin{enumerate}[(I)]
\item Construct a level-1 $C_\beta|_{U'_\beta}\to C_\alpha|_{U'_\alpha}$:

Choose an invariant metric $g_{\alpha\beta}$ on $U_\alpha$, and apply the proposition \ref{LEVEL1CCHANGE} to $C_\beta|_{U_\beta}\to C_\alpha|_{U_\alpha}$. Notice in that proposition, we can fix $U'_{\beta}$ beforehand, and as long as $U'_{\beta}$ is a precompact shrinking of $U_\beta$, we can find an embedded tubular neighborhood $W'_{\alpha\beta}$ of $\phi_{\alpha\beta}(U'_{\alpha\beta})$ in $U_\alpha$, where $U'_{\alpha\beta}$ is the domain of $C_\beta|_{U'_\beta}\to C_\alpha|_{U'_\alpha}$ and $\pi_{\alpha\beta}$ is a fiber bundle ep-groupoid projection functor from $W'_{\alpha\beta}$.

\item Construct $\pi_{\alpha\gamma}$ compatible with the existing $\pi_{\beta\gamma}$ and $\pi_{\alpha\beta}$ up to a precompact shrinking of $\{U'_\beta, U'_\alpha\}$, accomplished in the following steps:

\begin{enumerate}[(1)]
\item Place $\pi_{\beta\gamma}$ and $\pi_{\alpha\beta}$ into the same setting, namely, in $U'_\alpha$:

We extend the submersion $\pi_{\alpha\beta}$ to $\pi_{\alpha\beta}^o$ by the morphisms acting on $U'_\alpha$, and we pushforward $\pi_{\beta\gamma}$ to $(\phi_{\alpha\beta})_\ast\pi_{\beta\gamma}$ and extend it to $((\phi_{\alpha\beta})_\ast\pi_{\beta\gamma})^o$ by the moprhisms in $U'_\alpha$. We observe that $$((\phi_{\alpha\beta})_\ast\pi_{\beta\gamma})^o: \text{orbit}(\phi_{\alpha\beta}(W'_{\beta\gamma}))\to \text{orbit}(\phi_{\alpha\beta}(\phi_{\beta\gamma}(U'_{\beta\gamma}))),$$ where $\phi_{\alpha\beta}(W'_{\beta\gamma}):=\phi_{\alpha\beta}(U'_{\alpha\beta}\cap W'_{\beta\gamma})$, is an invariant open subset of a fiber bundle ep-groupoid (whose projection is a vector-bundle-like submersion), but not a fiber bundle ep-groupoid in general. 

\item Make $((\phi_{\alpha\beta})_\ast\pi_{\beta\gamma})^o$ into a fiber bundle ep-groupoid by a precompact shrinking $U^{(4)}_\alpha$ of $U'_\alpha$:

Choose a precompact shrinking $U^{(3)}_\alpha$ of $U'_\alpha$ in the good coordinate system $\{C_\delta|_{U'_\delta}\}_{\delta\in \{\gamma, \beta,\alpha\}}\bigsqcup \{C_\epsilon|_{U_\epsilon}\}_{\epsilon\in S\backslash \{\gamma, \beta, \alpha\}}$ as in \ref{PRECOMPACTCSHRINK}. Denote the domain of the coordinate change $C_\beta|_{U_\beta}\to C_\alpha|_{U^{(3)}_\alpha}$ by $U^{(3)}_{\alpha\beta}$.

Then an invariant fiber-shrinking $W^{(3)}_{\beta\gamma}$ of $W'_{\beta\gamma}$ as in \ref{FSHRINKING} will have the property that $$\phi_{\alpha\beta}((\pi_{\beta\gamma}|_{W^{(3)}_{\beta\gamma}})^{-1}(\pi_{\beta\gamma}(U^{(3)}_{\alpha\beta}\cap W^{(3)}_{\beta\gamma})))\subset U'_\alpha.$$
Similar to figure \ref{figure2_2}, we define $U_\alpha^{(4)}:=$
$$(U^{(3)}_\alpha\backslash \text{orbit}(\phi_{\alpha\beta}((\pi_{\beta\gamma}|_{W^{(3)}_{\beta\gamma}})^{-1}(\pi_{\beta\gamma}((U'_{\alpha\beta}\backslash U^{(3)}_{\alpha\beta})\cap W^{(3)}_{\beta\gamma})))))\bigcup$$ $$\text{orbit}(\phi_{\alpha\beta}((\pi_{\beta\gamma}|_{W^{(3)}_{\beta\gamma}})^{-1}(\pi_{\beta\gamma}(U^{(3)}_{\alpha\beta}\cap W^{(3)}_{\beta\gamma}))))\;\subset U'_\alpha.$$ 
This is a trick of invariantly shrinking a chart guaranteeing that it will be a strong neighborhood around the image of the domain of the coordinate change into it.
Denote the domain of the coordinate change $C_\beta|_{U'_\beta}\to C_\alpha|_{U^{(4)}_\alpha}$ as $U^{(4)}_{\alpha\beta}$. Then, $$((\phi_{\alpha\beta})_\ast\pi_{\beta\gamma})^o:\text{orbit}(\phi_{\alpha\beta}(U^{(4)}_{\alpha\beta}\cap W^{(3)}_{\beta\gamma}))\to \text{orbit}(\phi_{\alpha\beta}(U^{(4)}_{\alpha\beta}\cap \phi_{\beta\gamma}(U'_{\beta\gamma})))$$ is a fiber bundle ep-groupoid projection and a vector-bundle-like submersion.

\item Construct a Riemannian metric simultaneously recovering vector-bundle-like submersions obtained as well as the composition:

By the construction, $$\pi_{\alpha\beta}^o:\text{orbit}((\pi_{\alpha\beta}|_{W'_{\alpha\beta}})^{-1}(\phi_{\alpha\beta}(U^{(4)}_{\alpha\beta})))\to \text{orbit}(\phi_{\alpha\beta}(U^{(4)}_{\alpha\beta}))\;\;\text{and}$$
$$((\phi_{\alpha\beta})_\ast\pi_{\beta\gamma})^o:\text{orbit}(\phi_{\alpha\beta}(U^{(4)}_{\alpha\beta})\cap W^{(3)}_{\beta\gamma}))\to \text{orbit}(\phi_{\alpha\beta}(U^{(4)}_{\alpha\beta}\cap \phi_{\beta\gamma}(U'_{\beta\gamma})))$$ are both vector-bundle-like submersions as defined in definition \ref{VBLIKE}. Denote the first fiber bundle ep-groupoid projection restricted to the invariant open subset $\text{orbit}(\phi_{\alpha\beta}(U^{(4)}_{\alpha\beta})\cap W^{(3)}_{\beta\gamma})$ of its base by $\pi: A\to B$, and denote the second fiber bundle ep-groupoid projection by $\pi':=B\to C$. Now we are in the position to apply proposition \ref{METRICCOMPATIBILITY}. In particular, we have an invariant Riemannian metric $g_A$ on $$A:=\text{orbit}((\pi_{\alpha\beta}|_{W'_{\alpha\beta}})^{-1}(\phi_{\alpha\beta}(U^{(4)}_{\alpha\beta})))|_{\text{orbit}(\phi_{\alpha\beta}(U^{(4)}_{\alpha\beta})\cap W^{(3)}_{\beta\gamma}))}$$ which recovers the vector-bundle-like submersions $\pi: A\to B$ and $\pi':B\to C$, and directly recovers $\pi'\circ\pi$ in the sense of \ref{SUBMERSIONMODEL}.

\item Precompactly shrink $U'_\beta$ to $U''_\beta$ and create a transition region to extend the metric $g_A$ constructed in item (3):

Now denote $U''_\gamma:=U'_\gamma$. Use the above trick in item (2) to choose a precompact shrinking $U''_\beta$ of $U'_\beta$ in the good coordinate system $$\{C_\gamma|_{U''_\gamma}, C_\beta|_{U'_\beta}, C_\alpha|_{U^{(4)}_\alpha}\}\bigsqcup\{C_\delta|_{U_\delta}\}_{\delta\in S\backslash \{\gamma, \beta, \alpha\}},$$
such that $(U''_\beta, U'_\beta, W''_{\beta\gamma})$ is an invariant strong open neighborhood of $\phi_{\beta\gamma}(U''_{\beta\gamma})$. Here $U''_{\beta\gamma}$ is the domain of $C_\gamma|_{U''_\gamma}\overset{\text{level-1}}{\to} C_\beta|_{U''_\beta}$ and $W''_{\beta\gamma}$ is an invariant fiber-shrinking of  $W^{(3)}_{\beta\gamma}|_{\phi_{\beta\gamma}(U''_{\beta\gamma})}$.

We also have the induced $C_\beta|_{U''_\beta}\to C_\alpha|_{U^{(4)}_\alpha}$ with the domain of the coordinate change denoted by $U''_{\alpha\beta}$, and the associated $W''_{\alpha\beta}\subset U^{(4)}_\alpha$ is an invariant fiber-shrinking of $W'_{\alpha\beta}|_{\phi_{\alpha\beta}(U''_{\alpha\beta})}$. 

Denote $W^{(4)}_{\alpha\beta}:=W'_{\alpha\beta}|_{\phi_{\alpha\beta}(U^{(4)}_{\alpha\beta})}$, where $U^{(4)}_{\alpha\beta}$ is defined in item (2).

Now via the transition region $\text{orbit}(W^{(4)}_{\alpha\beta}\backslash W''_{\alpha\beta})$, $g_A|_{\text{orbit}(W''_{\alpha\beta})}$ extends to an invariant Riemannian metric $g_{\alpha}$ on $U_\alpha^{(4)}$.

\item Construction of $\pi_{\alpha\gamma}$:

Use $g_{\alpha}$ to produce an immersed tubular neighborhood $W'_{\alpha\gamma}$ of the image $\phi_{\alpha\gamma}(U'_{\alpha\gamma})$ of the domain $U'_{\alpha\gamma}$ of $C_{\gamma}|_{U_\gamma}\to C_\alpha|_{U_{\alpha}}$, such that the submersion $\pi_{\alpha\gamma}:W'_{\alpha\gamma}\to\phi_{\alpha\gamma}(U'_{\alpha\gamma})$ is a functor between ep-groupoids. A sufficiently small invariant fiber-shrinking $W^m_{\alpha\gamma}$ of $W'_{\alpha\gamma}|_{\phi_{\alpha\gamma}(U^m_{\alpha\gamma})}$ restricted to the image of the domain $U^m_{\alpha\gamma}$ of the coordinate change $C_{\gamma}|_{U''_\gamma}\to C_\alpha|_{U^{(4)}_{\alpha}}$ will be embedded in $U^{(4)}_\alpha$. So here we never need to further shrink $U''_\gamma=U'_\gamma$ chosen before. Using the idea of the trick in item (II)(2) again (also see figure \ref{figure2_2}), we shrink and modify $U^{(4)}_\alpha$ away from $W^{(4)}_{\alpha\beta}$ into $U^{(5)}_\alpha\subset U^{(4)}_\alpha$ to achieve the strong open neighborhood property of $W^{(5)}_{\alpha\gamma}$ for $\phi_{\alpha\gamma}(U^{(5)}_{\alpha\gamma})$, where $U^{(5)}_{\alpha\gamma}$ is the domain of $C_\gamma|_{U''_\gamma}\to C_\alpha|_{U^{(5)}_\alpha}$ and $W^{(5)}_{\alpha\gamma}\subset U^{(5)}_\alpha$ is an invariant fiber-shrinking of $W^m_{\alpha\gamma}|_{\phi_{\alpha\gamma}(U^{(5)}_{\alpha\beta})}$ such that $\text{orbit}(W^{(5)}_{\alpha\gamma}|_{\phi_{\alpha\gamma}(U^{(5)}_{\alpha\gamma}\cap U''_{\beta\gamma})})=(\pi_{\alpha\beta}^o)^{-1}(\text{orbit}(\phi_{\alpha\beta}(\text{orbit}(W''_{\beta\gamma}))))$ $\equiv\text{orbit}((\pi_{\alpha\beta})^{-1}(\phi_{\alpha\beta}(W''_{\beta\gamma})))$ (by fiber-shrinking strictly inside the transition region if necessary).
Here, we do not need to change the $\text{orbit}((\pi_{\alpha\beta})^{-1}(\phi_{\alpha\beta}(W''_{\beta\gamma})))$ part of $U^{(4)}_\alpha$ during shrinking $U^{(4)}_\alpha$ into $U^{(5)}_\alpha$, so the previously constructed data for the other two coordinate changes are intact: The domain of $C_\beta|_{U'_\beta}\to C_\alpha|_{U^{(5)}_\alpha}$ is still $U^{(4)}_{\alpha\beta}$ with $W^{(4)}_{\alpha\beta}\subset U^{(5)}_\alpha$ (which implies that the domain of $C_\beta|_{U''_\beta}\to C_\alpha|_{U^{(5)}_\alpha}$ is still $U''_{\alpha\beta}$ with $W''_{\alpha\beta}\subset U^{(5)}_\alpha$), etc. We just constructed the vector-bundle-like submersion $\pi_{\alpha\gamma}: W^{(5)}_{\alpha\gamma}\to \phi_{\alpha\gamma}(U^{(5)}_{\alpha\gamma})$, which is compatible with $\pi_{\alpha\beta}|_{W''_{\alpha\beta}}$ and $\pi_{\beta\gamma}|_{W''_{\beta\gamma}}$ by the construction.
\end{enumerate}

\item Construct $\tilde\pi_{\alpha\gamma}$ compatible with the existing $$\tilde\pi_{\alpha\beta}|_{(\tilde E_{\alpha\beta}|_{W''_{\alpha\beta}})}\text{ and }\tilde\pi_{\beta\gamma}|_{(\tilde E_{\beta\gamma}|_{W''_{\beta\gamma}})}.$$

\begin{enumerate}[(i)]
\item Construct a connection $\nabla_{\beta\gamma}$ on $E_\beta|_{W^{(3)}_{\beta\gamma}}$ recovering $\pi_{\beta\gamma}$:

Observe that $\tilde \pi_{\beta\gamma}: \tilde E_{\beta\gamma}|_{W^{(3)}_{\beta\gamma}}\to \phi_{\beta\gamma}(E_\gamma|_{U'_{\beta\gamma}})$ provides an invariant trivialization of $\tilde E_{\beta\gamma}\to W^{(3)}_{\beta\gamma}$ along each fiber of $\pi_{\beta\gamma}$. Once we pick an invariant bundle connection on $\phi_{\beta\gamma}(E_\gamma|_{U'_{\beta\gamma}})$ and use the trivial connection along each fiber of $\pi_{\beta\gamma}$, we arrive at a connection $\nabla_1$ on $\tilde E_{\beta\gamma}|_{W^{(3)}_{\beta\gamma}}$. Consider $\ker(\hat\pi_{\beta\gamma}: E_\beta|_{W^{(3)}_{\beta\gamma}}\to \tilde E_{\beta\gamma}|_{W^{(3)}_{\beta\gamma}})$ which is a vector bundle ep-groupoid over $W^{(3)}_{\beta\gamma}$. We pick an invariant connection $\nabla_2$ on this bundle $\ker(\hat\pi_{\beta\gamma})\to W^{(3)}_{\beta\gamma}$. Then define the invariant direct sum connection $\nabla_{\beta\gamma}$ on $E_{\beta}|_{W^{(3)}_{\beta\gamma}}\cong \tilde E_{\beta\gamma}\oplus_{W^{(3)}_{\beta\gamma}}\ker(\hat\pi_{\beta\gamma})$. $\nabla_{\beta\gamma}$ recovers $\tilde\pi_{\beta\gamma}$.

\item Construct $\nabla_{\alpha\beta}$ extending the pushforward of $\nabla_{\beta\gamma}$ and recovering $\pi_{\alpha\beta}$:

We can repeat item (i) for $\tilde\pi_{\alpha\beta}: \tilde E_{\alpha\beta}|_{W^{(4)}_{\alpha\beta}}\to \phi_{\alpha\beta}(W^{(4)}_{\alpha\beta})$, except that here instead of using $\nabla_1$ on $\tilde E_{\alpha\beta}|_{W^{(4)}_{\alpha\beta}}$, we will use the following $\nabla'_1$: We extend the invariant pushforward $((\phi_{\alpha\beta})_\ast(\nabla_{\beta\gamma}))^o$ on $\hat\phi_{\alpha\beta}(E_\beta|_{\text{orbit}(W^{(3)}_{\beta\gamma})\cap U^{(4)}_{\alpha\beta}})$ invariantly over $\hat\phi_{\alpha\beta}(E_\beta|_{U^{(4)}_{\alpha\beta}})$, then to an invariant connection $\nabla'_1$ on $\tilde E_{\alpha\beta}|_{W^{(4)}_{\alpha\beta}}$ which is trivial along fiber of $\pi_{\alpha\beta}$. Let $\nabla_{\alpha\beta}$ denote the resulting connection on $E_\alpha|_{W^{(4)}_{\alpha\beta}}$ by repeating (i) using $\nabla'_1$ in our current setting. $\nabla_{\alpha\beta}$ recovers $\tilde\pi_{\alpha\beta}$, and $\pi_{\alpha\beta}$ is trivial on each fiber of $\pi_{\alpha\gamma}$.

\item Use the transition region to construction an invariant connection on $E_\alpha|_{U^{(5)}_\alpha}$, where $U^{(5)}_\alpha$ is defined in item (II)(5):

Use the tranisition region $\text{orbit}(W^{(4)}_{\alpha\beta}\backslash W''_{\alpha\beta})$ to extend $(\nabla_{\alpha\beta})|_{W''_{\alpha\beta}}$ to an invariant connection on $E_\alpha|_{U^{(5)}_{\alpha}}$. Use this connection, we can parallel transport $\hat\phi_{\alpha\gamma}(E_\gamma|_{U^{(5)}_{\alpha\gamma}})$ to $\tilde E_{\alpha\gamma}|_{W^{(5)}_{\alpha\gamma}}$ along the geodesic $g_{\alpha}$ restricted to the fibers of $\pi_{\alpha\gamma}$. Recall those fibers of $\pi_{\alpha\gamma}$ are totally geodesic by proposition \ref{METRICCOMPATIBILITY}) and the contruction in item (II)(3). This defines $\tilde\pi_{\alpha\gamma}: \tilde E_{\alpha\gamma}|_{W^{(5)}_{\alpha\gamma}}\to \phi_{\alpha\gamma}(U^{(5)}_{\alpha\gamma})$.

\item The compatibility of $\tilde\pi_{\ast\ast}$ follows by the construction:

Notice that since the connection restricted to $\text{orbit}(\tilde E_{\alpha\beta}|_{W''_{\alpha\beta}})$ is trivial along each fiber of $\pi_{\alpha\gamma}^o=((\phi_{\alpha\beta})_\ast(\pi_{\beta\gamma}))^o\circ\pi_{\alpha\beta}^o$. Therefore, parallel transporting along the geodesic of $g_{\alpha}$ restricted to each fiber of $\pi_{\alpha\gamma}$ yields the same result as parallel transporting along the geodesic in the fiber of $((\phi_{\alpha\beta})_\ast(\pi_{\beta\gamma})^o$ first then followed by the geodesic in the fiber $\pi_{\alpha\beta}^o$. This proves that $\tilde\pi_{\alpha\gamma}$ is compatible with $\tilde\pi_{\beta\gamma}|_{(\tilde E_{\beta\gamma}|_{W''_{\beta\gamma}})}$ and $\tilde\pi_{\alpha\beta}|_{(\tilde E_{\alpha\beta}|_{W''_{\alpha\beta}})}$.

\item Make the quotient section transverse by a shrinking $U''_\alpha$ of $U^{(5)}_\alpha$, and define $W''_{\ast\ast}$ to finish the construction of all $\pi_{\ast\ast}$ and $\tilde\pi_{\ast\ast}$:

We can apply the construction in proposition \ref{LEVEL1CCHANGE}, namely by shrinking $U^{(5)}_\alpha$ to $U''_\alpha$ and fiber-shrink $W^{(5)}_{\alpha\gamma}|_{\phi_{\alpha\gamma}(U''_{\alpha\gamma})}$ to $W''_{\alpha\gamma}$ in order to achieve the transversality of the quotient section. Here $U''_\gamma=U'_\gamma$ stays fixed as we can change the domain of the coordinate change by shrinking the target chart alone, just as in \ref{LEVEL1CCHANGE}. Notice that we do not need to modify the $W''_{\alpha\beta}$ part of $U^{(5)}_\alpha$. So we have the level-1 structure for $C_\gamma|_{U''_\gamma}\to C_{\alpha}|_{U''_\alpha}$ except $\hat\pi_{\alpha\gamma}$.

\end{enumerate}

\item Construct $\hat\pi_{\alpha\gamma}: E_\alpha|_{W''_{\alpha\gamma}}\to \tilde E_{\alpha\gamma}$:

Via $E_\beta|_{W''_{\beta\gamma}}\cong \ker\hat\pi_{\beta\gamma}\oplus_{W''_{\beta\gamma}}\tilde E_{\beta\gamma}$, define an invariant bundle metric $h_{\beta\gamma}$ on $E_\beta|_{W''_{\beta\gamma}}$, making $\ker\hat\pi_{\beta\gamma}$ and $\tilde E_{\beta\gamma}$ orthogonal. Extend $((\hat\phi_{\alpha\beta})_\ast h_{\beta\gamma})^o$ to an invariant bundle metric $h_1$ on $\hat\phi_{\alpha\beta}(E_\beta|_{U^{(4)}_{\alpha\beta}})$. Using $\tilde\pi_{\alpha\beta}^o$, extend $h_1$ to a bundle metric $h_2$ on $\tilde E_{\alpha\beta}|_{W^{(4)}_{\alpha\beta}}$. Using $E_\alpha|_{W^{(4)}_{\alpha\beta}}\cong \ker\hat\pi_{\alpha\beta}\oplus_{W^{(4)}_{\alpha\beta}}\tilde E_{\alpha\beta}$, define an $E_\alpha|_{W^{(4)}_{\alpha\beta}}$-invariant direct sum bundle metric $h_3$ which makes the summands orthogonal and restricts to the first factor to $h_2$. Using the transition region $E_\alpha|_{\text{orbit}(W^{(4)}_{\alpha\beta}\backslash W''_{\alpha\beta})}$, extend $h_3^o$ to an invariant bundle metric $h_4$ on $E_\alpha|_{U^{(5)}_\alpha}$. Define $h_{\alpha\gamma}:=h_4|_{W^{(5)}_{\alpha\gamma}}$. Then over $\text{orbit}(W^{(5)}_{\alpha\gamma})\cap \text{orbit}(W''_{\alpha\beta})$, $h_{\alpha\gamma}^o$ recovers $\hat\pi_{\alpha\beta}^o$. Define $\hat\pi_{\alpha\gamma}: E_\alpha|_{W''_{\alpha\beta}}\to \tilde E_{\alpha\beta}$ using $h_{\alpha\gamma}$. Then $\hat\pi_{\alpha\gamma}$ is compatible with $\hat\pi_{\beta\gamma}$ and $\hat\pi_{\alpha\beta}$.
\end{enumerate}
\end{proof}

\begin{remark} The above general three-chart construction shows that further constructions do not modify the existing constructions except only shrinking the existing data, and Riemannian metrics and connections are not needed in the level-1 data and we can always construct those Riemmanian metrics and connections to recover the previously constructed $\pi_{\ast\ast}, \tilde\pi_{\ast\ast}$ and $\hat\pi_{\ast\ast}$.
\end{remark}

\subsection{Constructing a level-1 good coordinate system}
\begin{theorem}[existence of a level-1 good coordinate system]\label{EXLEVELONEGCS} Let $\mathcal{G}$ be a strongly intersecting Hausdorff good coordinate system indexed by a total order $(S, \leq)$. Then there exists a level-1 good coordinate system $\mathcal{G}'$ for $\mathcal{G}$ in the sense of definition \ref{LEVELONEGCS}.
\end{theorem}
\begin{proof}\label{PERTNCOMPATIBILITY} We prove by induction according to the total order $(S,\leq)$, by systematically applying generalized versions of proposition \ref{ALPHABETAGAMMA}.

\begin{enumerate}
\item The first step to start the induction process is this: Let $\beta:=\min S$ and let $\alpha:=\min(S\backslash \{\beta\})$ and apply the proposition \ref{LEVEL1CCHANGE}.

\item Denote $S_\beta:=\{\gamma\in S\;|\; \gamma\leq \beta\}$. For $\beta\in S$, we now define \emph{$S_\beta$-step}:

\begin{definition}[$S_\beta$-step] For a precompact shrinking $\{U'_{\gamma}\}_{\gamma\in S_\beta}$ of $\{U_\gamma\}_{\gamma\in S_\beta}$ in $\mathcal{G}$, we can construct level-1 $C_{\gamma_1}|_{U'_{\gamma_1}}\overset{\text{level-1}}{\to} C_{\gamma_2}|_{U'_{\gamma_2}}$ for $\gamma_1, \gamma_2\in S_\beta$ with $(\gamma_1,\gamma_2)\in I(\mathcal{G})$ such that the underlying coordinate changes $C_{\gamma_1}|_{U'_{\gamma_1}}\to C_{\gamma_2}|_{U'_{\gamma_2}}$ are restricted from the coordinate changes in $\mathcal{G}$ and those level-1 coordinate changes are compatible.
\end{definition}
\item We now define $I_\alpha$ and $J_\alpha$, and state the inductive step:

For $\alpha\in S$, denote $I_\alpha:=\{\beta\in S\;|\;(\beta,\alpha)\in I(\mathcal{G})\}$. $\beta\in I_\alpha$ is said to be an \emph{initial index} if there does not exist $\gamma\in I_\alpha$ such that $(\gamma, \beta)\in I(\mathcal{G})$. Denote the set of initial indices in $I_\alpha$ by $J_\alpha$.

The statement of the \emph{inductive step} is the following. For $\beta\in S$, assume that $S_{\beta}$-step is true and denote $\alpha:=\min{S\backslash S_{\beta}}$, then we can establish the $S_\alpha$-step for $U''_{\gamma}, \gamma\in S_{\alpha}$, such that \begin{enumerate}
\item if $\gamma\in (S_\alpha\backslash I_\alpha)\bigsqcup J_\alpha$, then $U''_\gamma=U'_\gamma$;
\item if $\gamma\in I_\alpha\backslash (J_\alpha\bigsqcup\{\alpha\})$, then $U''_\gamma$ is a precompact shrinking of $U'_\gamma$ obtained in $S_\beta$-step, as defined in \ref{PRECOMPACTCSHRINK} (with $J=\{\gamma\}$ there);
\item if $\gamma=\alpha$, then $U''_\alpha$ is a precompact shrinking of $U_\alpha$; and
\item $\{C_\gamma|_{U''_\gamma}\}_{\gamma\in S_\alpha}\bigsqcup \{C_\delta|_{U_\delta}\}_{\delta\in S\backslash S_\alpha}$ is a shrinking of $\mathcal{G}$.
\end{enumerate}

We can summarize items (a)-(d) by saying $\{U''_\gamma, U''_\alpha\}_{\gamma\in I_\alpha\backslash (J_{\alpha}\bigsqcup\{\alpha\})}$ is a precompact shrinking of $\{U'_\gamma, U_\alpha\}_{\gamma\in I_\alpha\backslash (J_{\alpha}\bigsqcup\{\alpha\})}$ in the good coordinate system obtained after $S_\beta$-step.
\item We show that the general inductive step formulated above is a slight generalization of \ref{ALPHABETAGAMMA} regarding three charts to $S_\alpha$-many charts:

\begin{enumerate}[(i)]
\item Let $\alpha_{-1}:=\max (I_\alpha\backslash \{\alpha\})$, $\alpha_{-k}:=\max(I_\alpha\backslash \{\alpha,\cdots, \alpha_{-(k-1)}\})$, and $N:=|I_\alpha|-1$. So, $I_\alpha=\{\alpha_{-N},\alpha_{-N+1},\cdots,\alpha_{-1},\alpha\}$.

\item Since $S_\beta=S_\alpha\backslash\{\alpha\}$ is one chart less, by the inductive hypothesis, we can assume that we already have level-1 coordinate changes among the charts $C_\gamma|_{U'_\gamma}, \gamma\in S_\beta$, hence in particular among the charts $\{C_\gamma|_{U'_\gamma}\}_{\gamma\in I_\alpha\backslash\{\alpha\}}\equiv\{C_{\alpha_{-i}}|_{U'_{\alpha_{-i}}}\}_{1\leq i\leq N}$.

\item We apply proposition \ref{ALPHABETAGAMMA} for the case $(\alpha,\beta,\gamma)=(\alpha,\alpha_{-1},\alpha_{-2})$, and obtain level-1 structure for a chart shrinking $U^{(-2)}_\delta$ of $U'_{\delta}$ for $\delta=\alpha_{-2}, \alpha_{-1}, \alpha$ with $U_{\alpha_{-2}}^{(-2)}=U'_{\alpha_{-2}}$. Here the shrinking $U^{(-2)}_{\alpha_{-1}}$ respects the $\pi_{\ast\ast}$-structure already in place, namely the $\pi_{\alpha_{-1}\gamma}|_{W^{(-2)}_{\alpha_{-1}\gamma}}$ of the restricted level-1 coordinate changes from any chart in $\{C_\gamma|_{U^{(-2)}_\gamma=U'_\gamma}\}_{\gamma\in I_\alpha\backslash\{\alpha,\alpha_{-1}\}}$ to $C_{\alpha_{-1}}|_{U^{(-2)}_{\alpha_{-1}}}$ and $\pi_{\gamma\gamma'}|_{W'_{\gamma\gamma'}}, \gamma,\gamma'\in I_\alpha\backslash\{\alpha, \alpha_{-1}\}$ from (ii) are still compatible. This step includes the trivial case where $(\alpha_{-2}, \alpha_{-1})\not\in I(\mathcal{G})$.

\item Construct $C_{\alpha_{-3}}|_{U^{(-3)}_{\alpha_{-3}}=U'_{\alpha_{-3}}}\overset{\text{level-1}}{\to} C_{\alpha}|_{U^{(-3)}_\alpha}$ with compatible level-1 coordinate changes among the $\{\alpha_{-3},\alpha_{-2},\alpha_{-1},\alpha\}$-charts:

\begin{enumerate}[(a)]
\item Consider the charts indexed by $\{\alpha, \alpha_{-1}, \alpha_{-3}\}$ (this item is omitted if $(\alpha_{-3},\alpha_{-1})\not\in I(\mathcal{G})$):

By applying \ref{METRICCOMPATIBILITY} as in item (II) of the proof of \ref{ALPHABETAGAMMA}, we can find an invariant Riemannian metric $g_{\alpha\alpha_{-1}}$ on $$W_1:=\text{orbit}(W^{(-2)}_{\alpha\alpha_{-1}}|_{\phi_{\alpha\alpha_{-1}}(W^{(-2)}_{\alpha\alpha_{-1}})\cap\text{orbit}((\pi_{\alpha_{-1}\alpha_{-3}})^{-1}(\phi_{\alpha_{-1}\alpha_{-3}}(U^{(-2)}_{\alpha_{-1}\alpha_{-3}})))})$$ that recovers $\pi_{\alpha\alpha_{-1}}^o|_{W_1}$ and $((\phi_{\alpha\alpha_{-1}})_\ast\pi_{\alpha_{-1}\alpha_{-3}})^o$.

\item By the same token, consider the $\{\alpha, \alpha_{-2}, \alpha_{-3}\}$-charts (this item is omitted if $(\alpha_{-3},\alpha_{-2})\not\in I(\mathcal{G})$), we can find an invariant Riemannian metric $g_{\alpha\alpha_{-2}}$ on $$W_2:=\text{orbit}(W^{(-2)}_{\alpha\alpha_{-2}}|_{\phi_{\alpha\alpha_{-2}}(W^{(-2)}_{\alpha\alpha_{-2}})\cap\text{orbit}((\pi_{\alpha_{-2}\alpha_{-3}})^{-1}(\phi_{\alpha_{-2}\alpha_{-3}}(U'_{\alpha_{-2}\alpha_{-3}})))})$$ that recovers $\pi_{\alpha\alpha_{-2}}^o|_{W_2}$ and $((\phi_{\alpha\alpha_{-2}})_\ast\pi_{\alpha_{-2}\alpha_{-3}})^o$.

\item $g_{\alpha\alpha_{-1}}$ and $g_{\alpha\alpha_{-2}}$ on the overlap give the same vector-bundle-like submersion $\pi_{-3}$ from the invariant tubular neighborhood $W_1\cap W_2$ onto the image in $\text{orbit}(\phi_{\alpha\alpha_{-3}}(U_{\alpha\alpha_{-3}}))$, thanks to the compatibility of $\{\alpha, \alpha_{-1}, \alpha_{-3}\}$ and $\{\alpha, \alpha_{-2}, \alpha_{-3}\}$.

Using a partition of unity that is constant along the fibers of $\pi_{-3}$, we can find an invariant metric $g_{-3}$ on $W_1\cup W_2$ that recovers $\pi_{\alpha\alpha_{-1}}^o$ and $((\phi_{\alpha\alpha_{-1}})_\ast\pi_{\alpha_{-1}\alpha_{-3}})^o$ on $W_1$, and $\pi_{\alpha\alpha_{-2}}^o$ and $((\phi_{\alpha\alpha_{-2}})_\ast\pi_{\alpha_{-2}\alpha_{-3}})^o$ on $W_2$.

We then precompactly shrink $U^{(-2)}_{\alpha_{-1}}$ to $U^{(-3)}_{\alpha_{-1}}$ such that $\text{orbit}(W^{(-3)}_{\alpha\alpha_{-1}})$ is a strong open neighborhood over $$U_1:=((\alpha\alpha_{-1})_\ast\pi_{\alpha_{-1}\alpha_{-3}})^{-1}(\text{orbit}(\phi_{\alpha\alpha_{-3}}(U^{(-3)}_{\alpha\alpha_{-3}})))\cap \text{orbit}(\phi_{\alpha\alpha_{-1}}(U^{(-3)}_{\alpha\alpha_{-1}}));$$ and precompactly shrink $U^{(-2)}_{\alpha_{-2}}$ to $U^{(-3)}_{\alpha_{-2}}$ with $\text{orbit}(W^{(-3)}_{\alpha\alpha_{-2}})$ is a strong open neighborhood over $$U_2:=((\alpha\alpha_{-2})_\ast\pi_{\alpha_{-2}\alpha_{-3}})^{-1}(\text{orbit}(\phi_{\alpha\alpha_{-3}}(U^{(-3)}_{\alpha\alpha_{-3}})))\cap \text{orbit}(\phi_{\alpha\alpha_{-2}}(U^{(-3)}_{\alpha\alpha_{-2}})),$$
such that $(\text{orbit}(W^{(-3)}_{\alpha\alpha_{-1}}))|_{U_1\cap U_2}=(\text{orbit}(W^{(-3)}_{\alpha\alpha_{-2}}))|_{U_1\cap U_2}$ and we also ensure that shrunken charts together with remaining charts is still a shrinking of $\mathcal{\mathcal{G}}$. Due to the compatibility, each strong open neighborhood over $U_1\cap U_2$ is a double fiber bundle ep-groupoid over $$\text{orbit}(\phi_{\alpha\alpha_{-3}}(U^{(-3)}_{\alpha\alpha_{-3}}))\cap \text{orbit}(\phi_{\alpha\alpha_{-1}}(U^{(-3)}_{\alpha\alpha_{-1}}))\cap \text{orbit}(\phi_{\alpha\alpha_{-2}}(U^{(-3)}_{\alpha\alpha_{-2}})).$$ 

Then we use the transition region $$(W_1\cup W_2)\backslash ((\text{orbit}(W^{(-3)}_{\alpha\alpha_{-1}}))|_{U_1}\cup (\text{orbit}(W^{-2}_{\alpha\alpha_{-3}}))|_{U_2})$$ to construct a Riemannan metric $g^{(-3)}_\alpha$ over $U^{(-2)}_{\alpha}$ (for constructing $\pi_{\alpha\alpha_{-3}}$, so we only needed to recover over $W_1\cup W_2$).

We use $g^{(-3)}_\alpha$ and shrink $U^{(-2)}_\alpha$ to $U^{(-3)}_\alpha$ (leaving the existing strong open neighborhoods intact) to contruct a vector-bundle-like submersion $\pi_{\alpha\alpha_{-3}}$ compatible with existing $\pi_{\ast\ast}$ in the data of compatible level-1 coordinate changes among charts $C_{\delta}|_{U^{(-3)}_\delta}, \delta=\alpha,\alpha_{-1},\alpha_{-2}, \alpha_{-3}$. Similarly, we can recover connections and bundle metrics to finish compatible data in the level-1 coordinate changes among those four charts. Observe that above general consideration include the cases where there is no coordinate changes between some of the charts with index $\delta\not=\alpha$.
\end{enumerate}
\item Construction of $C_{\alpha_{-j}}|_{U^{(-j)}_{\alpha_{-j}}=U'_{\alpha_{-j}}}\overset{\text{level-1}}{\to} C_{\alpha}|_{U^{(-j)}_\alpha}$ with the compatibility of level-1 coordinate changes among the charts at $(\alpha_{-j}, \cdots, \alpha_{-1},\alpha)$ (allowing no coordinate changes among some of the charts) is a straightforward generalization of item (iv), and builds on the construction of level-1 coordinate changes among $\{C_{\alpha_{-j+1}}|_{U^{(-j+1)}_{\alpha_{-j+1}}},\cdots,C_\alpha|_{U^{(-j+1)}_\alpha}\}$. After this step, we have compatibile level-1 coordinate changes among the charts $C_{\delta}|_{U^{(-j)}_\delta}, \delta=\alpha,\alpha_{-1},\cdots, \alpha_{-j}$. 

\item Compatible level-1 coordinate changes among $\{C_{\delta}|_{U^{(-N)}_\delta}\}_{\delta\in I_\alpha}$, hence among $\{C_\gamma|_{U''_\gamma=U'_\gamma}\}_{\gamma\in S_\alpha\backslash I_\alpha}\cup\{C_{\delta}|_{U''_\delta=U^{(-N)}_\delta}\}_{\delta\in I_\alpha}$, are obtained after $N$ steps. The inductive step is completed.
\end{enumerate}

\item After finite many steps, we reach $(S_\alpha)$-step with $\alpha=\max S$. Namely, we have compatible level-1 coordinate changes among the charts $\{C_{\gamma}|_{U'_\gamma}\}_{\gamma\in S_\alpha=S}$, the desired level-1 good coordinate system.
\end{enumerate}
\end{proof}

\begin{proposition}[controlling the size of a precompact shrinking]\label{CONTROLLINGTHESIZE} Let $\mathcal{G}$ be a strongly intersecting Hausdorff good coordinate system indexed by a total order $(S, \leq)$. Let $\tilde{\mathcal{G}}:=\{\tilde U_\alpha\}_{\alpha\in S}$ be a precompact shrinking of $\{U_\alpha\}_{\alpha\in S}$ in $\mathcal{G}$. Then there exists a level-1 good coordinate system $\mathcal{G}'$ for $\mathcal{G}$ such that $\tilde{\mathcal{G}}$ is a shrinking of the underlying good coordinate system of $\mathcal{G}'$.
\end{proposition}
\begin{proof} Notice that in choosing precompact shrinkings of charts, we merely require the precompactness, not the size at all. So we can just choose those precompact shrinkings $U^{(-i)}_{\alpha_{-k}}$ in the proof for theorem \ref{EXLEVELONEGCS} to include the corresponding bases $\hat U_{\alpha_{-k}}$ in $\hat{\mathcal{G}}$ in each step such that $\hat U_{\alpha_{-k}}$ is a precompact shrinking of $U^{(-i)}_{\alpha_{-k}}$. We need the last requirement so that we have room to precompact shrink $U^{(-i)}_{\alpha_{-k}}$ again if required in the construction, and then the resulting precompact shrinking still contains $\hat U_{\alpha_{-k}}$.
\end{proof}

\begin{remark}
Partitions of the index set in constructing level-1 good coordinate systems can be finer than by the bundle dimension, and the method is unaffected as long as we fix a total order partition, which by definition is compatible with the order by the bundle dimension. The level-1 structure for $C_\beta|_{U'_\beta}\to C_\alpha|_{U'_\alpha}$ is trivial, ($W'_{\alpha\beta}=\phi_{\alpha\beta}(U'_{\alpha\beta})$, $\tilde E_{\alpha\beta}=E_\alpha|_{\phi_{\alpha\beta}(U'_{\alpha\beta})}$, $\pi_{\alpha\beta}=Id$, $\tilde\pi_{\alpha\beta}=Id$ and $\hat\pi_{\alpha}=Id$), if the bundle dimensions of $E_\beta$ and $E_\alpha$ are the same, and this is implicit in the above proof.
\end{remark}

\subsection{From a concerted Kuranishi embedding to a concerted level-1 Kuranishi embedding}

\begin{theorem}\label{EMBEDDING0TO1EXTENDINGLEVEL1GCS} Let $\mathcal{G}\Rightarrow \tilde{\mathcal{G}}$ be a concerted Kuranishi embedding between two strongly intersecting Hausdorff good coordinate systems indexed by the same total order $(S,\leq)$. Let $\mathcal{G}''$ be a level-1 good coordinate system for $\mathcal{G}$. By composing, we have a concerted Kuranishi embedding $\mathcal{G}''\Rightarrow \tilde{\mathcal{G}}$ where we have compatible level-1 coordinate changes among charts $C_\alpha|_{U''_\alpha},\alpha\in S$.

Then there exists a concerted level-1 Kuranishi embedding $\mathcal{G}'\overset{\text{level-1}}{\Rightarrow}\tilde{\mathcal{G}}'$ for $\mathcal{G}''\Rightarrow \tilde{\mathcal{G}}$ in the sense of definition \ref{LEVEL1KEMBED}. (Hence, it is also a level-1 Kuranishi embedding for $\mathcal{G}\Rightarrow\tilde{\mathcal{G}}$).
\end{theorem}

\begin{proof}
\begin{enumerate}[(I)]
\item Denote $\alpha:=\max S$, by proposition \ref{LEVEL1CCHANGE} with $\tilde U^{(0)}_\alpha:=W^{(0)}_\alpha$, we construct a level-1 chart embedding $C_\alpha|_{U^{(0)}_\alpha}\overset{\text{level-1}}{\to} \tilde C_\alpha|_{\tilde U^{(0)}_\alpha}$ for the chart embedding $C_\alpha|_{U''_\alpha}\to\tilde C_\alpha|_{\tilde U_\alpha}$, where $U^{(0)}_\alpha$ is a precompact shrinking of $U''_\alpha$.

\item Establish $C_\beta|_{U''_\beta}\overset{\text{level-1}}{\to} \tilde C_\alpha|_{\tilde U^{(0)}_\alpha}$ for $(\beta,\alpha)\in I(\mathcal{G})$ with $\beta\not=\alpha$ for free to be used in the next step:

We have a coordinate change $C_{\beta}|_{U''_\beta}\to \tilde C_\alpha|_{\tilde U^{(0)}_\alpha}$ for $\beta\in S\backslash \{\alpha\}$ with $(\beta,\alpha)\in \mathcal{G}$, induced by composing $C_\beta|_{U''_\beta}\to C_\alpha|_{U^{(0)}_{\alpha}}\to \tilde C_\alpha|_{\tilde U^{(0)}_\alpha}$. Since the two coordinate changes in the composition are already level-1, we automatically have the level-1 structure on $C_{\beta}|_{U''_\beta}\to \tilde C_\alpha|_{\tilde U^{(0)}_\alpha}$ by composing the level-1 structures.

\item We will describe the simplest non-trivial inductive step. 
\begin{enumerate}[(1)]
\item Whilst in the level-1 good coordinate system case, we establish the level-1 structure for $(\gamma, \beta)$ first, then for $(\beta,\alpha)$ and finally for $(\gamma,\alpha)$; here we are given the level-1 structure for $(\gamma, \alpha)$ for free, and we will first construct the level-1 structure on $(\gamma,\beta)$, then on $(\beta,\alpha)$. Moreover, the domains of the chart-embeddings are the entire bases of the first charts.
\item Step up the notation:

Define $I_{\alpha}:=\{\beta\in S\;|\; (\beta,\alpha)\in I(\mathcal{G})\}$ and $\alpha_{-1}:=\max (I_\alpha\backslash \{\alpha\})$. We now construct a level-1 chart embedding for $C_{\alpha_{-1}}|_{U''_{\alpha_{-1}}}\to \tilde C_{\alpha_{-1}}|_{\tilde U_{\alpha_{-1}}}$ and a level-1 coordinate change for $\tilde C_{\alpha_{-1}}|_{\tilde U_{\alpha_{-1}}}\to \tilde C_\alpha|_{\tilde U^{(0)}_\alpha}$ with the domain of coordinate change denoted by $\tilde U^{(3)}_{\alpha\alpha_{-1}}$.

Denote the domain of coordinate change for $C_{\alpha_{-1}}|_{U''_{\alpha_{-1}}}\to C_\alpha|_{U^{(0)}_\alpha}$ by $U^{(3)}_{\alpha\alpha_{-1}}$ and $W^{(3)}_{\alpha\alpha_{-1}}:=\pi_{\alpha\alpha_{-1}}^{-1}(\phi_{\alpha\alpha_{-1}}(U^{(3)}_{\alpha\alpha_{-1}}))$.

\item We construct an invariant Riemannian metric $\tilde g_\alpha$ which over a part of the base recovers a vector-bundle-like submersion $\tau_{\alpha\alpha_{-1}}$ compatible with the existing structures:

We can find an invariant Riemannian metric $\tilde g_\alpha$ for recovering $$\Lambda_{\alpha\alpha_{-1}}:=((\phi_{\alpha})_\ast\pi_{\alpha\alpha_{-1}})^o\circ \Pi_{\alpha}$$ on $$\Pi_\alpha^{-1}(\text{orbit}(W^{(3)}_{\alpha\alpha_{-1}}))\subset W^{(0)}_\alpha=\tilde U^{(0)}_\alpha,$$ and $\tilde g_\alpha$ is totally geodesic on each fiber of $$\Lambda_{\alpha\alpha_{-1}}: \Pi_\alpha^{-1}(\text{orbit}(W^{(3)}_{\alpha\alpha_{-1}}))\to \phi_{\alpha\alpha_{-1}}(U^{(3)}_{\alpha\alpha_{-1}}).$$

We want to define a submersion $\Pi_{\alpha_{-1}}: W_{\alpha_{-1}}\to \phi_{\alpha_{-1}}(U''_{\alpha_{-1}})$. We already know a part of the definition. We use the pullback metric $(\tilde \phi_{\alpha\alpha_{-1}})^\ast (\tilde g_\alpha)$ to define the vector-bundle-like submersion $\tau_{\alpha\alpha_{-1}}$ on an immersion invariant tubular neighborhood $$(\tilde{\phi}_{\alpha\alpha_{-1}})^{-1}(\tilde{\phi}_{\alpha\alpha_{-1}}(\tilde U^{(3)}_{\alpha\alpha_{-1}})\cap \Pi_\alpha^{-1}(\text{orbit}(W^{(3)}_{\alpha\alpha_{-1}})))$$ of $\phi_{\alpha_{-1}}(U^{(3)}_{\alpha\alpha_{-1}})$.

Take a precompact shrinking of $U^{(-1)}_{\alpha_{-1}}$ of $U''_{\alpha_{-1}}$ and a precompact shrinking $U^{(-0.5)}_\alpha$ of $U^{(0)}_\alpha$, and denote the domain of coordinate change $C_{\alpha_{-1}}|_{U^{(-1)}_{\alpha_{-1}}}\to C_{\alpha}|_{U^{(-0.5)}_\alpha}$ by $U^{(-0.5)}_{\alpha\alpha_{-1}}$.

Then, we have a fiber-shrinking $O_{\alpha\alpha_{-1}}$ of $(\tau_{\alpha\alpha_{-1}})^{-1}(\phi_{\alpha_{-1}}(U^{(-0.5)}_{\alpha\alpha_{-1}}))$ being an invariant embedded strong open neighborhood of $$\phi_{\alpha_{-1}}(U^{(-0.5)}_{\alpha\alpha_{-1}}).$$

\item Create a transition region to extend the pullback metric to $\tilde g_{\alpha_{-1}}$.

Take another precompact shrinking $U^{(-1)}_{\alpha}$ of $U^{(-0.5)}_{\alpha}$ with the domain of the coordinate change $C_{\alpha_{-1}}|_{U^{(-1)}_{\alpha_{-1}}}\to C_{\alpha}|_{U^{(-1)}_\alpha}$ denoted by $U^{(-1)}_{\alpha\alpha_{-1}}$.

Using $((\tau_{\alpha\alpha_{-1}})^{-1}(\phi_{\alpha_{-1}}(U^{(-0.5)}_{\alpha\alpha_{-1}})))\backslash O_{\alpha\alpha_{-1}}$ as a transition region, we can extend the metric $(\tilde\phi_{\alpha\alpha_{-1}})^\ast(\tilde g_\alpha|_{(\tau_{\alpha\alpha_{-1}})^{-1}(\phi_{\alpha_{-1}}(U^{(-1)}_{\alpha\alpha_{-1}}))})$ to an invariant Riemannian metric $\tilde g_{\alpha_{-1}}$ on $\tilde U_{\alpha_{-1}}$. 

\item Use $\tilde g_{\alpha_{-1}}$ on $\tilde U_{\alpha_{-1}}$ to construct $\Pi_{\alpha_{-1}}$ in the level-1 chart embedding for index $\alpha_{-1}$:

We use this metric $\tilde g_{\alpha_{-1}}$ to produce a vector-bundle-like submersion $\Pi_{\alpha_{-1}}: \tilde U^{(-1)}_{\alpha_{-1}}=W^{(-1)}_{\alpha_{-1}}\to \phi_{\alpha_{-1}}(U^{(-1)}_{\alpha_{-1}})$ such that $W^{(-1)}_{\alpha_{-1}}$ is an invariant strong open neighborhood of $\phi_{\alpha_{-1}}(U^{(-1)}_{\alpha_{-1}})$ and $$\Pi_{\alpha_{-1}}^{-1}(\phi_{\alpha_{-1}}(U^{(-1)}_{\alpha\alpha_{-1}}))= O_{\alpha\alpha_{-1}}.$$

$\Pi_{\alpha_{-1}}$ is a part of the level-1 structure of the chart embedding $$C_{\alpha_{-1}}|_{U^{(-1)}_{\alpha_{-1}}}\to \tilde C_{\alpha_{-1}}|_{\tilde U^{(-1)}_{\alpha_{-1}}},$$ which extends $\tau_{\alpha\alpha_{-1}}: O_{\alpha\alpha_{-1}}\to\phi_{\alpha_{-1}}(U^{(-1)}_{\alpha\alpha_{-1}})$.

\item Similarly, we can construct the connection and bundle metric to recover the $\tilde\Lambda_{\alpha\alpha_{-1}}:=((\hat\phi_{\alpha})_\ast \tilde\pi_{\alpha\alpha_{-1}})^o\circ\tilde\Pi_{\alpha}$ and $$\hat\Lambda_{\alpha\alpha_{-1}}:=((\tilde \Pi_\alpha)^\ast((\hat\phi_\alpha)_\ast \hat\pi_{\alpha\alpha_{-1}}))^o\circ \hat\Pi_\alpha.$$ We proceed as follows:

\item Construct the connection in a similar way recovering the subbundle extension over a part of the base, pullback and extend it:

One can contrast the following arguments with the ones in the proof of proposition \ref{ALPHABETAGAMMA}. 

For the simplicity of notation, denote $U:=\tilde\phi_{\alpha\alpha_{-1}}(\phi_{\alpha_{-1}}(U^{(-0.5)}_{\alpha\alpha_{-1}}))$. Then $\tilde\Lambda_{\alpha\alpha_{-1}}$ gives a trivialization of $\tilde F_\alpha|_{\Lambda_{\alpha\alpha_{-1}}^{-1}(U)}$ along each fiber of $\Lambda_{\alpha\alpha_{-1}}$. By picking an invariant connection on the vector bundle ep-groupoid $\ker\hat\Lambda_{\alpha\alpha_{-1}}\to \Lambda_{\alpha\alpha_{-1}}^{-1}(U)$, we can trivialize it along each fiber of $\Lambda_{\alpha\alpha_{-1}}$. So, once we pick an invariant direct-sum connection on $\tilde{E}_\alpha|_{U}= \tilde F_{\alpha}|_{U}\oplus_{U}\ker\hat\Lambda_{\alpha\alpha_{-1}}|_U$, we have an invariant connection $\nabla$ on $\tilde{E}_\alpha|_{\Lambda_{\alpha\alpha_{-1}}^{-1}(U)}= \tilde F_{\alpha}|_{\Lambda_{\alpha\alpha_{-1}}^{-1}(U)}\oplus_{\Lambda_{\alpha\alpha_{-1}}^{-1}(U)}\ker\hat\Lambda_{\alpha\alpha_{-1}}$. We pullback $\nabla$ to $(\widetilde{\hat\phi_{\alpha\alpha_{-1}}})^\ast\nabla$, and use the above transition region $\tau_{\alpha\alpha_{-1}}(U^{(-0.5)}_{\alpha\alpha_{-1}})\backslash O_{\alpha\alpha_{-1}}$ to extend it to an invariant connection on $\tilde E_{\alpha_{-1}}|_{U^{(-1)}_{\alpha_{-1}}}$. We use this conention to parallel transport $\hat\phi_{\alpha_{-1}}(E_{\alpha_{-1}}|_{U^{-1}_{\alpha_{-1}}})$ into $\tilde F_{\alpha_{-1}}$ along $\tilde g_{\alpha_{-1}}$-geodesics in the toally geodesic fibers of $\Pi_{\alpha_{-1}}$. This defines $\tilde\Pi_{\alpha_{-1}}$. Define $\tilde U^{(-1)}_{\alpha_{-1}}:=W^{(-1)}_{\alpha_{-1}}$.

\item By shrinking $U^{(-1)}_{\alpha_{-1}}$ and fiber-shrinking $W^{(-1)}_{\alpha_{-1}}$ restricted to the shrunken base (without changing the established level-1 structures where the transversality already holds), we achieve the transversality of the quotient section. This (due to the domain of a chart embedding being the entire base) is a distinction to level-1 good coordinate system, where the $\gamma$-chart does not need to shrink again for the transversality in \ref{ALPHABETAGAMMA}. We restrict the level-1 structures to the shrinking, and keep the same notation.

\item Construct the bundle metric (the level-1 $\alpha_{-1}$-chart embedding):

Now we construct an invariant bundle metric which gives $\hat \Lambda_{\alpha\alpha_{-1}}$. First construct an invariant direct-sum bundle metric for $\tilde{E}_\alpha|_{U}= \tilde{F}_{\alpha}|_{U}\oplus_{U}\ker\hat\Lambda_{\alpha\alpha_{-1}}|_U$, then use the identification along the fibers of $\Lambda_{\alpha\alpha_{-1}}$ to extend it to an invariant bundle metric $h$ on $$\tilde{E}_\alpha|_{\Lambda_{\alpha\alpha_{-1}}^{-1}(U)}= \tilde F_{\alpha}|_{\Lambda_{\alpha\alpha_{-1}}^{-1}(U)}\oplus_{\Lambda_{\alpha\alpha_{-1}}^{-1}(U)}\ker\hat\Lambda_{\alpha\alpha_{-1}}.$$ We pullback $h$ to $(\widetilde{\hat\phi_{\alpha\alpha_{-1}}})^\ast h$, and use the above transition region $\tau_{\alpha\alpha_{-1}}(U^{(-0.5)}_{\alpha\alpha_{-1}})\backslash O_{\alpha\alpha_{-1}}$ to extend it to an invariant bundle metric on $\tilde{E}_{\alpha_{-1}}|_{U^{(-1)}_{\alpha_{-1}}}$. We use this bundle metric to define $\hat\Pi_{\alpha_{-1}}$.

So, we have a level-1 chart embedding $C_{\alpha_{-1}}|_{U^{(-1)}_{\alpha_{-1}}}\overset{\text{level-1}}{\to}\tilde C_{\alpha_{-1}}|_{\tilde U^{(-1)}_{\alpha_{-1}}}$.

\item Construct the level-1 coordinate change in $\tilde{\mathcal{G}}$ indexed by $(\alpha_{-1},\alpha)$:

We use the above $\tilde g_{\alpha}$ to define the vector-bundle-like submersion $\widetilde{\pi_{\alpha\alpha_{-1}}}:\tilde W^{(-1)}_{\alpha\alpha_{-1}}\to \tilde\phi_{\alpha\alpha_{-1}}(\tilde U^{(-1)}_{\alpha\alpha_{-1}})$ in the usual way. Note that $$\widetilde{\pi_{\alpha\alpha_{-1}}}=(((\tilde\phi_{\alpha})_\ast\pi_{\alpha\alpha_{-1}})^o\circ \Pi_\alpha)|_{(((\tilde\phi_{\alpha})_\ast\pi_{\alpha\alpha_{-1}})^o\circ \Pi_\alpha)^{-1}(\phi_{\alpha\alpha_{-1}}(\Pi_{\alpha_{-1}}(\tilde U^{(-1)}_{\alpha\alpha_{-1}})))}.$$ We parallel transport using the above connection $\nabla$ along $\tilde g_{\alpha}$-geodesics in the fibers of $\widetilde{\pi_{\alpha\alpha_{-1}}}$, which belong to totally geodesic fibers of $\Lambda_{\alpha\alpha_{-1}}$, we obtain $\widetilde{\tilde\pi_{\alpha\alpha_{-1}}}$. Using the above bundle metric $h$, we can define $\widetilde{\hat\pi_{\alpha\alpha_{-1}}}$.

We have a level-1 coordinate change $\tilde C_{\alpha_{-1}}|_{\tilde U^{(-1)}_{\alpha_{-1}}}\overset{\text{level-1}}{\to} \tilde C_{\alpha}|_{\tilde U^{(-1)}_{\alpha}}$ now, such that $\tilde U^{(-1)}_{\alpha_{-1}}:=\Pi_{\alpha_{-1}}^{-1}(U^{(-1)}_{\alpha_{-1}})$ is a predetermined precompact shrinking of $\tilde U_{\alpha_{-1}}$ and $\tilde U^{(-1)}_\alpha:=\Pi_{\alpha}^{-1}(U^{(-1)}_{\alpha})$ is a predetermined precompact shrinking of $\tilde U^{(0)}_{\alpha}$.

\item The commutativity of the square follows by the construction:

Because each fiber of $\Lambda_{\alpha\alpha_{-1}}$ is totally geodesic, and the connection restricted to $\tilde E_\alpha$ and the bundle metric are trivial along the fibers of $\Lambda_{\alpha\alpha_{-1}}$, the following level-1 square commutes up to the morphisms in $\tilde C_\alpha|_{\tilde U^{(-1)}_\alpha}$:

$$\begin{CD}
\tilde C_{\alpha_{-1}}|_{\tilde U^{(-1)}_{\alpha_{-1}}}@>\text{level-1}>>\tilde C_\alpha|_{\tilde U^{(-1)}_\alpha}\\
@A\text{level-1}AA @AA\text{level-1}A\\
C_{\alpha_{-1}}|_{U^{(-1)}_{\alpha_{-1}}}@>\text{level-1}>>C_\alpha|_{U^{(-1)}_\alpha}
\end{CD}\;\;\;\;.$$
\end{enumerate}

\item Now we describe the general inductive step, which has exactly the same structure as the above simplest non-trivial case. 

\begin{enumerate}
\item Step up the notations and the inductive hypothesis:

Fix $\alpha\in S$. Denote $T_\alpha:=\{\beta\in S\;|\;\alpha\leq \beta\}$ and $i:=|T_\alpha|-1$.

Suppose we already established level-1 chart embeddings $$C_{\beta}|_{U^{(-i)}_\beta}\overset{\text{level-1}}{\to}\tilde C_\beta|_{\tilde U^{(-i)}_\beta}$$ for all $\beta\in T_\alpha$ and compatible level-1 coordinate changes $$\tilde C_\gamma|_{\tilde U^{(-i)}_\gamma}\overset{\text{level-1}}{\to}\tilde C_\beta|_{\tilde U^{(-i)}_\beta}$$ for all $(\gamma,\beta)\in I(\tilde{\mathcal{G}})=I(\mathcal{G})$ with $\gamma,\beta\in T_\alpha$ such that $U^{(-i)}_\beta$ and $\tilde U^{(-i)}_\beta$, $\beta\in T_\alpha$ are precompact shrinkings of $U''_\beta$ and $\tilde U_\beta$, $\beta\in T_\alpha$ and we have the level-1 commutative squares $$\begin{CD}
\tilde C_\gamma|_{\tilde U^{(-i)}_\gamma}@>\text{level-1}>>\tilde C_\beta|_{\tilde U^{(-i)}_\beta}\\
@A\text{level-1}AA @AA\text{level-1}A\\
C_{\gamma}|_{U^{(-i)}_\gamma}@>\text{level-1}>>C_\beta|_{U^{(-i)}_\beta}
\end{CD}$$ for all $(\gamma,\beta)\in I(\mathcal{G})$ with $\gamma, \beta\in T_{\alpha}$.

Denote $\gamma:=\max(S\backslash T_\alpha)$ and $K_\gamma:=\{\beta\in S\;|\; (\gamma, \beta)\in I(\mathcal{G})\}\backslash\{\gamma\}$. Then $K_\gamma\subset T_\alpha$.
\item A trivial case:

If $K_\gamma=\emptyset$, we just proceed as in item (I), to get a level-1 chart embedding $C_\gamma|_{U^{(-i-1)}_\gamma}\overset{\text{level-1}}{\to} \tilde C_\gamma|_{\tilde U^{(-i-1)}_\gamma}$, where $U^{(-i-1)}_\gamma$ is a shrinking of $U''_\gamma$ and $\tilde U^{(-i-1)}_\gamma=W^{(-i-1)}_\gamma$ is a shrinking of $\tilde U_\gamma$. For $\beta\in T_\alpha$, define $U^{(-i-1)}_\beta:=U^{(-i)}_\beta$ and $\tilde U^{(-i-1)}_\beta:=\tilde U^{(-i)}_\beta$.

\item The non-trivial alternative and generalize items (I)-(III):

If $K_\gamma\not=\emptyset$, we proceeed similarly to item (I), (II) and (III). We first get a level-1 chart embedding $C_\gamma|_{U^{(i)}_\gamma}\overset{\text{level-1}}{\to} \tilde C_\gamma|_{\tilde U^{(i)}_\gamma}$: Compose the level-1 $C_\gamma|_{U''_\gamma}\overset{\text{level-1}}{\to} C_\beta|_{U^{(i)}_\beta}$ with $C_\beta|_{U^{(i)}_\beta}\overset{\text{level-1}}{\to}\tilde C_\beta|_{\tilde U^{(i)}_\beta}$ to have level-1 coordinate changes $C_\gamma|_{U''_\gamma}\overset{\text{level-1}}{\to} \tilde C_\beta|_{\tilde U^{(i)}_\beta}$ for all $\beta\in K_\gamma$.

\item Construct the metric:

For each $\beta\in K_\gamma$, denote $\Lambda_{\beta\gamma}:=((\phi_\beta)_\ast\pi_{\beta\gamma})^o\circ\Pi_{\beta}$ and find an invariant metric $\tilde g_\beta$ recovering it, in particular $\tilde g_\beta$ is totally geodesic on the fibers of the vector-bundle-like submersion $\Lambda_{\beta\gamma}$. We pull them back to $(\tilde\phi_{\beta\gamma})^\ast\tilde g_\beta$ for all $\beta\in K_\gamma$. Use them to define the vector-bundle-like submersions $\tau_{\beta\gamma}: O_{\beta\gamma}\to \phi_{\gamma}(U^{(-i-0.5)}_{\beta\gamma})$, where $U^{(-i-0.5)}_{\beta\gamma}$ is the domain of the coordinate changes between charts of precompact shrinkings $U^{(-i-1)}_{\gamma}$ and $U^{(-i-0.5)}_\beta$ of $C_\gamma|_{U^{(-i)}_\gamma}$ and $C_\beta|_{U^{(-i)}_\beta}$. Observe that $\tau_{\beta_1\gamma}=\tau_{\beta_2\gamma}$ over $$\phi_\gamma(U^{(-i-0.5)}_{\beta_1\gamma})\cap \phi_\gamma(U^{(-i-0.5)}_{\beta_2\gamma})=\phi_\gamma(U^{(-i-0.5)}_{\beta_1\gamma}\cap U^{(-i-0.5)}_{\beta_2\gamma})$$ for $\beta_1, \beta_2\in K_\gamma$. Use a partition of unity $\{\mu_\beta\}_{\beta\in K_\gamma}$ for the open sets $$\phi_\beta(U_{\beta\gamma}), \beta\in K_\gamma,$$ we have a partition of unity $\mu_\beta\circ\tau_{\beta\gamma}$ for the open cover $\{O_{\beta\gamma}\}_{\beta\in K_\gamma}$ for $\bigcup_{\beta\in K_\gamma} O_{\beta\gamma}$. We define the invariant metric $$\tilde g'_\gamma:=\sum_{\beta\in K_\gamma}(\mu_\beta\circ \tau_{\beta\gamma}) ((\tilde\phi_{\beta})^\ast\tilde g_\beta).$$ Use $\tilde g'_\gamma$, we can define the vector-bundle-like submersion $$\tau_\gamma: O_\gamma\to\phi_\gamma(\bigcup_{\beta\in K} U^{(-i-0.5)}_{\beta\gamma})$$ such that the fiber $(O_\gamma)_z\subset\bigcap_{\beta\in K_\gamma, z\in \phi_\gamma(U^{(-i-0.5)}_{\beta\gamma})} (O_{\beta\gamma})_z$. We have that $\tau_\gamma=\tau_{\beta\gamma}$ over the common domain of the definitions.

Take another precompact shrinking $U^{(-i-1)}_\gamma$ of $U^{(-i-0.5)}_\gamma$ and a fiber shrinking $O'_\gamma$ of $O_\gamma$. Denote the domains of the coordinate changes of $C_{\gamma}|_{U^{(-i-1)}_\gamma}\to C_{\beta}|_{U^{(-i-1)}_\beta}$ by $U^{(-i-1)}_{\beta\gamma}$ for all $\beta\in K_\gamma$. Then use the transition region $$(\tau_{\gamma})^{-1}(\phi_{\gamma}(\bigcup_{\beta\in K_\gamma} U^{(-i-0.5)}_{\beta\gamma}))\backslash O'_{\gamma}|_{\phi_{\gamma}(\bigcup_{\beta\in K_\gamma} U^{(-i-1)}_{\beta\gamma})}$$ to extend the Riemannian metric $\tilde g'_\gamma|_{(O'_{\gamma}|_{\phi_{\gamma}(\bigcup_{\beta\in K_\gamma} U^{(-i-1)}_{\beta\gamma})})}$ to an invariant Riemannian metric $\tilde g_\gamma$ on $\tilde U_{\gamma}$. We use $\tilde g_\gamma$ to produce a vector-bundle-like submersion $\Pi_{\gamma}:\tilde U^{(-i-1)}_\gamma=W^{(-i-1)}_\gamma\to \phi_{\gamma}(U^{(-i-1)}_{\gamma})$ such that $W^{(-i-1)}_\gamma$ is an invariant strong open neighborhood of $\phi_\gamma(U^{(-i-1)}_\gamma)$ and $$\Pi_\gamma^{-1}(\phi_\gamma(\bigcup_{\beta\in K_\gamma}U^{(-i-1)}_{\beta\gamma}))\subset O'_\gamma.$$

\item We  now construct the remaining data for the level-1 chart embedding $$C_\gamma|_{U^{(-i-1)}_\gamma}\overset{\text{level-1}}{\to}\tilde C_\gamma|_{\tilde U^{(-i-1)}_\gamma},$$ by first constructing the connection:

Within this item, we denote $U^\beta:=\tilde\phi_{\beta\gamma}(\phi_\gamma(U^{(-i-0.5)}_{\beta\gamma}))$. Then $$\tilde\Lambda_{\beta\gamma}:=((\phi_{\beta})_\ast \tilde\pi_{\beta\gamma})^\circ \tilde\Pi_\beta$$
gives a trivialization of $\tilde F_\beta|_{\Lambda_{\beta\gamma}^{-1}(U^\beta)}$. 

Denote $\hat\Lambda_{\beta\gamma}:=((\tilde\Pi_{\beta})^\ast((\hat\phi_\beta)_\ast\hat\pi_{\beta\gamma}))^o\circ\hat\Pi_\beta$.
 
By picking an invariant bundle connection on $\ker\hat\Lambda_{\beta\gamma}\to \Lambda^{-1}_{\beta\gamma}(U^\beta)$, we can trivialize it along each fiber of $\Lambda_{\beta\gamma}$. So, we pick an invariant direct-sum connection on $\tilde{E}_\beta|_{U^\beta}=\tilde F_\beta|_{U^\beta}\oplus_{U^\beta}\ker\hat\Lambda_{\beta\gamma}|_{U^\beta}$, then we have an invariant connection $\nabla^\beta$ on $$\tilde{E}_\beta|_{\Lambda_{\beta\gamma}^{-1}(U^\beta)}=\tilde F_\beta|_{\Lambda_{\beta\gamma}^{-1}(U^\beta)}\oplus_{\Lambda_{\beta\gamma}^{-1}(U^\beta)}\ker\hat\Lambda_{\beta\gamma}|_{\Lambda_{\beta\gamma}^{-1}(U^\beta)}.$$ 

We pullback $\nabla^\beta$ to $(\widetilde {\hat\phi_{\beta\gamma}})^\ast \nabla^\beta$ on $\tilde E_\gamma|_{\tau_{\beta\gamma}^{-1}(\phi_\gamma(U^{(-i-0.5)}_{\beta\gamma}))}$. Denote the bundle projection $\widetilde E_\gamma \to \tilde U_\gamma$ by $\lambda_\gamma$. We use the partition of unity $\mu_\beta\circ\tau_{\beta\gamma}\circ\lambda_{\gamma}, \beta\in K_\gamma$ to get an invariant connection $\nabla'_\gamma$ on $\tilde E_\gamma|_{O_\gamma}$ and use the transition region $$(\tau_{\gamma})^{-1}(\phi_{\gamma}(\bigcup_{\beta\in K_\gamma} U^{(-i-0.5)}_{\beta\gamma}))\backslash (O'_{\gamma}|_{\phi_{\gamma}(\bigcup_{\beta\in K_\gamma} U^{(-i-1)}_{\beta\gamma})})$$
to extend  $\nabla'_\gamma|_{(O'_{\gamma}|_{\phi_{\gamma}(\bigcup_{\beta\in K_\gamma} U^{(-i-1)}_{\beta\gamma})})}$ to an invariant bundle connection $\nabla_\gamma$ on $\tilde E_\gamma|_{\tilde U^{(-i-1)}_\gamma}$. We use this conneciton to parallel transport $\hat \phi_\gamma(E_\gamma|_{U^{(-i-1)}_\gamma})$ into $\tilde F_\gamma$ along $\tilde g_\gamma$-geodesics in the totally geodesic fibers of $\Pi_\gamma$. This defines $\tilde\Pi_\gamma$. Define $\tilde U^{(-1)}_\gamma:=W^{(-i-1)}_\gamma$. By shrinking the base then a fiber-shrinking (keeping the part with the validity of the transversality from the existing level-1 structures intact), we obtain the transversality of the quotient section. We then restrict other level-1 structures adapted to this shrinking and we still keep the same notation.

\item We construct the bundle metric:

Now we construct the invariant bundle metric which gives $\hat\Lambda_{\beta\gamma}$ defined above. We first construct an invariant direct-sum bundle metric for $\tilde{E}_\beta|_{U^\beta}= \tilde F_{\beta}|_{U^\beta}\oplus_{U^\beta}\ker\hat\Lambda_{\beta\gamma}|_{U^\beta}$, then use the identification along the fibers of $\Lambda_{\beta\gamma}$ to extend it to an invariant bundle metric $h_\beta$ on $\tilde{E}_\beta|_{\Lambda_{\beta\gamma}^{-1}(U^\beta)}= \tilde{F}_{\beta}|_{\Lambda_{\beta\gamma}^{-1}(U^\beta)}\oplus_{\Lambda_{\beta\gamma}^{-1}(U^\beta)}\ker\hat\Lambda_{\beta\gamma}$. We pullback $h_\beta$ to $(\widetilde{\hat\phi_{\beta\gamma}})^\ast h_\beta$. Using the partition of unity $\mu_\beta\circ\pi_{\beta\gamma}\circ\lambda_\gamma$, we can get an invariant bundle metric $h'_\gamma$ on $\tilde E_\gamma|_{O_\gamma}$, and use the above transition region $$(\tau_{\gamma})^{-1}(\phi_{\gamma}(\bigcup_{\beta\in K_\gamma} U^{(-i-0.5)}_{\beta\gamma}))\backslash O'_{\gamma}|_{\phi_{\gamma}(\bigcup_{\beta\in K_\gamma} U^{(-i-1)}_{\beta\gamma})}$$ to extend  $h'_\gamma|_{(O'_{\gamma}|_{\phi_{\gamma}(\bigcup_{\beta\in K_\gamma} U^{(-i-1)}_{\beta\gamma})})}$ to an invariant bundle metric $h_\gamma$ on $\widetilde {E_{\gamma}}|_{U^{(-i-1)}_{\gamma}}$. We use this bundle metric to define $\hat\Pi_{\gamma}$.

So, we have a level-1 chart embedding $C_{\gamma}|_{U^{(-i-1)}_{\gamma}}\overset{\text{level-1}}{\to}\tilde C_{\gamma}|_{\tilde U^{(-i-1)}_{\gamma}}$.

\item Construct new level-1 coordinate changes in $\tilde{\mathcal{G}}$:

We use above $\tilde g_\beta$ on $\Lambda_{\beta\gamma}^{-1}(U^\beta)$ to define the vector-bundle-like submersion $\widetilde{\pi_{\beta\gamma}}:\tilde W^{(-i-1)}_{\beta\gamma}\to \tilde\phi_{\beta\gamma}(\tilde U^{(-i-1)}_{\beta\gamma})$ in the usual way (by defining $\widetilde{\pi_{\beta\gamma}}$ over $\tilde\phi_{\beta\gamma}(\tilde U^{(-i-0.5)}_{\beta\gamma})$ first). Note that we then have $$\widetilde{\pi_{\beta\gamma}}=(((\tilde\phi_{\beta})_\ast\pi_{\beta\gamma})^o\circ \Pi_\beta)|_{(((\tilde\phi_{\beta})_\ast\pi_{\beta\gamma})^o\circ \Pi_\beta)^{-1}(\phi_{\beta\gamma}(\Pi_{\gamma}(\tilde U^{(-i-1)}_{\beta\gamma})))}.$$ We parallel transport using the above connection $\nabla^\beta$ along $\tilde g_{\beta}$-geodesics in the fibers of $\widetilde{\pi_{\beta\gamma}}$, which belongs to totally geodesic fibers of $\Lambda_{\beta\gamma}$, we obtain $\widetilde{\tilde\pi_{\beta\gamma}}$. The transversality of the quotient section is automatic.

Using the above bundle metric $h_\beta$, we can define $\widetilde{\hat\pi_{\beta\gamma}}$.

Therefore, we constructed a level-1 coordinate change $$\tilde C_{\gamma}|_{\tilde U^{(-i-1)}_{\gamma}}\overset{\text{level-1}}{\to} \tilde C_{\beta}|_{\tilde U^{(-i-1)}_{\beta}}$$ for all $\beta\in K_\gamma$, where $$\tilde U^{(-i-1)}_{\gamma}:=\Pi_{\gamma}^{-1}(U^{(-i-1)}_{\gamma})$$ is a predetermined precompact shrinking of $\tilde U_{\gamma}$ and $$\tilde U^{(-i-1)}_\beta:=\Pi_{\beta}^{-1}(U^{(-i-1)}_{\beta})$$ is a predetermined precompact shrinking of $\tilde U^{(-i)}_{\beta}$.

\item The level-1 commutativity of the level-1 squares constructed so far follows by the construction:

Because each fiber of $\Lambda_{\beta\gamma}$ is totally geodesic, and the connection restricted to $\tilde E_\beta$ and the bundle metric are trivial along fibers of $\Lambda_{\beta\gamma}$, the following level-1 squares commute for all $\beta\in K_\gamma$:

$$\begin{CD}
\tilde C_{\gamma}|_{\tilde U^{(-i-1)}_{\gamma}}@>\text{level-1}>>\tilde C_\beta|_{\tilde U^{(-i-1)}_\beta}\\
@A\text{level-1}AA @AA\text{level-1}A\\
C_{\gamma}|_{U^{(-i-1)}_{\gamma}}@>\text{level-1}>>C_\beta|_{U^{(-i-1)}_\beta}
\end{CD}\;\;\;\;.$$

Notice that we heavily relied on the compatibility of the previously constructed structures in the above construction.

Define $U^{(-i-1)}_\delta:=U^{(-i)}_\delta$ for all $\delta\in T_\alpha\backslash K_\gamma$. Then we completed the inductive step.
\end{enumerate}
\item After finitely many steps, we finish the construction of a concerted level-1 Kuranishi embedding between two level-1 good coordinate systems. We relabel the resulting data to match the conclusion.
\end{enumerate}
\end{proof}

\section{Perturbation of the level-1 good coordinate system}\label{PERTN}

In theorem \ref{EXLEVELONEGCS}, starting from a strongly intersecting Hausdorff good coordinate system $\mathcal{G}$ indexed by a total order $(S, \leq)$, we can obtain a level-1 good coordinate system $\mathcal{G}'$ for $\mathcal{G}$. In this section, we explain how to use the level-1 structure to naturally lift a perturbation that makes the section $s_\beta$ in one chart $C_\beta|_{U'_\beta}$ transverse to a perturbation in another charts $C_\alpha|_{U'_\alpha}$ of a higher order that automatically makes the section $s_\alpha|_{W'_{\alpha\beta}}$ transverse (because the tangent bundle condition holds over the perturbed zeros in $W'_{\alpha\beta}$ too). This is important for globally constructing a compact invariant perturbation by induction.

Before discussing perturbations by multisections, we need to recall some definitions regarding the orientation for Kuranishi structures and good coordinate systems from Fukaya-Oh-Ohta-Ono \cite{FOOO}, which we will use.

\begin{definition}[orientation, \cite{FOOO}]\label{ORIENTATION}
\begin{enumerate}
\item A Kuranishi chart $C_x|_{U_x}$ is \emph{orientable} if $TU_x\oplus_{U_x} (E_x|_{U_x})^*$ is orientable, where $(E_x|_{U_x})^*$ denotes the dual bundle of $E_x|_{U_x}$. An \emph{orientation} for an orientable Kuranishi chart is a choice of an orientation for $TU_x\oplus (E_x|_{U_x})^*$. An \emph{orientable Kuranishi chart} with a chosen orientation is called an oriented Kuranishi chart.
\item Let $C_y|_{U_y}\to C_x|_{U_x}$ be a coordinate change from an oriented chart $C_y|_{U_y}$ to another oriented chart $C_x|_{U_x}$. The orientation on the chart $C_x|_{U_x}$ induces a natural orientation on $T(\phi_{xy}(U_{xy}))\oplus (\hat\phi_{xy}(E_y|_{U_{xy}}))^*$ because of the tangent bundle condition. The orientation on $$T(\phi_{xy}(U_{xy}))\oplus (\hat\phi_{xy}(E_y|_{U_{xy}}))^*$$ in turn naturally induces an orientation on $TU_y\oplus (E_y|_{U_y})^*$ through the identification via the embedding. If the induced orientation on $TU_y\oplus (E_y|_{U_y})^*$ obtained this way agrees with the given orientation, we say that the coordinate change is \emph{orientation-preserving}.
\item A Kuranishi structure is said to be \emph{orientable} if each Kuranishi chart is orientable and can be equipped with an orientation such that all coordinate changes are orientation-preserving. An \emph{orientation} for an orientable Kuranishi structure is the data consisting of a choice of orientation for each Kuranishi chart such that all coordinate changes are orientation-preserving. An \emph{oriented Kuranishi structure} is an orientable Kuranishi structure equipped with an orientation.
\item Orientability, an orientation and being oriented can also be defined in the same way for good coordinate systems and level-1 good cooordinate systems. A (level-1) good coordinate system of an oriented Kuranishi structure is naturally oriented, and we will always use such an orientation.\end{enumerate}
\end{definition}

\begin{definition}[multisection and its transversality] One way is to extract a version from \cite{PolyfoldIII} for each vector bundle ep-groupoid $E_\alpha|_{U_\alpha}\to U_\alpha$, since the notation there has the usual meaning for the finite dimensional differential geometry. Moreover, this meaning agrees with the M-polyfold meaning in finite dimensions, as we have the usual smoothness and retracts are submanifolds and hence can be replaced by the standard local models. We require weights and branch structures to be preserved under coordinate changes. Alternatively one can get a working definition from how the perturbation is constructed in the genericity trick \ref{GENERICITYTRICK} and the proof of theorem \ref{CITPERTURBATION} below.
\end{definition}

\begin{definition}[lifting a perturbation] In a level-1 good coordinate system $\mathcal{G}'$, let $(\beta,\alpha)\in I(\mathcal{G}')$, $$s_\alpha|_{W'_{\alpha\beta}}/\tilde E_{\alpha\beta}: W'_{\alpha\beta}\to E_\alpha|_{W'_{\alpha\beta}}/\tilde E_{\alpha\beta}$$ has a well-defined linearization over the points in its zero set $\phi_{\alpha\beta}(U'_{\alpha\beta})$ and is transverse over the entire $W'_{\alpha\beta}$ by the definition. Using this structure, we can lift a perturbation $\tau$ of $s_\beta|_{U'_{\alpha\beta}}$ to a perturbation of $s_\alpha|_{W'_{\alpha\beta}}$, $$(\pi_{\alpha\beta}, \tilde\pi_{\alpha\beta})^*\tau=(z\mapsto (\tilde\pi_{\alpha\beta}|_{(\tilde E_{\alpha\beta})_z})^{-1}(\hat\phi_{\alpha\beta}(\tau((\phi_{\alpha\beta}^{-1}|_{\phi_{\alpha\beta}(U'_{\alpha\beta})})(\pi_{\alpha\beta}(z))))).$$
\end{definition}

In other words, this lifted perturbation $$(\pi_{\alpha\beta},\tilde\pi_{\alpha\beta})^\ast\tau: W'_{\alpha\beta}\to \tilde E_{\alpha\beta}\subset E_\alpha|_{W'_{\alpha\beta}}$$ is defined by first looking at the image of $\tau$ under $\hat\phi_{\alpha\beta}$, then extend it to the entire $W'_{\alpha\beta}$ by using the fiberwise identification $\tilde\pi_{\alpha\beta}$. Recall that $\tilde\pi_{\alpha\beta}$ identifies a vector in the fiber $(\tilde E_{\alpha\beta})_w=(\hat\phi_{\alpha\beta}(E_\beta|_{U'_{\alpha\beta}}))_w, w\in \phi_{\alpha\beta}(U'_{\alpha\beta})$ to a vector in the fiber $(\tilde E_{\alpha\beta})_z$ for all $z\in (\pi_{\alpha\beta})^{-1}(w)$.

Notice that $(\pi_{\alpha\beta},\tilde\pi_{\alpha\beta})^\ast(s_\beta|_{U'_{\alpha\beta}})=s_\alpha|_{W'_{\alpha\beta}}$.

\begin{lemma}\label{LIFTINGPERT} If $s_\beta|_{U'_{\alpha\beta}}+\tau$ is transverse, then $$(\pi_{\alpha\beta},\tilde\pi_{\alpha\beta})^*(s_\beta|_{U'_{\alpha\beta}}+\tau)=s_\alpha|_{W'_{\alpha\beta}}+(\pi_{\alpha\beta},\tilde\pi_{\alpha\beta})^*\tau$$ is automatically transverse over the \emph{perturbed zeros}.
\end{lemma}
\begin{proof} By the definition of the level-1 structure.
\end{proof}

\begin{lemma}[genericity trick, \cite{PolyfoldII}]\label{GENERICITYTRICK} Let $E\to U$ be a finite dimensional vector bundle ep-groupoid, and let $s: U\to E$ be a section (where the orbit space of the zero set $\underline{s^{-1}(0)}$ is not necessarily compact). For $i=1, 2$, let $U'_i$ be an invariant open subset of $U$ such that $\overline{U'_i}$ is compact in $U$, and $U'_1\subset U'_2$. Let $O$ be an invariant (possibly empty) open subset of $U$ such that $s^{-1}(0)\subset U'_2\bigcup O$ and, letting $\tilde W'_1:=U'_2\bigcup ((\overline{U'_2}\backslash U'_2)\bigcap (\overline{U'_1}\bigcup O))$, $s^{-1}(0)\cap \tilde W'_1$ is compact in $\tilde W'_1$.

Suppose $s|_{\overline{U'_1}}$ is transverse, then we can choose an invariant multisectional perturbation $\tau$ with its support being compact inside $(U'_2\bigcup O)\backslash \overline{U'_1}$ such that $s+\tau$ is transverse over $\overline{U'_2}$. The same statement also holds when $s$ is a multisection.
\end{lemma}

\begin{proof}
Since $s$ is transverse over a compact invariant $\overline{U'_1}$ hence it is transverse over an open invariant neighborhood $V_1\subset U$ containing $\overline{U'_1}$. For each $z\in s^{-1}(0)\bigcap \overline{U'_2\backslash V_1}$, we choose local non-invariant sections $t_{(z, i)}, i\leq m_z$ compactly supported aroud $z$ inside $(U'_2\bigcup O)\backslash\overline{U'_1}$, such that the collection $t_{(z,i)}$ is invariant under the morphisms, and $s, t_{(z,i)}, i\leq m_z$ span $E_z$ hence span for $E|_{V_z}$ for some open set $V_z\subset (U'_2\bigcup O)\backslash\overline{U'_1}$ around $z$. Since $s^{-1}(0)\bigcap\overline{U'_2\backslash V_1}$ compact, we can choose finite many $z_1, \cdots, z_n$ such that $s^{-1}(0)\bigcap\overline{U'_2\backslash V_1}\subset \bigcup_{j\leq n} V_{z_j}$. Then define $$s+a\cdot 
t: U\times \oplus_{i\leq n}\R^{\oplus m_{z_j}}\to E,$$ $$(x, (a_{(z_j, i)})_{ i\leq m_{z_j}, j\leq n})\mapsto s(x)+\sum_{j\leq n}\sum_{i\leq m_{z_j}} a_{(z_j, i)} t_{(z_j, i)}(x).$$

There are morphisms acting on $\oplus_{i\leq n}\R^{\oplus m_{z_j}}$ induced by the morphisms acting on the collection of the local sections. Notice that $s+a\cdot t$ is not equivariant but its symmetrization under the morphisms is. $s+a\cdot t$ is transverse over $(s|_{\overline{U'_2}})^{-1}(0)\times\{0\}$, so it is transverse over the set of the form $\overline{V_2}\times B_\epsilon$, where $V_2$ is an invariant open subset of $U$ containing $\overline{U'_2}$, and $B_\epsilon$ is an $\epsilon$-ball centered at $0$ in $\oplus_{i\leq n}\R^{\oplus m_{z_j}}$.  $Q:=((s+a\cdot t)|_{\overline{V_2}\times B_\epsilon})^{-1}(0)$ is a manifold (which in general is not invariant and possibly has a non-regular boundary). Define the projection $\text{pr}_a: \overline{V_2}\times B_\epsilon\to B_\epsilon, (x,a)\mapsto a$. Then $\text{pr}_\text{a}|_Q: Q\to B_{\epsilon}$ is a proper map and its regular values $a$ can be chosen to be as close to 0 as possible but in general not invariant under the morphisms. Any of its regular value has the property that the symmetrization of $s+a\cdot t$ under the morphisms, which equals to $s$ added with the symmetrization $\text{symm}(a\cdot t)$ of $a\cdot t$, is an invariant transverse multisection over $\overline{U'_2}$. Here $\text{symm}(a\cdot t)$ is a rational linear combination of the local sections $t_{(z_j, i)}, i\leq m_{z_j}, j\leq n$ and the desired perturbation $\tau$.
\end{proof}

\begin{theorem}\label{CITPERTURBATION} Let $\mathcal{G}':=\{C_x|_{U'_x}\}_{x\in S}$ be a level-1 good coordinate system for an oriented strongly intersecting Hausdorff $\mathcal{G}$. Then $\mathcal{G}$ admits an invariant compact transverse multisectional perturbation.
\end{theorem}

\begin{proof}
\begin{enumerate}
\item We can assume that $\text{dim} U'_x-\text{rank} E_x=2\text{dim} U'_x- \text{dim} E_x$ is constant for all $x\in S$ without the loss of generality. (Otherwise, we just repeat the following for each connected component. See \ref{CONNECTEDNESSREMARK}.)

\item Choose a total order partition $K$ of $S$ as follows: 

In the notation, we have used abstract total order partitions to perform inductive constructions so far, an example of which for the case of a single good coordinate system can always be the partition by the bundle dimension. Therefore, an index in a total order partition for a connected (component of) good coordinate system also indexes the bases with the same dimension. If we explicitly use the total order partition by the bundle dimension, we have to increase the dimension of base by 2 each time. With the current hypothsis of $\text{dim} U'_x-\text{rank} E_x$ being constant, we are in the case of either bundle dimensions are all odd or all even for the good coordinate system, and it would be also cumbersome to carry the multiple of 2 in the notation. We will use the partition by the bundle rank as a total order partition instead, which achieve the same effect for any discussion regarding a single good coordinate system.

Define $\alpha_i:=\{x\in S\;|\; \text{rank} E_x=i\}$. Let $$N:=\max\{{\text{rank} E_x\;|\; x\in S}\}.$$ $K:=\{\alpha_1, \cdots, \alpha_N\}$. Here, any other total order partition also works. 

We organize $C_x|_{U'_x}, x\in\alpha_i$  into $C_{\alpha_i}|_{U'_{\alpha_i}}$ as usual, and denote the latter by $C_i|_{U'_i}$ for short. Then by item 1 and the assumption, $U'_i$ is an ep-groupoid of a fixed dimension (where the Hausdorffness is a part of the definition of an ep-groupoid). We can also index $\mathcal{G}$ in the same way. If $\alpha_i=\emptyset$, then the $i$th step below is vacuously true. 

\item Precompactly shrink while preserving the strong open neighborhood property:

We take a precompact shrinking $\{U''_i\}_{i\leq N}$ of $\{U'_i\}_{i\leq N}$, where we use the previous trick to ensure $(U''_i, U'_i, W''_{ij})$ is an invariant strong open neighborhood of $\phi_{ij}(U''_{ij})$ for all $j<i$, $W''_{ij}$ is a fiberwisely precompact fiber shrinking of $W'_{ij}|_{\phi_{ij}(U''_{ij})}$, and the double-primed level-1 structures are compatible. Let $U^m_{ij}$ denote the domain of the coordinate change $C_j|_{U'_j}\to C_i|_{U''_i}$. We use the \textbf{convention} that if there is no coordinate change between the $j$-chart and the $i$-chart, the corresponding notation for the level-1 structure is omitted.

\item The special case $C_0|_{U'_0}$:

We start with $C_0|_{U'_0}$. The section $s_0|_{U'_0}$ is already transverse. So we do not need to do peturbation and just take the restriction $s_0|_{U''_0}$.

\item The initial step of the induction by the bundle rank is to perturb $C_1|_{U'_1}$:

$s_1$ is already transverse over $\overline{W''_{10}}$ with the closure taken in $U'_0$. We apply lemma \ref{GENERICITYTRICK} for the case $$(s: U\to E, U'_1, U'_2, O):=(s_1: U'_1\to E_1|_{U'_1}, W''_{10}, U''_1, \bigcup_{2\leq j\leq N}U^m_{j1}),$$
so we can find a multisectional perturbation $\tau_1$ compactly supported in $(U''_1\cup\bigcup_{2\leq j\leq N} U^m_{j1})\backslash \overline{W''_{10}}$ such that $s_1+\tau_1$ that is defined over $U'_1$ is transverse over $\overline{U''_1}$.  We fix $(s_1+\tau_1)|_{U''_1}$ as the perturbed section chosen in this initial step.

\item The simplest non-trivial inductive step:

Since $\tau_1$ is defined on $U'_1\backslash\overline{W''_{10}}$, $(\pi_{21}, \tilde\pi_{21})^\ast \tau_1$ is defined on $W'_{21}\backslash \overline{W''_{20}}$, and hence $s_2+(\pi_{21}, \tilde\pi_{21})^\ast \tau_1$ is defined on $U'_2$.

We have that $s_1+\tau_1$ is transverse on $\overline{U''_1}$ hence on $\overline{U''_{21}}$. By lemma \ref{LIFTINGPERT}, $s_2+(\pi_{21}, \tilde\pi_{21})^\ast \tau_1=(\pi_{21},\tilde\pi_{21})^\ast(s_1+\tau_1)$ is transverse on $\overline{W''_{21}}$. 

On $\overline{W''_{20}}$, $s_2+(\pi_{21}, \tilde\pi_{21})^\ast \tau_1=s_2=(\pi_{20},\tilde\pi_{20})^\ast s_0$ is transverse, by lemma \ref{LIFTINGPERT} and item (4).

We thus have that $s_2+(\pi_{21}, \tilde\pi_{21})^\ast \tau_1$ is transverse on $\overline{W''_{20}\cup W''_{21}}$.

By applying lemma \ref{GENERICITYTRICK} for the case $$(s: U\to E, U'_1, U'_2, O):=$$ $$(s_2+(\pi_{21}, \tilde\pi_{21})^\ast \tau_1: U'_2\to E_2|_{U'_2}, W''_{20}\cup W''_{21}, U''_2, \bigcup_{3\leq j\leq N}U^m_{j2}),$$
we can find a multisectional perturbation $\tau_2$ compactly supported in $$(U''_2\cup\bigcup_{3\leq j\leq N} U^m_{j2}))\backslash \overline{W''_{20}\cup W''_{21}}$$ such that $s_2+(\pi_{21}, \tilde\pi_{21})^\ast \tau_1+\tau_2$ is transverse on $\overline{U''_2}$. We fix $$(s_2+(\pi_{21}, \tilde\pi_{21})^\ast \tau_1+\tau_2)|_{U''_2}$$ as the perturbed section chosen in this first inductive step.

\item The general $i$th inductive step:

For $k\leq i-1$, we have constructed $\tau_{k}$, which is defined on $$U'_{k}\backslash\overline{\bigcup_{j\leq k-1}W''_{k j}}.$$ So, $(\pi_{i k}, \tilde\pi_{i k})^\ast \tau_{k}$ is defined on $$W'_{i k}\backslash \overline{\bigcup_{j\leq k-1}W''_{i j}},$$ and hence $(\pi_{i k}, \tilde\pi_{i k})^\ast \tau_{k}$ is defined on $U'_i$. $s_i+\sum_{k\leq i-1}(\pi_{i k}, \tilde\pi_{i k})^\ast \tau_{k}$ is defined on $U'_i$.

For $k$ such that $0\leq k\leq i-1$, we have that $$s_{k}+\sum_{j\leq k-1}(\pi_{k j}, \tilde\pi_{k j})^\ast \tau_j+\tau_{k}$$ is transverse on $\overline{U''_{k}}$, hence on $\overline{U''_{i k}}$. So, by lemma \ref{LIFTINGPERT},
\begin{align*}&\;\;\;\;\;(s_i+\sum_{j\leq i-1}(\pi_{i j}, \tilde\pi_{i j})^\ast \tau_{j})|_{\overline{W''_{ik}}}\\
&=(s_i+\sum_{j\leq k}(\pi_{i j}, \tilde\pi_{i j})^\ast \tau_{j})|_{\overline{W''_{ik}}}\\
&=(\pi_{i k}, \tilde\pi_{i k})^\ast(s_k+\sum_{j\leq k-1}(\pi_{i j}, \tilde\pi_{i j})^\ast \tau_{j}+\tau_k)|_{\overline{W''_{ik}}}
\end{align*}
is transverse on $\overline{W''_{ik}}$. 

Thus, $s_i+\sum_{j\leq i-1}(\pi_{i j}, \tilde\pi_{i j})^\ast \tau_{j}$ is transverse on $\overline{\bigcup_{0\leq k\leq i-1} W''_{i k}}$.

Invoking lemma \ref{GENERICITYTRICK} for the case $(s: U\to E, U'_1, U'_2, O):=$ $$(s_i+\sum_{j\leq i-1}(\pi_{i j}, \tilde\pi_{i j})^\ast \tau_{j}: U'_i\to E_i|_{U'_i}, \bigcup_{0\leq k\leq i-1} W''_{i k}, U''_i, \bigcup_{i+1\leq n\leq N}U^m_{ni}),$$
we can find a multisectional perturbation $\tau_i$ compactly supported in $$(U''_i\cup\bigcup_{i+1\leq n\leq N}U^m_{ni})\backslash \overline{\bigcup_{0\leq k\leq i-1} W''_{i k}}$$ such that $s_i+\sum_{j\leq i-1}(\pi_{i j}, \tilde\pi_{i j})^\ast \tau_{j}+\tau_i$ is transverse on $\overline{U''_i}$. We fix $$(s_i+\sum_{j\leq i-1}(\pi_{i j}, \tilde\pi_{i j})^\ast \tau_{j}+\tau_i)|_{U''_i}$$ as the perturbed section chosen in this $i$th inductive step. We have used the strong open neighborhood properties and the compatibility of level-1 structures in the above.

\item After the $N$th step, we have completed the construction by induction. Note that in the $N$th step, $O=\emptyset$. 

\item Notice that the union of the solution sets of the perturbed sections in the identification space lies inside the union of the supports of perturbation $\tau_i, 1\leq i\leq N$, which is in turn compactly included in $(\bigcup_{0\leq 1\leq N}\underline{U''_i})/\sim=M(\mathcal{G}'')$ by the construction. Therefore, we have constructed a compact invariant transverse multisectional perturbation of the level-1 good coordinate system $\mathcal{G}''$.
\end{enumerate}
\end{proof}

\begin{remark} Observe that by using a level-1 good coordinate system and perturbing in the above manner, there is no issue of zeros leaking through the dimension-jumping region while perturbing the section in the target chart of a coordinate change. The above argument also ensures the compactness of the \emph{perturbation} of a good coordinate system, which is the collection of the perturbed sections.
\end{remark}

\section{Tripling process and fiber products}
\subsection{Tripling process}

In this section, we will describe a unique way to chart-refine a strongly intersecting Hausdorff level-1 good coordinate system ordered by a total order to address the following two scenarios.

\begin{enumerate}[A.]
\item We have explained how to make a concerted Kuranishi embedding into level-1 mudulo some precompact shrinking of the domain good coordinate system. A tripling provides a way to make any Kuranishi embedding $\mathcal{G}\to\tilde{\mathcal{G}}$ into a concerted Kuranishi embedding, up to chart-refinement. Namely, we chart-refine $\mathcal{G}$ and $\tilde{\mathcal{G}}$ to resolve conflicting dimension jumpings in coordinate changes between intersecting charts indexed by $\beta$ and $\alpha$ in $\mathcal{G}$ versus $\tilde{\mathcal{G}}$. This solution will be found in subsection \ref{ADMISSIBLETRIPLING}.
\item Suppose that we have a pair of concerted level-1 Kuranishi embeddings $\mathcal{G}\to \mathcal{G}^1$ and $\mathcal{G}\to \mathcal{G}^2$ indexed by $S$ with the same order $(S, \leq)$. The level-1 structures in the chart embeddings of Kuranishi embeddings allow us to form a fiber product Kuranishi chart $(C^1_\alpha|_{U^1_\alpha})_{\Pi^1_\alpha}\times_{C_\alpha|_{U_\alpha},\; \Pi^2_\alpha}(C^2_\alpha|_{U^2_\alpha})$ for each $\alpha\in S$. Let $\beta,\alpha\in S$ distinct with $\beta\leq \alpha$. Then $\text{dim} E_\beta\leq \text{dim} E_\alpha$, $\text{dim} E^1_\beta\leq \text{dim} E^1_\alpha$ and $\text{dim} E^2_\beta\leq \text{dim} E^2_\alpha$. However, the bundle dimension in the fiber product chart indexed by $\alpha$ can be less than that indexed by $\beta$, hence there is no possible coordinate change from the fiber product chart at $\beta$ to that at $\alpha$ if the fiber product charts are defined this way. The tripling process provides a way to form a fiber product good coordinate system, explained in subsection \ref{FIBPROD}. Here we chart-refine $\mathcal{G}$ so that charts are indexed by certain subsets $T$ of $S$, and two charts indexed by $T$ and $T'$ intersect if and only if one of $T$ and $T'$ is included in the other, and in this setting we can always construct the fiber product charts with naturally induced coordinate changes among them.
\end{enumerate}

\begin{definition}[tripling]\label{TRIPLING} Let $\mathcal{G}:=\{C_\alpha|_{U_\alpha}\}_{\alpha\in S}$ be a strongly intersecting Hausdorff level-1 good coordinate system indexed by a total order $(S, \leq)$ for some other good coordinate system $\hat{\mathcal{G}}:=\{C_\alpha|_{\hat U_\alpha}\}_{\alpha\in S}$. Here for the notational simplicity, in changing from ``level-1 $\mathcal{G}'$ for $\mathcal{G}$'' to ``level-1 $\mathcal{G}$ for $\hat{\mathcal{G}}$'', all primes in the level-1 structures are removed. Recall \ref{LEVEL0GCS}(2)(b) in the general definition of a good coordinate system, we ask the order to be compatible with the order by the bundle dimension.

Denote $I^\ast (\mathcal{G}):=\{(\beta,\alpha)\in S\times S\;|\; \beta\not=\alpha\})\cap I(\mathcal{G})$.

For any $(\beta,\alpha)\in I^\ast(\mathcal{G})$, denote $U_{(\beta,\alpha)}^3:=\text{orbit}(W_{\alpha\beta})$ appeared in the level-1 coordinate change for $(\beta,\alpha)$, and choose invariant shrinkings $U^1_{(\beta,\alpha)}$ and $U^2_{(\beta,\alpha)}$ of $U_\beta$ and $U_\alpha$ (see figure \ref{figure_tripling}) such that 
\begin{enumerate}[(a)]
\item $U_\beta\backslash U_{\alpha\beta}\subset U^1_{(\beta,\alpha)}$,
\item $U_\alpha\backslash \text{closure}^{\text{fiber}}_{\hat U_\alpha}({U^3_{(\beta,\alpha)}})\subset U^2_{(\beta,\alpha)}$, where $\text{closure}^{\text{fiber}}_{\hat U_\alpha}({U^3_{(\beta,\alpha)}})$ denotes the union of the orbits of the fiber closure of $W_{\alpha\beta}$ under the morphisms in $\hat U_\alpha$ (note that if $\text{dim} E_\beta=\text{dim} E_\alpha$, $\text{closure}^{\text{fiber}}_{\hat U_\alpha}({U^3_{(\beta,\alpha)}}):=\text{orbit}(\phi_{\alpha\beta}(U_{\alpha\beta}))$),
\item $(U^2_{(\beta,\alpha)}, \hat U_{\alpha}, W_{\alpha\beta}\cap U^2_{(\beta,\alpha)})$ is an invariant strong open neighborhood of\\
$\phi_{\alpha\beta}(U_{\alpha\beta})\cap U^2_{(\beta,\alpha)}$\footnote{Just pick an invariant open subset $V^1_{(\beta,\alpha)}\subset U_{\alpha\beta}$ disjoint from $U^1_{(\beta,\alpha)}$ in $U_\beta$ such that $$((\pi_{\alpha\beta})^{-1}(\phi_{\alpha\beta}(V^1_{(\beta,\alpha)})))\cup(U_\alpha\backslash \text{closure}^{\text{fiber}}_{\hat U_\alpha}({U^3_{(\beta,\alpha)}}))$$ is an open neighborhood of $U_\alpha\backslash \text{closure}^{\text{fiber}}_{\hat U_\alpha}({U^3_{(\beta,\alpha)}})$.}, and
\item the domain of the restricted coordinate change $C_\beta|_{U^1_{(\beta,\alpha)}}\to C_\alpha|_{U^2_{(\beta,\alpha)}}$ is empty.
\end{enumerate}

\begin{figure}[htb]
  \begin{center}

\begin{tikzpicture}[scale=3/5]
\hspace{0.2 cm}

\filldraw [black!20] (2.5,8.5) rectangle (7,11.5);
\filldraw [black!40] (2.5,9.2) rectangle (4,10.8);

\draw (0,10)--(4,10);

\begin{scope}[shift={(9,0)}]
\filldraw [black!20] (2.5,8.5) rectangle (7,11.5);
\filldraw [yellow!60] (2.5,9.2) rectangle (4,10.8);

\draw[very thick, dashed, blue!90] (2.5,9.2)--(3.7, 9.2)--(3.7,10.8)--(2.5,10.8)--(2.5,11.5)--(7,11.5)--(7,8.5)
--(2.5,8.5)--(2.5,9.2);

\draw (0,10)--(4,10);
\draw[very thick, red!90] (0,10.03)--(3,10.03);
\end{scope}

\filldraw [black!20]  (1,3) circle (2 and 1);
\filldraw [black!20]  (4.5,3) circle (3 and 2.5);

\draw [very thick, dotted, black!60]  (1,3) circle (2 and 1);
\draw [very thick, dotted, black!60]  (4.5,3) circle (3 and 2.5);

\filldraw [black!20]  (11,3) circle (2 and 1);
\filldraw [black!20]  (14.5,3) circle (3 and 2.5);

\draw [very thick, dotted, black!60]  (11,3) circle (2 and 1);
\draw [very thick, dotted, black!60]  (14.5,3) circle (3 and 2.5);

\begin{scope}[shift={(10,0)}]
\clip (1,3) circle (2 and 1);
\clip (4.5,3) circle (3 and 2.5);
\filldraw [yellow!60] (1,3) circle (2 and 1) (4.5,3) circle (3 and 2.5);

\end{scope}

\begin{scope}
\clip (9, 1.5) rectangle (11.8,4.5);

\draw [very thick, dashed, red!90]  (11,3) circle (2 and 1);

\end{scope}

\draw  [very thick, dashed, red!90]  (11.7, 2.1)-- (11.7,4);

\draw [very thick, dotted, black!60]  (14.5,3) circle (3 and 2.5);

\begin{scope}
\clip (11.8, 0) rectangle (18,6);

\draw [very thick, dashed, blue!90]  (14.5,3) circle (3 and 2.5);

\end{scope}

\draw  [very thick, dashed, blue!90]  (11.7, 2.1) ..controls (12.5, 2.1) and (12.5,3.9) .. (11.7,3.9);

\node at (0,7){$U_\beta$};
\path[->] (0,6.5)edge [bend right] (0,3);
\node at (4,7){$U_\alpha$};
\path[->] (4,6.5)edge [bend left] (5,4);

\node at (9,7)[red]{$U_{(\beta,\alpha)}^1$};
\node at (9,6.25)[red]{(open region enclosed};
\node at (9,5.5)[red]{by red contour)};
\path[->] (10,5)edge [bend right] (11,3);

\node at (16,7.5)[blue]{$U_{(\beta,\alpha)}^2$};
\node at (16,6.75)[blue]{(open region enclosed};
\node at (16,6)[blue]{by blue contour)};
\path[->] (17,5.5)edge [bend left] (15,3);

\node at (13,-1)[yellow]{$U_{(\beta,\alpha)}^3=W_{\alpha\beta}$};
\node at (13,-1.75)[yellow]{(open region shaded};
\node at (13,-2.5)[yellow]{in green)};
\path[->] (13,-.5)edge [bend right] (12.5,3);

\node at (8,-3.25){In the second example, one can imagine another dimension};
\node at (8,-4){of $U_\alpha$ perpendicular to the page.};

\end{tikzpicture}

  \end{center}
\caption[A notion in the definition of tripling.]{A notion in the definition of tripling.}
\label{figure_tripling}
\end{figure}

For $\alpha\in S$, $T\subset S$ nonempty, denote $I(T,\alpha):=\{\beta\in T\;|\; (\beta,\alpha)\in I^\ast (\mathcal{G})\}$ and $I(\alpha, T):=\{\beta\in T\;|\; (\alpha,\beta)\in I^\ast (\mathcal{G})\}$.

For a nonempty subset $T\subset S$ with $I(T,\max T)=T\backslash \{\max T\}$, the object $\bigcap_{\beta\in T\backslash\{\max T\}} U^3_{(\beta,\max T)}$ is defined\footnote{Observe that if $\bigcap_{\beta\in T\backslash\{\max T\}} U^3_{(\beta,\max T)}\not=\emptyset$, then we also have $I(\min T, T)=T\backslash\{\min T\}$.}.

Denote $\mathcal{S}_S:=$ $$\{T\subset S\;|\; T\not=\emptyset, I(T, \max T)=T\backslash\{\max T\},\;\;\bigcap_{\beta\in T\backslash\{\max T\}} U^3_{(\beta,\max T)}\not=\emptyset\}.$$

Now let $T\in \mathcal{S}_S$, we define $U_T:=$

$$\left (\bigcap_{\gamma\in I(S\backslash T, \max T)}U^2_{(\gamma,\max T)}\cap \bigcap_{\alpha\in I(\max T, S\backslash T) }U^1_{(\max T,\alpha)}\right)\cap\bigcap_{\beta\in T\backslash \{\max T\}} $$
$$\left(U^3_{(\beta,\max T)}\cap (\pi_{\max T \beta})^{-1}\left((\bigcap_{\delta\in I(S\backslash T, \beta)}U^2_{(\delta,\beta)})\cap(\bigcap_{\epsilon\in I(\beta, S\backslash T)} U^1_{(\beta, \epsilon)})\right)\right),$$
which is a nonempty invariant open subset of $U_{\max T}$.

\begin{remark} Because of the way in which $U^2_{(\beta,\alpha)}$'s are constructed and the compatibility of level-1 coordinate changes, we have $$U_T:=(\pi_{\max T\beta})^{-1}(\pi_{\max T\beta}(U_T))$$ for all $\beta\in T\backslash \{\max T\}$. Namely, $(U_T, \hat U_{\max T}, U_T)$ is a strong open neighborhood of $\pi_{\max T\beta}(U_T)$ for all $\beta\in T\backslash \{\max T\}$.
\end{remark}

The induced order $\leq$ on $\mathcal{S}_S$ is that $T_1\leq T_2$ if and only if $\max T_1\leq \max T_2$.

\begin{lemma}\label{INTERSECTISINCLUDE} Let $T_1, T_2\in\mathcal{S}_S$. If $U_{T_1}$ and $U_{T_2}$ intersect in the identification space $M(\mathcal{G})$, then $T_1\subset T_2$ or $T_2\subset T_1$. 
\end{lemma}

\begin{proof} If $T_1\not\subset T_2$ and $T_2\not\subset T_1$, then $T_1\backslash T_2$ is nonempty and $T_2\backslash T_1$ is nonempty. Define $\alpha:=\max (T_1\backslash T_2)$ and $\beta:=\max (T_2\backslash T_1)$. Then $\alpha\not=\beta$. Since $S$ is a total order, and since the statement about $T_1$ and $T_2$ is symmetric, we can assume $\alpha\leq \beta$ without the loss of generality by switching $T_1$ and $T_2$ if necessary. Then $\beta\not\leq \alpha$ since $S$ is a total order.

Denote $V_1:=U^1_{(\alpha,\beta)}$ if $\alpha=\max T_1$, or otherwise $V_1:=\pi_{\max T_1 \alpha}^{-1}(U^1_{(\alpha,\beta)})$. Denote $V_2:=U^2_{(\alpha,\beta)}$ if $\beta=\max T_2$, or otherwise $V_2:=\pi_{\max T_2 \beta}^{-1}(U^2_{(\alpha,\beta)})$. Then we have that $U_{T_1}\subset V_1$ and $U_{T_2}\subset V_2$, but $V_1$ and $V_2$ do not intersect in $M(\mathcal{G})$ by the property of $U^1_{(\alpha,\beta)}$ and $U^2_{(\alpha,\beta)}$. Therefore, $U_{T_1}$ and $U_{T_2}$ do not intersect in $M(\mathcal{G})$ and we have proved the contrapositive.
\end{proof}

The converse of lemma \ref{INTERSECTISINCLUDE} is true. Indeed, it is clear that if $T_1\subset T_2$, then $U_{T_1}$ intersect with $U_{T_2}$ in $M(\mathcal{G})$.

Let $T_1, T_2\in \mathcal{S}_S$ with $\max T_1\leq \max T_2$. Suppose that $U_{T_1}$ and $U_{T_2}$ intersect in the identification space $M(\mathcal{G})$ (by lemma \ref{INTERSECTISINCLUDE}, we have $T_1\subset T_2$ or we have $T_2\subset T_1$ with $\max T_2=\max T_1$). By the strong intersecting property, we have $(\max T_1, \max T_2)\in I(\mathcal{G})$, then we have a coordinate change $$C_{\max T_1}|_{U_{T_1}}\to C_{\max T_2}|_{U_{T_2}}$$ with the domain of the coordinate change $U_{T_2 T_1}$ determined from the restriction (or equivalently from the maximality condition). Denote $$I(\mathcal{G}_S):=\{(T_1, T_2)\subset \mathcal{S}_S\times \mathcal{S}_S\;|\; T_1\leq T_2, U_{T_1}\text{ and }U_{T_2}\text{ intersect in }M(\mathcal{G})\}.$$

Denote the strongly intersecting good coordinate system $$\mathcal{G}_S:=\{(\mathcal{S}_S, \leq), \{C_{\max T}|_{U_T}\}_{T\in\mathcal{S}_S}, \{C_{\max T_1}|_{U_{T_1}}\to C_{\max T_2}|_{U_{T_2}}\}_{(T_1, T_2)\in I(\mathcal{G}_S)}\},$$
and $\mathcal{G}_S$ is automatically a level-1 good coordinate system induced from $\mathcal{G}$.
\emph{Tripling} is the canonical process of replacing a strongly intersecting level-1 good coordinate system $\mathcal{G}$ indexed by a total order $S$ with $\mathcal{G}_S$ indexed by an order $\mathcal{S}_S$ (where ($\mathcal{S}_S, \leq)$ is not antisymmetric in general).
\end{definition}

\begin{remark} Observe that $(\mathcal{G}_S, (\mathcal{S}_S, \leq))$ chart-refines $(\mathcal{G}, (S, \leq))$. In fact, more is true. We can choose a total order partition of $\mathcal{S}_S$ which is isomorphic to $S$ as follows: If $\alpha\in S$, define $\mathcal{S}_\alpha:=\{T\subset \mathcal{S}_S\;|\; \max T=\alpha\}$. If we regroup Kuranishi charts in $\mathcal{G}_S$ by this partition, namely $\tilde C_\alpha:=C_\alpha|_{\bigcup_{T\in \mathcal{S}_\alpha} U_T}$, we arrive at a shrinking $\{\tilde C_\alpha\}_{\alpha\in S}$ of the original $\mathcal{G}$. Notice that here we have taken the union instead of disjoint union for it to a shrinking.
\end{remark}

\begin{remark} We could have defined the tripling inductively by $(S,\leq)$, where it might be easier to visualize but harder notation-wise. The point is that all data can be chosen simultaneously, and after the choice is made the resulting tripling is unambiguously defined immediately.
\end{remark}
\subsection{Agreement of the order and the admissibility}\label{ADMISSIBLETRIPLING}

In this section, we show how to construct a concerted chart-refinement of a general Kuranishi embedding. We also introduce the notion of an admissible pair of concerted (level-1) Kuranishi embeddings, as well as preparing the setting for forming a fiber product in the next subsection.

Readers should review the following notions before continuing: a Kuranishi embedding \ref{KEMB}, the concertedness \ref{CONCERTEDNESS}, a chart-refinement of a good coordinate system \ref{CHARTREFINEMENT}, a refinement of a good coordinate system \ref{REFINEMENTMAP}, a chart-refinement of a Kuranishi embedding \ref{CROFKEMB}, and a general embedding \ref{GEMB}.

We have also defined a level-1 Kuranishi embedding \ref{LEVEL1KEMBED}. Now we define the notion of a level-1 chart-refinement:

\begin{definition}[level-1 chart-refinement]\label{LEVEL1CRMENT} Let $\mathcal{G}=(X, (S, \leq), \{C_x|_{U_x}\}_{x\in S})$ and $\mathcal{G}'=(X, (S', \leq'), \{C'_x|_{U'_x}\}_{x\in S'})$ be two strongly intersecting Hausdorff level-1 good coordinate systems for $X$, indexed by $(S, \leq)$ and $(S, \leq')$ respectively. $\mathcal{G}'$ is said to level-1 chart-refine $\mathcal{G}$ if there is a map $\mu: S'\to S$ such that if $\beta'\leq'\alpha'$, then $\mu(\beta')\leq \mu(\alpha')$, and for each $\alpha'\in S'$, we have an open chart embedding $C'_{\alpha'}|_{U'_{\alpha'}}\to C_{\mu(\alpha')}|_{U_{\mu(\alpha')}}$ with the domain of the coordinate change being the entire $U'_{\alpha'}$, such that for all $(\beta', \alpha')\in I(\mathcal{G}')$
$$\begin{CD}
C_{\mu(\beta')}|_{U_{\mu(\beta')}}@>\text{level-1}>> C_{\mu(\alpha')}|_{U_{\mu(\alpha')}}\\
@AAA @AAA\\
C'_{\beta'}|_{U'_{\beta'}} @>\text{level-1}>> C'_{\alpha'}|_{U'_{\alpha'}}
\end{CD}$$ is level-1 commutative up to the morphisms (where the level-1 structures for the vertical chart-embedding are trivial).
\end{definition}

\begin{example} The tripling $\mathcal{G}_S$ is a level-1 chart-refinement of $\mathcal{G}$.
\end{example}

\begin{remark} In the discussions below, all the level-1 chart-refinements are inclusions of the restrictions.
\end{remark}

\begin{proposition}[concertedness trick]\label{CONCERTEDPROOF} Let $\mathcal{G}\Rightarrow\tilde{\mathcal{G}}$ be a Kuranishi embedding between strongly intersecting Hausdorff good coordinate systems indexed by $(\hat S, \hat\leq)$ and $(\hat S,\hat\leq_1)$. Then we can find a concerted Kuranishi embedding which chart-refines $\mathcal{G}\Rightarrow\tilde{\mathcal{G}}$.
\end{proposition}

\begin{proof} Concertedness fails when there exist squares of the following form:
$$\begin{CD}
\tilde C_\beta|_{\tilde U_\beta}@<<<\tilde C_\alpha|_{\tilde U_\alpha}\\
@AAA @AAA\\
C_\beta|_{U_\beta} @>>>C_\alpha|_{U_\alpha}
\end{CD}, $$
where $(\alpha,\beta)\not\in I(\mathcal{G})$ and $(\beta,\alpha)\not\in I(\tilde{\mathcal{G}})$. Namely, we cannot change the direction of one of the horizontal arrows to make them agree.

We define a partition $S$ of $\hat S$ by the double bundle dimensions as in \ref{LEVEL1KEMBED}, but choose a total order $\leq$ on $S$ by hand, which is compatible with the bundle dimension in $\mathcal{G}$. By \ref{EXLEVELONEGCS}, we obtain a level-1 $\mathcal{G}'$ for $\mathcal{G}$ such that the underlying good coordinate system of $\mathcal{G}'$ is a precompact shrinking for $\mathcal{G}$.

We then apply definition \ref{TRIPLING} to $\mathcal{G}'$, so we have a level-1 good coordinate system $\mathcal{G}'_S$. Then we consider the composition of $\mathcal{G}'_S\to \mathcal{G}'\to\mathcal{G}\to \tilde{\mathcal{G}}$. For every $T\in \mathcal{S}_S$, we can find $\alpha\in S$ such that there exists a Kuranishi chart embedding $C_{\max T}|_{U'_{T}}\to \tilde C_\alpha|_{\tilde U_\alpha}$. Indeed, just choose $\alpha:=\max T$. Since $\mathcal{S}_S$ is finite, we can choose an invariant open subset $\tilde U_T$ of $\tilde U_{\max T}$ for every $T\in \mathcal{S}_S$ such that $\phi_{\max T}(U'_T)\subset \tilde U_T$, and if $U'_{T_1}$ and $U'_{T_2}$ do not intersect in $M(\mathcal{G}')$, then $\tilde U_{T_1}$ and $\tilde U_{T_2}$ do not intersect in $M(\tilde{\mathcal{G}})$.

Observe that $\max T$ is not necessarily the $\widetilde{\max} T$ defined using the order $\tilde\leq$ by the bundle dimension in $\tilde {\mathcal{G}}$. For each $T$ with $\widetilde \max T\not=\max T$, since $$\pi_{\max T \widetilde{\max} T}(U'_T)\subset U'_{\max T\widetilde\max T},$$ we have that $\phi_{\widetilde {\max} T}(U'_{\max T\widetilde\max T})\subset \tilde\phi_{\widetilde\max T \max T}(\tilde U_{\widetilde\max T \max T})\subset \tilde U_{\widetilde\max T}$. Hence $\phi_{\widetilde{\max} T}(\pi_{\max T \widetilde{\max} T}(U'_T))\subset \tilde U_{\widetilde \max T}$. 

Define $\tilde\leq$ on $\mathcal{S}_S$ by $T_1\tilde\leq T_2$ if and only if $\widetilde\max T_1\tilde\leq \widetilde\max T_2$. Denote $\mathcal{S}'_S:=\{T\in\mathcal{G}_S\;|\;\widetilde \max T\not=\max T\}$. We do the following by induction on $\mathcal{S}'_S$ using the order $\tilde\leq$. Fix a general $T\in \mathcal{S}'_S$, denote $J:=\{T'\in \mathcal{S}'_S\;|\; T'\leq T\}$. Suppose we have $U^J_K, K\in \mathcal{S}_S$ from the previous induction. Choose an invariant open subset $\tilde U^T_{T}$ of $\phi_{\widetilde{\max} T}(\pi_{\max T \widetilde{\max} T}(U'_T))$ in $\tilde U_{\widetilde \max T}$ such that for every $K\in \mathcal{S}_S\backslash \{T\}$, $\tilde U^T_T$ intersects with $\tilde U^J_{K}$ if and only if $\tilde U^J_T$ intersects with $\tilde U^J_{K}$. Let $\tilde U^T_{K}:=\tilde U^J_{K}$ for all $K\in \mathcal{S}_S\backslash\{T\}$, which completes the inductive step.

After finitely many steps, we complete the induction, and denote the result $\tilde U'_{T}, T\in \mathcal{S}_S$. Denote $\tilde{\mathcal{G}}'_S:=\{\tilde C_{\widetilde \max T}|_{\tilde U'_T}\}_{T\in\mathcal{S}_S}$ with the induced coordinate changes among the restricted charts, then it is a good coordinate system and we have $\tilde U'_T\subset \tilde U_{\widetilde\max T}$. Moreover, the Kuranishi embedding $\mathcal{G}'_S\Rightarrow \tilde{\mathcal{G}}'_S$ is naturally induced from the composed map $\mathcal{G}'_S\to\tilde{\mathcal{G}}$.

We claim that $\mathcal{G}'_S\Rightarrow \tilde{\mathcal{G}}'_S$ is concerted now. Indeed, suppose $T_1, T_2\in \mathcal{S}_S$ and $U'_{T_1}$ and $U'_{T_2}$ intersect, then by lemma \ref{INTERSECTISINCLUDE}, we have $T_1\subset T_2$ or $T_2\subset T_1$, or both. In the first case, $\max T_1\leq T_2$ and $\widetilde \max T_1\tilde\leq \widetilde\max T_2$, hence $T_1\leq T_2$ and $T_1\tilde \leq T_2$; similarly in the second case, both orders also agree. Hence, we achieved concertedness for the Kuranishi embedding $\mathcal{G}'_{S}\Rightarrow\tilde{\mathcal{G}}'_S$ which chart-refines $\mathcal{G}\Rightarrow\tilde{\mathcal{G}}$.
\end{proof}

\begin{proposition}\label{GROUPING}
Let $(\mathcal{G}, (\hat S, \hat\leq))\Rightarrow (\mathcal{G}^1, (\hat S, \hat\leq_1))$ be a concerted Kuranishi embedding (by the definition the orders on $\hat S$ is compatible with the bundle dimensions in the respective good coordinate systems), then we can choose a total order partition $(S,\leq)$ for both $(\hat S, \hat\leq)$ and $(\hat S, \hat\leq_1)$.
\end{proposition}

\begin{proof}(\ref{LEVEL1KEMBED}) Let $S_{i,j}:=\{\alpha\in S\;|\; \text{dim} (E_\alpha|_{U_\alpha})=i\;\text{dim} (E^1_{U^1_\alpha}|_{U^1_\alpha})=j\}$. We define $S_{i,j}\leq S_{i',j'}$ if $i+j\leq i'+j'$. By the concertedness, when $\underline{U_{S_{i,j}}}$ and $\underline{U_{S_{i',j'}}}$ intersect in the identification space, we must have $i\leq i'$ and $j\leq j'$, or $i'\leq i$ and $j'\leq j$, or both. They agree with the order by the bundle dimension sum. This $\leq$ on $\{S_{i,j}\}_{i,j}$ defines a total order on this partition of $S$. We can define grouped charts $C_{i,j}|_{U_{i,j}}:=C_{S_{i,j}}|_{U_{S_{i,j}}}$ as in \ref{REORGANIZEBYTOT}.
\end{proof}

The above results have prepared the hypothesis of the following theorem:

\begin{theorem}\label{LEVELONEKEMB} This theorem is just theorem \ref{EMBEDDING0TO1EXTENDINGLEVEL1GCS}:
\begin{enumerate}
\item Let $\mathcal{G}\Rightarrow\tilde{\mathcal{G}}$ be a concerted Kuranishi embedding indexed by the same total order $(S,\leq)$, then we can construct a level-1 Kuranishi embedding $\mathcal{G}'\overset{\text{level-1}}{\Rightarrow}\tilde{\mathcal{G}}'$ whose underlying Kuranishi embedding chart-refines $\mathcal{G}\Rightarrow\tilde{\mathcal{G}}$, where $\mathcal{G}'$ and $\tilde{\mathcal{G}}'$ are precompactly shrinkings of $\mathcal{G}$ and $\tilde{\mathcal{G}}$ respectively. 
\item In (1), we can pick a level-1 $\mathcal{G}^0$ for $\mathcal{G}$ first and require the level-1 $\mathcal{G}'$ in the conclusion to be a precompact shrinking of the level-1 $\mathcal{G}^0$. Recall that this means: $W'_{\alpha\beta}$ is a precompact fiber shrinking of $W^0_{\alpha\beta}|_{\phi_{\alpha\beta}(U'_{\alpha\beta})}$ and $(U'_\alpha, U^0_\alpha, W'_{\alpha\beta})$ is a strong open neighborhood of $\phi_{\alpha\beta}(U'_{\alpha\beta})$.
\end{enumerate}
\end{theorem}

\begin{definition}[admissibility]\label{ADMISSIBLE} Let $\mathcal{G}{\Rightarrow}\mathcal{G}^1$ and $\mathcal{G}{\Rightarrow}\mathcal{G}^2$ be two concerted (level-1) Kuranishi embeddings of of strongly intersecting Hausdorff (level-1) good coordinate systems with the order $\leq, \leq_1, \leq_2$ respectively on a common $S$. They are an \emph{admissible pair of concerted (level-1) Kuranishi embeddings} if for every pair of indices $\beta,\alpha\in S$ such that $U_\beta$ and $U_\alpha$ intersect in the identification space, we have $\beta\leq \alpha, \beta\leq_1\alpha, \beta\leq_2\alpha$ hold at the same time or $\alpha\leq \beta,\alpha\leq_1\beta, \alpha\leq_2\beta$ hold at the same time (or both). We define $\alpha\leq^a \beta$ as $\alpha\leq\beta, \alpha\leq_1\beta,\alpha\leq_2\beta$ when the indexed charts intersect in the identification space, and as $\alpha\leq\beta$ otherwise. $\leq^a$ is a nonantisymmetric total order. 
\end{definition}

\begin{theorem}[achieving an admissible pair of concerted level-1 Kuranishi embeddings]\label{ADMISSIBLEPAIRKEMB} Let $\mathcal{G}\Rightarrow \hat{\mathcal{G}}$ and $\mathcal{G}\Rightarrow\tilde{\mathcal{G}}$ be two Kuranishi embeddings, where $\mathcal{G}$, $\hat{\mathcal{G}}$ and $\tilde{\mathcal{G}}$ are strongly intersecting Hausdorff good coordinate systems indexed by $(S,\leq)$, $(S, \hat\leq)$ and $(S,\tilde\leq)$ respectively. Then there exist a concerted level-1 Kuranishi embedding $\mathcal{G}'\overset{\text{level-1}}{\Rightarrow}\hat{\mathcal{G}}'$ whose underlying Kuranishi embedding chart-refines $\mathcal{G}\Rightarrow\hat{\mathcal{G}}$, and a concerted level-1 Kuranishi embedding for $\mathcal{G}'\overset{\text{level-1}}{\Rightarrow}\tilde{\mathcal{G}}'$ whose underlying Kuranishi embedding chart-refines $\mathcal{G}\Rightarrow\tilde{\mathcal{G}}$, where $\mathcal{G}'$, $\hat{\mathcal{G}}'$ and $\tilde{\mathcal{G}}'$ are indexed by $(S',\leq')$, $(S', \hat\leq_1')$ and $(S',\tilde\leq_2')$ respectively, such that the new pair is an admissible pair.
\end{theorem}

\begin{proof} (a proof by just quoting the above results but requiring two triplings)
\begin{enumerate}
\item Construct a concerted level-1 Kuranishi embedding $\mathcal{G}^{(2)}\overset{\text{level-1}}{\Rightarrow}\hat{\mathcal{G}}^{(2)}$ whose underlying Kuranishi embedding chart-refines $\mathcal{G}\Rightarrow\hat{ \mathcal{G}}$. The order $(S^{(1)}, \leq^{(1)})$ for $\mathcal{G}^{(2)}$ is obtained from applying the tripling to $(\mathcal{G} , (S', \leq'))$, where $S'$ is the partition of $S$ by the \textbf{triple bundle dimensions} and $\leq'$ is a total order made by hand from the order by the bundle dimension in $\mathcal{G}$; and $(S^{(1)}, \hat\leq^{(1)})$ is induced from tripling the non-level-1 $(\hat{\mathcal{G}}, (S', \hat\leq'))$ as in \ref{CONCERTEDPROOF}, where $\leq'$ is the order by the bundle dimension in $\hat{\mathcal{G}}$:

By proposition \ref{CONCERTEDPROOF}, we can find a concerted Kuranishi embedding $\mathcal{G}^{(1)}\Rightarrow\hat{\mathcal{G}}^{(1)}$ with the orders $(S^{(1)}, \leq^{(1)})$ and $(S^{(1)},\hat\leq^{(1)})$ whose underlying Kuranishi embedding chart-refines $\mathcal{G}\Rightarrow\hat{\mathcal{G}}$. 

Then by \ref{GROUPING} and \ref{LEVELONEKEMB}, we can obtain a concerted level-1 Kuranishi embedding $\mathcal{G}^{(2)}\overset{\text{level-1}}{\Rightarrow}\hat{\mathcal{G}}^{(2)}$ with the common total order $(S^{(2)}, \leq^{(2)})$. 

We can revert the common total order $(S^{(2)}, \leq^{(2)})$ of $\mathcal{G}^{(2)}\overset{\text{level-1}}{\Rightarrow}\hat{\mathcal{G}}^{(2)}$ back the the orders $(S^{(1)}, \leq^{(1)})$ and $(S^{(1)},\hat\leq^{(1)})$ (which are the orders after tripling in proposition \ref{CONCERTEDPROOF} and right before grouping in proposition \ref{GROUPING}). Recall that we can do this because in shrinking the grouped charts in theorem \ref{LEVELONEKEMB}, we are required to preserve how charts indexed by the original indices (indices after tripling in this case) intersect among each other. Then we still have concertedness in $\mathcal{G}^{(2)}\overset{\text{level-1}}{\Rightarrow}\hat{\mathcal{G}}^{(2)}$, but the orders $(S^{(1)}, \leq^{(1)})$ on $\mathcal{G}^{(2)}$ and $(S^{(1)},\hat\leq^{(1)})$ on $\hat{\mathcal{G}}^{(2)}$ might not agree, and they might no longer be total orders.

\item Establish the second concerted level-1 Kuranishi embedding:

We can compose $\mathcal{G}^{(2)}\overset{\sim}{\to} \mathcal{G}\Rightarrow \tilde{\mathcal{G}}$ and chart-refine $\tilde{\mathcal{G}}$ to $\tilde{\mathcal{G}}^{(2)}$, so we have a Kuranishi embedding $\mathcal{G}^{(2)}\Rightarrow \tilde{\mathcal{G}}^{(2)}$ with the common index set $S^{(1)}$ and the induced order $\tilde\leq^{(1)}$ on $S^{(1)}$ for $\tilde{\mathcal{G}}^{(2)}$.

We again in the same setting as in (1) except we already have a level-1 structure on $\mathcal{G}^{(2)}$ which we will use in repeating above by applying theorem \ref{LEVELONEKEMB}.(2). As a result we have a concerted level-1 Kuranishi embedding $\mathcal{G}^{(3)}\overset{\text{level-1}}{\Rightarrow}\hat{\mathcal{G}}^{(3)}$ which chart-refines $\mathcal{G}^{(2)}\Rightarrow \tilde{\mathcal{G}}^{(2)}$. Here the order $(S^{(3)}, \leq^{(3)})$ on $\mathcal{G}^{(3)}$ is induced from the tripling of $(\mathcal{G}^{(2)}, (S^{(1)},\leq^{(1)}))$, and denote the order on $\tilde{\mathcal{G}}^{(3)}$ by $(S^{(3)},\tilde\leq^{(3)})$.

\item Revisit the first level-1 Kuranishi embedding and chart-refine:

We compose $\mathcal{G}^{(3)}\overset{\sim}{\to}\mathcal{G}^{(2)}\overset{\text{level-1}}{\Rightarrow}\hat{\mathcal{G}}^{(2)}$. Define the order $\hat\leq^{(3)}$ on $S^{(3)}$ by $T_1\hat\leq^{(3)} T_2$ if and only if $\widehat\max T_1 \hat\leq^{(1)}\widehat\max T_2$. For all $T\in S^{(3)}$, define $$\hat U^{(3)}_T:=(\Pi_{\widehat\max T^{(1)}})^{-1}(\pi_{\max^{(1)}T\; \widehat\max^{(1)} T }U^{(3)}_T)$$ and define $\hat{\mathcal{G}}^{(3)}:=\{\hat C^{(1)}_{\widetilde\max^{(1)} T}|_{\hat U^{(3)}_T}\}_{T\in S^{(3)}}$ with the induced coordinate changes among the restricted charts. Then $\hat{\mathcal{G}}^{(3)}$ indexed by $(S^{(3)}, \hat\leq^{(3)})$ is a chart-refinement of $\hat{\mathcal{G}}^{(2)}$.

Define $\hat\phi_T:=\hat\phi^{(2)}_{\widehat\max T\;\,\max^{(1)} T}\circ\phi_{\max^{(1)} T}$, and $$\Pi_{T}:=\Pi_{\max^{(1)} T}\circ\widehat{\pi_{\widehat\max^{(1)}}}_{T\;\max^{(1)} T}.$$ We now have a level-1 Kuranishi embedding $\mathcal{G}^{(3)}\overset{\text{level-1}}{\Rightarrow}\hat{\mathcal{G}}^{(3)}$, which is still concerted and is a restriction from (hence chart-refines) the level-1 Kuranishi embedding $\mathcal{G}^{(2)}\overset{\text{level-1}}{\Rightarrow}\hat{\mathcal{G}}^{(2)}$.

Here recall that $S^{(3)}:=\mathcal{S}_{S^{(1)}}$ is obtained from tripling $\mathcal{G}^{(2)}\Rightarrow\tilde{\mathcal{G}}^{(2)}$ in achieving the concertedness.

\item Confirm that we already have an admissible pair:

To summarize, we now have level-1 Kuranishi embeddings $$\mathcal{G}^{(3)}\overset{\text{level-1}}{\Rightarrow}\hat{\mathcal{G}}^{(3)}\text{ and }\mathcal{G}^{(3)}\overset{\text{level-1}}{\Rightarrow}\tilde{\mathcal{G}}^{(3)},$$ where the orders for $\mathcal{G}^{(3)}, \hat{\mathcal{G}}^{(3)}$ and $\tilde{\mathcal{G}}^{(3)}$ are $(S^{(3)}, \leq^{(3)})$, $(S^{(3)}, \hat\leq^{(3)})$ and $(S^{(3)}, \tilde\leq^{(3)})$ respectively. Observe that this is an admissible pair.

Indeed, suppose $T_1, T_2\in S^{(3)}$ and $U^{(3)}_{T_1}$ intersect with $U^{(3)}_{T_2}$ in the identification space, then by lemma \ref{INTERSECTISINCLUDE}, we have $T_1\subset T_2$ or $T_2\subset T_1$ as subsets of $S^{(1)}$, or both. In any case, we have three orders $\leq^{(5)}$, $\hat\leq^{(5)}$, and $\tilde\leq^{(5)}$ all agree by the definition of the induced orders.\end{enumerate}
\end{proof}

Now we are ready to form a fiber product in the next subsection.

\subsection{Fiber products}\label{FIBPROD}

At the end of the last subsection, we have an admissible pair of concerted level-1 Kuranishi embeddings $\mathcal{G}\overset{\text{level-1}}{\Rightarrow}\mathcal{G}^1$ and $\mathcal{G}\overset{\text{level-1}}{\Rightarrow}\mathcal{G}^2$, where $\mathcal{G}$, $\mathcal{G}^1$ and $\mathcal{G}^2$ are strongly intersecting Hausdorff level-1 good coordinate systems indexed by the orders $(S^0, \leq^0)$, $(S^0, \leq^1)$ and $(S^0, \leq^2)$ respectively.

We take a common total order partition $(S,\leq)$ for these three orders, similar to proposition \ref{GROUPING}: We partition $S^0$ by the triple bundle dimensions into $S$: Denote $\mathbf{i}:=(i^0, i^1, i^2)$ and $$\alpha_{\mathbf{i}}:=\{\alpha\in S^0\;|\; \text{dim}(E_\alpha|_{U_\alpha})=i^0,\;\text{dim}(E^1_\alpha|_{U^1_\alpha})=i^1,\text{ and }\text{dim}(E^2_\alpha|_{U^2_\alpha})=i^2\}.$$ Denote $S=\{\alpha_{\mathbf{i}}\;|\;\alpha_{\mathbf{i}}\not=\emptyset\}\cong\{\mathbf{i}\;|\;\alpha_{\mathbf{i}}\not=\emptyset\}$. Define $\alpha_{\mathbf{j}}\leq \alpha_{\mathbf{i}}$ (or $\mathbf{j}\leq\mathbf{i}$) if and only if $j^0\leq^0 i^0$, $j^1\leq i^1$ and $j^2\leq i^2$. It is a total order. Define $\text{max}$ by taking the the maximal element using $\leq$, which is consistent with the orders by the $i$-coordinate, $i^1$-th coordinate and $i^2$-coordinate of $\mathbf{i}:=(i, i^1, i^2)$ in $\alpha_{\mathbf{i}}$ respectively.

Apply tripling for $\mathcal{G}, (S,\leq)$ to get $\mathcal{G}_S, (\mathcal{S}_{S}, \leq)$ as in definition \ref{TRIPLING}, and this induces triplings $\mathcal{G}^1_S$ and $\mathcal{G}^2_S$ for $\mathcal{G}^1$ and $\mathcal{G}^2$, and an admissible pair of concerted level-1 Kuranishi embeddings $\mathcal{G}_S\overset{\text{level-1}}{\Rightarrow}\mathcal{G}^1_S$ and $\mathcal{G}_S\overset{\text{level-1}}{\Rightarrow}\mathcal{G}^2_S$. Indeed, define $U^1_T:=(\Pi^1_{\max T})^{-1}(U_T)$ for $T\in\mathcal{S}_{S}$ and $\mathcal{G}^1:=\{C^1_{\max^1 T}|_{U^1_T}\}_{T\in\mathcal{S}_S}$ with the induced coordinate changes among the restricted charts. Define $\phi^1_{T}:=\phi^1_{\max T}$, $\hat\phi^1_{T}:=\hat\phi^1_{\max T}$, $\Pi^1_T:=\Pi^1_{\max T}$, $\tilde\Pi^1_T:=\tilde\Pi^1_{\max T}$ and $\hat\Pi^1_T:=\hat\Pi^1_{\max T}$, which gives the level-1 Kuranishi embedding $\mathcal{G}_S\overset{\text{level-1}}{\Rightarrow}\mathcal{G}^1_S$. The same goes for $\mathcal{G}^2_S$ by changing the superscripts $1$ to $2$.

\begin{definition}[notation for a fiber product Kuranishi chart] Suppose $C_\alpha|_{U_\alpha}$ is a Kuranishi chart, and $C_\alpha|_{U_\alpha}\overset{\text{level-1}}{\to} C^i_\alpha|_{U^i_\alpha}$ is a level-1 Kuranishi chart embedding with the vector-bundle like submersion $\Pi^i_\alpha: U^i_\alpha\to \phi^i_\alpha(U_\alpha)$, the fiberwisely isomorphic bundle maps $\tilde\Pi^i_\alpha: \tilde F^i_\alpha\to \hat\phi_\alpha^i(E_\alpha|_{U_\alpha})$ covering $\Pi^i_\alpha$ and the bundle projection $\hat \Pi^i_\alpha: E^i_\alpha|_{U^i_\alpha}\to\tilde F^i_\alpha$. Here, $i=1, 2$.

Denote the fiber product base $\mathbf{U}_\alpha:=U^1_\alpha\times_{U_\alpha} U^2_\alpha:=U^1_\alpha\;_{\Pi_\alpha}\times_{U_\alpha,\;\Pi^2_\alpha} U^2_\alpha$, and the fiber product bundle is \begin{align*}\mathbf{E}_\alpha|_{\mathbf{U}_\alpha}&:=({E^1_\alpha}\times _{E_\alpha}E^2_\alpha)|_{U^1_\alpha\times_{U_\alpha} U^2_\alpha}\\
&:=(E^1_\alpha|_{U^1_\alpha})\;_{(\tilde\Pi^1_\alpha\circ\hat \Pi^1_\alpha)}\times_{(E_\alpha|_{U_\alpha}),\;(\tilde\Pi^2_\alpha\circ\hat \Pi^2_\alpha)}(E^2_\alpha|_{U^2_\alpha})\\
&:=\{(v^1, v^2)\;|\;\tilde\Pi^1_\alpha(\hat\Pi_\alpha(v^1))=\tilde\Pi^2_\alpha(\hat\Pi^2_\alpha(v^2))\}.
\end{align*}

We have the naturally induced fiber product section $\mathbf{s}_\alpha:=s^1_\alpha\times_{s_\alpha} s^2_\alpha$. We denote 
$$\mathbf{C}_\alpha|_{\mathbf{U}_\alpha}:=(C^1_\alpha|_{U^1_\alpha})_{\hat\Pi^1_\alpha}\times_{(C_\alpha|_{U_\alpha}),\;\hat\Pi^2_\alpha}(C^2_\alpha|_{U^2_\alpha}):=(\mathbf{s}_\alpha: \mathbf{U}_\alpha\to \mathbf{E}_\alpha),$$ and call it a \emph{fiber product Kuranishi chart}. We have naturally induced level-1 Kuranishi chart embeddings $C^1_\alpha|_{U^1_\alpha}\overset{\text{level-1}}{\to}\mathbf{C}_\alpha|_{\mathbf{U}_\alpha}$ and $C^2_\alpha|_{U^2_\alpha}\overset{\text{level-1}}{\to}\mathbf{C}_\alpha|_{\mathbf{U}_\alpha}$, such that the following level-1 square is level-1 commutative:

$$\begin{CD}
C^1_\alpha|_{U^1_\alpha}@>\text{level-1}>> \mathbf{C}_\alpha|_{\mathbf{U}_\alpha}\\
@A\text{level-1}AA @AA\text{level-1}A\\
C_\alpha|_{U_\alpha} @>\text{level-1}>> C^2_\alpha|_{U^2_\alpha}
\end{CD}\;\;\;\;.$$
\end{definition}

\begin{remark} In order to have a fiber product good coordinate system globally, we have to define each fiber product Kuranishi chart using the data from the tripling in the above notion, rather than just simply applying the above notion to the charts at each index.
\end{remark}

\begin{definition}[fiber product good coordinate system $\mathcal{G}^{\text{FP}}$]\label{FPGCSLEVELONE}

We need a tripling with a slightly better property. For $\mathcal{G}$, we choose $\tilde U^1_{(\beta,\alpha)}$, $\tilde U^2_{(\beta,\alpha)}$, $\tilde U^3_{(\beta,\alpha)}=\text{orbit}(W_{\alpha\beta})$ for $(\beta,\alpha)\in I^\ast(\mathcal{G}):=\{(\beta,\alpha)\in I(\mathcal{G})\;|\;\beta\not=\alpha\}$ as before. $$\mathcal{G}_m:=\{C_\beta|_{\tilde U^1_{(\beta,\alpha)}}, C_\alpha|_{\tilde U^2_{(\beta,\alpha)}}, C_\alpha|_{\tilde U^3_{\alpha\beta}}\}_{(\beta,\alpha)\in I^\ast(\mathcal{G})}$$ with the induced coordinate changes is a good coordinate system. Choose a precompact shrinking $\{U^1_{(\beta,\alpha)}, U^2_{(\beta,\alpha)}\}_{(\beta,\alpha)\in I^\ast (\mathcal{G})}$ of $\{\tilde U^1_{(\beta,\alpha)}, \tilde U^2_{(\beta,\alpha)}\}_{(\beta,\alpha)\in I^\ast(\mathcal{G})}$ in $\mathcal{G}_m$ in the sense of \ref{PRECOMPACTCSHRINK} which respects the strong open neighborhoods in the level-1 structures. Define $U^3_{\alpha\beta}:=\tilde U^3_{\alpha\beta}$ and use those $U^i_{(\beta,\alpha)}$ to proceed with the tripling construction.

For $T\subset \mathcal{S}_S$, denote $$\hat\Pi^{\text{FP}}_{T,i}:=((\tilde \Pi^i_{\max T})^\ast ((\hat\phi^i_{\max T})_\ast \hat\pi_{\max T \min T}))\circ \hat\Pi^i_{\max T},\;\;\text{and}$$
\begin{align*}B_T:&=(\hat \phi_{\max T \min T})_\ast(C_{\min T}|_{(\phi_{\max T\min T})^{-1}(\pi_{\max T \min T} (U_T)}))\\&\equiv (\tilde\pi_{\max T\min T}\circ\hat\pi_{\max T\min T})_\ast(C_T|_{U_T}).
\end{align*}
Then we can define $$C^{\text{FP}}_T|_{U^{\text{FP}}_T}:=(C^1_{\max T}|_{U^1_T})_{\hat\Pi^{\text{FP}}_{T,1}}\times_{B_T,\; \hat\Pi^{\text{FP}}_{T, 2}}(C^2_{\max T}|_{U^2_T}).$$

Define $\text{dim} (T):=\text{dim} (C^{\text{FP}}|_{U^{\text{FP}}_T}):=\text{dim}(E^{\text{FP}}_T)=\text{dim} E^1_T+\text{dim} E^2_{T}-\text{dim} E_{\min T}$ which is the bundle dimension of $E^{\text{FP}}_T$. Define the order $T_1 \leq^{\text{FP}} T_2$ if and only if $\text{dim} T_1\leq \text{dim} T_2$. Observe that if $T\subset T'$, then $\text{dim} T\leq \text{dim} T'$.

Suppose that $U_{T_1}$ and $U_{T_2}$ intersect in the identification space $M(\mathcal{G}_S)$, then by the definition of $(S,\leq)$ via the triple bundle dimensions and lemma \ref{INTERSECTISINCLUDE}, we have $T_1\subset T_2$ or $T_2\subset T_1$, or both.

We only consider the case $T_1\subset T_2$, since the second case is symmetric with respect to switching $1$ and $2$. In this case, we have, $\text{dim}T_1\leq \text{dim}T_2$. We now define a coordinate change from $C^{\text{FP}}_{T_1}|_{U^{\text{FP}}_{T_1}}$ to $C^{\text{FP}}_{T_2}|_{U^{\text{FP}}_{T_2}}$ in the direction of this order:

Define $U_{T_2 T_1}:=$
$$U_{T_1}\cap (\phi_{\max T_2 \max T_1})^{-1}(\pi_{\max T_2\max T_1} (\text{orbit}(W_{\max T_2\max T_1})\cap U_{T_2}))$$ and $U^i_{T_2 T_1}:=(\Pi^{\text{FP}}_{T_1, i})^{-1}(U_{T_2 T_1})$. Then we have a natural coordinate change $C^{\text{FP}}_{T_1}|_{U^{\text{FP}}_{T_1}}\to C^{\text{FP}}_{T_2}|_{U^{\text{FP}}_{T_2}}$ with the domain $U^{\text{FP}}_{T_2 T_1}:=U^1_{T_2 T_1}\times_{U_{T_2 T_1}} U^2_{T_2 T_1}$ and this coordinate change is induced by the following three coordinate changes:
\begin{align*}
C^1_{\max T_1}|_{U^1_{T_2 T_1}}&\to  C^1_{\max T_2}|_{U^1_{T_2}}\to C^{\text{FP}}_{T_2}|_{U^{\text{FP}}_{T_2}}\\
C^2_{\max T_1}|_{U^2_{T_2 T_1}}&\to  C^2_{\max T_2}|_{U^2_{T_2}}\to C^{\text{FP}}_{T_2}|_{U^{\text{FP}}_{T_2}}\\
B_{T_2 T_1} &\overset{\text{diagonal embedding}}{\to}\\
((\phi^1_{\max T_2 \max T_1})_\ast((\phi^1_{\max T_1})_\ast & C_B))_{\hat\Pi^1_{\max T_2}}\times_{B_{T_2},\; \hat\Pi^2_{\max T_2}}\\((\phi^2_{\max T_2 \max T_1})_\ast((\phi^2_{\max T_1})_\ast& C_B))
\quad\quad\to C^{\text{FP}}_{T_2}|_{U^{\text{FP}}_{T_2}}
\end{align*}
$$\text{where }B_{T_2 T_1}:=(\hat \phi_{\max T_1 \min T_1})_\ast(C_{\min T_1}|_{\phi_{\max T_1\min T_1}^{-1}(\pi_{\max T_1 \min T_1} (U_{T_2 T_1}))})$$
$$\quad\quad\quad\; B_{T_2}:=(\hat \phi_{\max T_2 \min T_2})_\ast(C_{\min T_2}|_{\phi_{\max T_2\min T_2}^{-1}(\pi_{\max T_2 \min T_2} (U_{T_2}))}).$$

Hence, we have $C^1_{\max T_1}|_{U^1_{T_2 T_1}}\times_{B_{T_2 T_1}} C^2_{\max T_1}|_{U^2_{T_2 T_1}}\to C^{\text{FP}}_{T_2}|_{U^{\text{FP}}_{T_2}}$. We verify that $C^{\text{FP}}_{T_1}|_{U^{\text{FP}}_{T_2 T_1}}$ is exactly the left hand side $C^1_{\max T_1}|_{U^1_{T_2 T_1}}\times_{B_{T_2 T_1}} C^2_{\max T_1}|_{U^2_{T_2 T_1}}$, and 
$C^{\text{FP}}_{T_1}|_{U^{\text{FP}}_{T_2 T_1}}\to C^{\text{FP}}_{T_2}|_{U^{\text{FP}}_{T_2}}$ is a coordinate change, and moreover is naturally equipped with a level-1 structure induced from the given data, and this coordinate change is by the convention denoted as $C^{\text{FP}}_{T_1}|_{U^{\text{FP}}_{T_1}}\to C^{\text{FP}}_{T_2}|_{U^{\text{FP}}_{T_2}}$.

Define $I(\mathcal{G}^{\text{FP}}):=$ $$\{(T_1, T_2)\;|\; T_1\leq^{\text{FP}} T_2,\;\;\text{$U^{\text{FP}}_{T_1}$ and $U^{\text{FP}}_{T_2}$ intersect in the identification space}\}$$ $$\equiv\{(T_1, T_2)\in \mathcal{S}_S\times \mathcal{S}_S\;|\; T_1\subset T_2\}.$$
Here we observe that $U^{\text{FP}}_{T_1}$ and $U^{\text{FP}}_{T_2}$ intersect in the identification space if and only if $U_{T_1}$ and $U_{T_2}$ intersect in the identification space. 

We denote $$\mathcal{G}^{\text{FP}}:=(X, (\mathcal{S}_S, \leq^{\text{FP}}), \{C^{\text{FP}}_T|_{U^{\text{FP}}_T}\}_{T\in\mathcal{S}_S}, \{C^{\text{FP}}_{T_1}|_{U^{\text{FP}}_{T_1}}\to C^{\text{FP}}_{T_2}|_{U^{\text{FP}}_{T_2}}\}_{(T_1, T_2)\in I(\mathcal{G}^{\text{FP}})})$$ and it is automatically a level-1 good coordinate system. $\mathcal{G}^{\text{FP}}$ is called a \emph{fiber product good coordinate system}.
\end{definition}

Now we want to define $\{(\Psi^i_T, \hat\Psi^i_T, U^i_T)=(C^i_T|_{U^i_T}\to C^{\text{FP}}_T|_{U^{\text{FP}}_T})\}_{T\in \mathcal{S}_S}$ which will have naturally induced level-1 structures and hence give rise to level-1 Kuranishi embeddings $\mathcal{G}^i_S\Rightarrow \mathcal{G}^{\text{FP}}$, $i=1,2$, such that the following level-1 square of concerted level-1 Kuranishi embeddings commutes:

$$\begin{CD}
\mathcal{G}^1_S@>\text{level-1}>> \mathcal{G}^{\text{FP}}\\
@A\text{level-1}AA @AA\text{level-1}A\\
\mathcal{B}(\mathcal{G}_S) @>\text{level-1}>> \mathcal{G}^2_S
\end{CD}\;\;\;\;\;\;,$$
which means the following level-1 square is level-1 commutative.
$$\begin{CD}
C^1_T|_{U^1_T}@>\text{level-1}>> C^{\text{FP}}_T|_{U^{\text{FP}}_T}\\
@A\text{level-1}AA @AA\text{level-1}A\\
\mathcal{B}(C_T|_{U_T}) @>\text{level-1}>> C^2_T|_{U^2_T}
\end{CD}\;\;\;,$$
where $\mathcal{B}(\cdot)$ is some natural procedure applied to $\mathcal{G}_S$ to make the above level-1 square commute.

$C^i_T|_{U^i_T}\to C^{\text{FP}}_T|_{U^{\text{FP}}_T}$ for $T=\{\alpha\}\in \mathcal{S}_S$ with $\alpha\in S$ has to be the usual inclusion in the $i$-th factor of the fiber product, for example, the image of $C^1_{\{ \alpha\}}|_{U^1_{\{\alpha\}}}$ in $C^{\text{FP}}_{\{\alpha\}}|_{U^{\text{FP}}_{\{\alpha\}}}$ is $(C^1_{\{\alpha\}}|_{U^1_{\{\alpha\}}})\times_{(C_{\{\alpha\}}|_{U_{\{\alpha\}}})} ((\hat\phi^2_{\{\alpha\}})_\ast(C_{\{\alpha\}}|_{U_{\{\alpha\}}}))$.

For $T$ with $|T|>1$ (i.e. $T\in \mathcal{S}_S\backslash\{\{\alpha\}\;|\;\alpha\in S\}$), we need a twisting.

For brevity of the notations, denote $\overline{T}=\max T$ and $\underline{T}=\min T$. Recall $C^{\text{FP}}_T|_{U^{\text{FP}}_T}:=(C^1_{\overline{T}}|_{U^1_{T}})\times_{(\hat\phi_{\overline{T}\underline{T}})_\ast(C_{\underline{T}}|_{\phi^{-1}_{\overline{T}\underline{T}}(\pi_{\overline{T} \underline{T}}(U_T))})} (C^2_{\overline{T}}|_{U^2_T})$. 

Consider $(T,T')\in I(\mathcal{G}_S)$ with $I\not=I'$. Thus, $T\subsetneqq T'$. Let us look at the images of $C^1_T|_{(\Pi^1_T)^{-1}(\phi^1_{T}(U_{T'T}))}\equiv C^1_{\overline{T}}|_{(\Pi^1_{\overline{T}})^{-1}(\phi^1_{\overline{T}}(U_{T'T}))}$, where we have $U^1_{T'T}=(\Pi^1_T)^{-1}(\phi^1_{T}(U_{T'T}))$.

The image of $C^1_T|_{U^1_{T'T}}$ as the first factor inside $C^{\text{FP}}_T|_{U^{\text{FP}}_T}$ is $$C^1_{\overline{T}}|_{U^1_{T'T}}\times_{((\hat\phi_{\overline{T}\underline{T}})_\ast(C_{\underline{T}}|_{(\phi_{\overline{T}\underline{T}})^{-1}(\pi_{\overline{T} \underline{T}}(U_{T'T}))}))} (\hat\phi^1_{\overline{T}})_\ast((\hat\phi_{\overline{T}\underline{T}})_\ast(C_{\underline{T}}|_{(\phi_{\overline{T}\underline{T}})^{-1}(\pi_{\overline{T} \underline{T}}(U_{T'T}))}));$$
and the base of the image of the latter inside $C^{\text{FP}}_{T'}|_{U^{\text{FP}}_{T'}}$ is
$$\{((z,v,w),(z,v,0))\;|\;z\in \pi_{\overline{T'}\underline{T'}}(\phi_{\overline{T'}\;\overline{T}}(U_{T'T})),$$
$$ v\in (\pi_{\overline{T'}\underline{T'}}|_{\phi_{\overline{T'}\;\overline{T}}(U_{T'T})})^{-1}(z), w\in (\Pi^1_{\overline{T'}}|_{\phi^1_{\overline{T'}\;\overline{T}}(U^1_{T'T})})^{-1}((z,v))\}.$$

The image of $C^1_T|_{U^1_{T'T}}$ inside $C^1_{T'}|_{U^1_{T'}}$ is $(\hat\phi^1_{\overline{T'}\overline{T}})_\ast (C^1_{\overline{T}}|_{U^1_{T'T}})$; and the base of the image of the latter as the first factor inside $C^{\text{FP}}_{T'}|_{U^{\text{FP}}_{T'}}$ is
$$\{((z,v,w),(z,0,0))\;|\;z\in \pi_{\overline{T'}\underline{T'}}(\phi_{\overline{T'}\;\overline{T}}(U_{T'T})),$$
$$ v\in (\pi_{\overline{T'}\underline{T'}}|_{\phi_{\overline{T'}\;\overline{T}}(U_{T'T})})^{-1}(z), w\in (\Pi^1_{\overline{T'}}|_{\phi^1_{\overline{T'}\;\overline{T}}(U^1_{T'T})})^{-1}((z,v))\}.$$

Therefore to have the commutativity of the fiber product square, we need to define $\Psi^1_{T}: C^1_T|_{U^1_T}\to C^{\text{FP}}_T|_{U^{\text{FP}}_T}$ so that its image smoothly twists from the first factor embedding $((z,v,w),(z,0,0))$ over the strong neighborhood over the image $\phi_{T\{\underline{T}\}}(U_{T\{\underline{T}\}})$ of the domain of the coordinate change from $C^1_{\{\underline{T}\}}|_{U_{\{\underline{T}\}}}$ to the partial diagonal embedding $((z,v^1+v^2,w),(z,v^1,0))$ over the strong neighborhoood over the image $\phi_{TR}(U_{TR})$ of the domain of the coordinate change from $C_{R}|_{U^1_{R}}$ for each $R\subset T$\footnote{As $\underline{U_R}$ always intersect with $\underline{U_S}$ in the indentification space for $R\subset T\in\mathcal{S}_S$.}, where $((z,v^1,0),(z,v^1,0))$ is the diagonal embedding into $\pi_{\overline{T}\;\overline{R}}(U_{TR})\times_{\pi_{\overline{T}\underline{T}}(U_{TR})}\pi_{\overline{T}\;\overline{R}}(U_{TR})$. Moreover, we need to do this systematically to have the global compatibility, so the images of $C^1_T|_{U^1_T}, T\in \mathcal{S}_S$ in $\mathcal{G}^{\text{FP}}$ with the induced coordinate changes is a good coordinate system, and is level-1 with the naturally induced level-1 structures for those coordinate changes. The same goes for $\Psi^2$.

To achieve this, we define a twisting sytem:

\begin{definition}[twisting system] Recall from \ref{FPGCSLEVELONE}. Let $\mathcal{G}$ be a level-1 good coordinate system indexed by a total order $(S,\leq)$. At the start of the tripling process, for each $(\beta,\alpha)\in I^\ast(\mathcal{G})$, we have chosen $\tilde U^i_{(\beta,\alpha)}, i=1,2,3$ respecting the strong neighborhoods in the level-1 structure, giving rise to a chart-refinement\footnote{This is defined in \ref{FPGCSLEVELONE}, and it is not $\mathcal{G}_S$ as we have not taken the tripling yet.} $\mathcal{G}_m$ of $\mathcal{G}$. Choose a precompact $\{U^1_{(\beta,\alpha)},U^2_{(\beta,\alpha)}\}_{(\beta,\alpha)\in I^\ast(\mathcal{G})}$ of $\{\tilde U^1_{(\beta,\alpha)},\tilde U^2_{(\beta,\alpha)}\}_{(\beta,\alpha)\in I^\ast(\mathcal{G})}$ in $\mathcal{G}_m$ in the sense of \ref{PRECOMPACTCSHRINK}. Denote the result of the tripling using $\{U^1_{(\beta,\alpha)},U^2_{(\beta,\alpha)}, U^3_{(\beta,\alpha)}=\tilde U^3_{(\beta,\alpha)}=\text{orbit}(W_{\alpha\beta})\}_{(\beta,\alpha)\in I^\ast(\mathcal{G})}$ by $\{U_T\}_{T\in \mathcal{S}_S}$. Denote $I_{\alpha}=\{\beta\in S\;|\;(\beta,\alpha)\in I^\ast(\mathcal{G})\}$. Treating $$\pi^o_{\alpha\beta}:\text{orbit}(W_{\alpha\beta})\to\text{orbit}(\phi_{\alpha\beta}(U_{\alpha\beta}))$$ as a fiber bundle ep-groupoid with the total space $\text{orbit}(W_{\alpha\beta})$. A collection of bundle ep-groupoid morphisms $\lambda_{\alpha\beta}: \text{orbit}(W_{\alpha\beta})\to\text{orbit}(W_{\alpha\beta})$ covering the identity on the bases\footnote{Namely, a smooth section functor $\lambda_{\alpha\beta}\in \Gamma(\text{orbit}(\phi_{\alpha\beta}(U_{\alpha\beta})), \text{End}(\text{orbit}(W_{\alpha\beta})))$.}, $\beta\in I_\alpha$, $\alpha\in S$, is called a \emph{twisting system} if
\begin{enumerate}
\item $\lambda_{\alpha\beta}=\lambda_{\alpha\gamma}$ on $\text{orbit}(W_{\alpha\beta})\cap \text{orbit}(W_{\alpha\gamma})$, so $\lambda_{\alpha\beta}$ fits into a smooth ep-groupoid map $\lambda_\alpha: W^{\cup}_\alpha\to W^{\cup}_\alpha$ with $W^{\cup}_\alpha:=\bigcup_{\beta\in I_\alpha}\text{orbit}(W_{\alpha\beta})$;
\item $\lambda_\beta|_{U_{\alpha\beta}}=(\phi_{\alpha\beta})^{-1}\circ\pi_{\alpha\beta}\circ\lambda_\alpha\circ \phi_{\alpha\beta}$;
\item denoting $\alpha_{-1}:=\max I_\alpha$, $$\lambda_{\alpha}|_{(W^{\cup}_\alpha)\cap U^1_{(\alpha_{-1},\alpha)}}=\pi_{\alpha \alpha_{-1}}\text{ and }\lambda_\alpha|_{(W^{\cup}_\alpha)\cap U^2_{(\alpha_{-1},\alpha)}}=Id;$$
\item on $(W^{\cap}_\alpha)\cap (\pi_{\alpha\beta}^{-1}(\phi_{\alpha\beta}(U^1_{(\beta,\alpha)})))$, $\lambda_\alpha=\phi_{\alpha\beta}\circ\lambda_\beta\circ(\phi_{\alpha\beta}^{-1})\circ\pi_{\alpha\beta}$ for all $\beta\in I_\alpha$, and
\item $\lambda_\alpha|_{U^2_{(\alpha,\gamma)}\backslash (\bigcup_{\beta\in I_\alpha\backslash(I_\gamma\cup\{\gamma\})\;:\;(\gamma,\beta)\in I(\mathcal{G})}\text{orbit}(W_{\beta\gamma}))}=Id$ for all $\gamma\in I_\alpha$.
\end{enumerate}
We will write $\lambda_{\alpha\beta},\beta\in I_\alpha$ as $\lambda_\alpha$ in the following.
\end{definition}

We can build this inductively. Note that the precompact shrinking from $\{\tilde U^1_{(\beta,\alpha)},\tilde U^2_{(\beta,\alpha)}\}_{(\beta,\alpha)\in I^\ast(\mathcal{G})}$ to $\{U^1_{(\beta,\alpha)},U^2_{(\beta,\alpha)}\}_{(\beta,\alpha)\in I^\ast(\mathcal{G})}$ makes room for bump functions as well as creating transition regions for the compatibility.

\begin{definition}[fiber product square, $\mathcal{G}^{\text{FP}}=\mathcal{G}^1_S\times_{\mathcal{B}(\mathcal{G}_S)}\mathcal{G}^2_S$]

For $T=\{\alpha\}$, $C^i_T|_{U^i_T}\to C^{\text{FP}}_T|_{U^{\text{FP}}_T}$ is the $i$-th factor inclusion in the fiber product. For $|T|>1$, define $(C^1_T|_{U^1_T}\to C^{\text{FP}}_T|_{U^{\text{FP}}_T})=(\Psi^1_T,\hat\Psi^1_T, U^1_T)$, using a twisting system $\{\lambda_\alpha\}_{\alpha\in S}$ as $\Psi^1_{T}: U^1_T\mapsto U^1_T\times_{\phi_{\overline{T}\underline{T}}(\pi_{\overline{T}\underline{T}}(U_T))}U^2_T, w\mapsto (w, \lambda_{\overline{T}}(\Pi^1_T(w)))$. By the property of the twisting system $\{\lambda_\alpha\}_{\alpha\in S}$, $\Psi^1_{T'}\circ\phi_{T'T}=\phi_{T'T}^{\text{FP}}\circ\Psi^1_T$ for all $(T,T')\in I(\mathcal{G}^{\text{FP}})$. Use the preimage, we can define $\Psi^1_T$. Then $\{(\Psi^1_T,\hat\Psi^1_T, U^1_T)\}_{T\in\mathcal{S}_S}$ is a Kuranishi embedding and in fact a concerted level-1 Kuranishi embedding $\mathcal{G}^1_S\overset{\text{level-1}}{\Rightarrow}\mathcal{G}^{\text{FP}}$ with the induced level-1 structures. Define $\Psi^2_T$ by swapping factors and using the same $\{\lambda_\alpha\}_{\alpha\in S}$ as the above. Define $\mathcal{B}(\mathcal{G}_S):=\{(\tilde\pi_{\max{T}\min{T}}\circ\hat\pi_{\max{T}\min{T}})_\ast (C_T|_{U_T})\}_{T\in\mathcal{S}_S}$. For $i=1, 2$, using $\{\phi^1_T|_{\pi_{\max T\min T}(U_T)}\}_{T\in\mathcal{S}_S}$, $\mathcal{B}(\mathcal{G}_S)\Rightarrow\mathcal{G}^1_S$ is a Kuranishi embedding and chart-refines $\mathcal{G}\Rightarrow\mathcal{G}^i$. Although $\mathcal{B}(\mathcal{G}_S)\Rightarrow\mathcal{G}^i_S$ is not concerted but it has a level-1 structure induced from the concerted level-1 Kuranishi embedding.

We have the fiber product property:
$$\begin{CD}
\mathcal{G}^1_S@>\text{level-1}>> \mathcal{G}^{\text{FP}}\\
@A\text{level-1}AA @AA\text{level-1}A\\
\mathcal{B}(\mathcal{G}_S) @>\text{level-1}>> \mathcal{G}^2_S
\end{CD}\;\;\;\;\;\;.$$

We say $\mathcal{G}^{\text{FP}}$ is the \emph{fiber product level-1 good coordinate system} $\mathcal{G}_S^1\times_{\mathcal{B}(\mathcal{G}_S)}\mathcal{G}^2_S$ \emph{formed from an admissible pair of concerted level-1 Kuranishi embeddings} $\mathcal{G}\overset{\text{level-1}}{\Rightarrow} \mathcal{G}^1$ and $\mathcal{G}\overset{\text{level-1}}{\Rightarrow}\mathcal{G}^2$.
\end{definition}

See figure \ref{FP} for illustrating the bases in a fiber product square.

\begin{figure}[htb]
\begin{center}

\begin{tikzpicture}[scale=.75]
\hspace{0 cm}

\filldraw[black!20] (13,-1)--(13.5,-1)--(13.5,1)--(13,1)--(13,1.5)--(15,1.5)--(15,-1.5)--(13,-1.5)--(13,-1);
\draw[very thick](10,0)--(12,0);
\draw [very thick] (13,-1)--(13.5,-1)--(13.5,1)--(13,1)--(13,1.5)--(15,1.5)--(15,-1.5)--(13,-1.5)--(13,-1);
\draw [very thick, red](11.5,0.03)--(14,0.03);

\begin{scope}[shift={(-5,5)}]

\filldraw[black!20] (13,-1)--(13.5,-1)--(13.5,1)--(13,1)--(13,1.5)--(15,1.5)--(15,-1.5)--(13,-1.5)--(13,-1);

\filldraw[black!40] (10,-1.2) rectangle (12,1.2);

\filldraw[path fading=south,color=red!70] (11.5,-1) rectangle (14,1);
\draw[very thick](10,0)--(12,0);
\draw[magenta, very thick] (11.5,-1) rectangle (14,1);
\draw[yellow] (10,-1.2) rectangle (12,1.2);
\draw [very thick] (13,-1)--(13.5,-1)--(13.5,1)--(13,1)--(13,1.5)--(15,1.5)--(15,-1.5)--(13,-1.5)--(13,-1);
\draw [very thick, red](11.5,0.03)--(14,0.03);
\end{scope}

\begin{scope}[shift={(5,5)}]

\filldraw[black!20] (13,-1)--(13.5,-1)--(13.5,1)--(13,1)--(13,1.5)--(15,1.5)--(15,-1.5)--(13,-1.5)--(13,-1);

\filldraw[blue!40] (10,-1.4) rectangle (12,1.4);

\filldraw[path fading=south,color=red!70] (11.5,-1) rectangle (14,1);
\draw[very thick](10,0)--(12,0);
\draw[blue, very thick] (11.5,-1) rectangle (14,1);
\draw[cyan] (10,-1.4) rectangle (12,1.4);
\draw [very thick] (13,-1)--(13.5,-1)--(13.5,1)--(13,1)--(13,1.5)--(15,1.5)--(15,-1.5)--(13,-1.5)--(13,-1);
\draw [very thick, red](11.5,0.03)--(14,0.03);
\end{scope}

\begin{scope}[shift={(0,10)}]

\filldraw[black!20](14.38,-.825)--(12.38,-.825)--(12.6,-.55)--(13.1,-.55)-- (13.9,.55)--(13.4,.55)--(13.62, .825)--(15.62, .825)--(14.38,-.825);

\filldraw[blue!40] (10,-1.4) rectangle (12,1.4);

\filldraw[path fading=west, color=black!25] (9.5,-.85)--(10.5,-1.95)--(10.5,.85)--(9.5,1.95)--(9.5,-.85);

\begin{scope}[shift={(2,0)}]
\filldraw[path fading=west, color=black!25] (9.5,-.85)--(10.5,-1.95)--(10.5,.85)--(9.5,1.95)--(9.5,-.85);
\end{scope}

\filldraw[color=black!35] (10.5,-1.95)--(12.5,-1.95)--(12.5,.85)--(10.5, .85)--(10.5,-1.95);

\filldraw[path fading=east, color=black!45] (10.5,.85)--(12.5,.85)--(11.5,1.95)--(9.5, 1.95)--(10.5,.85);

\filldraw[path fading=south,color=red] (11.5,-1) rectangle (14,1);

\filldraw[color=red!55] (11.9,-1.45)--(14.4,-1.45)--(14.4,.55)--(11.9,.55)--(11.9,-1.45);
\filldraw[path fading=west,color=red!45] (11.9,-1.45)--(11.1,-.55)--(11.1,1.45)--(11.9,.55)--(11.9,-1.45);
\filldraw[path fading=east,color=red!65] (11.1,1.45)--(11.9,.55)--(14.4,.55)--(13.6,1.45)--(11.1,1.45);


\draw[very thick](10,0)--(12,0);
\draw[cyan] (10,-1.4) rectangle (12,1.4);

\draw [very thick](14.30,-.825)--(12.30,-.825)--(12.52,-.55)--(13.02,-.55)-- (13.9,.55)--(13.4,.55)--(13.62, .825)--(15.62, .825)--(14.30,-.825);

\draw [blue, very thick] (13.3,.55)--(14.4,.55)--(13.52, -.55)--(12.42, -.55);
\draw [blue, very thick] (12.05, 1)--(11.5,1)--(11.5,-1)--(12.05,-1);
\draw [blue, very thick](12.05,-1)..controls (12.35,-.95) and (12.4,-.60) .. (12.42,-.55);
\draw [blue, very thick] (12.05,1).. controls (12.8,1) and (13,.55).. (13.3,.55);

\draw [magenta, very thick] (11.7,0.45)--(11.1,0.45)--(11.9,-.45)--(12.5,-.45);
\draw [magenta, very thick] (13.33,.52)--(14.43,.52)--(13.55, -.58)--(12.45, -.58);
\draw [magenta, very thick] (12.5,-.45)..controls (13.1, -.4) and (13.15, .5).. (13.33,.52);
\draw [magenta, very thick] (11.7,0.45)..controls (12,0.37) and (12.3,.4)..(12.45,-.58);
\draw [very thick, red](11.5,0.03)--(14,0.03);
\draw [yellow] (10.5,-0.55)--(12.5,-0.55)--(11.5,0.55)--(9.5, 0.55)--(10.5,-0.55);
\end{scope}

\path[->](6,7) edge (8.5,9);

\path[->](18,7) edge (15.5,9);

\path[->] (15.5,1) edge (18,3);

\path[->] (8.5,1) edge (6,3);

\node at (12.5, -2.5){tripling and fiber product};

\end{tikzpicture}

\end{center}
\caption[Forming a fiber product using a tripling.]{Forming a fiber product using a tripling in the simplest non-trivial case where the na\"{i}ve fiber product without a tripling does not work.}
\label{FP}
\end{figure}
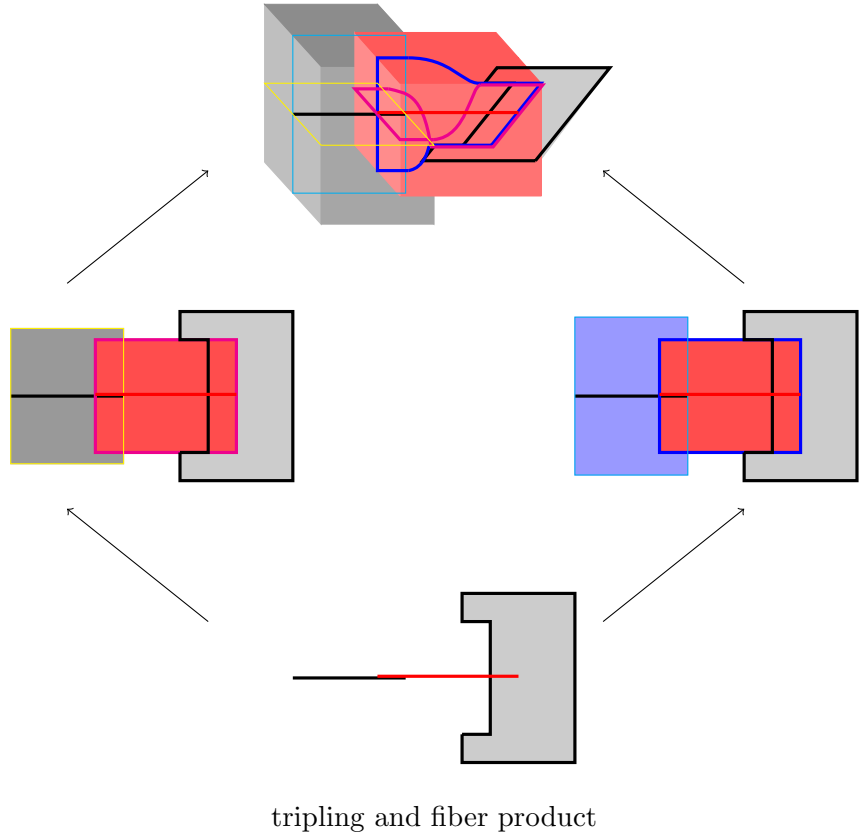

\begin{remark} Since $\mathcal{G}^i_S$ chart-refines $\mathcal{G}^i$ for $i=1, 2$, $\mathcal{G}^{\text{FP}}$ is a common refinement of $\mathcal{G}^1$ and $\mathcal{G}^2$. We can view the level-1 Kuranishi embedding $\mathcal{B}(\mathcal{G}_S)\overset{\text{level-1}}{\Rightarrow} \mathcal{G}^i_S$ as a projection from $\mathcal{G}^i_S$ to the image of $\mathcal{B}(\mathcal{G}_S)$ in $\mathcal{G}^i_S$, and this justifies the name of a \emph{fiber product}. The fiber product is not unique, but different constructions are equivalent as level-1 good coordinate systems. Namely, any two constructions of fiber product this way are level-1 refined by a common level-1 good coordinate system.
\end{remark}

\section{Independence of choices}\label{INDEP}

We make various choices in each stage of the construction above, and we now explain how the notion of a common refinement (with the chart-refinement inverted and followed by a general embedding \ref{GEMB}) absorbs the dependence of those choices. We will compare the different choices made at each stage here. For a general comparison of all choices at once, see \ref{GENERALCHOICECOMPARISON}.

\subsection{Equivalence of the good coordinate systems for a given Kuranishi structure}\label{INDEPGCSYSTEM}

We start by examining the independence of choices of good coordinate systems. Two different choices of good coordinate systems can be shown to refine a common good coordinate system. Then this picture can be made into a pair of concerted level-1 Kuranishi embeddings of strongly intersecting Hausdorff good coordinate systems, up to chart-refinements. Then by using the fiber product, we can show that both shrunken (level-1) good coordinate systems are then (level-1) refined by a common (level-1) good coordinate system, and therefore the two starting good coordinate systems are refined by a common good coordinate system, namely, equivalent.

Now follow the details. First, recall the following two results:

\begin{theorem}[part of theorem \ref{BIGHAUSDORFF}]\label{SIHSHRINKING} Let $\mathcal{G}$ be a good coordinate system. Then we can find a shrinking $\mathcal{G}'$ of $\mathcal{G}$ such that $\mathcal{G}'$ is strongly intersecting and Hausdorff.
\end{theorem}

\begin{proposition}\label{ENCORECOMMONREFINED} Let $\tilde{\mathcal{G}}^1$ and $\tilde{\mathcal{G}}^2$ be two good coordinate systems for a Kuranishi structure $\mathcal{K}$ as defined in \ref{GCSFORKS}. Then $\tilde{\mathcal{G}}^1$ and $\tilde{\mathcal{G}}^2$ both refine some common good coordinate system $\mathcal{G}$. In fact, we can find strongly intersecting Hausdorff good coordinate systems $\mathcal{G}$, $\mathcal{G}^1$ and $\mathcal{G}^2$ with the same index set $S$ but with possibly different orders on $S$ such that:
\begin{enumerate}
\item $\mathcal{G}^1$ is a chart-refinement of $\tilde{\mathcal{G}}^1$,
\item $\mathcal{G}^2$ is a chart-refinement of $\tilde{\mathcal{G}}^2$, and
\item there exist Kuranishi embeddings $\mathcal{G}\Rightarrow\mathcal{G}^1$ and $\mathcal{G}\Rightarrow\mathcal{G}^2$.
\end{enumerate}
\end{proposition}

\begin{proof} Invoke theorem \ref{SIHSHRINKING} in the proof of proposition \ref{COMMONREFINED}.
\end{proof}

The results of the last section and this subsection culminate in the following theorem: 

\begin{theorem} Any two choices of good coordinate systems $\tilde{\mathcal{G}}^1$ and $\tilde{\mathcal{G}}^2$ obtained from a Kuranishi structure $\mathcal{K}$ are equivalent. In fact, we estabish that there exists strongly intersecting Hausdorff level-1 good coordinate systems $\mathcal{G}^1_S$, $\mathcal{G}^2_S$ and $\mathcal{G}^{\text{FP}}$ such that for $i=1, 2$, we have that underlying good coordinate system $\mathcal{G}^i_S$ chart-refines $\tilde{\mathcal{G}}^i$ and there exist concerted level-1 Kuranishi embeddings $\mathcal{G}^i_S\overset{\text{level-1}}{\Rightarrow}\mathcal{G}^{\text{FP}}$. (So, $\mathcal{G}^{\text{FP}}$ is a common refinement of $\tilde{\mathcal{G}}^1$ and $\tilde{\mathcal{G}}^2$).
\end{theorem}

\begin{proof} Let $\tilde{\mathcal{G}}^1$ and $\tilde{\mathcal{G}}^2$ be any two good coordinate systems obtained from the Kuranishi structure $\mathcal{K}$. 

Applying proposition \ref{ENCORECOMMONREFINED}, we have a chart-refinement $\hat{\mathcal{G}}^1$ of $\tilde{\mathcal{G}}^1$ and a chart-refinemement $\hat{\mathcal{G}}^2$ of $\tilde{\mathcal{G}}^2$ such that there exist natural Kuranishi embeddings $\hat{\mathcal{G}}\Rightarrow\hat{\mathcal{G}}^1$ and $\hat{\mathcal{G}}\Rightarrow\hat{\mathcal{G}}^2$ for some $\hat{\mathcal{G}}$, where $\hat{\mathcal{G}}$, $\hat{\mathcal{G}}^1$ and $\hat{\mathcal{G}}^2$ are all strongly intersecting, Hausdorff and indexed by the same set $S$.

Applying \ref{ADMISSIBLEPAIRKEMB}, we have an admissible pair of concerted level-1 Kuranishi embeddings $\mathcal{G}\Rightarrow\mathcal{G}^1$  and $\mathcal{G}\Rightarrow\mathcal{G}^2$, whose underlying Kuranishi embeddings chart-refines $\hat{\mathcal{G}}\Rightarrow\hat{\mathcal{G}}^1$  and $\hat{\mathcal{G}}\Rightarrow\hat{\mathcal{G}}^2$.

Applying construction \ref{FPGCSLEVELONE}, we have a level-1 good coordinate system $\mathcal{G}^{\text{FP}}:=\mathcal{G}^1_S\times_{\mathcal{G}_S}\mathcal{G}^2_S$ whose underlying good coordinate system still denoted by $\mathcal{G}^{FP}$ (even level-1) refines $\mathcal{G}^1$ and $\mathcal{G}^2$, which in turn chart-refine $\hat{\mathcal{G}}^1$ and $\hat{\mathcal{G}}^2$ respectively, which in turn chart-refine $\tilde{\mathcal{G}}^1$ and ${\mathcal{G}}^2$ respectively.

Therefore $\mathcal{G}^{\text{FP}}$ is a common refinement of $\tilde{\mathcal{G}}^1$ and $\tilde{\mathcal{G}}^2$. Therefore, $\tilde{\mathcal{G}}^1$ and $\tilde{\mathcal{G}}^2$ are equivalent.
\end{proof}

\subsection{Independence of the choices of level-1 structures}

The different choices of Hausdorff level-1 good coordinate systems are also equivalent, as a common refinement can be found by using the tripling and a fiber-product-like construction.

\begin{proposition}\label{INDEPLEVEL1} Let $\mathcal{G}^1:=\{C_\alpha|_{U^1_\alpha}\}_{\alpha\in S}$ and $\mathcal{G}^2:=\{C_\alpha|_{U^2_\alpha}\}_{\alpha\in S}$ be two choices of level-1 good coordinate systems whose underlying good coordinate systems are precompact shrinkings of a strongly intersecting Hausdorff good coordinate system $\mathcal{G}:=\{C_\alpha|_{U_\alpha}\}_{\alpha\in S}$. Here $(S, \leq)$ is a total order.

Then there exists level-1 good coordinate systems $\mathcal{G}^3$ indexed by $(S', \leq)$ and level-1 chart-refinements $\mathcal{G}^4$ and $\mathcal{G}^5$ of $\mathcal{G}^1$ and $\mathcal{G}^2$ respectively such that $\mathcal{G}^4$ and $\mathcal{G}^5$ both concertedly level-1 Kuranishi embed into $\mathcal{G}^3$. Namely, $\mathcal{G}^3$ is a common level-1 refinement for $\mathcal{G}^1$ and $\mathcal{G}^2$.
\end{proposition}

\begin{proof}
The difficulty lies in that we cannot find a common shrinking of the underlying good coordinate systems of $\mathcal{G}^1$ and $\mathcal{G}^2$.
\begin{enumerate}
\item Prepare an admissible pair of Kuranishi embeddings to form a fiber product but both Kuranishi embeddings are not level-1 with respect to given level-1 structures on level-1 good coordinate systems:

Define $V_\alpha:=U^1_\alpha\cup U^2_\alpha$, then $\{C_\alpha|_{V_\alpha}\}_{\alpha\in S}$ is precompact shrinking of $\mathcal{G}$. Using proposition \ref{CONTROLLINGTHESIZE}, choose a level-1 good coordinate system $\mathcal{G}^6$ for $\mathcal{G}$ such that $\{C_\alpha|_{V_\alpha}\}_{\alpha\in S}$ is a precompact shrinking of the underlying good coordinate system of $\mathcal{G}^6$.

For each $x\in X$, define $\alpha^1_x:=\min\{\alpha\in S\;|\; x\in X_\alpha|_{U^1_\alpha}\}$ and $\alpha^2_x:=\min\{\alpha\in S\;|\; x\in X_\alpha|_{U^2_\alpha}\}$. Denote $\alpha_x:=\min\{\alpha^1_x, \alpha^2_x\}$.

Choose $U^7_{\alpha_x}$ to be an invariant shrinking of $$((\phi_{\alpha^1_x\alpha_x})^{-1}(\pi^6_{\alpha^1_x\alpha_x} (W^6_{\alpha^1_x \alpha_x})))\cap ((\phi_{\alpha^2_x\alpha_x})^{-1}(\pi^6_{\alpha^2_x\alpha_x}(W^6_{\alpha^2_x\alpha_x})))\footnote{Here convention of this notation is that $\pi^6_{\alpha^i_x\alpha_x} (W^6_{\alpha^i_x \alpha_x})=\phi_{\alpha^i_x\alpha_x}(U^6_{\alpha_x})$ if $\alpha_x=\alpha^i_x$.}$$ such that 
\begin{enumerate}
\item $x\in X_{\alpha_x}|_{U^7_{\alpha_x}}$, 
\item $\phi^6_{\alpha^i\alpha_x}(U^7_{\alpha_x})\subset U^i_{\alpha^i_x}$ for $i=1, 2$, and 
\item $(U^7_{\alpha_x}, U_{\alpha_x}, U^7_{\alpha_x\beta})$ is an invariant strong open neighborhood of $\phi_{\alpha_x\beta}(U_{\alpha_x\beta})$ for all $\beta\in \alpha_x$, here $U_{\alpha_x\beta}$ is the domain of the coordinate change $C_\beta|_{U^6_\beta}\to C_\alpha|_{U^7_{\alpha_x}}$.
\end{enumerate}

Choose a finite set $J\subset S$ such that $X=\bigcup_{x\in J} X_{\alpha_x}|_{U_{\alpha_{x}}}$. Then $$\mathcal{G}^7:=\{C_{\alpha_x}|_{U^7_{\alpha_x}}\}_{x\in J}$$ is a level-1 good coordinate system induced from $\mathcal{G}^6$, and by construction, we have non-level-1 Kuransihi embeddings $\mathcal{G}^7\Rightarrow\mathcal{G}^1$ and $\mathcal{G}^7\Rightarrow\mathcal{G}^2$.

We can apply theorem \ref{ADMISSIBLEPAIRKEMB} to have an admissible pair of concerted level-1 Kuranishi embeddings $\mathcal{G}^7\overset{\text{level-1}}{\Rightarrow}\tilde{\mathcal{G}}^1$ and $\mathcal{G}^7\overset{\text{level-1}}{\Rightarrow}\tilde{\mathcal{G}}^2$ indexed by some order $(S'', \leq'')$ and the underlying good coordinate system $\tilde{\mathcal{G}}^i$ chart-refines the underlying good coordinate system $\mathcal{G}^i$ for $i=1, 2$.

\item For $i=1, 2$, compare level-1 structures on $\mathcal{G}^i$ and $\tilde{\mathcal{G}}^i$, where the underlying good coordinate systems $\tilde{\mathcal{G}}^i$ chart-refines the underlying good coordinate systems $\mathcal{G}^i$:

We can choose a chart-refinement $\hat{\mathcal{G}}^i$ of $\mathcal{G}^i$ such that it is a level-1 good coordinate system, the underlying good coordinate systems $\tilde{\mathcal{G}}^i$ chart-refines the underlying good coordinate systems $\hat{\mathcal{G}}^i$ and both $\tilde{\mathcal{G}}$ and $\hat{\mathcal{G}}^i$ are now indexed by $(S'', \leq)$, and $\text{dim} \tilde E_\alpha=\text{dim} \hat E^i_\alpha$ for all $\alpha\in S''$. In particular, we can ensure that the image of $\tilde W^i_{\alpha\beta}$ under the Kuranishi embedding belongs to $\hat W^i_{\alpha\beta}$.

Now we perform the tripling for $\tilde{\mathcal{G}}^i$ to arrive at $\tilde U^i_T, T\in \mathcal{S}_{S''}$. Define $\hat U^i_T:=(\hat\pi^i_{\max T \min T})^{-1}(\hat \pi^i_{\max T\min T}(\tilde U_T)$. Form the fiber product $$\mathbf{C}^i_T:=\tilde C^i_{T}|_{\tilde U^i_T}\times_{\hat C^i_{\min T}|_{\phi_{\max T\min T}^{-1}(\hat\pi^i_{\max T\min T}(\hat U^i_T)})} \hat C^i_{T}|_{\hat U^i_T}$$ with induced fiber product coordinate changes among the charts. Denote this level-1 good coordinate change by $\mathbf{G}^i$ indexed by $\mathcal{S}_{S''}$. We have naturally induced concerted level-1 Kuranishi embeddings $\tilde{\mathcal{G}}^i_S\overset{\text{level-1}}{\to}\mathbf{G}^i$ and $\hat{\mathcal{G}}^i_S\overset{\text{level-1}}{\to}\mathbf{G}^i$ by diagonal embeddings for $i=1, 2$.

\item Forming the fiber product, which reaches the conclusion:

So far, we have:
$$\begin{CD}
\mathcal{G}^1@<\sim<<\hat{\mathcal{G}}^1@<\sim<<\hat{\mathcal{G}}^1_{S''}@>>>\mathbf{G}^1@.@.\\
@.@.@. @AAA @. @. \\
@.@.@. \tilde{\mathcal{G}}^1_{S''}@.@.\\
@.@. @. @AAA @. @.\\
@.@.@.\mathcal{G}^7_{S''}@>>>\tilde{\mathcal{G}}^2_{S''}@>>>\mathbf{G}^2\\
@. @.@. @. @. @AAA\\
@.@.@.@.@.\hat{\mathcal{G}}^2_{S''}\\
@. @. @. @.@. @VV\sim V\\
@.@.@.@.@.\hat{\mathcal{G}}^2\\
@. @. @. @.@. @VV\sim V\\
@.@.@.@.@.\mathcal{G}^2
\end{CD}. $$

From the above picture, we can form a fiber product and compose some maps:

$$\begin{CD}
\mathcal{G}^1@<\sim<<\hat{\mathcal{G}}^1_{S''}@<\sim<<(\hat{\mathcal{G}}^1_{S''})_{\mathcal{S}_{S''}}@>>>(\mathbf{G}^1)_{\mathcal{S}_{S''}}@>>>\mathcal{G}^3\\
@.@. @. @AAA @AAA \\
@.@.@.\mathcal{B}((\mathcal{G}^7_{S''})_{\mathcal{S}_{S''}})@>>>(\mathbf{G}^2)_{\mathcal{S}_{S''}}\\
@. @.@. @. @AAA\\
@.@.@.@.(\hat{\mathcal{G}}^2_{S''})_{\mathcal{S}_{S''}}\\
@. @. @. @. @VV\sim V\\
@.@.@.@.\hat{\mathcal{G}}^2_{S''}\\
@. @. @. @. @VV\sim V\\
@.@.@.@.\mathcal{G}^2
\end{CD}, $$
where $\mathcal{G}^3:=(\mathbf{G}^1)_{\mathcal{S}_{S''}}\times_{\mathcal{B}((\mathcal{G}^7_{S''})_{\mathcal{S}_{S''}})}(\mathbf{G}^2)_{\mathcal{S}_{S''}}$

Denote $\mathcal{G}^4:=(\hat{\mathcal{G}}^1_{S''})_{\mathcal{S}_{S''}}$ and $\mathcal{G}^5:=(\hat{\mathcal{G}}^2_{S''})_{\mathcal{S}_{S''}}$. Then we arrive at the conclusion:

$$\begin{CD}
\mathcal{G}^1@<\sim<<\mathcal{G}^4@>>>\mathcal{G}^3\\
@. @. @AAA \\
@.@.\mathcal{G}^5\\
  @. @. @VV\sim V\\
@.@.\mathcal{G}^2
\end{CD}.$$

Namely, we have constructed a common level-1 refinement $\mathcal{G}^3$ for $\mathcal{G}^1$ and $\mathcal{G}^2$. 
\end{enumerate}
\end{proof}

\subsection{Comparing perturbations of different level-1 good coordinate systems for a Kuranishi structure}

We have now transferred every choice made during the constructions leading to a level-1 good coordinate system for perturbation into a Kuranishi embedding concept, up to chart-refinement. Next, we examine choices of different perturbations and compare them in the setting of the Kuranishi embeddings and chart-refinements.

A perturbation on $\mathcal{G}$ (supported in a precompact shrinking of $\mathcal{G}$) can be modified in a cobordant manner to be contained in the image of the chart-refinement $\mathcal{G}'\to\mathcal{G}$ hence pullback to $\mathcal{G}'$, and we can use $(\Pi_\alpha, \tilde\Pi_\alpha)$ in the level-1 structure of a level-1 Kuranishi embedding $\mathcal{G}'\Rightarrow\mathcal{G}''$ to lift the perturbation from $\mathcal{G}'$ to $\mathcal{G}''$. Those moves keep the weighted branched orbifold solution set unchanged after the cobordant change in the initial modification before these two pullbacks.

Two different perturbations on two different choices of level-1 good coordinate systems for a Kuranish structure can be compared in a common level-1 refinement using this method, and perturbations hence the solution sets will be cobordant.

\begin{theorem}\label{COMPAREPERTURBATION} Let $\mathcal{G}'$ and $\mathcal{G}''$ be two choices of strongly intersecting Hausdorff level-1 good coordinate systems for an oriented Kuranishi structure $\mathcal{K}$. Suppose $\mathcal{T}'$ and $\mathcal{T}''$ are two compact invariant transverse multisectional perturbations constructed on chart-refinements $\tilde{\mathcal{G}}'$ and $\tilde{\mathcal{G}}''$ of $\mathcal{G}'$ and $\mathcal{G}''$, respectively, using the method in subsection \ref{PERTN}. Then there exist concerted level-1 Kuranishi embeddings from the chart-refinements $\hat{\mathcal{G}}'$ and $\hat{\mathcal{G}}''$ of $\tilde{\mathcal{G}}'$ and $\tilde{\mathcal{G}}''$ into a level-1 good coordinate system $\hat{\mathcal{G}}$. The compact perturbation $\mathcal{T}'$ on $\tilde{\mathcal{G}}'$ is cobordant to a compact perturbation $\mathcal{T}'''$ on $\tilde{\mathcal{G}}'$, which pulls back to a compact perturbation on $\hat{\mathcal{G}}'$, which in turn lifts to a compact perturbation $\hat{\mathcal{T}}'$ on $\mathcal{G}'$. In the same way, the compact perturbation $\mathcal{T}''$ on $\tilde{\mathcal{G}}''$ is cobordant to a perturbation $\mathcal{T}''''$ on $\tilde{\mathcal{G}}''$, which gives rise to a compact perturbation $\hat{\mathcal{T}}''$on $\mathcal{G}''$. Moreover, the two choices of compact perturbations $\hat{\mathcal{T}}'$ and $\hat{\mathcal{T}}''$ on $\hat{\mathcal{G}}$ are cobordant. So all the choices made in the construction will yield the same invariant if the invariant mechanism factors through cobordisms.
\end{theorem}

\begin{proof} The proof follows from the following observation:

The idea behind pulling back a chosen compact perturbation using the method in \ref{PERTN} through refinement is as follows: Shrinking $U'_\alpha$ into $\tilde U'_\alpha$ will cut off some of the compact perturbation, but it only affects the region that is not further covered by other charts $\tilde U'_\beta, \beta\not=\alpha$, and over that region, we can replace\footnote{We may end up with a different linear combination, but the result is cobordant to the previous choice.} local perturbations with perturbations that have support close enough to $X$ that they can be compact and completely contained in chart-refinement, and then these can be pulled back through a chart-refinement and then lifted along a level-1 Kuranishi embedding by pulling back via $(\Pi_\alpha,\tilde\Pi_\alpha)$.
\end{proof}

\section{The equivalences are equivalence relations}
\subsection{The equivalence of good coordinate systems is an equivalence relation}

\begin{corollary} \label{GCSEQUIVPF} For two good coordinate systems, admitting a common refinement is transitive and hence an equivalence relation.
\end{corollary}

\begin{proof}
Suppose $\mathcal{G}^1$ and $\mathcal{G}^2$ have a common refinement $\mathcal{G}^{4}$, that is, there exist a chart-refinement of $\tilde{\mathcal{G}}^i$ of $\mathcal{G}^i$ and a chart-refinement $\tilde{\mathcal{G}}_i^4$ of $\mathcal{G}^4$ such that there exists Kuranishi embeddings $\tilde{\mathcal{G}}^i\Rightarrow\tilde{\mathcal{G}}_i^{4}$ for $i=1, 2$.

Suppose $\mathcal{G}^2$ and $\mathcal{G}^3$ have a common refinement $\mathcal{G}^{5}$, that is, there exist a chart-refinement of $\hat{\mathcal{G}}^i$ of $\mathcal{G}^i$ and a chart-refinement $\hat{\mathcal{G}}_i^5$ of $\mathcal{G}^5$ such that there exists Kuranishi embeddings $\hat{\mathcal{G}}^i\Rightarrow\hat{\mathcal{G}}_i^{5}$ for $i=2, 3$.

Observe that admitting a common chart-refinement for good coordinate systems is not an equivalence relation. Therefore, the best we can do is similar to before, we can find a $\mathcal{G}^6$ such that there exist $\mathcal{G}^6\Rightarrow\tilde{\mathbf{G}}^2\overset{\sim}{\to}\tilde{\mathcal{G}}^2$ and $\mathcal{G}^6\Rightarrow\hat{\mathbf{G}}^2\overset{\sim}{\to}\hat{\mathcal{G}}^2$.

We chart-refine $\tilde{\mathcal{G}}_2^4$ by $\tilde{\mathbf{G}}^4$ to have $\tilde{\mathbf{G}}^2\Rightarrow\tilde{\mathbf{G}}^4$; and we chart-refine $\hat{\mathcal{G}}_2^5$ by $\hat{\mathbf{G}}^5$ to have $\hat{\mathbf{G}}^2\Rightarrow\hat{\mathbf{G}}^5$.

By composing, we now have a pair of Kuranishi embeddings $\mathcal{G}^6\Rightarrow\tilde{\mathbf{G}}^4$ and $\mathcal{G}^6\Rightarrow\hat{\mathbf{G}}^5$.

Then by theorem \ref{ADMISSIBLEPAIRKEMB}, we have an admissible pair of concerted level-1 Kuranishi embeddings $\mathbf{G}^6\Rightarrow\mathbf{G}^4$ and $\mathbf{G}^6\Rightarrow\mathbf{G}^5$ indexed by $S$, whose underlying Kuranishi embeddings chart-refine $\mathcal{G}^6\Rightarrow\tilde{\mathbf{G}}^4$ and $\mathcal{G}^6\Rightarrow\hat{\mathbf{G}}^5$ respectively.

Then by construction \ref{FPGCSLEVELONE}, we have a fiber product good coordinate system $\mathcal{G}^7:=\mathbf{G}^4_S\times_{\mathcal{B}(\mathbf{G}^6_S)}\mathbf{G}^5_S$ into which both $\mathbf{G}^4_S$ and $\mathbf{G}^5_S$ Kuranishi embed.

So putting the above together, we have a string of chart-refinements $$\mathbf{G}^4_S\overset{\sim}{\to}\mathbf{G}^4\overset{\sim}{\to}\tilde{\mathbf{G}}^4\overset{\sim}{\to}\tilde{\mathcal{G}}^4_2\overset{\sim}{\to}\mathcal{G}^4.$$

By composing the above, and together with the rest picture, we have (where $\rightarrow$ in the diagram always denotes $\Rightarrow$):

$$\begin{CD}
\mathcal{G}^1@<\sim<<\tilde{\mathcal{G}}^1@>>>\tilde{\mathcal{G}}^4_1@>\sim>>\mathcal{G}^4@.\\
@.@.@. @A\sim AA @. \\
@.@.@.\mathbf{G}^4_S@>>>\mathcal{G}^7
\end{CD}.$$

We can find a chart-refinement $\tilde{\mathbf{G}}^1$ of $\tilde{\mathcal{G}}^1$ to complete the following picture:

$$\begin{CD}
\mathcal{G}^1@<\sim<<\tilde{\mathcal{G}}^1@>>>\tilde{\mathcal{G}}^4_1@>\sim>>\mathcal{G}^4@.\\
@.@A\sim AA@. @A\sim AA @. \\
@.\tilde{\mathbf{G}}^1@>>>\mathbf{G}^1@>\sim>>\mathbf{G}^4_S@>>>\mathcal{G}^7
\end{CD}.$$

$$\begin{CD}
\mathcal{G}^1@<\sim<<\tilde{\mathcal{G}}^1@>>>\tilde{\mathcal{G}}^4_1@>\sim>>\mathcal{G}^4\\
@.@A\sim AA@. @A\sim AA  \\
@.\tilde{\mathbf{G}}^1@>>>\mathbf{G}^1@>\sim>>\mathbf{G}^4_S\\
@.@.@VVV @VVV  \\
@.@.\mathbf{G}^7@>\sim >> \mathcal{G}^7
\end{CD}.$$

After composing, we have $\mathcal{G}^1\overset{\sim}{\leftarrow}\tilde{\mathbf{G}}^1\Rightarrow\mathbf{G}^7\overset{\sim}{\to}\mathcal{G}^7$, namely, $\mathcal{G}^7$ refines $\mathcal{G}^1$.

By the same token, $\mathcal{G}^7$ refines $\mathcal{G}^3$. $\mathcal{G}^7$ is the desired common refinement of $\mathcal{G}^1$ and $\mathcal{G}^3$.

Thus it shows that admitting a common refinement is transitive. The reflexivity and symmetry properties are trivial.
\end{proof}

\subsection{The equivalence of Kuranishi structures is an equivalence relation}\label{KEQUIV}

\begin{theorem}\label{GCSEMBEDDINGFORKS} Let $\Phi:\mathcal{K}\Rightarrow\mathcal{K}'$ be a Kuranishi embedding between Kuranishi structures. Then there exists a concerted level-1 Kuranishi embedding $\mathcal{G}\Rightarrow \mathcal{G}'$, where $\mathcal{G}$ and $\mathcal{G}'$ are strongly intersecting Hausdorff level-1 good coordinate systems obtained from $\mathcal{K}$ and $\mathcal{K}'$ respectively, such that the underlying embedding $\mathcal{G}\Rightarrow\mathcal{G}'$ is induced by $\mathcal{K}\Rightarrow\mathcal{K}'$.
\end{theorem}

\begin{proof}
Just follow the proof of theorem \ref{EXISTGCS}.

As $X$ is metrizable, choose a metric $d$ on $X$. We can choose $B_{r_p}(p)\subset X_p$ for some $r_p>0$ depending on $p$.

Denote $quot_p: V_p\to \underline{V_p}$, and define an invariant open subset of $V_p$:
$$U_p:=V_p\backslash (s_p^{-1}(0)\backslash quot_p^{-1}(\psi_p^{-1}(B_{r_p/2}(p))))$$ such that $p\in X_p|_{U_p}$, so $X_p|_{U_p}=B_{r_p/2}(p)$.

Choose a finite index set $S$, such that $X=\bigcup_{x\in S}X_x|_{U_x}$.

Define the order $\leq$ on $S$ as for $y, x\in S$, $y\leq x$ if and only if
\begin{enumerate}[(i)]
\item $\text{dim} E_y<\text{dim} E_x$, or
\item $\text{dim} E_y=\text{dim} E_x$ and $|G_y|\leq |G_x|$.
\end{enumerate}

We then construct coordinate changes among charts $\{C_x|_{U_x}\}_{x\in S}$ induced from $\mathcal{K}$ as in construction \ref{EXISTGCS}. We have a good coordinate system $\mathcal{G}$.

Since $B_{r_p}(p)\subset X_p\subset X'_p$, we can define $U'_p\subset V'_p$ similarly, and define $\leq'$ similarly. Then we have a good coordinate system $\mathcal{G}'$ induced from $\mathcal{K}'$.

The Kuranishi embedding induces a Kuranishi embedding $\mathcal{G}\Rightarrow\mathcal{G}'$, indexed by $S$ with the orders $(S, \leq)$ and $(S, \leq')$ respectively.

Then applying the first part of the proof of theorem \ref{ADMISSIBLEPAIRKEMB} (which is obtained by applying proposition \ref{CONCERTEDPROOF}, proposition \ref{GROUPING} and theorem \ref{LEVELONEKEMB}), we get the desired conclusion.
\end{proof}

\begin{theorem} For two Kuranishi structures, admitting a common refinement is a transitive relation.
\end{theorem} 

\begin{proof}
Let $\mathcal{K}^4$ be a common refinement for $\mathcal{K}^1$ and $\mathcal{K}^2$ and let $\mathcal{K}^5$ be a common refinement for $\mathcal{K}^2$ and $\mathcal{K}^3$. Explicitly, we have refinements $\mathcal{K}^2\overset{\sim}{\leftarrow}\mathbf{K}'\Rightarrow\mathcal{K}^4$ and $\mathcal{K}^2\overset{\sim}{\leftarrow}\mathbf{K}''\Rightarrow\mathcal{K}^5$. We can find a common chart-refinement $\mathbf{K}$ for both $\mathbf{K}'$ and $\mathbf{K}''$.

Composing maps, we have $\mathbf{K}\Rightarrow\mathcal{K}^4$ and $\mathbf{K}\Rightarrow\mathcal{K}^5$. This induces a pair of Kuranishi embedding of good coordinate systems $\mathcal{G}\Rightarrow\mathcal{G}^4$ and $\mathcal{G}\Rightarrow\mathcal{G}^5$.

Applying theorem \ref{ADMISSIBLEPAIRKEMB} we have an admissible pair of concerted level-1 Kuranishi embeddings $\mathbf{G}\Rightarrow\mathbf{G}^4$ and $\mathbf{G}\Rightarrow\mathbf{G}^5$ with a common index set $S$, whose underlying Kuranishi embedding chart-refines the existing pair.

After construction \ref{FPGCSLEVELONE}, we have $\mathbf{G}^4_S\Rightarrow\mathcal{G}^{\text{FP}}$ and $\mathbf{G}^5_S\Rightarrow\mathcal{G}^{\text{FP}}$ where $\mathcal{G}^{\text{FP}}:=\mathbf{G}^4_S\times_{\mathcal{B}(\mathbf{G}_S)}\mathbf{G}^5_S$, such that $\mathbf{G}^4_S$ chart-refines $\mathbf{G}^4$ and $\mathbf{G}^5_S$ chart-refines $\mathbf{G}^5$.

By proposition \ref{GCSKS}, we have $\mathcal{K}(\mathbf{G}^4_S)\Rightarrow\mathcal{K}(\mathcal{G}^{\text{FP}})$. 

Since $\mathbf{G}^4_S$ is obtained from $\mathcal{K}^4$, by proposition \ref{KREFINEGCS}, we have $\mathcal{K}(\mathbf{G}^4_S)$ refines $\mathcal{K}^4$, namely $\mathcal{K}^4\overset{\sim}{\leftarrow}\mathcal{K}^6\Rightarrow \mathcal{K}(\mathbf{G}^4_S)$ for some $\mathcal{K}^6$. So $\mathcal{K}^4\overset{\sim}{\leftarrow}\mathcal{K}^6\Rightarrow \mathcal{K}(\mathcal{G}^{\text{FP}})$.

Therefore, for some $\tilde{\mathcal{K}}^1$, we have
$$\begin{CD}
\mathcal{K}^1@<\sim<<\tilde{\mathcal{K}}^1@>>>\mathcal{K}^4@.\\
@.@. @A\sim AA @. \\
@.@.\mathcal{K}^6@>>>\mathcal{K}(\mathcal{G}^{\text{FP}})
\end{CD}.$$

We can find a chart-refinement $\tilde{\mathbf{K}}$ of $\tilde{\mathcal{K}}^1$ to complete the following picture:

$$\begin{CD}
\mathcal{K}^1@<\sim<<\tilde{\mathcal{K}}^1@>>>\mathcal{K}^4@.\\
@.@A\sim AA @A\sim AA @. \\
@.\tilde{\mathbf{K}}@>>>\mathcal{K}^6@>>>\mathcal{K}(\mathcal{G}^{\text{FP}})
\end{CD}.$$

By composing maps, we have $\mathcal{K}^1\overset{\sim}{\leftarrow}\tilde{\mathbf{K}}\Rightarrow\mathcal{K}(\mathcal{G}^{\text{FP}})$. Namely, $\mathcal{K}(\mathcal{G}^{\text{FP}})$ is a refinement of $\mathcal{K}^1$.

By the same token, $\mathcal{K}(\mathcal{G}^{\text{FP}})$ is also a refinement of $\mathcal{K}^2$, hence is a common refinement for $\mathcal{K}^1$ and $\mathcal{K}^2$.

So the relation is transitive.
\end{proof}

\begin{corollary}\label{KSEQUIVPF} The equivalence of Kuranishi structures is an equivalence relation.
\end{corollary}
\begin{proof} We have established the transitivity; and the reflexivity and symmetry are trivial.
\end{proof}

\section{Perturbation factoring through equivalence and independence of most general choices}\label{GENERALCHOICECOMPARISON}

\begin{definition}[strong refinement] A refinement from a good coordinate system $\mathcal{G}$ to another $\mathcal{G}'$ is \emph{strong}, if it is of the form $$\mathcal{G}\overset{\text{chart-refinement}}{\leftarrow}\mathcal{G}''\overset{\text{concerted Kuranishi embedding}}{\to}\mathcal{G}',$$ without the need of the last chart-refinement map in definition \ref{REFINEMENTMAP}.
\end{definition}

Before showing the full choice-independence of the perturbation theory, we collect some results we will need in constructing a common refinement. The proofs of these results have been covered in the earlier sections.
\begin{theorem}\label{PRELIMTHM} The following statements are true:
\begin{enumerate}
\item Two level-1 chart-refinements of a good coordinate system admits a common strong level-1 refinement.
\item Two good coordinate systems of equivalent Kuranishi structures have a common strong level-1 refinement.
\item We can lift perturbn across level-1 Kuranishi embeddings. Therefore, perturbations of two level-1 good coordinate systems can be modified cobordantly to push through level-1 refinements to become perturbations on the common level-1 good coordinate system, which then are cobordant.
\end{enumerate}
\end{theorem}

\begin{theorem}[full choice-independence of the perturbation theory]\label{FULLINDEP} Let $[\mathcal{K}]$ be a Kuranishi structure equivalence class. We choose two Kuranishi structures $\mathcal{K}^i$ in $[\mathcal{K}]$, then choose a good coordinate system $\mathcal{H}^i$ for $\mathcal{K}^i$, then choose a level-1 good coordinate system $\mathcal{G}^i$ for $\mathcal{H}^i$, then choose a perturbation on a level-1 precompact shrinking $\mathcal{J}^i$ of $\mathcal{G}^i$. Then $\mathcal{J}^1$ and $\mathcal{J}^2$ have a common strong level-1 refinement to compare perturbations.
\end{theorem}

\begin{proof}
We summarize the proof in figure \ref{full_choice_ind}.

\begin{figure}[htb]
  \begin{center}

\begin{tikzpicture}[scale=1]

\node at (-2, 3.45) {$(\mathcal{Q}^1_{\mathcal{S}_{(\mathcal{S}_S)}})_{\mathcal{S}_{(\mathcal{S}_{(\mathcal{S}_S)})}}$};
\node at (3.9, 3.55) {$(\mathcal{Q}^2_{\mathcal{S}_{(\mathcal{S}_S)}})_{\mathcal{S}_{(\mathcal{S}_{(\mathcal{S}_S)})}}$};

\node at (-4.2, 2.2){$\widetilde{(\tilde{\mathcal{L}}^1_S)_{\mathcal{S}_S}}$};
\node at (0.2, 2){$(\mathcal{N}_S)_{\mathcal{S}_S}$};
\node at (4.6, 2.3){$\widetilde{(\tilde{\mathcal{L}}^2_S)_{\mathcal{S}_S}}$};

\matrix(m)[matrix of math nodes, row sep=2.25em, column sep=1.5em,
text height=1.5ex, text depth=0.25ex]{
& & & & Q^{\text{FP}}\\
& & {\;} & & & & {\;}\\
& {\;} & \mathcal{Q}^1 &  & {\;\;\;\;} & & \mathcal{Q}^2 & {\;}\\
\tilde{\tilde{\mathcal{J}}}^1 & \tilde{\mathcal{L}}^1_{\mathcal{S}_S} & & & \mathcal{N} & & & \tilde{\mathcal{L}}^2_{\mathcal{S}_S} & \tilde{\tilde{\mathcal{J}}}^2\\
& \tilde{\mathcal{L}}^1 & & \tilde{\mathcal{N}}^1 & & \tilde{\mathcal{N}}^2 & & \tilde{\mathcal{L}}^2\\
& \mathcal{L}^1 & \tilde{\tilde{\mathcal{H}}}^1 & \mathcal{N}^1 & & \mathcal{N}^2 & \tilde{\tilde{\mathcal{H}}}^2 & \mathcal{L}^2\\
\tilde{\mathcal{J}}^1 & & \tilde{\mathcal{H}}^1 & & \mathcal{H} & & \tilde{\mathcal{H}}^2 & & \tilde{\mathcal{J}}^2\\
\mathcal{J}^1 & \mathcal{G}^1 & \mathcal{H}^1 & \hat{\mathcal{H}}^1 & & \hat{\mathcal{H}}^2 & \mathcal{H}^2 & \mathcal{G}^2 & \mathcal{J}^2\\
};
\draw[double,>= angle 60,->](m-2-3) -- (m-1-5);
\draw[double,>= angle 60,->](m-2-7) -- (m-1-5);
\draw[double,>= angle 60,->](m-3-2) -- (m-2-3);
\draw[double,>= angle 60,->](m-3-8) -- (m-2-7);
\draw[double,>= angle 60,->](m-4-9) -- (m-3-8);
\draw[double,>= angle 60,->](m-4-1) -- (m-3-2);
\draw[double,>= angle 60,->](m-4-2) -- (m-3-3);
\draw[double,>= angle 60,->](m-4-8) -- (m-3-7);
\draw[double,>= angle 60,->](m-3-5) -- (m-3-3);
\draw[double,>= angle 60,->](m-3-5)-- (m-3-7);
\draw[double,>= angle 60,->](m-6-3) -- (m-5-2);
\draw[double,>= angle 60,->](m-6-3)-- (m-5-4);
\draw[double,>= angle 60,->](m-5-4) -- (m-4-5);
\draw[double,>= angle 60,->](m-6-7) -- (m-5-8);
\draw[double,>= angle 60,->](m-6-7) -- (m-5-6);
\draw[double,>= angle 60,->](m-5-6) -- (m-4-5);
\draw[double,>= angle 60,->](m-7-1) -- (m-6-2);
\draw[double,>= angle 60,->](m-7-3) -- (m-6-2);
\draw[double,>= angle 60,->](m-7-3) -- (m-6-4);
\draw[double,>= angle 60,->](m-8-4) -- (m-7-5);
\draw[double,>= angle 60,->](m-8-6) -- (m-7-5);
\draw[double,>= angle 60,->](m-7-7) -- (m-6-6);
\draw[double,>= angle 60,->](m-7-7) -- (m-6-8);
\draw[double,>= angle 60,->](m-7-9) -- (m-6-8);
\path[thick, >= angle 60, ->]
(m-4-1) edge node [below, sloped]{$\sim$}(m-7-1)
(m-4-9) edge node [above, sloped]{$\sim$}(m-7-9)
(m-7-1) edge node [below, sloped]{$\sim$}(m-8-1)
(m-7-9) edge node [above, sloped]{$\sim$}(m-8-9)
(m-3-2) edge node [below, sloped]{$\sim$}(m-4-2)
(m-2-3) edge node [below, sloped]{$\sim$}(m-3-3)
(m-2-7) edge node [below, sloped]{$\sim$}(m-3-7)
(m-3-8) edge node [below, sloped]{$\sim$}(m-4-8)
(m-4-2) edge node [below, sloped]{$\sim$}(m-5-2)
(m-4-8) edge node [below, sloped]{$\sim$}(m-5-8)
(m-3-5) edge node [below, sloped]{$\sim$}(m-4-5)
(m-5-2) edge node [below, sloped]{$\sim$}(m-6-2)
(m-6-3) edge node [below, sloped]{$\sim$}(m-7-3)
(m-5-4) edge node [below, sloped]{$\sim$}(m-6-4)
(m-5-6) edge node [below, sloped]{$\sim$}(m-6-6)
(m-6-7) edge node [below, sloped]{$\sim$}(m-7-7)
(m-5-8) edge node [below, sloped]{$\sim$}(m-6-8)
(m-8-1) edge node [above, sloped]{$\sim$}(m-8-2)
(m-8-2) edge node [above, sloped]{$\sim$}(m-8-3)
(m-8-4) edge node [above, sloped]{$\sim$}(m-8-3)
(m-7-3) edge node [above, sloped]{$\sim$}(m-8-4)
(m-6-4) edge node [above, sloped]{$\sim$}(m-7-5)
(m-6-6) edge node [above, sloped]{$\sim$}(m-7-5)
(m-7-7) edge node [above, sloped]{$\sim$}(m-8-6)
(m-8-6) edge node [above, sloped]{$\sim$}(m-8-7)
(m-8-8) edge node [above, sloped]{$\sim$}(m-8-7)
(m-8-9) edge node [above, sloped]{$\sim$}(m-8-8);
\end{tikzpicture}

  \end{center}
\caption[Independence of most general choices made in the perturbation theory.]{Independence of most general choices made in the perturbation theory.}
\label{full_choice_ind}
\end{figure}
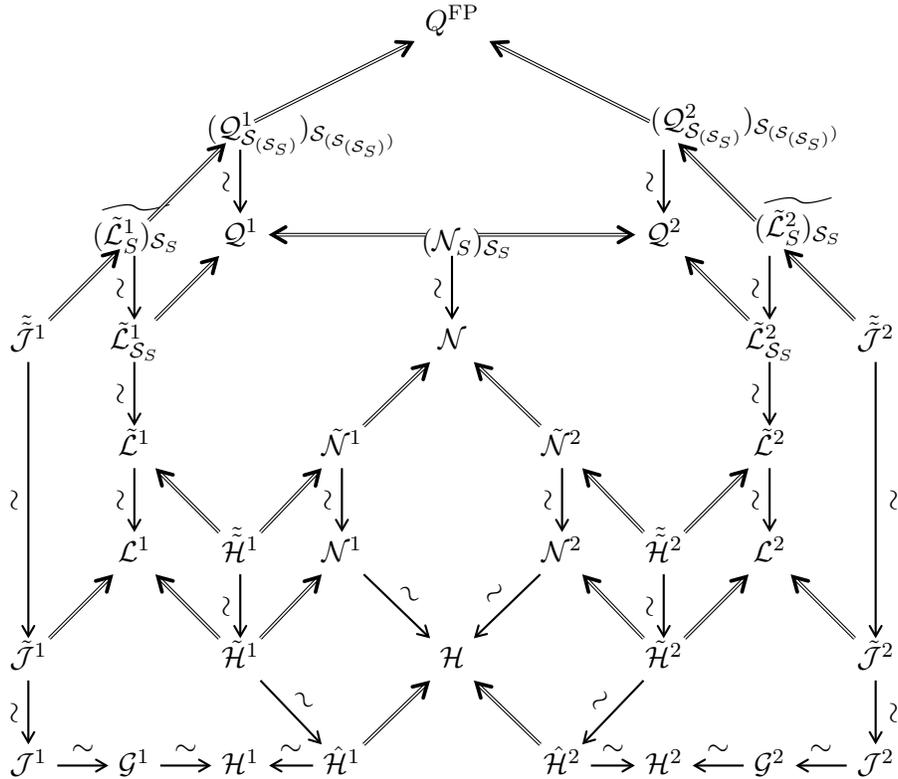

By theorem \ref{PRELIMTHM}.(2), $\mathcal{H}^1$ and $\mathcal{H}^2$ admit a common strong level-1 refinement $\mathcal{H}$.

Since both $\mathcal{J}^i$ and $\hat{\mathcal{H}}^i$ are level-1 chart-refinement of $\mathcal{H}^i$, they have a common strong level-1 refinement $\mathcal{L}^i$, by theorem \ref{PRELIMTHM}.(1).

We can lift a concerted level-1 Kuranishi embedding $\hat{\mathcal{H}}^i\Rightarrow \mathcal{H}$ through a chart-refinement of the domain to a concerted level-1 Kuranishi embedding $\tilde {\mathcal{H}}^i\Rightarrow \mathcal{N}^i$.

Since both $\mathcal{N}^1$ and $\mathcal{N}^2$ are common chart-refinement of $\mathcal{H}$, we can find a common strong level-1 refinement $\mathcal{N}$, by theorem \ref{PRELIMTHM}.(1).

We can lift the concerted level-1 Kuranishi embedding $$\tilde{\mathcal{H}}^i\Rightarrow \mathcal{N}^i$$ through a chart-refinement of the target to a concerted level-1 Kuranishi embedding $\tilde{\tilde {\mathcal{H}}}^i\Rightarrow \tilde{\mathcal{N}}^i$.

We can lift the level-1 Kuranishi embedding $$\tilde{\mathcal{H}}^i\Rightarrow \mathcal{L}^i$$ through a chart-refinement of the domain to a level-1 Kuranishi embedding $\tilde{\tilde {\mathcal{H}}}^i\Rightarrow \tilde{\mathcal{L}}^i$.

We now have a pair of concerted level-1 Kuranishi embeddings $\tilde{\tilde{H}}^i\Rightarrow\tilde{L}^i$ and $\tilde{\tilde{H}}^i\Rightarrow\mathcal{N}$. Observe, we perform tripling once to get the admissibility, and when form a fiber product $\mathcal{Q}^i$ of the admissible pair of concerted level-1 Kuranishi embeddings between the triplings (which level-1 chart-refines the original pair of concerted level-1 Kuranishi embeddings), we do a tripling again, hence the notation $\mathcal{Q}^i:=((\tilde{L}^i_S)_{\mathcal{S}_S})\times_{\mathcal{B}((\tilde{\tilde{\mathcal{H}}}^i_S)_{\mathcal{S}_S})}\times((\mathcal{N}_S)_{\mathcal{S}_S})$ indexed by $\mathcal{S}_{(\mathcal{S}_S)}$. Notice we can manage to use the same tripling $(\mathcal{N}_S)_{\mathcal{S}_S}$ for both fiber products. Indeed after picking $(U^j_{(\beta,\alpha)})^i, j=1,2,3$ for $\tilde{\tilde{\mathcal{H}}}^i, i=1,2$, we use $$\Pi^i_\beta(\bigcap_{k=1}^2\Pi^{-1}_\beta((U^1_{(\beta,\alpha)})^k)), \Pi^i_\alpha(\bigcap_{k=1}^2\Pi^{-1}_\alpha((U^2_{(\beta,\alpha)})^k)), \Pi^i_\alpha(\bigcap_{k=1}^2\Pi^{-1}_\alpha((U^3_{(\beta,\alpha)})^k))$$ instead to initiate the tripling. We do this again for the second tripling.

We have a pair of concerted level-1 Kuranishi embeddings $(\mathcal{N}_S)_{\mathcal{S}_S}\Rightarrow \mathcal{Q}^i$ commonly indexed by $\mathcal{S}_{(\mathcal{S}_S)}$ now. Again, we can form a fiber product $$\mathcal{Q}^{\text{FP}}:=(\mathcal{Q}^1_{\mathcal{S}_{(\mathcal{S}_S)}})_{\mathcal{S}_{(\mathcal{S}_{(\mathcal{S}_S)})}}\times_{\mathcal{B}((\mathcal{N}_{\mathcal{S}_{(\mathcal{S}_S)}})_{\mathcal{S}_{(\mathcal{S}_{(\mathcal{S}_S)})}})}(\mathcal{Q}^2_{\mathcal{S}_{(\mathcal{S}_S)}})_{\mathcal{S}_{(\mathcal{S}_{(\mathcal{S}_S)})}}$$ of the admissible pair of concerted level-1 Kuranishi embeddings between triplings commonly indexed by $\mathcal{S}_{(\mathcal{S}_{(\mathcal{S}_S)})}$ (which level-1 chart-refines the original pair of level-1 Kuranishi embeddings).

We can lift the concerted level-1 Kuranishi embedding $$(\tilde{\mathcal{L}}^i_S)_{\mathcal{S}_S}\Rightarrow\mathcal{Q}^i$$ through a chart-refinement of the target to a concerted level-1 Kuranishi embedding $\widetilde{(\tilde{\mathcal{L}}^i_S)_{\mathcal{S}_S}}\Rightarrow (\mathcal{Q}^i_{\mathcal{S}_{(\mathcal{S}_S)}})_{\mathcal{S}_{(\mathcal{S}_{(\mathcal{S}_S)})}}$.

We can lift the concerted level-1 Kuranishi embedding $$\tilde{\mathcal{J}}^i\Rightarrow\mathcal{L}^i$$ through a chart-refinement of the target to a concerted level-1 Kuranishi embedding $\tilde{\tilde{\mathcal{J}}}^i\Rightarrow \widetilde{(\tilde{\mathcal{L}}^i_S)_{\mathcal{S}_S}}$.

Now, if we compose the outmost edges of the diagram, we have for $i=1, 2$, $$\mathcal{J}^i\underset{\text{level-1}}{\overset{\sim}{\leftarrow}}\tilde{\tilde{\mathcal{J}}}^i\underset{\text{level-1}}{\Rightarrow}\mathcal{Q}^{\text{FP}}.$$

Therefore, $\mathcal{Q}^{\text{FP}}$ is the common strong level-1 refinement where compact invariant transverse perturbations on $\mathcal{J}^i$ can be modified in cobordant ways and lifted through level-1 refinements to be compared on $\mathcal{Q}^{\text{FP}}$, using theorem \ref{PRELIMTHM}.(3).

Therefore, the perturbation theory factors through the equivalence of Kuranishi structures and is independent of all the choices made, up to several cobordisms.
\end{proof}

\section{Subtleties in constructing a Kuranishi theory}

Having introduced all the ingredients of a choice-independent Kuranishi structure theory, we now outline some of the subtleties that have been encountered and dealt with:

\begin{enumerate}[(A)]

\item In a minimal formulation of Kuranishi structures, the tangent bundle condition is most naturally specified for zeros; however, extensions to higher charts are chosen after perturbing sections from lower charts, and the tangent bundle condition for perturbed sections may fail; given the constraint that we do not want to modify perturbed sections of lower charts, compact perturbation can thus be problematic and messy. This is one of many reasons why we introduced the level-1 structures.

We illustrate this issue with an example of a coordinate change; we can also think of this as a local part of $X$ and can add other charts to make $X$ compact. Let $$E_y=(\R\times (-5,1))\times \R\text{ and }E_x=(\R\times (-1,5)\times (-R,R)\times \R^2,$$ and let $\phi_{xy}(z_1, z_2)=(z_1, z_2, 0)$ with the domain of coordinate change $\R\times (-1, 1)$. Let $s_y(z_1, z_2)=(z_1, z_2, \eta (z_1) )$, where $\eta$ has zeros in $(-\epsilon, \epsilon)$, and let $s_x(z_1, z_2, z_3)=(z_1, z_2, z_3, \eta(z_1), \tau (z_1, z_3))$. We can define $\tau(z_1, \cdot)$ to be a cubic polynomial parameterized by $z_1$ that always has a zero at $z_3=0$, with $(\partial_{z_3}\tau)(z_1, 0)>0$ for $z_1\in(-\epsilon,\epsilon)$; and when $z_1$ increases from $2$ to $4$, say, two more zeros of $\tau(z_1,\cdot)$ gradually start to appear at $z_1=2$ and disappear at $z_1=4$; and at $z_1=3$ one of these newly-born zeros coincides with the existing zero and the other zero has $z_1$ larger than $R$. We can add one more chart to make $X$ compact. Suppose we perturb $\eta$ to have a single transverse zero at $3$, which is a perfectly legitimate transverse perturbation for $s_y$. Suppose that we also extend this perturbation by pulling back along the $z_3$ coordinate near $z_3=0$; in order to keep the perturbed $s_y$ intact, we have to perturb over $V_x\backslash \phi_{xy}(V_{xy})$ to make this extension transverse. A generic perturbation with the constraint that the existing choice must be kept fixed will create a noncompact component of $X$ that cannot be cancelled out. The reason is that here we have to perturb the section on $V_x\backslash\phi_{xy}(V_{xy})$, which is not compact. Observe that since the tangent bundle condition is specified at zeros, the tangent bundle condition will not be available for use without due care. This example can also be modified into an example showing that the tangent bundle condition is crucial.

This issue is solved by a level-1 coordinate change that makes full use of the tangent bundle condition and addresses the above compactness issues (using the fact that the part of $X$ covered by $V_{pq}$ is open in $X$).  

\item A coordinate change moving up in dimensions will create a quotient topology that is not locally compact, so there is a need to be careful to obtain a compact perturbation when moving up in dimensions.

The quotient topology of two charts identified by a coordinate change strictly increasing in dimension is not locally compact. Let us assume $V_p$ is a precompact shrinking of a bigger $\tilde V_p$, so we can talk about the boundary $\partial V_p:=\overline{V_p}\backslash V_p$, with the closure $\overline{V_p}$ taken in $\tilde V_p$. An immersed invariant tubular neighborhood $W_{pq}$ of $\phi_{pq}(V_{pq})$ (introducing no additional zeros) becomes increasingly narrow towards $\partial V_p$, where it is embedded. One intuition is to shrink $V_p$ to $V'_p$ in the direction normal to the boundary so that the resulting tubular neighborhood $W'_{pq}$ does not shrink to $\phi_{pq}(V_{pq})$ as it approaches $\partial V_p$; and if there is a way to ensure the extended perturbation of $(\phi_{pq})_\ast (s_q|_{V_{pq}})$ in $W'_{pq}$ is already transverse, we can use a partition of unity to glue it with a compact perturbation in $V'_p$ supported away from $V'_p\backslash \phi_{pq}(V'_{pq}$). This intuition is justified because it is impossible that $z_n\in X$ coming from $V_p\backslash W_{pq}$ converges to $z\in X_q|_{V_q}\cap X_p|_{\partial V_p}$, where the convergence is defined using the topology of $X$. For if this were possible, then we would have a contradiction because $V_{pq}$ induces an open set in $X$ by the definition of a Kuranishi structure (indeed, $V_{pq}$ is open in $V_q$ and $V_q$ induces an open set in $X$), and this limiting point $z$ lies in the open set in $X$ induced by $V_{pq}$, equivalently $\phi_{pq}(V_{pq})$, so eventually the limiting sequence will lie inside the open set induced by $\phi_{pq}(V_{pq})$, which is a contradiction. With this understood, we do not actually need to shrink $V_p$, and as there is no limiting zero to the boundary, we can always perturb the section over $V_p$ in a compact set away from places responsible for a non-locally-compact topology. This is a consequence of the coverage on $X$ by a Kuranishi chart being an open set in $X$ and a way of dealing with item (A) so that the extended perturbation is already transverse (for example, the existence of the level-1 structure for a coordinate change).

\item In order to analyze the various choices involved and to show that Kuranishi structure equivalence is an equivalence relation, we need to form a fiber product whose natural definition requires additional structures. This requirement is met by an admissible pair of concerted level-1 Kuranishi embeddings and the tripling process.

Let us illustrate with a simple example of a good coordinate system consisting of two charts, which illustrates the issue where dimension becomes an obstacle and why the tripling is necessary. Suppose that we have an admissible pair of concerted level-1 Kuranishi embeddings $C_y|_{U_y}\overset{\text{level-1}}{\to} C_x|_{U_x}$ to $C^i_{y}|_{U^i_y}\overset{\text{level-1}}{\to} C^i_x|_{U^i_x}$  for $i=1, 2$. The most natural candidate for a fiber product will not work in general, because, for example, if $\text{dim} E^i_y=\text{dim} E^i_x=\text{dim} E_x=6$ and $\text{dim} E_y=2$, we cannot have a coordinate change from the fiber product chart of dimension 10 indexed by $y$ to the fiber product chart of dimension 6 indexed by $x$.

\item Because we often shrink charts in our theory, it is essential to be able to determine the domains of the coordinate changes from the charts alone. This is achieved by the maximality condition, which rules out pathological examples that cause the incompatibility when extending perturbations. We discussed them before and recall two situations here:

One possible failure of the maximality is that two regions in $U_x$ are identified in higher charts, while the $G_x$-action does not map one to the other. By the definition of a Kuranishi chart, the quotient of the zero set maps injectively to $X$. Thus, the two regions have to be away from the zero set, but it is possible that the two regions still exist in any neighborhood of the zero set in $U_x$. Then if we perturb inductively by order, the two regions will in general have nonisomorphic perturbations chosen in the earlier stages of induction, and this will make it impossible to extend to a higher chart as they are identified there.

We could also have $C_y\to C_x$ with $U_{xy}$ not being as large as it should be, so that it is possible that although a part of $U_y\backslash U_{xy}$ is not identified with a part of $U_x\backslash\phi_{xy}(U_{xy})$ by this direct coordinate change, they are identified via some coordinate changes into higher-order charts (those parts need not be away from zeros like in the scenario of the previous paragraph). Then perturbations for $C_y$ and $C_x$ that are chosen to be compatible over $U_{xy}$ might not be compatible in the higher-order chart.

\item We only know that $X$ is Hausdorff, but the Hausdorffness of the space formed by gluing together charts in a good coordinate system is not apriori true with respect to any reasonable topology, and we need to have the Hausdorffness in this glued space to do inductive constructions and gluing (using a partition of unity). We also need a better control of coordinate changes in a good coordinate system, namely, we want to have that a point of $\underline{U'_x}$ and a point of $\underline{U'_y}$ are identified if and only if there exists a coordinate change between chart indexed by $x$ and $y$ in the direction determined by the order $\leq$. Namely, the control is upgraded from the coverages on $X$ by the charts to the (quotiented) bases of the charts. This is the strongly intersecting property.

The topological matching condition in the definition of a Kuranishi structure $\mathcal{K}$ allows a strongly intersecting shrinking $\mathcal{G}'$ for any good coordinate system $\mathcal{G}$ obtained from $\mathcal{K}$, and after defining the relative topology, to obtain a strongly intersecting Hausdorff shrinking from $\mathcal{G}'$.

\item To have a naturally defined and useful definition of the germ of Kuranishi structures for our purpose, we need to consider Kuranishi structure equivalence, where we allow the shrinking of charts and coherently replacing charts by possibly higher dimensional charts. This can be used to show the choice-independence throughout each stage of the theory (the maximality condition is needed to show the existence of a well-defined good coordinate system from a given choice). We also need to show that this is indeed an equivalence and that perturbation factors through this equivalence. This is established by using fiber product constructions in several stages.

\item For a given dimension, we need a recipe to glue perturbations extended from lower dimensional charts and perturbations chosen elsewhere.

\item The level-1 structure chosen on top of the minimal definition of a Kuranishi structure to carry out the above always exists based on the minimal definition and this is essential for the theory to be useful in general. A part but not the whole structure of the level-1 structure is actually naturally given if the Kuranishi structure from which the good coordinate system is obtained is from a \emph{polyfold Fredholm structure} via the \emph{forgetful construction}, as discussed in the next installment \cite{DII} of this series (after the forgetful construction is established). However, even if we restricting to those Kuranishi structures, constructing the remaining data for the level-1 structure is not trivial and still requires the vast majority of the ideas and methods of this paper.

\end{enumerate}

\newpage

\bibliographystyle{amsplain}

\end{document}